\PassOptionsToPackage{usenames,dvipsnames}{xcolor}
\documentclass[11pt]{article} 
\usepackage[a4paper, total={6in, 9in}]{geometry}
\usepackage{amssymb}
\usepackage[toc,page]{appendix}
\usepackage{amsmath}
\DeclareFontFamily{U}{stix2bb}{\skewchar\font127 }
\DeclareFontShape{U}{stix2bb}{m}{n} {<-> stix2-mathbb}{}
\DeclareMathAlphabet{\mathbb}{U}{stix2bb}{m}{n}
\usepackage{setspace}
\usepackage{slashed} 
\usepackage{eqparbox}
\newcommand{\eqmathbox}[2][N]{\eqmakebox[#1]{$\displaystyle#2$}}
\usepackage{enumerate}
\usepackage{enumitem}
\setlist[enumerate]{wide=16pt, label=\arabic*)}
\setlist[itemize]{wide=16pt}
\usepackage{changepage}
\usepackage{todonotes}
\usepackage{amsfonts}
\usepackage{amssymb}
\usepackage{tikz-cd}
\usepackage[perpage]{footmisc}

\usetikzlibrary{decorations.markings, arrows.meta}
\tikzset{
    marrow/.style={decoration={markings,mark=at position 0.5 with {\arrow{#1}}}, postaction=decorate}
}

\usetikzlibrary{matrix,arrows}
\usepackage{adjustbox}
\usepackage{mathrsfs}
\usepackage{relsize} 
\usepackage{amsbsy}
\usepackage{stmaryrd}
\usepackage{xcolor}
\usepackage{extarrows}
\usepackage{varwidth}
\usepackage{array}
\usepackage{tikz}
\usepackage{tikz-cd}
\usetikzlibrary{matrix}
\tikzcdset{cramped}
\usepackage{gensymb}
\usepackage{bm}
\usepackage{graphicx}
\usepackage{tensor}
\usepackage[colorlinks = true,
            linkcolor = black,
            urlcolor  = blue,
            citecolor = red,
            anchorcolor = green,
            hidelinks]{hyperref}
\hypersetup{%
    pdfborder = {0 0 0}
}
\usepackage{mathtools}
\usepackage{braket}
\usepackage{bbm} 
\usepackage[skip=8pt plus2pt, indent=10pt]{parskip}
\usepackage{amsthm}
\usepackage{etoolbox}

\makeatletter

\patchcmd\deferred@thm@head
  {\addvspace{-\parskip}}
  {}
  {}{\typeout{\string\deferred@thm@head patch failed!}}

\makeatletter 

\usepackage{url} 

\theoremstyle{plain}
\newtheorem{theorem}{Theorem}[section]
\newtheorem{lemma}[theorem]{Lemma}
\newtheorem{claim}[theorem]{Claim}
\newtheorem{proposition}[theorem]{Proposition}
\newtheorem{corollary}[theorem]{Corollary}

\theoremstyle{definition}
\newtheorem{definition}[theorem]{Definition}
\newtheorem{example}[theorem]{Example}
\newtheorem{remark}[theorem]{Remark}
\newtheorem{importantremark}[theorem]{Important remark!}

\newtheorem{prop/def}[theorem]{Proposition/Definition}
\newtheorem{assumption}[theorem]{Assumption}

\numberwithin{figure}{section}
\numberwithin{table}{section}
\numberwithin{equation}{section}

\makeatletter
\newcommand*{\transpose}{%
  {\mathpalette\@transpose{}}%
}
\newcommand*{\@transpose}[2]{%
  \raisebox{\depth}{$\m@th#1\intercal$}%
}

\usepackage{stackengine,xcolor}
\input pdf-trans
\newbox\qbox
\def\usecolor#1{\csname\string\color@#1\endcsname\space}
\newcommand\bordercolor[1]{\colsplit{1}{#1}}
\newcommand\fillcolor[1]{\colsplit{0}{#1}}
\newcommand\colsplit[2]{\colorlet{tmpcolor}{#2}\edef\tmp{\usecolor{tmpcolor}}%
  \def\tmpB{}\expandafter\colsplithelp\tmp\relax%
  \ifnum0=#1\relax\edef\fillcol{\tmpB}\else\edef\bordercol{\tmpC}\fi}
\def\colsplithelp#1#2 #3\relax{%
  \edef\tmpB{\tmpB#1#2 }%
  \ifnum `#1>`9\relax\def\tmpC{#3}\else\colsplithelp#3\relax\fi
}
\newcommand\outline[1]{\leavevmode%
  \def\maltext{#1}%
  \setbox\qbox=\hbox{\maltext}%
  \boxgs{Q q 2 Tr \thickness\space w \fillcol\space \bordercol\space}{}%
  \copy\qbox%
}
\newcommand\mathcalbb[2][1]{\outline{$\mathcal{#2}$}}
\bordercolor{black}
\fillcolor{white}
\def\thickness{.1}

\usepackage[utf8]{inputenc}
\newcommand*\dg{\textnormal{dg}}
\newcommand*{\gr}{\textnormal{gr}}
\newcommand*\al{\alpha}
\newcommand*\be{\beta}
\newcommand*\ga{\gamma}

\newcommand*\bph{\overline{\phi}}

\newcommand*\lam{\lambda}
\newcommand*\per{\textnormal{per}}
\newcommand*\BM{\textnormal{BM}}
\newcommand*\red{\textnormal{red}}
\newcommand*\lpe{\textnormal{lper}}
\newcommand*\sig{\sigma}

\makeatother

\newcommand*\At{\textnormal{At}}

\newcommand*\cs{\textnormal{cs}}

\newcommand*\Fib{\textnormal{Fib}}

\newcommand\Fun{\textnormal{Fun}}
\newcommand\Diag{\textnormal{Diag}}

\newcommand\cyc{\textnormal{cyc}}
\newcommand\op{\textnormal{op}}

\newcommand*\rig{\textnormal{rig}}

\newcommand*{\rk}{\textnormal{rk}}
\newcommand*{\Rk}{\textnormal{Rk}}
\newcommand*{\Ho}{\textnormal{Ho}}

\newcommand*\Ext{\textnormal{Ext}}
\newcommand*\Edg{\textnormal{Edg}}
\newcommand*\bEdg{\overline{\textnormal{E}}_{\dg}}
\newcommand*\wEdg{\wt{\textnormal{E}}_{\dg}}
\newcommand*\mExt{\mathcal{E}\textnormal{xt}}

\newcommand*\bM{\boldsymbol{\mathcal{M}}}
\newcommand*\bN{\boldsymbol{N}}
\newcommand*\FFna{\wt{\mathbb{F}}}
\newcommand*\Rep{\textnormal{Rep}}
\newcommand*\Sh{\textnormal{Sh}}
\newcommand*\Ver{\textnormal{Ver}}
\newcommand*\Tot{\textnormal{Tot}}

\newcommand*\Art{\mathscr{A}\!\textit{rt}}

\newcommand{\coker}{\textnormal{coker}}

\newcommand*\Perf{\textnormal{Perf}}
\newcommand*\PPerf{\mathscr{P}\!\textit{erf}}

\newcommand{\ddr}{\textnormal{d}_{\textnormal{dR}}}

\newcommand*\Flag{\textnormal{Fl}}
\newcommand*\iFlag{\!\textit{Fl}}
\newcommand*\MS{\textnormal{MS}}
\newcommand*\Hilb{\textnormal{Hilb}}

\newcommand*\SSym{\textnormal{SSym}}

\newcommand*\Vect{\textnormal{Vec}}
\newcommand*\Spec{\textnormal{Spec}}
\newcommand*\bSpec{\textbf{Spec}}
\newcommand*\ch{\textnormal{ch}}
\newcommand*\GL{\textnormal{GL}}

\newcommand*\vdim{\textnormal{vdim}}
\newcommand\Hom{\textnormal{Hom}}
\newcommand\mHom{\mathcal{H}\textnormal{om}}

\newcommand{\msC}{{\mathscr C}}

\newcommand{\msE}{{\mathscr E}}

\newcommand{\msN}{{\mathscr N}}

\newcommand{\msS}{{\mathscr S}}

\newcommand{\Ff}{{\mathfrak{f}}}

\newcommand{\Fl}{{\mathfrak{l}}}
\newcommand{\Fm}{{\mathfrak{m}}}

\newcommand{\Ft}{{\mathfrak{t}}}

\newcommand{\Fv}{{\mathfrak{v}}}

\newcommand*\chart{\textnormal{char}}
\newcommand*\bTot{\textbf{Tot}}
\newcommand*\bmM{\boldsymbol{\mathcal{M}}}
\newcommand*\bmN{\boldsymbol{\mathcal{N}}}
\newcommand*\Map{\textnormal{Map}}
\newcommand*\End{\textnormal{End}}

\newcommand*\pt{\textnormal{pt}}
\newcommand*\BB{\mathbb{B}}
\newcommand*\CC{\mathbb{C}}

\newcommand*\EE{\mathbb{E}}
\newcommand*\FF{\mathbb{F}}
\newcommand*\GG{\mathbb{G}}
\newcommand*\HH{\mathbb{H}}

\newcommand*\kk{\mathbb{k}}
\newcommand*\LL{\mathbb{L}}
\newcommand*\MM{\mathbb{M}}
\newcommand*\NN{\mathbb{N}}

\newcommand*\PP{\mathbb{P}}
\newcommand*\QQ{\mathbb{Q}}
\newcommand*\RR{\mathbb{R}}
\newcommand*\TT{\mathbb{T}}

\newcommand*\VV{\mathbb{V}}
\newcommand*\WW{\mathbb{W}}

\newcommand*\ZZ{\mathbb{Z}}

\newcommand{\mA}{{\mathcal A}}
\newcommand{\mB}{{\mathcal B}}
\newcommand{\mC}{{\mathcal C}}
\newcommand{\mD}{{\mathcal D}}
\newcommand{\mE}{{\mathcal E}}
\newcommand{\mF}{{\mathcal F}}

\newcommand{\mH}{{\mathcal H}}
\newcommand{\mI}{{\mathcal I}}

\newcommand{\mK}{{\mathcal K}}
\newcommand{\mL}{{\mathcal L}}
\newcommand{\mM}{{\mathcal M}}
\newcommand{\mN}{{\mathcal N}}
\newcommand{\mO}{{\mathcal O}}
\newcommand{\mP}{{\mathcal P}}

\newcommand{\mS}{{\mathcal S}}
\newcommand{\mT}{{\mathcal T}}

\newcommand{\mV}{{\mathcal V}}
\newcommand{\mW}{{\mathcal W}}
\newcommand{\mX}{{\mathcal X}}
\newcommand{\mY}{{\mathcal Y}}
\newcommand{\mZ}{{\mathcal Z}}

\newcommand{\mbB}{\mathcalbb{B}}
\newcommand{\mbN}{\mathcalbb{N}}
\newcommand{\T}{\mathsf{T}}
\newcommand{\V}{\mathsf{V}}
\newcommand{\sv}{\mathsf{v}}
\renewcommand{\S}{\mathsf{S}}

\newcommand*\Coh{\textnormal{Coh}}
\newcommand*\blangle{\big\langle}
\newcommand*\brangle{\big\rangle}

\newcommand*\ar{\textnormal{arg}}

\newcommand*{\pe}{\textnormal{per}}
\newcommand*{\Ker}{\textnormal{Ker}}
\newcommand*{\loc}{\textnormal{loc}}
\newcommand*{\vir}{\textnormal{vir}}
\newcommand*{\svir}{\textnormal{svir}}
\newcommand*{\id}{\textnormal{id}}

\newcommand{\OT}{\textnormal{OT}}

\newcommand{\QCoh}{\textnormal{QCoh}}

\newcommand{\RHom}{\textnormal{RHom}}
\newcommand*{\ttop}{\textnormal{top}}

\newcommand{\PT}{\textnormal{PT}}

\newcommand{\Ksst}{K^0_{\textnormal{sst}}}

\newcommand{\Masi}{\mathcal{M}^{\sigma}_{\alpha}}
\newcommand{\Masio}{\mathcal{M}^{\sigma}_{\alpha_1}}
\newcommand{\Masit}{\mathcal{M}^{\sigma}_{\alpha_2}}

\newcommand{\Masvir}{\big[M^{\sigma}_{\alpha}\big]^{\vir}}
\newcommand*{\Masig}{\big\langle \mathcal{M}^{\sigma}_{\alpha}\big\rangle}

\newcommand*{\JS}{\textnormal{JS}}

 \newcommand*{\nv}{\textsf{v}}
 \DeclareMathOperator{\tr}{tr}
  \DeclareMathOperator{\Tr}{Tr}

\newcommand\numberthis{\addtocounter{equation}{1}\tag{\theequation}}

\newcommand\Item[1][]{%
  \ifx\relax#1\relax  \item \else \item[#1] \fi
  \abovedisplayskip=0pt\abovedisplayshortskip=0pt~\vspace*{-\baselineskip}}
  \renewcommand{\text}[1]{\textnormal{#1}}

\newcommand{\un}[1]{\ensuremath{\underline{#1}}}
\newcommand{\wt}[1]{\ensuremath{\widetilde{#1}}}
\newcommand{\bs}[1]{\ensuremath{\boldsymbol{#1}}}

\newcommand{\ov}[1]{\overline{#1}}

\newcommand{\B}[1]{{\color{Blue} #1}}
\newcommand{\RX}[1]{{\color{Red} #1}}
\newcommand{\BuOr}[1]{{\color{BurntOrange} #1}}
\newcommand{\Cya}[1]{{\color{Cyan} #1}}
\newcommand{\FG}[1]{{\color{ForestGreen} #1}}
\newcommand{\Gre}[1]{{\color{Green} #1}}
\newcommand{\Ma}[1]{{\color{Maroon} #1}}
\newcommand{\Pur}[1]{{\color{Purple} #1}}

\newcommand{\nocontentsline}[3]{}
\newcommand{\tocless}[1]{\addtocontents{toc}{\protect\setcounter{tocdepth}{1}}
  #1
  \addtocontents{toc}{\protect\setcounter{tocdepth}{2}}
}

\title{Wall-crossing for Calabi--Yau fourfolds: framework, tools, and applications}
\author{Arkadij Bojko\thanks{Shanghai Institute for Mathematics
and Interdisciplinary Sciences, Fudan University; Block A, International Innovation Plaza, No. 657 Songhu Road, Yangpu District, Shanghai, China;
E-mail: abojko@simis.cn
}}
\date{}
\begin{document}
\maketitle
\begin{abstract}
This work develops new ideas and tools to establish equivariant wall-crossing in Calabi--Yau four categories. In the process, I introduce several necessary formalisms, including an equivariant deformation of Joyce's vertex algebras built on a novel and explicitly computable definition of equivariant homology for stacks. The proof of the wall-crossing formula is then given for Calabi--Yau four quivers and local CY fourfolds. A crucial part of the problem is showing that the generalized invariants counting stable objects are well-defined. Using a conceptual argument akin to the quantum Lefschetz principle, I show that for torsion-free sheaves, this follows from the wall-crossing formula for Joyce--Song stable pairs.  In joint work with Kuhn--Liu--Thimm, these frameworks together with functoriality established here for Park's virtual pullback diagrams are used to prove the general CY4 wall-crossing conjecture.

\end{abstract}
\tableofcontents
\section{Introduction}
Invariants associated to moduli commonly depend on a choice of a parameter that I will henceforth call the \textit{stability condition}. When stability conditions form manifolds, the moduli may change when entering or crossing real codimension one submanifolds called \textit{walls}. The resulting phenomenon called \textit{wall-crossing} can be studied on two levels: the change of the invariants and the change of the moduli themselves. Here, I will focus on the former.

 The focal point of this work is the equivariant counting of objects in Calabi--Yau four (CY4) categories. An example of such a category is given by the coherent sheaves $\Coh(X)$ on a Calabi--Yau fourfold -- a smooth quasi-projective variety of dimension $4$ with a fixed trivialization of its canonical bundle $K_X\cong \mO_X$ -- with an action of a torus $\T$ preserving its holomorphic volume form. Wall-crossing for moduli of sheaves is a deeply researched subject in lower dimensions   with \cite{Tha96, mochizuki, JoyceSong, KS, GJT, JoyceWC} among the pivotal contributions.
 
Counting sheaves on Calabi--Yau fourfolds has its roots in \cite{DT}, while \cite{BJ,OT} laid the formal foundations of the subject\footnote{Some examples were already studied in \cite{CL}.}. Wall-crossing of these invariants has been formulated only conjecturally in \cite[§4.4]{GJT} (see also \cite[§2.5]{bojko2}). Even so, it has already seen its application in \cite{bojko2, bojko3}. This work is the first contribution that proves and refines Joyce's wall-crossing conjecture for $\T$-equivariant CY4 theories in terms of $\T$-deformations of vertex algebras. Specifically, I establish $\T$-equivariant wall-crossing for CY4 dg-quivers and local CY fourfolds. At the Simons Collaboration on Special Holonomy Meeting in Janaury 2023, I announced a proof that I believed to also work for sheaves on CY fourfolds $X$ with $H^1(\mO_X)=0$. I have since then found a gap in my argument explained in §\ref{sec:introproofWC} and §\ref{sec:sheaves}. An approach to work around this issue has now appeared in a follow-up joint work \cite{BKLT}. 

 Some applications of the results proved in the current work are summarized in §\ref{sec:introquivloc}. Additionally, multiple ongoing projects are using the equivariant CY4 quiver wall-crossing to address existing conjectures and to study the structure of equivariant descendent integrals -- see §\ref{sec:futurework}. The general result, obtained in \cite{BKLT} and building on the present work, will be used to prove many existing conjectures about curve and surface counting invariants in \cite{CK3, CT1, CT2, CT2-1, CT3, BKP2}. Additionally, all results in \cite{bojko2, bojko3} can be established unconditionally as theorems thanks to \cite{BKLT}.  A more detailed discussion of future applications is given in §\ref{sec:futurework}. 

\tocless{\subsection{Equivariant counting of objects in CY4 categories}}

Let $\mA$ be a CY4 abelian category (in the sense explained in §\ref{sec:dgsetup}) with a possible action of a torus $\T$. When defining invariants from moduli problems, one fixes two pieces of data first:
\begin{enumerate}
\item For a quotient of $K^0(\mA)\twoheadrightarrow \ov{K}(\mA)$ of the Grothendieck group of $\mA$, choose a class $\al\in \ov{K}(\mA)$.
\item Determine a (weak) stability condition $\sig$ in the sense of §\ref{sec:assstability}.
\end{enumerate}
Let $M^{\sig}_{\al}$ be the moduli space parametrizing $\sig$-stable objects in the class $\al$. Relying on the presence of $-2$-shifted symplectic structures, it was described in \cite{BJ} as a real derived manifold. This introduced the question of \textit{orientability}. Denote by $T^{\vir}_{M^{\sig}_{\al}}$ the natural virtual tangent bundle with Serre duality
$$
T^{\vir}_{M^{\sig}_{\al}}\cong \big(T^{\vir}_{M^{\sig}_{\al}}\big)^\vee[-2]
$$
chiinduced by the $-2$-shifted symplectic structure. A choice of orientation for the associated derived manifold is equivalent to a choice of a trivialization 
$$
\det\big(T^{\vir}_{M^{\sig}_{\al}}\big)\cong \un{\CC}
$$
compatible with Serre duality as made precise in Definition \ref{def:orientation}. 
\begin{example}
\leavevmode
\vspace{-4pt}
\begin{enumerate}
\item When $\mA$ is a heart in $D^b(X)$ for a CY fourfold $X$, the existence of orientations was reduced to a gauge-theoretic problem in \cite{CGJ, bojko}. The original proof of gauge-theoretic orientations in \cite{CGJ} was, however, flawed. A correction with an additional constraints on $H^3(X,\ZZ_2)$ appeared in \cite{JU} for projective $X$. For the quasi-projective case this imposes, as shown in  \cite{bojko, KT-Pardon}, the condition recalled in Example \ref{ex:sheavesandreps0}.
\item For representations of CY4 quivers, the existence of orientations is shown in full generality in Corollary \ref{cor:quiveror} and in \cite[§4.4.1]{Schmier}.
\end{enumerate}
\end{example}

Once orientations are determined, the machinery of \cite{BJ} produces a virtual fundamental class 
\begin{equation}
\label{eq:MasigvirBJ}
\big[M^{\sig}_{\al}\big]^{\vir}\in H_*\big(M^{\sig}_{\al},\ZZ\big)
\end{equation}
assuming properness of the moduli space. The contribution of each connected component of $M^{\sig}_{\al}$ to $\big[M^{\sig}_{\al}\big]^{\vir}$ changes its sign upon switching to the opposite orientation. The more algebro-geometric approach of \cite{OT} produces equivariant virtual fundamental cycles 
\begin{equation}
\label{eq:MasigvirOT}
\big[M^{\sig}_{\al}\big]^{\vir}\in A^{\T}_*\big(M^{\sig}_{\al},\ZZ[2^{-1}]\big)
\end{equation}
assuming that the \textit{(real) virtual dimension} given by the rank of $T^{\vir}_{M^{\sig}_{\al}}$ is even. The signs of \eqref{eq:MasigvirOT} are affected by orientations in the same way. When $M^{\sig}_{\al}$ is projective, the main result in \cite{OT2} shows that the cycle-class map takes \eqref{eq:MasigvirOT} to \eqref{eq:MasigvirBJ} in $H_*\big(M^{\sig}_{\al},\ZZ[2^{-1}]\big)$ . In the present work, I will always assume that the virtual dimension is even, and I will work with the classes in homology over $\ZZ[2^{-1}]$ or $\QQ$.

Park's virtual pullback \cite{Park} is a crucial tool for proving wall-crossing for the above virtual fundamental classes. It is based on Manolache's \cite{Manolache} and is used to relate virtual fundamental classes along quasi-smooth morphisms, as recalled in §\ref{sec:basics}. Briefly, it states that if $f:N\to M$ is a morphism between two CY4 moduli spaces that admits a perfect obstruction theory fitting into the diagram \eqref{eq:introdiagsym}, then
$$[N]^{\vir}=f^![M]^{\vir}\,.$$
Constructing Park's virtual pullback diagrams (Pvp diagrams for short) is one of the main technical hurdles in the present work. They can be recovered from a derived Lagrangian correspondence for $-2$-shifted derived stacks as explained in \cite[Lemma 4.1]{LagCor}, see also Remark \ref{rem:dagJScomp}. This description is not needed in the current work, so I only remark throughout whenever there ought to be such an underlying structure.
\tocless{\subsection{$\T$-deformations of vertex algebras}}

Consider the moduli stack $\mM_{\mA}$ of all objects in $\mA$. To formulate wall-crossing, Joyce \cite{Joycehall} has introduced vertex algebras on the homology $V_* = H_{*+\vdim}(\mM_{\mA},\QQ)$ shifted by the virtual dimension of each connected component of $\mM_{\mA}$. These, in turn, induce Lie brackets on the quotients
\begin{equation}
\label{eq:Lieintro}
L_* = V_{*+2}/TV_*
\end{equation}
where $T$ is the translation operator of $V_*$.

In the present work, I will discuss $\T$-equivariant refinements of his construction in terms of \textit{local} and \textit{global} $\T$\textit{-deformations of vertex algebras}. The benefit of working with both approaches is explained in Example \ref{ex:toric}. Additionally, it is clear that the local one is more explicit, and the equivariant structure is more straightforward.
\begin{enumerate}[wide, align=left]
\item[\textbf{Local approach}] Instead of working with the full stack, one replaces it by the $\T$-fixed point stack $\mM^{\T}_{\mA}$. For any ring $R$ and $\mathbb{k}$ its ring of fractions, let $R\llbracket \Ft\rrbracket$ be the ring of power series on $\Ft = \textnormal{Lie}(\T)$. I will also denote by $ \mathbb{k}\llparenthesis\Ft\rrparenthesis$ the result of inverting all homogeneous polynomials in $R\llbracket \Ft\rrbracket$. One then takes the localized $\T$-equivariant homology of $\mM^{\T}_{\mA}$ which is just
\begin{equation}
\label{eq:localequivarianthom}
V_{\loc,*-\vdim}:=H_*\big(\mM^{\T}_{\mA},R\big)\llparenthesis\Ft\rrparenthesis^{\gr}\,.
\end{equation}
As explained to me by Anton Mellit (see Example \ref{ex:equivariantpushforw}), without allowing power-series, equivariant pushforward would not be defined. Joyce's construction applied to this setting produces \textit{$\T$-deformations of vertex algebras} that will be defined in \cite{Bo26} (see also \cite[pp. 51-57]{BoHeidelberg} where I called them additive deformations). These explicitly contain negative powers of $(z+u)$ for $e^u$ an irreducible character of $\T$. To obtain wall-crossing formulae, one expands such expressions in $|z|>|u|$ using \eqref{eq:explicitexpand}. This produces deformations of vertex algebras in the sense of \cite[§5]{HaiLi} (see also \cite[p. 5]{BoMseminartalk} for a more down-to-earth formulation), which in turn give rise to deformations of the non-equivariant Lie algebras. 
\item[\textbf{Global approach}] For this setup, one requires a definition of equivariant (generalized) homology theory for Artin stacks. Due to Example \ref{ex:equivariantpushforw}, the operational K-homology of \cite{Liu} is not suitable for this purpose.\footnote{This is apparent from looking at \cite[Theorem 2.3.2]{Liu} which contradicts the observation following Example \ref{ex:equivariantpushforw}.} To remedy this, we introduce in \cite{BB1} a construction of equivariant homology theories from bivariant theories of \cite{FMP} for stacks. Using the 6-functor formalism as in \cite{KhanEH} produces such a bivariant theory. This is summarized in Appendix \ref{app:B}, where I reproduce all the necessary operations on $H^{\T}_*$ from the bivariance properties described in \cite[§2.2]{FMP}. Expanding in $|z|>|u|$ again, this approach leads to a more general version of $\T$-deformations of vertex algebras which will be explained in \cite{BB1}. To be suitable for wall-crossing, one still needs to localize the homology as before.
\end{enumerate}
In both situations, one obtains deformations of Lie algebras which are Lie algebras themselves. As they are associated to localized equivariant homologies, I will denote them by $L_{\loc,*}$. All statements made in the rest of the introduction apply to both approaches equivalently. 
\begin{remark}
\leavevmode
\vspace{-4pt}
    \begin{enumerate}
        \item The axioms of $\T$-deformations of vertex algebras, together with our current construction, will be used in future work with Emile Bouaziz, aiming towards equivariant Virasoro constraints.
        \item The expansion $|u|>|z|$ of $\T$-deformations of vertex algebras also has a geometric interpretation as will be explained in \cite{Bo26}. It appears when twisting the enumerative invariants one wall-crosses with by insertions in the form of Chern classes. It will be used in the proof of the $\PT^{0}/\PT^{1}$ correspondence for elliptic fibrations conjectured in \cite{BKP3}.
        \item As shown in Example \ref{ex:trivialThom}.2), the equivariant homology theory from Appendix \ref{app:B} is quite computable. An ongoing project aims to leverage this to, for example, compute equivariant descendent integrals for Hilbert schemes on $\CC^4$. The final goal is to describe an explicit wall-crossing relation for the qq-characters of the equivariant DT and PT vertex as defined in \cite{KimGo, KimGo2}. Note that the authors also provide a quiver description of the corresponding invariants, thereby bringing these problems within the scope of the present theorems. 
    \end{enumerate}
\end{remark}

\tocless{\subsection{Equivariant wall-crossing}\label{sec:introWCeqhom}}
Whenever $M^{\sig}_{\al}$ is an open substack of the rigidification $\mM^{\rig}_{\mA}$ and its $\T$-fixed point locus is proper, its equivariantly localized virtual fundamental class $[M^{\sig}_{\al}]^{\vir}_{\T,\loc}$  induces an element in $L_{\loc, 0}$ of the same name. Note that one needs to work with rings $R$ that are $\ZZ[2^{-1}]$-algebras to apply equivariant localization in \cite{OT}. I will, from now on, omit both $\T$ and $\loc$ from the subscript of the above classes to avoid cluttering the introduction with notation. 

 Let $\sig, \sig'$ be two stability conditions connected by a continuous path.  Then Joyce's wall-crossing formula, first proposed in \cite{GJT}, predicts that the classes $\big[M^{\sig'}_{\al}\big]^{\vir}$ can be expressed as linear combinations of iterated Lie brackets acting on $\big[M^{\sig}_{\al_i}\big]^{\vir}$ where $\un{\al} = (\al_1,\al_2,\cdots,\al_n)$ are partitions of $\al$. For this to make sense, one needs to include the cases when there are strictly $\sig$-semistable objects, as without them the fixed point loci are not proper. Let us, for now, assume that there exists a predescribed procedure for defining these classes that I will denote by 
 $$
 \big\langle \mM ^{\sig}_{\al} \big\rangle\in L_{\loc,*}\,.
 $$
They have to satisfy $\big\langle \mM ^{\sig}_{\al} \big\rangle = \big[M^{\sig}_{\al}\big]^{\vir}$ whenever the right-hand side is defined. Proving wall-crossing can then be split into two separate problems.
\begin{enumerate}[wide, align=left]
\item[\textbf{Problem (I)}] Prove the wall-crossing formula
\begin{equation}
\label{eq:introWC}
\big\langle \mM ^{\sig'}_{\al} \big\rangle =  \sum_{\un{\al}\vdash_{\mA} \al}\widetilde{U}(\underline{\alpha};\sigma,\sigma') \Big[\Big[\cdots \Big[\big\langle\mM^{\sigma}_{\alpha_1}\big\rangle,\big\langle \mathcal{M}^{\sigma}_{\alpha_2}\big\rangle\Big],\ldots\Big],\big\langle\mM^{\sigma}_{\alpha_n}\big\rangle\Big]
\end{equation}
for $\sig,\sig'$ in a small enough neighbourhood $\mathscr{N}$ such that all $\big\langle\mM^{\sigma}_{\beta}\big\rangle$ and $\big\langle\mM^{\sigma'}_{\beta}\big\rangle$ are defined by the same procedure. Here, $\widetilde{U}(\un{\al};\sig,\sig')$ are the coefficients explained in \cite[§3.2]{JoyceWC}. Additionally, wall-crossing into Joyce--Song pairs stated in \eqref{eq:JSWC} should hold for all $\sig\in \mathscr{N}$. Together, the two statements form the local version of \cite[§4.4]{GJT}, \cite[Conjecture 2.9]{bojko2}.
\item[\textbf{Problem (II)}] If $K$ is a countable set of ways $k$ to define $\big\langle \mM ^{\sig}_{\al} \big\rangle$ as $\big\langle \mM ^{\sig}_{\al} \big\rangle^k$ for $\sig$ in an associated neighborhood $\msN_k$, show that 
$$
\big\langle \mM ^{\sig}_{\al} \big\rangle^{k_1} = \big\langle \mM ^{\sig}_{\al} \big\rangle^{k_2}\qquad \text{for}\quad k_1,k_2\in K\,, \sig\in \msN_{k_1}\cap \msN_{k_2}\,.
$$
The neighbourhoods $\mathscr{N}_k$ can then be patched together by taking a union ranging over $k\in K$.  
\end{enumerate}
\begin{remark}
\label{rem:welldefinednotneeded}
Sometimes, one only cares about relating actual virtual fundamental classes for two different stability conditions (e.g., DT/PT wall-crossing). In such cases, one can pick an appropriate $k\in K$ for which Problem (I) can be solved. Since we are not interested here in the classes $\big\langle \mM ^{\widetilde{\sig}}_{\al} \big\rangle^{k}$ for $\wt{\sig}$ on the intermediate walls, this is entirely sufficient. 
\end{remark}
In the next subsection, I will first address Problem (II) from a different, more conceptual perspective compared to what exists in the literature. The proof also works in the setting of \cite{JoyceWC} which, for example, includes torsion-free sheaves on surfaces. Using an observation akin to the quantum Lefschetz-principle for CY4 moduli spaces, I will show that Problem (II) reduces to Problem (I). 

To address Problem (I) and related questions in the future, I develop a $\infty$-categorical framework for Pvp diagrams and prove their functoriality. This allows me to formulate a set of ideal assumptions needed for the proof of (I). Unfortunately, these assumptions turn out to be too ideal to be fully general as explained in detail in §\ref{sec:sheaves}. Nevertheless, I show that they hold for in following two cases.
\begin{claim}\leavevmode
\vspace{-4pt}
\begin{enumerate}[wide, align=left]
\item[Problem (I)] can be proved for compactly supported sheaves on local CY fourfolds and representations of CY4 quivers. The former case can extended to include stable pairs.
\item[Problem (II)] can be resolved for any projective Calabi--Yau fourfold and semistable torsion-free sheaves on it by reducing to Problem (I). In the case of local Calabi--Yau fourfolds, Problem (II) is addressed in full generality (see also Theorem \ref{thm:introlocalindep}).
\end{enumerate}
\end{claim}
The joint work \cite{BKLT} addresses the general case by using Jouanolou devices and building on the techniques and results of the present work. It also proves wall-crossing for Joyce--Song pairs which combined with the second claim above gives another proof that the invariants $ \big\langle \mM ^{\sig}_{\al} \big\rangle$ are well-defined for torsion-free sheaves.
 
\tocless{\subsection{Reducing Problem (II) to Problem (I)}}

Fix $\mA=\Coh(X)$ and a line bundle $L$ on $X$. For a fixed stability condition $\sig$, assume that all $\sig$-semistable sheaves $E$ of class $\llbracket E\rrbracket = \al\in \ov{K}(\mA)$ satisfy 
$$
H^i(E\otimes L^*) = 0 \qquad \text{whenever}\quad i>0\,.
$$
The homology of the moduli stack of ($\T$-equivariant) triples $V\otimes L\xrightarrow{s}E$, where $V$ a vector space, $E$ a sheaf on $X$, and $s$ a morphism, can also be endowed with a vertex algebra structure $W^L_{\loc, *}$ inducing a Lie algebra. For the Joyce--Song stability condition $\sig_{\JS}$ associated to $\sigma$, one considers the moduli space $N^{\JS}_{L,\al}$ of $\sig_{\JS}$-stable triples with $V=\CC$ and $\llbracket E\rrbracket=\al$. Its virtual fundamental class is expressed by wall-crossing into Joyce--Song stable pairs mentioned in Problem (I) as
\begin{equation}
\label{eq:introJSWC}
\big[N^{\JS}_{L,\al}\big]^{\vir} = \sum_{\begin{subarray}{c}
        \un{\alpha}\vdash_{\mA} \alpha\,, \\
          \phi(\alpha_i) =\phi(\alpha) \end{subarray}}\frac{1}{n!}\Big[\big\langle \mM^\sigma_{\alpha_n}\big\rangle, \cdots \Big[\big\langle \mM^\sigma_{\alpha_1}\big\rangle,  e^{(1,0)}_L\Big]\cdots\Big]
\end{equation}
where the Lie bracket comes from $W^L_{\loc, *}$. Here $\phi(\al)$ denotes the ``phase" of $\al$ with respect to $\sig$, and $e^{(1,0)}_L$ is the point class of the triple $L\to 0$. 

If \eqref{eq:introJSWC} is proved, one could define the classes $\big\langle\mM^\sigma_{\alpha}\big\rangle$ depending on a choice of $L$ inductively by noting that $\Big[-,  e^{(1,0)}_L\Big]$ is injective (see Lemma \ref{lem:pushJSandinj}). Instead, one uses \eqref{eq:Masigdef} which is a projection of the above along the map to sheaves. In other words, formula \eqref{eq:introJSWC} is a lift of the equation defining $\big\langle\mM^\sigma_{\alpha}\big\rangle$ for a fixed $L$.

With this in mind, the main result addressing Problem (II) is as follows. 
\begin{theorem}[Theorem \ref{thm:independence!}]
Let $\al$ be of positive rank and $\sig$ be a stability condition such that the $\sig$-semistable sheaves of positive rank are torsion-free. The preceding definition of $\big\langle \mM^\sigma_{\alpha}\big\rangle$ is independent of $L$ because \eqref{eq:introJSWC} from Problem (I) has now been shown to hold in \cite{BKLT}.
\end{theorem}

To explain the idea of the proof given in §\ref{sec:proof}, it is easier to present it for torsion-free sheaves on a surface $S$. In this case Problem (II) is already solved in \cite[§9.2, §9.3]{JoyceWC} by applying the method in \cite[Proposition 7.19]{mochizuki}.  Mochizuki's idea was to construct a larger moduli space $\MM^{\JS}_{L_1,L_2,\al}$ that contains both $N^{\JS}_{L_1,\al}$ and $N^{\JS}_{L_2,\al}$ for two different line bundles $L_1,L_2\to S$. A localization with respect to a natural $\GG_m$-action on $\MM^{\JS}_{L_1,L_2,\al}$ produced a formula relating the invariants $\big[N^{\JS}_{L_i,\al}\big]^{\vir}$  but not the moduli spaces themselves. After a computation in \cite[§9.3]{JoyceWC}, this eventually implies
\begin{equation}
\label{eq:introindependence}
\big\langle \mM^\sigma_{\alpha}\big\rangle^{L_1} = \big\langle \mM^\sigma_{\alpha}\big\rangle^{L_2}\,.
\end{equation}

I sought an alternative argument of this result because i) I wanted to give a more conceptual geometric proof which begins be relating the moduli spaces themselves ii) 
$\MM^{\JS}_{L_1,L_2,\al}$ seemed like an unnatural moduli space without a deeper interpretation.

My proof  reduces to the situation when $L_2 = L_1\otimes \mO_S(-D_2)$ for a smooth divisor $D_2$. Using the section $L_2\to L_1$ that vanishes on $D_2$, there is a closed embedding 
$$
\begin{tikzcd}
N^{\JS}_{L_1,\al}\arrow[r,hookrightarrow,"\iota"]&N^{\JS}_{L_2,\al}
\end{tikzcd}
$$
that identifies $\iota\big(N^{\JS}_{L_1,\al}\big)$ with the vanishing locus of a section of a vector bundle $\VV_1$. Denoting the universal sheaf on $S\times N^{\JS}_{L_2,\al}$ by $\mF$, the projection from the product to the second factor by $p$, and the pullback of a line bundle $L$ along the projection to $S$ by $L_S$, the vector bundle is given by 
$$
\VV_1 = p_*\Big(\mE\otimes (L^*_2)_S|_{D_2\times N^{\JS}_{L_2,\al}}\Big)\,.
$$
There is then the virtual pullback diagram given by the black part in
\begin{equation}
    \label{eq:introJSdiag}
    \begin{tikzcd}[column sep=large] \arrow[d,violet,"\ov{\kappa}"]\Pur{\widetilde{\FF}_1^\vee[2]}\arrow[r,violet,"{\ov{\mu}}"]&\Pur{\FF_1}\arrow[d,violet, "\mu"]\arrow[r,violet]&\Pur{\VV_1^\vee[1]}\arrow[d,violet,equal]\arrow[r, violet]&\Pur{\widetilde{\FF}_1^\vee[3]}\arrow[d,violet]\\
   \arrow[d] \iota^*\big(\FF_{2}\big)\arrow[r,"\kappa"]&\Pur{\widetilde{{\color{Black}\FF}}}_1\arrow[r]\arrow[d]&\VV_1^\vee[1]\arrow[d,equal]\arrow[r]&\iota^*\big(\FF_{2}\big)[1]\arrow[d]\\   \iota^*\big(\mathbb{L}_{N^{\JS}_{L_2, \al}}\big)\arrow[r]&\mathbb{L}_{N^{\JS}_{L_1, \al}}\arrow[r]&\LL_{\iota}\arrow[r]&\iota^*\big(\mathbb{L}_{N^{\JS}_{L_2, \al}}\big)[1]
    \end{tikzcd}\,.
\end{equation}
The first two vertical arrows are the standard pair obstruction theories. Together with the equality they form a morphism of distinguished triangles. By \cite{KKP, Manolache}, such a diagram implies
\begin{equation}
\label{eq:introJScomparison}
\iota_*\Big[N^{\JS}_{L_1,\al}\Big]^{\vir} = \Big[N^{\JS}_{L_2,\al}\Big]^{\vir}\cap e(\VV_1)
\end{equation}
where $e(-)$ denotes the Euler class of a vector bundle. This observation replaces Mochizuki's argument in my proof as it compares the pair moduli spaces and their invariants. 

The rest of the argument is simple as it combines \eqref{eq:introJSWC} with an argument along the lines of what I used in \cite[§3.1 and §3.2]{bojko2}. Some details that are specified in \ref{sec:proof} are omitted here. 
I first extend $-\cap e(\VV_1)$ to a morphism of vertex algebras and their associated Lie algebras. By \eqref{eq:NJSD1Dvir}, one can equate 
$$
(\iota)_*\Big(\textnormal{RHS of }\eqref{eq:introJSWC}\textnormal{ for }L_1\Big) =\Big(\textnormal{RHS of }\eqref{eq:introJSWC}\textnormal{ for }L_2\Big)\cap e\big(\VV_1\big)
$$
where $\iota_*$ is an appropriate extension of the pushforward in \eqref{eq:introJScomparison}. This implies \eqref{eq:introindependence}.

Let us now return to working with a CY fourfold X. In this case, the obstruction theories $\FF_i$ of $N^{JS}_{L_i,\al}$ are given by $\RHom(I,I)_0$ at $I\cong\{L_i\to F\}$ in $D^b(X)$. To conclude \eqref{eq:introJScomparison}, one now needs the full Pvp diagram \eqref{eq:introJSdiag} where $\ov{\theta} = \theta^\vee[2]$ for a map of complexes. The proof that the natural morphisms fit into this diagram forms the bulk of §\ref{sec:proof}. The rest of the argument remains the same.

One can apply the above reasoning to yet another scenario.
\begin{example}
\label{ex:introspectral}
Let $Y$ be a smooth projective three-fold with Kodaira dimension $\kappa(Y) = -\infty$ -- for example, a Fano 3-fold or any three-fold admitting a non-trivial $\T$-action. Construct $X=\Tot(K_Y)\xrightarrow{r}Y$ for the canonical bundle $K_Y$.  Let $\al$ have 3-dimensional compact support, then all sheaves of class $\al$ are set-theoretically supported on the zero section by Lemma \ref{lem:3dim0section}. Let $\sig$ be Gieseker or slope stability, and $L=r^*L_Y$ for a line bundle $L_Y$ on $Y$. Everything discussed in this subsection applies to $\big[N^{\JS}_{L,\al}\big]^{\vir}$ and $\big\langle \mM^\sig_{\al}\big\rangle^{L}$ with $\T$-equivariance included.
\end{example}
Using spectral correspondence in §\ref{sec:localCY4}, it is shown that wall-crossing \eqref{eq:introJSWC} holds in this situation. The two statements together imply the next theorem.
\begin{theorem}[Corollary \ref{cor:spectralindep}]
\label{thm:introlocalindep}
In the situation of Example \ref{ex:introspectral}, the invariants $\big\langle \mM^\sig_{\al}\big\rangle^{L}$ are independent of $L$. 
\end{theorem}
\tocless{\subsection{Stable $\infty$-Pvp diagrams and their functoriality}}

After diagram-chasing like the one in §\ref{sec:proof}, it becomes clear that stable $\infty$-categories provide a more suitable framework for constructing diagrams of the form \eqref{eq:introJSWC}. Therefore, most of this work is written in this language with the necessary background recalled in Appendix \ref{appendix}. When mentioning
   \begin{equation}
   \label{eq:introdiagsym}
    \begin{tikzcd}[column sep=large]   \Pur{\MM[-1]}\arrow[d,equal, violet]\arrow[r,violet]&\arrow[d,"\ov{\kappa}",violet]\Pur{\wt{\FF}^\vee[2]}\arrow[r,"{\ov{\mu}}", violet]&\Pur{\FF}\arrow[d,"\mu", violet]\arrow[r, violet]&\Pur{\MM}\arrow[d,equal, violet]\\
    \Pur{\MM[-1]}\arrow[d]\arrow[r,"\eta", violet]&\arrow[d] \Pur{f^*\big(\EE\big)}\arrow[r,"\kappa", violet]&\Pur{\FFna}\arrow[r, violet]\arrow[d]&\Pur{\MM}\arrow[d]\\
   \mathbb{L}_{f}[-1]\arrow[r]&   f^*\big(\mathbb{L}_{M}\big)\arrow[r]&\mathbb{L}_{N}\arrow[r]&\mathbb{L}_{f}
    \end{tikzcd}
\end{equation}
in this subsection, I will only mean the \Pur{purple part}
because the lower half is well understood. The precise formulation of this diagram in the language of stable $\infty$-categories is explained in §\ref{Asec:functoriality}.

Due to $\infty$-functoriality of cones, it is enough to work with the self-dual homotopy commutative diagram  (see Proposition \ref{prop:sympullback})\footnote{For pretriangualated categoroes this was independently remarked in \cite{LiuVW}.}
 \begin{equation}
   \label{eq:introstartingIxI}
\begin{tikzcd}
\arrow[d]\MM[-1]\arrow[r,"\eta"]&f^*(\EE)\arrow[d,"\ov{\eta}"]\\
  0\arrow[r]&\MM^\vee[3]
\end{tikzcd}\,.
\end{equation}
Interestingly, it does not seem possible, without further conditions, to extend it to \eqref{eq:introdiagsym} when working strictly within triangulated categories. This is not the case for a different part of the Pvp diagram as discussed in \cite[Appendix C]{Park}.

When it comes to proving wall-crossing, one needs to construct Pvp diagrams along multiple consecutive morphisms
$$
 \begin{tikzcd}
     N_2\arrow[r,"f_2"] \arrow[rr,bend left = 50, "f"]&N_1\arrow[r, "f_1"]&M
 \end{tikzcd}\,.
$$
 The final result should be the same as applying the construction along the composition $f$. This is what I call \textit{functoriality} of $\infty$-Pvp diagrams proved in Theorem \ref{thm:functsympull} along with converses. This result is used repeatedly in \cite{BKLT} to prove that the construction there are well-defined and satisfy the necessary conditions.

\tocless{\subsection{Proving wall-crossing in the presence of obstruction theories and why these do not always exist}}\label{sec:introproofWC}
The proof of \eqref{eq:introWC} from Problem (I) follows the steps in \cite{JoyceWC} based on \cite{mochizuki}. There are two main differences: i) the master spaces need to be endowed with CY4 obstruction theories and ii) the signs coming from comparing orientations must be carefully computed. To address i), I rely on the tools for manipulating $\infty$-Pvp diagrams discussed in the previous subsection. For ii), I list all the necessary orientation conventions in §\ref{sec:orconventions}. They are intended to bridge the gap between the conventions in \cite[Theorem 1.15]{CGJ} leading to the sign comparison under taking direct sums and the ones in \cite[Theorem 7.1]{OT} used for equivariant localization. In the process, I realized an inconsistency in \cite[§7]{OT}. Corrections in §\ref{sec:orconventions} lead to a cleaner equivariant localization formula \eqref{eq:vireqloc} which is explicitly independent of choices of resolutions. Using this equation, the computation in §\ref{sec:WCflags} convincingly recovers (the equivariant version of) Joyce's vertex algebra as predicted in \cite{GJT}. 

The main idea of the proof, in a few sentences, is to refine wall-crossing for objects in $\mA$ to an enhanced category $\mB_{A_r}$ which consists of objects in $\mA$ framed by flags. They can be represented by the $A_r$ quiver for some $r$ where the last vertex takes values in $\mA$ rather than in vector spaces. Projecting back to $\mA$ implies \eqref{eq:introJSWC} but can also give rise to refined formulae with insertions as will be the case in \cite{Bo26}. To prove the enhanced wall-crossing for flags, one constructs an \textit{enhanced master space} $\MS$. It approximates a $\PP^1$-bundle over the moduli of semistable objects in $\mB_{A_r}$. More concretely, one imposes an appropriate stability condition on the category represented by the following quiver.
\begin{center}
\includegraphics[scale=1]{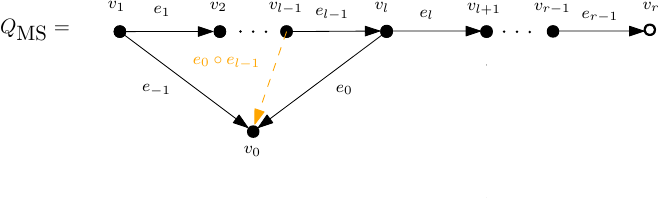}
\end{center}
 The vertex $\circ$ is again replaced by an object in $\mA$, and the \BuOr{dashed} arrow is a relation.

To do all this on the level of enumerative invariants, one needs a CY4 obstruction theory $\FF_{\MS}$ on MS. For a CY4 dg-quiver, the obstruction theories of the moduli of its representations are explicitly described by Lemma \ref{lem:cotangentofQ}. This allows me to translate $\FF_{\MS}$ into the data of a CY4 completion of $Q_{\MS}$. This completion is given by
\begin{equation}
\label{eq:introCY4QMS}
    \includegraphics{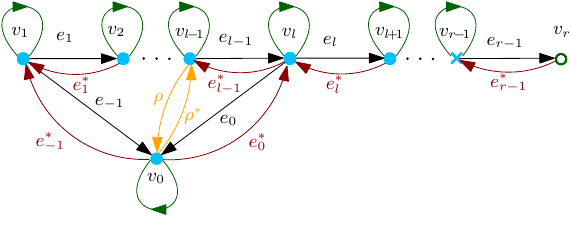}
\end{equation}
with superpotential
$$
\BuOr{\mH_{\MS}}= \BuOr{\rho^*}\circ e_0\circ e_{l-1}\,.
$$
Different colors are used to represent different degrees: black for degree 0, orange for degree \BuOr{$-1$}, red for degree \Ma{$-2$}, and green for degree \FG{$-3$}. The contribution of \FG{$\circ$} is the obstruction theory $\EE$ for $\mA$.

Of course, one needs to induce $\FF_{\MS}$ via the diagram \eqref{eq:introstartingIxI}. In Example \ref{ex:counterexamples}, such a direct construction is shown to be impossible for sheaves on a general CY fourfold. This verifies that an additional argument is needed. This missing argument is now provided in \cite{BKLT} by using Jouanolou devices.

Nevertheless, two cases included in the next theorem can be addressed here directly as long as properness of $\T$-fixed point loci of the moduli in question is satisfied (see Assumption \ref{ass:stab}.g) and  Assumption \ref{ass:obsonflag}.a)).
\begin{theorem}[Corollary \ref{cor:WCconsequences}]
\label{thm:intro}
Equivariant wall-crossing from Problem (I) and the independence of $k$ from Problem (II) hold when $\mA$ is the category of 
\begin{enumerate}
\item compactly supported sheaves on a CY fourfold $X$ from Example \ref{ex:introspectral} where $Y$ can now be quasi-projective,
\item degree 0 representations of a CY4 dg-quiver.
\end{enumerate}
This also applies to stable pair wall-crossing on $X$ from 1), see §\ref{sec:pairWC} for more details.
\end{theorem}
For the first result, I combine the CY4 dg-quiver in question with \eqref{eq:introCY4QMS} to construct a new one. Its stable representations are parametrized by MS with the appropriate obstruction theory $\FF_{\MS}$. For local CY fourfolds, I use spectral correspondence and the observation of \cite[(6.11)]{TT2} that describes the obstruction theory for pairs on $X$ as the $-2$-shifted cotangent bundle of the obstruction theory for pairs on $Y$. Since MS parametrizes generalized pairs, this reasoning applies as well.  In \cite{LiuVW}, spectral correspondence was already used to construct obstruction theories on $\MM^{\JS}_{L_1,L_2,\alpha}$ in the case of 2-dimensional sheaves on local CY threefolds. This also motivated the current approach.\footnote{In \cite{KLT}, the authors studied a related problem for CY threefolds. They developed an approach that relies on almost perfect obstruction theories, which have, unfortunately, not yet been developed for CY fourfolds.}
\tocless{\subsection{Local CY fourfolds and quivers}}\label{sec:introquivloc}
I have already discussed one application of Theorem \ref{thm:intro} in Theorem \ref{thm:introlocalindep}. Using the stable pair wall-crossing, one can also prove the DT/PT correspondence for local CY fourfolds from Theorem \ref{thm:intro}.1) on the level of equivariant virtual fundamental classes. 

 Fix a compactly supported curve class $\beta\in H^*_{\textnormal{cs}}(X)$ and $m\in\ZZ$. Then the moduli spaces 1) DT$_{\beta,m}$, 2) PT$_{\beta,m}$ parametrize pairs $\mO_X\xrightarrow{s} F$ with one-dimensional $F$ such that $\ch(F)=\beta$ and
 \begin{enumerate}
 \item $s$ is surjective,
 \item $\coker(s)$ is 0-dimensional, and $F$ is pure,
 \end{enumerate}
 respectively. As long as the conditions of Lemma \ref{lem:compactcurve} are satisfied, the moduli spaces admit virtual fundamental classes induced by the traceless obstruction theories of the complexes $I=\{\mO_X\to F\}$. The following is an immediate consequence of Theorem \ref{thm:intro}.
\begin{corollary}
\label{cor:introDTPT}
Set $X =\Tot(K_Y)$ for a smooth quasi-projective 3-fold $Y$ with a $\T$ action satisfying the assumptions of Lemma \ref{lem:compactcurve}. Let $\mA$ be the abelian category from Example \ref{ex:DTPT}.  Then the wall-crossing formula
$$
\sum_{m\in\ZZ}\big[\textnormal{DT}_{\beta,m}\big]^{\vir}q^m = \exp\bigg\{\sum_{n>0}\Big[\big\langle \mM_{np}\big\rangle,- \Big]q^n\bigg\}\sum_{m_0\in \ZZ}\big[\textnormal{PT}_{\beta,m_0}\big]^{\vir}q^{m_0}
$$
holds for compatible choices of orientations. Here $p$ is the Poincaré dual of a point, and $\big\langle \mM_{np}\big\rangle$ are the invariants counting 0-dimensional sheaves defined uniquely using \eqref{eq:introJSWC}.
\end{corollary}

Choosing $Y=\CC^3$ in Example \ref{ex:introspectral} and $\beta =0$ in Corollary \ref{cor:introDTPT} leads to a formula computing $\big[\Hilb^m(X)\big]^{\vir}$ in terms of $\big\langle \mM_{np}\big\rangle$. However, it is more suitable to formulate this wall-crossing in terms of quivers. An ongoing work with multiple co-authors aims to refine the arguments in \cite{bojko2} to provide an alternative proof of Nekrasov's conjecture \cite{Nekrasov1, Nekrasov2} and, more generally, to address $\big[\CC^4/\Gamma\big]$ for finite subgroups $\Gamma\subset SU(4)$. Let 
$$
 \includegraphics[scale=1.2]{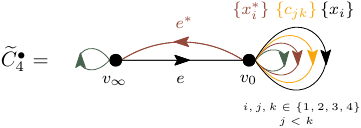}
$$
be the CY4 dg-quiver with superpotential 
$$
\BuOr{\mH} =\frac{1}{4} \sum_{\sig\in S_4}(-1)^{\sig}\BuOr{c_{\sig(1)\sig(2)}}\big[x_{\sig(3)},x_{\sig(4)}\big]\,.
$$
For a stability condition $\mu_>$ and the dimension vector $(d_{\infty},d_0) = (1,n)$, the $\mu_>$-semistable degree-0 representations of $\wt{C}_4^{\bullet}$ correspond to ideal sheaves of length $n$ on $\CC^4$. Corollary \ref{cor:hilbWC}  contains the wall-crossing formula
$$
\sum_{n>0}\big[\Hilb^n(\CC^4)\big]^{\vir}q^n= \exp\bigg\{\sum_{n>0}\Big[\big\langle \mM_{np}\big\rangle,- \Big]q^n\bigg\}e^{(1,0)}_{\mO_X}\,.
$$ 
\tocless{\subsection{Future work}}\label{sec:futurework}
As mentioned, the last formula and its appropriate versions for $\big[\CC^4/\Gamma\big]$ will be used in future work addressing the conjectures in \cite{CK1, Nekrasov1, Nekrasov2} and their adaptations to orbifolds appearing in \cite[Conjecture 1.6]{Schmier}. A similar argument will also be used to prove \cite[Conjecture 2.34]{AKL}.  Moreover, an ongoing project relies on the explicit description of the equivariant homology in Example \ref{ex:trivialThom} to give a natural formula for descendent integrals over $\Hilb^n\CC^4)$. 

Vertex DT/PT correspondences can also be naturally formulated in terms of equivariant CY4 quiver wall-crossing (or CY3 quiver wall-crossing for the 3-fold vertex). Together with the formula for the descendent integrals on $\Hilb^n(\CC^4)$, this ought to produce a closed correspondence between the vertex DT and PT qq-characters of \cite{KimGo, KimGo2}.

A generalization of Theorem \ref{thm:intro} will appear in a joint work \cite{BKLT}. The proof offered there will rely on the framework laid out here. It will address sheaf and stable pair wall-crossing on all CY fourfolds. In a separate work \cite{Bo26}, I will develop a new perspective on wall-crossing with cohomological insertions, which will be formulated in terms of formal deformations of vertex algebras. Combined, the two works will be used to prove DT/PT-like correspondences similar to the one conjectured in \cite{CK3} and \cite{BKP2, BKP3}, where the latter studies surface counting pair invariants. One of the major applications of this following \cite{BKPdiscussions} is a DT/PT correspondence for Fano 3-folds obtained by dimensional reduction.  K-theoretic refinement of the above program will also be developed.
\section*{Acknowledgements}
I am immensely grateful to Hyeonjun Park for explaining his work to me and for suggesting to use stable $\infty$-categories. A special thanks goes to Nikolas Kuhn and Henry Liu, who helped me understand the correct expansion for equivariant wall-crossing. I would also like to thank Noah Arbesfeld, Younghan Bae, Emile Bouaziz, Ionut Ciocan-Fontanine, Ben Davison, Michael Eichmair,  Dominic Joyce, Bernhard Keller, Adeel Khan, Nikolas Kuhn, Anton Mellit, Denis Nesterov, Georg Oberdieck, Rahul Pandharipande, Maximilian Schimpf, and Andy Jiang for many discussions that helped shape this work. 

I started my work on this project in autumn 2021 when I was supported by ERC-2017-AdG-786580-MACI. During my stay at ETH, this project received funding from the European Research Council (ERC) under the European Union Horizon 2020 research and innovation program (grant agreement No 786580). A large part of this article was written while I stayed at Academia Sinica as a Research Scholar. I completed this work at the Shanghai Institute for Mathematics and Interdisciplinary Sciences.
\maketitle

\section{Some notation}
I will use the following notation: 
\begin{itemize}
\item All vector spaces, (derived) schemes, and stacks are over $\CC$.
    \item I will write $\Vect$ for the category of vector spaces.
    \item All algebraic stacks (and thus algebraic spaces) will be assumed to be locally of finite type.
    \item Standard symbols like $X$ will denote classical algebraic spaces.
    \item Symbols like $\mathcal{X}$ will be used for algebraic stacks.
    \item Bold symbols like $\boldsymbol{X}$ and $\boldsymbol{\mathcal{X}}$ will be reserved for derived stacks, their morphisms, and objects living on them.
    \item Classical restrictions of morphisms between derived stacks and objects on derived stacks are going to be denoted by the same symbol but not in bold.
    \item The notation $\T$ will be reserved for a complex torus and $\Ft$ for its Lie algebra. 
    \item For a ring $R$, I will write $R[\Ft]$ and $R\llbracket\Ft\rrbracket$ for the ring of $R$-valued polynomials, respectively, power-series on $\Ft$. The fraction field of $R$ will be denoted by $\kk$.
    \item The usual rank of a K-theory class or its Chern character is denoted by $\Rk(-)$, but I also use an artificially defined rank for a choice a stability condition $\sig$ which is denoted by $\rk_{\sig}(-)$.
    \item The notation $c_{z^{-1}}(-)$ denotes the total Chern class $\sum_{i\geq 0}c_{i}(-)z^{-i}$.
    \item The (deformations of) vertex algebras and their associated Lie algebras will always be automatically $\ZZ$-graded.
    \item For a (higher) stack $\mX$, I will use $\mD^b(\mX)$ to denote its bounded derived $\infty$-category which is a stable $\infty$-category. Its homotopy (triangulated) category will be labeled by $D^b(\mX)$.
    \item Let $\EE\xrightarrow{\phi} \LL_{\mX}$ be a morphism in $D^b(\mX)$ which is an obstruction theory. Often, I will only write $\EE$ when the rest of the data is understood. I will also call any lift of $\phi$ to a morphism in $\mD^b(\mX)$ an obstruction theory when working with stable $\infty$-categories. 
    \item If $f:\mX\to \mY$ is an open embedding of stacks and $\FF\xrightarrow{\psi}\LL_{\mY}$ is an obstruction theory, I will say that $f^*\FF\xrightarrow{f^*\psi}\LL_{\mX}$ is the pullback of $\FF$ along $f$.

\end{itemize}

\section{Counting objects in CY4 categories}
I will begin by recalling the necessary conditions for the construction of virtual fundamental classes using \cite{BJ, OT}. In the process, I will also summarize some tools like the virtual pullback from \cite{Park} and the equivariant localization from \cite{OT}. An inconsistency regarding working with orientations in \cite[§7]{OT} is corrected using the conventions carefully noted down in §\ref{sec:orconventions}.

When reading the first 3 subsections, one can substitute $\Coh(X)$ where $X$ is a Calabi--Yau fourfold for the Calabi--Yau four (CY4) abelian category $\mA$. In §\ref{sec:dgsetup}, a general case of $\mA$ is considered leading to the discussion of CY4 dg-quivers in §\ref{sec:dgquivers}. In this final subsection, I fix some terminology that will be used throughout the article, and I discuss an example.

\subsection{Virtual fundamental classes}
\label{sec:basics}
Let $\mA$ be a CY4 abelian category in the sense recalled later in §\ref{sec:dgsetup}. I will restrict myself to working with the subgroup of \textit{even K-theory classes}
\begin{equation}
\label{eq:Keven}
K^0_e(\mA) = \big\{\alpha\in K^0(\mA):\chi(\alpha,\alpha)\in 2\ZZ\big\}
\end{equation}
which due to \cite{OT2} does not lose out on any information when working over $\ZZ[2^{-1}]$. Let $M$ be an algebraic moduli space of some stable objects $E$ of class $\llbracket E\rrbracket =\al$.  when working with the \textit{truncated cotangent} complex $\tau^{\geq -1}(\LL_{M})$, I will sometimes omit specifying $\tau^{\geq -1}$ where it is not necessary. When given an obstruction theory
\begin{equation}
\label{eq:EEobtheory}
\begin{tikzcd}\EE\arrow[r]&\LL_{M}\end{tikzcd}\,,
\end{equation}
I will often abuse terminology by calling $\EE$ the obstruction theory when the rest of the data in \eqref{eq:EEobtheory} is understood.

The next definition describes when there exist virtual fundamental classes by the construction in \cite{OT}.
 \begin{definition}
 \label{def:viradmcl}
Let $M$ be a separated sheme. The obstruction theory $\EE\to \LL_{M}$ is said to be CY4 if it is perfect of tor-amplitude $[-2,0]$ and satisfies the following properties:
\begin{enumerate}[wide, align=left]
    \item[\textit{(Self-duality)}] There is an isomorphism
    \begin{equation}
    \label{eq:sigma}
    \begin{tikzcd} \mathbb{i}_q:\EE\arrow[r,"\sim"]&{\EE^\vee[2]}\end{tikzcd}\,, \qquad \mathbb{i}_q^\vee[2] = \mathbb{i}_q\,.
    \end{equation}
    \item[\textit{(Orientability)}] There exists a trivialization 
    $
    o: \un{\CC}\xrightarrow{\sim}\det(\EE)
    $
satisfying 
\begin{equation}
\label{eq:EEorient}
(o^*)^{-1}\circ o^{-1}= \det(\mathbb{i}_q):\begin{tikzcd}\det(\EE)\arrow[r,"\sim"]& \det(\EE)^*\end{tikzcd}\,.
\end{equation}
In this case, I will say that $M$ is orientable for given \eqref{eq:EEobtheory}.
    \item[\textit{(Isotropy of cones)}] Let $\mathfrak{C}(\EE)$ be the virtual normal cone recalled in §\ref{Sec:pullback} and 
    \begin{equation}
    \label{eq:quadratic}
    \begin{tikzcd}q:\mathfrak{C}(\EE)\arrow[r]&\un{\CC}
    \end{tikzcd}
    \end{equation}
    be the quadratic form induced by $\mathbb{i}_q$. This form should vanish under the restriction along the embedding 
    $$
    \begin{tikzcd}
    \mathfrak{C}_M\arrow[r,hookrightarrow]&\mathfrak{C}(\EE)
    \end{tikzcd}
   $$
   of the intrinsic normal cone.
\item[\textit{(Evenness)}] The rank of $\EE$ is even. 
\end{enumerate}
 \end{definition}
 For a torus $\T$ acting on $M$, all of the above conditions make sense $\T$-equivariantly. If they are satisfied, then \cite{OT} constructs an equivariant virtual cycle
$$
[M]^{\vir}_{\T}\in A^{\T}_*\big(M,\ZZ[2^{-1}]\big)\,.
$$
If $M$ is connected then changing the choice of orientation $o$ changes the sign of $[M]^{\vir}_{\T}$. When $M$ is projective and $\T$ trivial, the work \cite{OT2} shows that $[M]^{\vir}:=[M]^{\vir}_{\T}$ maps to the Borisov--Joyce class \cite{BJ} in $H_*\big(M,\ZZ[2^{-1}]\big)$. This implies that the resulting class still denoted by $[M]^{\vir}$ lies in the image of $H_*\big(M,\ZZ\big)\to H_*\big(M,\ZZ[2^{-1}]\big)$. The second conclusion in \cite{OT2} states that the class from \cite{BJ} vanishes in $H_*\big(M,\ZZ[2^{-1}]\big)$ if evenness does not hold. This is also the motivation behind working with \eqref{eq:Keven}.
\begin{remark}
 Oh--Thomas construct virtual fundamental classes only when $M$ is quasi-projective. In the present and all subsequent works on wall-crossing, I will usually only assume $M$ to be separated or proper. For this reason, I will be using the generalization provided by 
\cite{Park,Parknote}. In this more general situation, it still constructs the cycle $[M]^{\vir}_{\T}\in A^{\T}_*\big(M,\ZZ[2^{-1}]\big)$. 
\end{remark}
The two main examples of $M$ to keep in mind are as follows.
\begin{example}
 \label{ex:CYquiver}
     \leavevmode
\vspace{-4pt}
\begin{enumerate}
\item The moduli space $M$ parametrizes semistable sheaves on a Calabi--Yau fourfold $X$.
\item Fix a quiver $Q$ with relations that are governed by a CY4 dg-structure as in §\ref{sec:dgquivers}.  Then I will consider moduli spaces of its semistable representations. 
\end{enumerate} 
\end{example}
For the rest of the subsection, I will fix an obstruction theory $\EE$ and assume that the conditions in Definition \ref{def:viradmcl} hold. Additionally, I will require that $M$ satisfies the following.
\begin{enumerate}[wide, align=left]
    \item[\textit{(Equivariant resolution property)}]  Any $\T$-equivariant perfect complex of tor-amplitude $[a,b]$ on $M^{\T}$ admits a $\T$-equivariant locally free resolution in degrees $[a,b]$.
\end{enumerate}
If $E^\bullet$ is such a locally free resolution of $\EE$, I will recall a convention from  \cite[Proposition 4.2]{OT} that uses an orientation on
\begin{equation}
\label{eq:symcomplex}
    E^\bullet = \{\begin{tikzcd}
        T\arrow[r,"\alpha"]&E\arrow[r,"\alpha^*"]&T^*
    \end{tikzcd}\}
\end{equation}
to induce one on $E$ in the sense of Definition \ref{def:orientation}.i). Here, the musical isomorphism $E \xrightarrow{i_q} E^*$ is required to be compatible with the quasi-isomorphism $\mathbb{i}_q$. In this case, an orientation of $E$ is given by a trivialization     $
    o_E: \un{\CC}\xrightarrow{\sim}\det(E)
    $
satisfying 
$$
(o_E^*)^{-1}\circ o_E^{-1}= \det(i_q):\begin{tikzcd}\det(E)\arrow[r,"\sim"]& \det(E)^*\end{tikzcd}\,.
$$
It makes $E$ into an $SO(2n,\CC)$ bundle. In the special case that $E\cong V\oplus V^*$, it has a natural orientation $o_{V}:\un{\CC}\xrightarrow{\sim}\det(E)\cong\det(V)\det(V)^*$ described in Definition \ref{def:orientation}.
\begin{definition}[{\cite[Proposition 4.2]{OT}}]
\label{def:Einducedor}
    Let $\un{\CC}\xrightarrow{o^\bullet}\det(E^\bullet)$ be an orientation of $E^\bullet$ in the sense of \eqref{eq:EEorient}.   The induced orientation of $E$ is defined as the composition of the consecutive arrows
    $$
\begin{tikzcd}\un{\CC}\arrow[r,"o^\bullet"]&\det(E^\bullet)  \cong \det(T)\det(T)^* \det(E)^*\arrow[r,"o^{-1}_{\T}\otimes \id"]&[1cm]\det(E)^*\arrow[r,"\det(i_q)^{-1}"]&[1cm]\det(E)\end{tikzcd}\,.
    $$
\end{definition}
The only explicit localization computation done in this work will use $\T = \GG_m$. If $M$ has the $\T$-equivariant resolution property, consider a resolution 
\begin{align*}
\EE|_{M^{\T}}&\cong \begin{tikzcd}[ampersand replacement=\&]
   \big\{T_{\T}\arrow[r]\&E_{\T}\arrow[r]\&\big(T_{\T}\big)^{*}\big\}=:E^\bullet_{\T}
\end{tikzcd}
\end{align*}
of the form \eqref{eq:symcomplex} for $\T$-equivariant vector bundles $T_{\T},E_{\T}\to M^{\T}$.  It admits a splitting into the fixed part and the moving part
$$
E^\bullet_{\T} = E^\bullet_m\oplus E^\bullet_f\,.
$$
The moving part describes the \textit{virtual normal bundle} 
\begin{equation}
\label{eq:Nvir}
N^{\vir} = E^\bullet_f[-2]=\begin{tikzcd}\Big\{T^m\arrow[r]&E^m\arrow[r]&\big(T^m\big)^*\Big\}\end{tikzcd}
\end{equation}
while the fixed part $\EE^f := E^\bullet_f$ determines a CY4 obstruction theory on $M^{\T}$. Let $t$ be the weight one $\T$-equivariant trivial line bundle on $M^{\T}$. Because $N^{\vir}$ contains only non-zero weights, it can be decomposed as 
\begin{equation}
\label{eq:Nvirdecomp}
N^{\vir}  = t\cdot N^{\geq }\oplus t^{-1}\cdot N^{\leq }
\end{equation}
such that $N^{\geq}$ only contains non-negative weights and  $N^{\leq}\cong \big(N^{\geq}\big)^\vee[-2]$. The orientation of $N^\vir$ I will use here is determined by 
\begin{equation}
\label{eq:oNgeq}
 \begin{tikzcd}[column sep=large]
o\big(N^{\vir}\big):\un{\CC}\arrow[r,"(-1)^{\Rk(N^{\geq})}o_{N^{\geq }}"]&[1cm]\det\big(N^{\geq}\big)\det\big(N^{\leq}\big)\cong\det(N^\vir)\,.
\end{tikzcd}   
\end{equation}
Notice the extra sign $(-1)^{\Rk(N^{\geq})}o_{N^{\geq }}$ which appears here to relate this orientation with the one used in \cite[Theorem 7.1]{OT}.

If $M$ came with a fixed choice of orientation $\un{\CC}\xrightarrow{o}\det(\EE)$, then \eqref{eq:oNgeq} induces an orientation $o^f$ of $M^{\T}$ as the composition of the consecutive morphisms
\begin{equation}
\label{eq:EEonMGGmtrivialization}
\begin{tikzcd}
\un{\CC}\arrow[r,"o|_{M^{\T}}"]& [1cm]  \det\big(\EE|_{M^{\T}}\big)
\arrow[r,"\eqref{eq:detdistinguished}"]&[0.3cm]\det(N^{\vir})\det\big(\EE^f\big)\arrow[r,"o(N^{\vir})^{-1}\otimes \id"]&[1cm]
\det(\EE^f)\end{tikzcd}\,.
\end{equation}
The decompositin \eqref{eq:Nvirdecomp} gives a resolution
\begin{equation}
\label{eq:N+resol}
N^\geq = \begin{tikzcd} \Big\{T^{\geq}\arrow[r]&E^{\geq}\arrow[r]&\big(T^{\leq }\big)^*
\Big\}\end{tikzcd}\,.
\end{equation}
Working with it, my choice of orientation of $N^{\vir}$ differs from the one in \cite[(115)]{OT}\footnote{Strictly speaking, I am not using the orientation convention used from \cite{OT} because there is an inconsistency in it explained in Remark \ref{rem:orientation}. In truth, I am comparing $o\big(N^{\vir}\big)$ to a correction of the orientation in loc. cit.} by the sign $(-1)^{\Rk(T^\leq)}$ as shown in Lemma \ref{lem:2orientconv}.iv). This will have a positive effect on the virtual localization formula that in the form presented in \cite[(115)]{OT} seems to depend on the choice of a resolution. 

Using the orientation $o^f$, one constructs the virtual fundamental class $\big[M^{\T}\big]^{\vir}$ of the subscheme $M^{\T}\xhookrightarrow{\iota} M$. The equivariant virtual localization formula determines $[M]^{\vir}_{\T}$ in terms of $[M^{\T}]^{\vir}$ in the \textit{localized $\T$-equivariant homology} $$H^{\T}_*(M)_{\loc}:=H^{\T}_*(M)\otimes_{R}\kk\llparenthesis\Ft\rrparenthesis\,.$$
\begin{theorem}[{\cite[Theorem 7.1]{OT}}, {\cite[Proposition A.5]{Park}}]
\label{thm:eqvirloc}
    Let $t=e^{z}$ be the weight 1 line bundle for a $\T$-action on $M$ where $\T=\GG_m$. Then 
    \begin{equation}
    \label{eq:vireqloc}
        [M]^{\vir}_{\T} = \iota_*\frac{\big[M^{\T}\big]^\vir}{z^{\Rk(N^{\geq })}c_{z^{-1}}\big(N^{\geq}\big)} 
    \end{equation}
    in $H^{\T}_*(M)[z^{-1}]$. 
    
    For a higher dimensional torus $\T$, let $N^{\vir} = N^{> 0}\oplus N^{<0}$ be any splitting such that $N^{<}\cong \big(N^{>}\big)^\vee[-2]$ and the sets of $\T$-weights of the two summands are disjoint. If $o^{f}$ is the orientation \eqref{eq:EEonMGGmtrivialization} induced by $o_{N^{<}}$, then the induced class $\big[M^{\T}\big]^{\vir}$ satisfies
    \begin{equation}
    \label{eq:vireqlocgeneral}
        [M]^{\vir}_{\T} = \iota_*\frac{\big[M^{\T}\big]^\vir}{e_{\T}\big(N^{>}\big)} \in H^{\T}_*(M)_{\loc}
    \end{equation}
    where $e_{\T}(-)$ denotes the $\T$-equivariant Euler class with potential poles. 
    \end{theorem}
\begin{proof}
    The actual statement in \cite{OT} uses a different orientation $o_{N^{\geq}_{\OT}}$ of $N^{\vir}$ discussed in Remark \ref{rem:orientation}.iv). Because this orientation differs from $o(N^{\vir})$ by $(-1)^{\Rk(T^{\leq})}$ as shown in Lemma \ref{lem:2orientconv}.iv), the resulting class $\big[M^{\T}\big]^{\vir}_{\OT}$ appearing in the original formula \cite[Theorem 7.1]{OT} satisfies 
    $$
    \big[M^{\T}\big]^{\vir} = (-1)^{\Rk(T^{\leq})}\big[M^{\T}\big]^{\vir}_{\OT}\,.
    $$
   Using $\sqrt{e}\big(N^{\vir}\big) = \frac{e(T^m)}{\sqrt{e}(E^m)}$, where the orientation of $E^m = t\cdot E^{\geq }\oplus t^{-1}\big(E^{\geq }\big)^*$ induced by $o_{N^{\geq}_{\OT}}$ is $o_{E^\geq}$, the virtual localization formula can be stated as
$$
[M]^{\vir}_{\T}=\iota_*\frac{\big[M^{\T}\big]^\vir_{\OT}}{\sqrt{e}\big(N^{\vir}\big)}\,.
$$
The formulation in \eqref{eq:vireqloc} is recovered by using
$$
(-1)^{\Rk(T^\leq)}\sqrt{e}\big(N^{\vir}\big) = \frac{e\big(t\cdot T^\geq\big)\cdot e\big(t\cdot (T^{\leq})^*\big)}{e\big(t\cdot E^\geq  \big)} = z^{\Rk(N^{\geq})}\frac{c_{z^{-1}}\big(T^{\geq}\big)\cdot c_{z^{-1}}\big((T^{\leq})^*\big)}{c_{z^{-1}}\big(E^\geq\big)}\,.
$$

When $\T$ is of dimension greater than 1, there are many ways to choose a splitting into positive and negative weights in general. To prove \eqref{eq:vireqlocgeneral}, it is sufficient to prove it for a fixed such choice. One could define positivity by using the alphabetical order on weights for a fixed order of the splitting $\T = \GG_{m,1}\times\GG_{m,2}\times\cdots\times \GG_{m,n}$. The decomposition of $N^{\vir}$ generalizing \eqref{eq:Nvirdecomp} is now a direct sum of 
$$N^{>}:=t_1\cdot N^{\geq}_1\oplus t_2\cdot N^{\geq}_2 \oplus \cdots\oplus t_n\cdot N^{\geq}_n $$ 
and $(N^{>})^\vee[-2]$. Here $N^{\geq }_i$ does not contain any weights of $\GG_{m,j}$ for $j<i$ and only non-negative weights of $\GG_{m,i}$. Let $N^{\geq}_i = \Big(T_i^{\geq }\to E_i^{\geq}\to \big(T^{\leq}_i\big)^*\Big)$ be resolutions as in \eqref{eq:N+resol}. Then the orientation $\prod_{i=1}^n (-1)^{\Rk(N^{\geq}_i)} o_{N^{\geq}_i}$ induces $\prod_{i=1}^n (-1)^{\Rk(T^{\leq}_i)} o_{E^{\geq}_i}$ of $E^m$. If $[M^{\T}]^{\vir}$ is computed using the induced orientation $o^f$, then the localization formula generalizes to 
$$
        [M]^{\vir}_{\T} = \iota_*\frac{\big[M^{\T}\big]^\vir}{\prod_{i=1}^nz_i^{\Rk(N^{\geq }_i)}c_{z_i^{-1}}\big(N^{\geq}_i\big)} 
$$
which can be expressed as \eqref{eq:vireqlocgeneral}.
\end{proof}
\subsection{Park's virtual pullback}
\label{Sec:pullback}
Everything in this subsection also works $\T$-equivariantly. While I only need obstruction theories for moduli problems represented by algebraic spaces, I also require some intermediate results for stacks. This subsection develops this viewpoint using the techniques appearing in the proof of \cite[Theorem B.6]{BKP}. These ideas were suggested to me by Hyeonjun Park, and I thank him wholeheartedly for his help.  

The argument is more functorial if one uses higher cone stacks of \cite[§3]{AraPst}. More naturally these are viewed as truncations of their derived analogs. Starting from a perfect complex $\mE$ on a stack $\mM$, I consider the derived stack $\bTot_{\mM}(\mE^\vee[1])$ over $\mM$. It has the functor of points description 
    \begin{equation}
    \label{eq:derivedTot}
\bTot_{\mM}\big(\mE^\vee[1]\big)(\Spec(A^\bullet)\xrightarrow{p} \mM)  = \Map_{\mD(A^\bullet)}\big( p^*\mE[-1],A^\bullet\big)\,,
\end{equation}
where $\mD(A^\bullet)$ is the stable $\infty$-category of quasi-coherent modules of $A^\bullet$ and $\Map_{\mD(A^\bullet)}(-,-)$ denotes the mapping space. Applying the truncation functor $t_0: \textnormal{DSta}_{\CC}\to \textnormal{HSta}_{\CC}$, one recovers
$
\mathfrak{C}_{\mM}(\mE) = t_0\Big(\bTot_{\mM}(\mE^\vee[1])\Big) 
$
which is the cone stack of Aranha--Pstragowski  \cite[Definition 3.1]{AraPst}.
\begin{remark}
\label{rem:explicitcone}
    When $\mE$ has a global resolution on $\mM$ given by 
$$
\begin{tikzcd}
   \ldots\arrow[r,"d^{-2}"]& E^{-1}\arrow[r, "d^{-1}"]& E^0\arrow[r,"d^0"]& E^1\arrow[r, "0"]& 0 
\end{tikzcd}\,,
$$
    where $E^0$ is in degree 0, there is an explicit description of $\mathfrak{C}_{\mM}(\mE)$ which follows from the proof of Theorem \cite[Thm. 3.10]{AraPst} (see also \cite[Appendix B]{BKP}). Denoting by
    $$
    E_{\bullet}= (\begin{tikzcd}
   \ldots\arrow[r]& 0\arrow[r]& E_{-1}\arrow[r,"d_{-1}"]& E_0\arrow[r, "d_0"]& E_1 \arrow[r,"d_1"]& \cdots )
\end{tikzcd}\,,
    $$
    the dual of $E^\bullet$, one constructs the quotient $[E_0/E_{-1}]$ over $\mM$. There is then an action of $[E_0/E_{-1}]$ on $C_{E'}=\mathfrak{C}_{\mM}\big(E'[1]\big)$ which is the usual cone associated to 
    $
    E' = \coker(d^{-2}).
    $
    Using higher-categorical quotients, one may write
    $$
    \mathfrak{C}_{\mM}(\mE) =\Big[C_{E'}\Big/[E_0/E_{-1}]\Big]\,.
    $$
    In the special case when $\LL_{\mM}$ is the (not necessarily) truncated cotangent complex of $\mM$, I use the notation 
    $$
       \mathfrak{C}_{\mM} =  \mathfrak{C}_{\mM}(\LL_{\mM})\,.
    $$
    This is the usual \textit{intrinsic normal cone} of $\mM$. 
\end{remark}
The quadratic form in \eqref{eq:quadratic} can be either constructed directly by using the description in Remark \ref{rem:explicitcone} if it exists or by using \eqref{eq:derivedTot} as in \cite[Remark 1.8]{Park}. In the latter case, let $\mE\cong \mE^\vee[2]$ be a quasi-isomorphism, then there exists an induced quadratic form 
$$
\boldsymbol{q}: \begin{tikzcd}\bTot_{\mM}(\mE)\arrow[r]&\un{\CC}\end{tikzcd}\,.
$$
The quadratic form on $\mathfrak{C}_{\mM}(\mE)$ is its classical restriction.

In §\ref{sec:basics}, I have specified the conditions on the obstruction theories of moduli spaces. Here, I extend them to moduli stacks in a natural way so that they will be compatible with Park's virtual pullback diagrams. 
 \begin{definition}[{\cite[App. B]{BKP}}]
 \label{def:viradm}
 For an Artin stack $\mM$ and its truncated cotangent complex $\LL_{\mM}$, I say that an obstruction theory\footnote{Here I still mean that $h^0(\EE)\rightarrow h^0(\LL_{M})$ is an isomorphism and $h^1(\EE)\to h^1(\LL_{M})$ is surjective.} $\EE\to \LL_{\mM}$
 is CY4 if it is perfect of tor-amplitude $[-3,1]$ and satisfies self-duality, orientability, isotropy of cones, and evenness from Definition \ref{def:viradmcl} but for $\mM$.
 \end{definition}
 Because of working with $\infty$-stable categories, we will need a refinement of the self-duality. Fortunately, such a refinement is always available because it comes from $-2$-shifted symplectic structures. 
 \begin{enumerate}[wide, align=left]
     \item[\textit{(Higher self-duality.)}] There is an equivalence in $\mD^b(\mM)$
     $$
     \begin{tikzcd}
    \mathbb{i}^{\wedge}_q: \EE\arrow[r]&\EE^\vee[2]&
    \end{tikzcd}
     $$
    lifting $\mathbb{i}_q$ . Here, the functor $(-)^\vee$ denotes the derived $\infty$-dual of Example \ref{ex:inftyfunc} ii).
 \end{enumerate}
Next, I would like to consider the situation that $\mM$ is given a CY4 obstruction theory and 
$
f:\mN\rightarrow \mM
$
is a \textit{quasi-smooth} map from an Artin stack $\mN$. This means that there is an obstruction theory 
$
\MM \stackrel{\tau}{\rightarrow} \LL_{f}
$
with  $\MM$ being perfect of tor-amplitude $[-1,1]$. \textit{Park's virtual pullback (Pvp) diagram} then has the following form:

   \begin{equation}
   \label{eq:diagsym}
    \begin{tikzcd}[column sep=large]   \MM[-1]\arrow[d,equal]\arrow[r,"{\lambda}"]&\arrow[d,"\ov{\kappa}"]\FFna^\vee[2]\arrow[r,"{\ov{\mu}}"]&\FF\arrow[d,"\mu"]\arrow[r,"\nu"]&\MM\arrow[d,equal]\\
    \MM[-1]\arrow[d]\arrow[r,"\eta"]&\arrow[d,"f^*\psi"] f^*\big(\EE\big)\arrow[r,"\kappa"]&\FFna\arrow[r]\arrow[d,"\phi"]&\MM\arrow[d,"\tau"]\\
   \mathbb{L}_{f}[-1]\arrow[r]&   f^*\big(\mathbb{L}_{\mM}\big)\arrow[r]&\mathbb{L}_{\mN}\arrow[r]&\mathbb{L}_{f}
    \end{tikzcd}\,,
\end{equation}
where the horizontal rows are distinguished triangles, the vertical arrows induce morphisms between triangles and $\phi$ is an obstruction theory. The pullback $f^*(-)$ is understood to be derived, as I will not specify it in notation. Due to vanishing of cohomologies, it immediately follows that $\FF\xrightarrow{\phi\circ\mu}\LL_{\mN}$ is an obstruction theory. Given this diagram in $\mD^b(\mN)$, I will call it the \textit{$\infty$-Pvp diagram}.

The diagram \eqref{eq:diagsym} replaces the compatibility diagram of Manolache \cite[Definition 4.5]{Manolache} when comparing virtual fundamental classes under the virtual pullback 
$$
f^!: A_*(\mM)\longrightarrow A_*(\mN)
$$
assuming that $\mM$ and $\mN$ are separated schemes. To do so, one needs to choose orientations of $\mN$ and $\mM$ in a compatible way. I continue using the orientation conventions set in §\ref{sec:orconventions} which include the isomorphism \eqref{eq:detdistinguished} and the natural choice of orientations from Definition \ref{def:orientation}.iv). 
\begin{definition}
\label{def:orpullback}
In the situation \eqref{eq:diagsym} with $m=\Rk(\MM)$, suppose that there is an orientation $\un{\CC}\xrightarrow{o_{\mM}}(\EE)$. Then the induced orientation $o_{\mN}$ of $\FF$ determined by the Pvp diagram is defined as the composition of the consecutive morphisms
\begin{center}
\begin{tikzpicture}[descr/.style={fill=white,inner sep=1.5pt}]
      \matrix (m) [
           matrix of math nodes,
           row sep=2.2em,
            column sep=3.5em,
            text height=1.5ex, text depth=0.25ex
        ]
   {  \un{\CC} & \det\big(\MM^\vee[2]\big)\det\big(\MM\big) & \\
   \det\big(\MM^\vee[2]\big)\det\big(\EE\big)\det\big(\MM\big) &  \det\big(\MM^\vee[2]\big)\det\big(\wt{\FF}\big) & \\
            \det\big(\FF\big)\,. &  & \\
        };
    \path[overlay,->, font=\scriptsize,>=latex]
        (m-1-1) edge node[midway, above] {$o_{\MM}$} (m-1-2) 
        (m-1-2) edge[out=345,in=165]  node[descr,yshift=0.3ex] {$\id\otimes o_{\mM}\otimes \id$}  (m-2-1)
        (m-2-1) edge node[midway, above] {$\id\otimes \epsilon_{\EE,\MM}$} (m-2-2)
        (m-2-2) edge[out=345,in=165]  node[descr,yshift=0.3ex] {$\epsilon_{\MM^\vee[2], \wt{\FF}}$} (m-3-1);
       \end{tikzpicture}
       \end{center}
\end{definition}
\begin{theorem}[{\cite[Theorem 0.1]{Park}}]
\label{thm:virtualpullpush}
Suppose that there is a Pvp diagram \eqref{eq:diagsym} for the quasi-smooth morphism $f:\mN\to \mM$ of separated schemes. If $\EE$ is CY4 with orientation $o_{\mM}$, then so is $\FF$ with the induced orientation $o_{\mN}$ determined by Definition \ref{def:orpullback}. The resulting virtual fundamental classes satisfy 
\begin{equation}
\label{eq:fvirpull}
f^![\mM]^{\vir} =[\mN]^{\vir}\,.
\end{equation}
In particular, if $f$ is smooth of dimension $d$ with Euler characteristic $\chi(f)$, then 
$$
f_*\Big([\mN]^{\vir}\cap c_d(T_f)\Big)=\chi(f)[\mM]^{\vir}\,.
$$
\end{theorem}
Equation \eqref{eq:fvirpull} already implies that the orientations from Definition \ref{def:orpullback} behave functorialy when composing Pvp diagrams (at least when the classes are non-zero). To make sense of this statement and to give a direct argument for it, I will recall some conclusions from
 §\ref{sec:functorialitysympull}. There, I explain how to compose diagrams of the form \eqref{eq:diagsym} for the commutative diagram
$$
\begin{tikzcd}
     \mN_2\arrow[r,"f_2"] \arrow[rr,bend left = 50, "f"]&\mN_1\arrow[r, "f_1"]&\mM
 \end{tikzcd}
$$
of stacks. Consider a morphism of distinguished triangles
\begin{equation}
\label{eq:MM1napullback}
\begin{tikzcd}[row sep=large]
\MM_1[-1]\arrow[d,"{\tau_1[-1]}"]\arrow[r,"\eta_1"]&\arrow[d,"f_1^*(\psi)"] f^*_1\big(\EE\big)\arrow[r,"\kappa_1"]&\wt{\FF}_{1}\arrow[r]\arrow[d,"\phi_1"]&\MM\arrow[d,"\tau_1"]\\
   \mathbb{L}_{f_1}[-1]\arrow[r]&   f_1^*\big(\mathbb{L}_{\mM}\big)\arrow[r]&\mathbb{L}_{\mN_1}\arrow[r]&\mathbb{L}_{f_1}
    \end{tikzcd}
\end{equation}
and a commutative diagram 
\begin{equation}
\label{eq:MM2napullback}
\begin{tikzcd}[row sep=large]
\MM_2[-1]\arrow[d,"{\tau_2[-1]}"]\arrow[r,"\eta_2"]&\arrow[d,"f_2^*(\phi_1)"] f^*_2\big(\wt{\FF}_{1}\big)\\
\mathbb{L}_{f_2}[-1]\arrow[r]&   f_2^*\big(\mathbb{L}_{\mN_1}\big)\,,
\end{tikzcd}
\end{equation}
where the vertical morphisms in both diagrams are (pullbacks and shifts of) obstruction theories. Suppose that both diagrams are given an appropriate lift to stable $\infty$-categories (see Definition \ref{def:lifting}). One of the statements concluded in Theorem \ref{thm:functsympull} can be summarized as follows:

\begin{enumerate}
    \item[(*)]\textit{Starting from the lifts of \eqref{eq:MM1napullback} and \eqref{eq:MM2napullback} suppose that they can be completed to appropriate $\infty$-Pvp diagrams. For the induced homotopy commutative square}
\begin{equation}
\label{eq:MMtoEEfromfrunct}
\begin{tikzcd}
\MM[-1]\arrow[d,"{\tau[-1]}"]\arrow[r,"\eta"]&\arrow[d,"f^*(\psi)"] f^*\big(\EE\big)\\
\mathbb{L}_{f}[-1]\arrow[r]&   f^*\big(\mathbb{L}_{\mM}\big)\,,
\end{tikzcd}
\end{equation}
\textit{with $\MM$ being the cone of the natural map}
$$
\begin{tikzcd}
    \MM_2[-1]\arrow[r]&\MM_1
\end{tikzcd}\,,
$$
\textit{there exists a natural diagram \eqref{eq:diagsym} containing \eqref{eq:MMtoEEfromfrunct}. }
\end{enumerate}
Using the convention in Definition \ref{def:orpullback}, the following functoriality of the orientations is easy to show. 
\begin{lemma}
\label{lem:functoror}
    Given the situation \textnormal{($\star$)} above, the induced orientation by the $\infty$-Pvp diagram \eqref{eq:diagsym} is equal to the orientation on $\mN_2$ obtained by applying Definition \ref{def:orpullback} for $f_1$ and $f_2$ consecutively.
\end{lemma}
\subsection{Orientation conventions}
\label{sec:orconventions}
When working with complex determinant line bundles, I will use the conventions of \cite[Chapter 1]{KM}. They provide a functor from the category of complexes of vector bundles with quasi-isomorphisms between them to the category of $\ZZ$-graded line bundles with isomorphisms. The theory of determinant line bundles in \cite{KM} was used both by \cite{Joycehall, JTU, GJT, bojko}, and by \cite{OT} to define and study orientations on moduli spaces of sheaves on CY fourfolds. However, the two groups have chosen different definitions of orientations that can be related. As both virtual equivariant localization of \cite{OT} and virtual pull-back of \cite{Park} use the definition from \cite{OT}, I will also do so here. However, I will formulate it slightly differently and then explain why it is the same as the one of \cite{OT}.

I state the conventions explicitly only when restricted to algebraic spaces, because that is the only setting where this will be needed. However, it could also be formulated for (higher) stacks using the appropriate notions of locally free sheaves, and I will assume this more general definition when formulating the abstract framework in §\ref{sec:ingredients} and §\ref{sec:assump}.

Firstly, I will fix the rules of working with $\ZZ$-graded line bundles here. For any two line bundles $L_1,L_2$, I will use the shorter notation $L_1\,L_2$ for their tensor product $L_1\otimes L_2$. 
\begin{definition}
\label{def:Z2line}
I work with $\ZZ$-graded line bundles on an algebraic space $M$ with $|L|:M\to \ZZ$ being its degree. The morphisms between them are degree-preserving isomorphisms. In particular, there are trivial line bundles $\un{\CC}_r$ of degree $r$, and a \textit{trivialization} of a line bundle $L$ with $|L|=r$ is an isomorphism $\un{\CC}_r\xrightarrow{\sim}L$. The subscript $r$ will be omitted unless necessary as it can be deduced from degree of $L$.

The following rules are used when manipulating $\ZZ$-graded line bundles. All explicit morphisms are described Zariski locally.
\begin{enumerate}[label=\roman*)]
    \item For any two $\ZZ$-graded line bundles $L_1$ and $L_2$, I will always use the isomorphism
\begin{equation}
\label{eq:detswap}
\sigma_{L_1,L_2}: \begin{tikzcd}L_1\,L_2 \arrow[r,"\sim"]&L_2\,L_1\end{tikzcd}\,,\quad \begin{tikzcd}
  u\otimes v\arrow[r,mapsto]&(-1)^{|L_1||L_2|} v\otimes u 
\end{tikzcd}
\end{equation}
to permute their order. For another line bundle $L$, I will write the action of the dual $L^*$ on $L$ from the right. In other words, I will use the isomorphism 
\begin{equation}
\label{eq:pairingL}
p_L: \begin{tikzcd}L\otimes L^*\arrow[r,"\sim"]&\un{\CC}\end{tikzcd}\,, \quad \begin{tikzcd}u\otimes \alpha\arrow[r,mapsto]& \alpha(u) \end{tikzcd}\,.
\end{equation} 
In particular, the pairing $L^*\otimes L\xrightarrow{\sim} \un{\CC}$ is given by the composition $p_L\circ \sigma_{L^*,L}$.
\item When defining the dual $f^*:L_2^*\to L_1^*$ of an isomorphism $f:L_1\to L_2$, one needs to keep acting with the elements of the dual line bundles from the right. Abstractly, this condition requires that the composition of arrows
\begin{equation}
\label{eq:dualdetiso}
\begin{tikzcd}
    \un{\CC}\arrow[r,"p_{L_1}^{-1}"]&L_1\otimes L_1^*\arrow[r,"f\otimes (f^*)^{-1}"]&[1cm]L_2\otimes L_2^*\arrow[r,"p_{L_2}"]&\un{\CC}
\end{tikzcd}
\end{equation}
is equal to $\id_{\un{\CC}}$.
\item The isomorphism between $L$ and its double dual used here is 
\begin{equation}
\label{eq:detddual}
\diamondsuit_L: \begin{tikzcd}
    (L^*)^*\arrow[r]&L
\end{tikzcd}\,,\quad \begin{tikzcd}\big(\alpha\mapsto \textnormal{u}(\alpha)\big)\arrow[r,mapsto]&u:\alpha(u) = \textnormal{u}(\alpha)\end{tikzcd}\,.
\end{equation}
This may at first seem unusual because the orders of the symbols are interchanged but it is compatible with \cite{OT}.
\item Moreover, for any two line bundles $L_1, L_2$, I will use the isomorphism 
\begin{equation}
\label{eq:deltaisom}
\begin{tikzcd}\delta_{L_1,L_2}: \big(L_1 L_2\big)^*\arrow[r]& L_2^*\,L_1^*\end{tikzcd}
\end{equation}
where 
$$
\begin{tikzcd}\delta^{-1}_{L_1,L_2}:L_2^*\,L_1^*\arrow[r]&  \big(L_1 L_2\big)^*\end{tikzcd}\,,\quad 
\begin{tikzcd}
    \beta\otimes \alpha \arrow[r,mapsto]&\big(u\otimes v\mapsto \alpha(u)\beta(v)\big)\,.
    \end{tikzcd}
$$
\end{enumerate}
\end{definition}
Next, I will recall the conventions for working with determinant line bundles of vector bundles.
\begin{definition}
\label{def:determinants}
If $V$ is a vector bundle of rank $\Rk(V): M\to \ZZ$ and $L = \det(V)$ is the determinant line bundle, then $|L| = \Rk(V)$. The operation $\det(-)$, can be made compatible with duals and direct sums. I will follow the conventions of \cite{KM} in doing so:
\begin{enumerate}[label=\roman*)]
\item Given a short exact sequence of vector bundles
\begin{equation}
\label{eq:detexactseq}
\begin{tikzcd}
0\arrow[r]&U\arrow[r,"i"]&W\arrow[r,"p"]&V\arrow[r]&0\,,
\end{tikzcd}
\end{equation}
one defines the isomorphism 
\begin{equation}
\label{eq:exactseqtodet}
\begin{tikzcd}
    \epsilon_{U,V}:\det(W)\arrow[r,"\sim"]& \det(U) \det(V)\end{tikzcd}\qquad\begin{tikzcd}
    i(u_{\wedge})\wedge v_{\wedge}\arrow[r,mapsto]&u_{\wedge}\otimes  p(v_{\wedge})
\end{tikzcd}
\end{equation}
where $u_{\wedge} = u_1\wedge \cdots\wedge  u_{\Rk(U)}$, $v_{\wedge} = v_1\wedge \cdots \wedge v_{\Rk(V)}$, and $i,p$ act separately on each factor $u_i$, respectively $v_j$. When $U=0$, this also describes the induced isomorphism $$\det(W)\xrightarrow{\det(p)}\det(V)\,.$$
\item For any vector-bundle $V$ with $r=\Rk(V)$ and its dual $V^*$ define
\begin{equation}
\label{eq:dualtodet}
\begin{tikzcd}
    d_V: \det (V^*)\arrow[r,"\sim"]&\det(V)^*
\end{tikzcd}\,,
\end{equation}
where 
$$
d_V(\alpha_1\wedge\cdots \wedge \alpha_r)(v_r\wedge\cdots \wedge v_1) = \det\Big[\alpha_j(v_i)\Big]_{i,j=1}^r\,.
$$
\item Start with a bounded complex $E^\bullet=(\cdots\to E^{r-1} \to E^r\to E^{r+1}\to \cdots)$ where each $E^i$ is locally free, then 
\begin{equation}
\label{eq:detEbullet}
\det(E^\bullet) :=\cdots\det(E^{r-1})^{*_{r-1}} \det(E^r)^{*_r}\det(E^{r+1})^{*_{r+1}} \cdots 
\end{equation}
where 
$$
L^{*_i} =\begin{cases}
    L &\textnormal{if}\quad i  \textnormal{ is even}\,,\\
    L^* &\textnormal{if}\quad i \textnormal{ is odd}\,.
\end{cases}
$$
\item The isomorphism 
$$
\det(T): \begin{tikzcd}\det\big(E^\bullet[1]\big)\arrow[r]&\det\big(E^\bullet\big)^*\end{tikzcd}
$$
is determined by applying \eqref{eq:deltaisom}, \eqref{eq:detddual}, and \eqref{eq:detswap} multiple times to the dual of \eqref{eq:detEbullet}.
\end{enumerate}
\end{definition}
There are many potential variations of the above conventions. The reason for the present choice are the following compatibilities between the different points. 
\vspace{20pt}
\begin{remark}
    \label{rem:compatibilitysigns}
        \leavevmode
\vspace{-4pt}
    \begin{enumerate}[label=\roman*)]
    \item One needs to be consistent and treat the trivial line bundles $\un{\CC}_r$ as any other graded line bundle. The best way to achieve this in explicit computations is to use $\un{\CC}_r \cong \det(V)$ where $V = \bigoplus_{i=1}^r \mO\cdot e_i$ so that there is a canonical section $e_1\wedge\cdots \wedge e_r$ that is set equal to 1 in $\un{\CC}_r$. Additionally, the dual $\un{\CC}^*_r$ is identified with $\un{\CC}_r$ via 
    $$
    \begin{tikzcd}
    \det(V)^*\arrow[r,"d^{-1}_V"] &\det(V^*)\arrow[r,"\det(\textnormal{can})"]&[1cm]\det(V)\,.
    \end{tikzcd}
    $$
   The isomorphism $V^*\xrightarrow{\textnormal{can}}V$ is induced by the standard scalar product. For example, the map \eqref{eq:dualdetiso} becomes 
   $$
   P_{\un{\CC}_r}: \begin{tikzcd}
       \un{\CC}_r\otimes \un{\CC}_r\arrow[r]&\un{\CC}_{2r}\,, \qquad e_1\wedge\cdots \wedge e_r\otimes e_r\wedge\cdots \wedge e_1\arrow[r,mapsto]&1\,.
   \end{tikzcd}
   $$
   This is why there will be no additional signs in Definition \ref{def:orientation} compared to \cite{OT}.
    \item For a $\ZZ$-graded line bundle $L$, the expected diagram 
    $$
    \begin{tikzcd}[column sep = large, row sep = small]
         L\otimes L^*\arrow[dr,"p_L"]&\\
        &\un{\CC}\\
\arrow[uu,"\diamondsuit_L\otimes \id_{L^*}"](L^*)^*\otimes L^*\arrow[ur,"p_{L^*}\circ\sigma_{(L^*)^*,L^*}"']&
    \end{tikzcd}
    $$
    commutes only up to a sign $(-1)^{|L|}$. This is also how the discussion following \cite[(57)]{OT} implicitly defines $\diamondsuit_L$.
        \item Consider a direct sum of vector bundles $W=U\oplus V$, then one can view it as a split exact sequence \eqref{eq:detexactseq} or one with $U$ and $V$ interchanged. This leads to two choices of isomorphism $\epsilon_{U,V}$ and $\epsilon_{V,U}$. In terms of
        $$\sigma_{U,V}:=\sigma_{\det(U),\det(V)}\,,$$
        they can be related by the commutative diagram
        $$
        \begin{tikzcd}[row sep=large, column sep = 0]
            &\arrow[dl,"\epsilon_{U,V}"']\det\big(U\oplus V\big)\arrow[dr,"\epsilon_{V,U}"]&\\
            \det(U)\det(V)\arrow[rr,"\sigma_{U,V}"]&&
            \det(V)\det(U)
        \end{tikzcd}\,.
        $$
        \item The isomorphisms \eqref{eq:deltaisom} and \eqref{eq:dualtodet} have been chosen such that the following diagram commutes.
\begin{equation}
\label{eq:orpentagon}
\begin{tikzpicture}[commutative diagrams/every diagram]
\node (P0) at (90:2.3cm) {$\det(V^*)\det(U^*)$};
\node (P1) at (90+72:2cm) {$\det(V)^*\det(U)^*$} ;
\node (P2) at (90+2*72:2cm) {\makebox[5ex][r]{$\big(\det(U)\det(V)\big)^*$}};
\node (P3) at (90+3*72:2cm) {\makebox[5ex][l]{$\det\big(U\oplus V\big)^*$}};
\node (P4) at (90+4*72:2cm) {$\det\big((U\oplus V)^{*}\big)$};
\path[commutative diagrams/.cd, every arrow, every label]
(P0) edge node[swap] {$d_V\otimes d_U$} (P1)
(P1) edge node[swap] {$\delta_{U,V}^{-1}$} (P2)
(P2) edge node {$\epsilon^*_{U,V}$} (P3)
(P4) edge node {$d_{U\oplus V}$} (P3)
(P0) edge node {$\epsilon^{-1}_{V^*,U^*}$} (P4);
\end{tikzpicture}
\end{equation} 
\item Definition \ref{def:determinants} works also when $E^\bullet$ is a bounded complex with each $E^i$ a perfect sheaf. More generally, if $\EE$ is a perfect complex with $\mH^i(\EE)$ also perfect, then \cite[p. 43]{KM} define 
\begin{equation}
\label{eq:detEE}
\det(\EE) = \cdots \det\big(\mH^{i-1}(\EE)\big)^{*_{i-1}}\det\big(\mH^i(\EE)\big)^{*_i}\det\big(\mH^{i+1}(\EE)\big)^{*_{i+1}}\cdots\,,
\end{equation}
which is compatible with \eqref{eq:detEbullet}.
\item Let
$$
\begin{tikzcd}
\EE_1\arrow[r]&\EE\arrow[r]&\EE_2\arrow[r]&\EE_1[1]
\end{tikzcd}
$$
be a distinguished triangle in $D^b(M)$ where each complex is perfect with perfect cohomologies. By \cite[Corollary 2]{KM}, there is an isomorphism
\begin{equation}
\label{eq:detdistinguished}
\epsilon_{\EE_1,\EE_2}: \begin{tikzcd}
    \det(\EE)\arrow[r]&\det(\EE_1)\det(\EE_2)
\end{tikzcd}
\end{equation}
generalizing \eqref{eq:detexactseq}. It is constructed by taking the associated long exact sequence of cohomologies
$$
\begin{tikzcd}
 \mH^\bullet_{\EE_1,\EE_2} = \Big( \cdots \arrow[r]& \mH^{i}(\EE_1)\arrow[r]&\mH^i(\EE)\arrow[r]&\mH^i(\EE_2)\arrow[r]&\cdots\Big)
\end{tikzcd}\,.
$$
The terms of the resulting complex are perfect and it is acyclic, so the previous point determines an isomorphism $\det\big(\mH^\bullet_{\EE_1,\EE_2}\big)\cong \un{\CC}$ which translates into $\epsilon_{\EE_1,\EE_2}$ after reorganizing the line bundles. 
    \end{enumerate}
\end{remark}
I will use the above operations with determinants to define orientations for $O(r,\CC)$ vector bundles and self-dual complexes. The formulation is new but equivalent to the one in \cite{OT}. Because of working only with classes $\alpha\in K^0_e(\mA)$ from \eqref{eq:Keven}, I restrict to the situation when
$$
 \textnormal{the rank }r\textnormal{ is even.}
$$
In some obvious instances, I will not specify the isomorphisms $d_V$ while using them. I will also write 
$$
p_V:=p_{\det(V)}:\begin{tikzcd}\det(V)\det(V)^*\arrow[r]&\un{\CC}
\end{tikzcd}\,.
$$
\begin{definition}
    \leavevmode
\vspace{-4pt}
    \label{def:orientation}
    \begin{enumerate}[label=\roman*)]
        \item If $E$ is an $O(r,\CC)$ vector bundle, with the pairing $q: E^{\otimes 2}\longrightarrow \CC$ inducing the isomorphism 
        $$
        i_q:\begin{tikzcd}
            E\arrow[r]& E^*\,,
        \end{tikzcd}
        $$
then an \textit{orientation of $E$} is an isomorphism $o: \un{\CC}\xrightarrow{\sim}\det(E)$\footnote{Here $\un{\CC}$ is understood as a degree $\Rk(V)$ trivial line bundle for the isomorphism to be a $\ZZ$-graded one.} satisfying
\begin{equation}
\label{eq:Eor}
d_E\circ \det(i_q) =  (o^{*})^{-1} \circ o^{-1}: \begin{tikzcd}
            \det(E)\arrow[r]& \det(E)^*\,.
        \end{tikzcd}
\end{equation}
        \item If $r=2n$ and $E$ admits a maximal isotropic subbundle $V$ giving the short exact sequence
\begin{equation}
\label{eq:maxisotrop}
\begin{tikzcd}
  0\arrow[r]&  V\arrow[r,"i"]&E\arrow[r,"p"]&V^*\arrow[r]&0\,,
\end{tikzcd}
\end{equation}
then the induced orientation $o_V$ of $E$ with respect to $V$ is defined as the composition of the consecutive morphisms
$$
\begin{tikzcd}[column sep = 
large]\un{\CC}\arrow[r,"(-i)^np^{-1}_{\det(V)}"]&[0.5cm]\det(V)\det(V)^*\arrow[r,"\id_{\det(V)}\otimes d^{-1}_V"]&[1.7cm]\det(V)\det(V^*)\arrow[r,"\epsilon^{-1}_{V,V^*}"]&\det(E)\end{tikzcd}\,.
$$
\item Let $\EE$ be a perfect complex complex of rank $r$ with perfect cohomologies and a quasi-isomorphism
$$
\mathbb{i}_q: \begin{tikzcd}\EE\arrow[r]& \EE^\vee[2]\end{tikzcd}\,.
$$
An orientation of $\EE$ is an isomorphism $o:\un{\CC}\xrightarrow{\sim}\det\big(\EE\big)$ satisfying
\begin{equation}
\label{eq:Ebulletor}
\det(\mathbb{i}_q)=(o^*)^{-1}\circ o^{-1}: \begin{tikzcd}\det(\EE)\arrow[r]&\det(\EE)^*\end{tikzcd}\,,
\end{equation}
where I identified the targets via the isomorphism
$
\det\big(\EE^\vee[2]\big)\cong \det\big(\EE\big)^*
$
obtained by applying \eqref{eq:detEE} and \eqref{eq:dualtodet}.

\item Let $\EE$ be a complex as in the previous point with $r=2n$, and let $\alpha: \VV\to \EE$ be a morphism of perfect complexes with perfect cohomologies in $D^b(M)$. If there is a distinguished triangle
$$
\begin{tikzcd}
\VV\arrow[r,"{\alpha}"]&\EE\arrow[r,"{\alpha^\vee[2]}"]&\VV^{\vee}[2]\arrow[r]&\VV[1]\,,
\end{tikzcd}
$$
there is an induced orientation $o_{\VV}$ of $\EE$ with respect to $\VV$ that follows from the isomorphism
$$
\epsilon_{\VV,\VV^\vee[2]}: \begin{tikzcd}\det\big(\EE\big)\arrow[r,"\sim"]&\det\big(\VV\big)\det\big(\VV^\vee[2]\big)\end{tikzcd}
$$
recalled in Remark \ref{rem:compatibilitysigns}.v). It is defined as the composition of the consecutive arrows
$$
\begin{tikzcd}[column sep=large]
\un{\CC}\arrow[r,"(-i)^np^{-1}_{\det(\VV)}"]&[0.7cm]\det\big(\VV\big)\det\big(\VV\big)^*\arrow[r,"\sim"]&\det\big(\VV\big)\det\big(\VV^\vee[2]\big)\,,
\end{tikzcd}
$$
and I will say that $\VV$ is positive with respect to $o_{\VV}$.
    \end{enumerate}
\end{definition}
The benefit of Definition \ref{def:orientation}.i) and iii) is that compared to \cite{OT} it does not need to state their factor $(-1)^{\frac{r(r-1)}{2}}$ explicitly. The rest of this subsection is dedicated to recalling the conventions in \cite{OT} and comparing mine to theirs.
\goodbreak
\begin{remark}
    \leavevmode
\vspace{-4pt}
\label{rem:orientation}
\begin{enumerate}[label=\roman*)]
    \item Recall that \cite[Definition 2.1]{OT} defines an orientation $o_{\OT}$ of an $O(r,\CC)$ bundle $E$ as an isomorphism $o_{\OT}:\un{\CC}\to \det(E)$ with the composition of
\begin{equation}
\label{eq:OTorientation}
\begin{tikzcd}
\un{\CC}\otimes \un{\CC}\arrow[r,"o_{\OT}\otimes o_{\OT}"]&[1.2cm]\det(E)\otimes \det(E)\arrow[r,"\id\otimes (d_E\circ \det(i_q))"]&[1.7cm]\det(E)\otimes \det(E)^*\arrow[r,"p_{\det(E)}"]&[0.7]\un{\CC}
\end{tikzcd}
\end{equation}
being the fiberwise multiplication on $\un{\CC}$ scaled by $(-1)^{\frac{r(r-1)}{2}}$.
\item  Given a self-dual complex $\EE$ of rank $r$ as in Definition \ref{def:orientation}.iii), its orientations were defined in \cite[(59)]{OT} by replacing $E$ with $\EE$ in i). 
\item 
When $r=2n$ and there is a short exact sequence \eqref{eq:maxisotrop}, then there is a natural Oh--Thomas orientation $o_{V,\OT}$ of $E$ that makes $V$ into a \textit{positive maximal isotropic subbundle} in the sense of \cite[Definition 2.2]{OT}. If $\{b_1,\ldots, b_n\}$ is a Zariski local basis of $V$ and $\{d_1,\ldots,d_n\}$ its dual basis, then $o_{V,\OT}$ acts by 
$$
\begin{tikzcd}
    1\arrow[r, mapsto]&(-i)^nb_1\wedge\cdots \wedge b_n\wedge d_n\wedge\cdots \wedge d_1\,.
\end{tikzcd}
$$
This shows that 
$$o_{V,\OT} = o_V$$
for the orientation $o_V$ of Definition \ref{def:orientation}.ii). Due to Remark \ref{rem:compatibilitysigns}.i), one can also conclude that 
\begin{equation}
\label{eq:VorvsVdualor}
o_{V^*} = (-1)^{\Rk(V)}o_{V}\,. 
\end{equation}
Etalé locally, there always exists a split exact sequence \eqref{eq:maxisotrop}, so the definition of orientations from \eqref{eq:Eor} coincides with the one in \cite{OT}. I will give a more categorical proof of this in Lemma \ref{lem:2orientconv} where it also applies to complexes $E^\bullet$. 
\item The situation of Definition \ref{def:orientation}.iv) is tackled in \cite{OT} only when describing the orientations of virtual normal bundles  \eqref{eq:Nvir} in the equivariant localization formula \cite[(115)]{OT}. Recall that there is a decomposition \eqref{eq:Nvirdecomp}
such that $N^{\leq}\cong \big(N^{\geq }\big)^\vee[-2]$, and $N^{\geq}$ can be described as in \eqref{eq:N+resol}. I used 
$$o\big(N^{\vir}\big) = (-1)^{\Rk(N^{\geq})}o_{N^{\geq }}$$
as the orientation for $N^{\vir}$. This differs from the convention in \cite[Theorem 7.1]{OT} that uses the orientation of $N^{\vir}$ inducing $o_{ E^{\geq}}$ of $E^m$ by Definition \ref{def:Einducedor}. Despite this, the authors of \cite{OT} claim that these two orientations are the same due to a misuse of this definition. By reversing Definition \ref{def:Einducedor} and using $\det(i_q)\circ o_{E^\geq} = o_{E^\geq }$ stated in Lemma \ref{lem:2orientconv}, their choice requires that 
$$
N^{\geq }_{\OT} =  \Big(T^{\geq }\oplus T^{\leq }\oplus  E^{\geq} \Big)
$$
is the positive isotropic that defines the orientation
\begin{equation}
\label{eq:orOT}
o_{N^{\geq}_{\OT}}:\begin{tikzcd}\un{\CC}\arrow[r,"\sim"]&\det\big(N^{\geq }_{\OT}\big)\det\big(N^{\geq }_{\OT}\big)^*\cong \det\big(N^{\vir}\big) \end{tikzcd}\,.
\end{equation}
The second isomorphism here follows from \eqref{eq:detEbullet}.  The relation between $o_{N^{\geq }}$ and $o_{N^{\geq }_{\OT}}$ is stated in Lemma \ref{lem:2orientconv}.iv).
\end{enumerate}
\end{remark}
The next lemma is noted down for book-keeping purposes to find the sign error in \cite{OT} that prevented the localization computation in §\ref{sec:WCflags}. It describes some basic properties of the orientations used in Definition \ref{def:orientation} and relates them to the orientations used in \cite{OT}.
\begin{lemma}
    \leavevmode
\vspace{-4pt}
\label{lem:2orientconv}
\begin{enumerate}[label=\roman*)]
    \item Let $E_1, E_2$ be two orthogonal bundles with orientations 
    $$
    \begin{tikzcd}
            \un{\CC}\arrow[r,"o_i"]&\det(E_i)
        \end{tikzcd}\qquad \textnormal{for}\quad i=1,2\,.
    $$
    The product $o_1\otimes o_2: \un{\CC}\to \det(E_1)\det(E_2)$ is an orientation of $E_1\oplus E_2$. The first isomorphism uses the identification $\un{\CC}\cong \un{\CC}\otimes \un{\CC}$ of graded line bundles inverse to multiplication.  The same applies to orientations of complexes from Definition \ref{def:orientation}.iii).
    \item Let $E$ be an $O(r,\CC)$ vector bundle with an orientation $o:\un{\CC}\xrightarrow{\sim}\det(E)$, then $(o^{*})^{-1} = \det(i_q)\circ o $ is an orientation of $E^*$. Suppose that $E$ is as in \ref{eq:maxisotrop}, then $(o^{*}_V)^{-1} = o_V$. The same result holds for the orientations of complexes from Definition \ref{def:orientation}.iii) and iv).
    \item Let $\EE$ be a complex of rank $r$ as in Definition \ref{def:orientation}.iii). A trivialization $\un{\CC}\xrightarrow{o}\det(\EE)$ is an orientation in the sense of Definition \ref{def:orientation}.iii) if and only if it is an orientation in the sense of \cite[Definition 2.1, (59)]{OT}.
    \item  The two orientations constructed in Remark \ref{rem:orientation}.iv) are related by 
$$
o_{N^\geq }= (-1)^{\Rk(E^{\geq})+\Rk(T^\leq )}o_{N^{\geq}_{\OT}} \,.
$$
\end{enumerate}
\end{lemma}
\begin{proof}
    \begin{enumerate}[label=\roman*)]
        \item One sees right away that 
        $$
        \big((o_1\otimes o_2)^*\big)^{-1}\circ (o_1\otimes o_2)^{-1} = \sigma_{E_1^*,E_2^*}\circ\big((o_1^*)^{-1}\circ o_1^{-1}\big)\otimes \big((o_2^*)^{-1}\circ o_2^{-1}\big)\,.
        $$
      From \eqref{eq:Eor}, it follows that one needs to check that
        $$d_{E_1\oplus E_2}\circ \det(i_q) = \sigma_{E_1^*,E_2^*}\circ  \big(d_{E_1}\circ \det(i_{q_1})\big)\otimes  \big(d_{E_2}\circ \det(i_{q_1})\big)\,,$$
        where $i_q,i_{q_1}$, and $i_{q_2}$ are the musical isomorphisms of $E_1\oplus E_2$, $E_1$, and $E_2$ respectively. This follows immediately from applying \eqref{eq:orpentagon}. 
        \item This first result follows immediately from the defining equation \eqref{eq:Eor} and
        $$o\circ o^* = \big((o^*)^{-1}\circ o^{-1}\big)^{-1} = \det(i_q^{-1})\,.$$
        Then I use $(o^{*}_V)^{-1} = \det(i_q)\circ o_V = o_V$ where the last equality is easy to check. 
        \item Fix an isomorphism $\un{\CC}_r\xrightarrow{o}\det\big(\EE\big)$.The condition of Remark \ref{rem:orientation}.i) and ii) can be alternatively expressed as the composition of the consecutive arrows
$$
\begin{tikzcd}
\un{\CC}_{2r}\arrow[r,"p^{-1}_{\un{\CC}_r}"]&\un{\CC}_r\otimes \un{\CC}_r\arrow[r,"o\otimes o"]&\det\big(\EE\big)\otimes \det\big(\EE\big)\arrow[r,"\id\otimes \det(\mathbb{i}_q)"]&[0.8cm]\det\big(\EE\big)\otimes \det\big(\EE\big)^*\arrow[r,"p_{\det(\EE)}"]&[0.7]\un{\CC}_{2r}
\end{tikzcd}
$$
being equal to $\id_{\CC_{2r}}$. This follows from Remark \ref{rem:compatibilitysigns}.i) that implies that
$
p_{\un{\CC}_r}
$ is equal to $(-1)^{\frac{r(r-1)}{2}}$ times the fiberwise multiplication of the two copies of $\un{\CC}_r$. Comparing with the condition \eqref{eq:dualdetiso} defining $(o^{*})^{-1}$ which states that the composition of
$$
\begin{tikzcd}
    \un{\CC}_{2r}\arrow[r,"p^{-1}_{\un{\CC}_r}"]&\un{\CC}_r\otimes\un{\CC}_r\arrow[r,"o\otimes (o^*)^{-1}"]&[1cm]\det\big(\EE\big)\otimes \det\big(\EE\big)^*\arrow[r,"p_{\det(\EE)}"]&\un{\CC}_{2r}
\end{tikzcd}
$$
is equal to $\id_{\un{\CC}_{2r}}$, one sees that $(o^*)^{-1} = \det(\mathbb{i}_q)\circ o$.
\item  One can decompose \eqref{eq:orOT} into the orientations of summands as
$$
o_{N^{\geq }_{\OT}} = o_{T^\geq }\otimes o_{T^{\leq }}\otimes o_{E^{\geq}} \,.
$$
 Examining the construction of $o_{N^{\geq}}$ from Definition \ref{def:orientation}, one finds that it is equal to 
$$
o_{N^{\geq }}\otimes o_{(E^{\geq })^*}\otimes o_{(T^{\geq })^*}\,.
$$
Combined with \eqref{eq:VorvsVdualor}, this proves the statement.
\end{enumerate}
\end{proof}
\subsection{The dg-perspective}
\label{sec:dgsetup}
In this section, I will work with dg-categories as they are the appropriate framework to think in when it comes to moduli stacks of complexes of sheaves and shifted symplectic structures. They also offer a more fundamental perspective than triangulated categories. This subsection gives a review of CY4 dg-categories and introduces some notation and terminology. For those not interested in the subject, it may be worth just to assume that all the necessary stacks can be enriched to $-2$-shifted symplectic ones. Then, apart from the notation in Example \ref{ex:Rhoms}, this subsection can be skipped.

The two main examples of dg-categories I will consider are as follows:
\begin{example}
\label{ex:smoothproperdgs}
\begin{enumerate}
  \leavevmode
\vspace{-4pt}
    \item Let $X$ be a variety with the dg-category of complexes of quasi-coherent sheaves $\mD\textnormal{QCoh}(X)$ and $\mD_{\pe}(X)$ its full dg-subcategory consisting of perfect complexes (denoted by $L_{\textnormal{pe}}(X)$ in \cite[§3.5]{TV}). Then it is saturated -- therefore smooth and proper as a dg-category -- in the sense of \cite[Definition 2.4]{TV} when $X$ is smooth and proper.
    \item Take a dg-algebra $A^\bullet$ as can be constructed for example from a dg-quiver $Q^\bullet$ by forming the associated dg-path algebra $A^\bullet= \CC Q^\bullet$. It can be thought of as a dg-category with a single object. The dg-category $A^\bullet$ is proper if and only if $A^{\bullet}$ is a perfect complex. The examples considered here will be smooth but not proper. 
    \end{enumerate}
\end{example}
The authors of \cite{TV} introduce derived moduli stacks $\bM_{\mD}$ of objects in dg-categories $\mD$. However, they only parameterized certain modules of $\mD$ that can sometimes form a smaller dg-category. This is apparent from the next example discussed in more detail in my lecture notes \cite[§5.4]{bojkoLec}.
\begin{example}
\label{ex:notproperdg}
  \leavevmode
\vspace{-4pt}
\begin{enumerate}
    \item When $X$ in Example \ref{ex:smoothproperdgs}.i) is not proper but still smooth, then $$\bM_X := \bmM_{\mD_{\per}(X)}$$  parametrizes compactly supported perfect complexes.
    \item  For $A^\bullet$ as in Example \ref{ex:smoothproperdgs}.ii), the stack $\bM_{A^\bullet}$ parametrizes dg-modules of $A^\bullet$ with their underlying complexes being perfect.  This does not include $A^\bullet$ unless it is proper. 
\end{enumerate}
\end{example}
When $\mD$ is \textit{finite type} -- satisfied by all examples considered in the rest of this work, 
it is shown in \cite[Theorem 3.6]{TV} that $\bM_{\mD}$  is \textit{locally geometric derived stack locally of finite presentation}. This in particular implies that the cotangent complex $\LL_{\bM_{\mD}}$ exists and is perfect. Therefore, one can define shifted (closed) $p$-forms of $\bM_{\mD}$ and even shifted symplectic structures as in \cite{PTVV} and ask whether they exist. Calabi--Yau dg-categories give rise to shifted symplectic structures, but first, I will briefly summarize the references dealing with the former. 

Serre duality for Calabi--Yau fourfolds states for any two compactly supported perfect complexes that 
$$
\RHom(E^\bullet,F^\bullet)\simeq\RHom(F^\bullet,E^\bullet)^\vee[-4]\,.
$$
Kontsevich \cite{KontsevichCY} formalized it into a notion of \textit{Calabi--Yau triangulated categories} by requiring the above quasi-isomorphism to hold bi-functorially. When choosing a refinement for dg-categories, there are two options: 1)  \textit{left Calabi--Yau structures} as in \cite{Ginzburg, KV}, and \cite[Definition 1.2]{BD1} and 2) \textit{right Calabi--Yau structures} as in \cite{KSAinfty} and \cite[Definition 1.5]{BD1}. I will not recall the details and will instead refer to \cite{BD1} with a partial account also given in \cite[§6]{bojkoLec}.

Left Calabi--Yau structures give rise to all the other structures in this subsection. This is represented by the following diagram with arrows labeled by citations referring to each construction.
$$
\begin{tikzcd}
   &\arrow[ld,"{\textrm{\cite[Theorem 3.1]{BD1}}}"']\boxed{\begin{matrix}\textrm{left CY4} \\\textrm{structure}\\\textrm{of } \mD\end{matrix}}\arrow[rd,"{\textrm{\cite{BDII}}}"]&\\
   \arrow[d,"\cite{KSAinfty}"']\boxed{\begin{matrix}\textrm{right CY4}\\ \textrm{structure}\\\textrm{of } \mD_{\lpe}\end{matrix}}&&\boxed{\begin{matrix}-2\textrm{-shifted}\\ \textrm{symplectic structure}\\\textrm{on }\bmM_{\mD}\end{matrix}}\arrow[d,"{\textrm{\cite[Proposition 4.3]{OT}}}"]\\
\boxed{\begin{matrix} \textrm{CY4 structure}\\
   \textrm{on }H^0\big(\mD_{\lpe}\big)\end{matrix}}&&\boxed{\begin{matrix}\textrm{isotropy}\\\textnormal{of cones}\end{matrix}}
\,.\end{tikzcd}
$$
Here $\mD_{\lpe}$ denotes the full dg-subcategory of $\mD$ consisting of \textit{locally perfect} objects in the sense of \cite[§2.3]{BD1}. As an example, for $\mD = \mD_{\per}(X)$ with $X$ smooth, this subcategory contains all compactly supported perfect complexes which are precisely the objects one wants to study, in this case. The CY4 structure of the homotopy category $H^0\big(\mD_{\lpe}\big)$ is that of a triangulated category. The right vertical arrow represents the proof of the isotropy of cones for the induced obstruction theory on the moduli stacks of semistable objects. It was generalized to this setting in \cite[p. 135]{BKP}.

I will use the following terminology for abelian categories.
\begin{definition}
\label{def:CY4abelian}
\leavevmode
\vspace{-4pt}
\begin{enumerate}[label=\roman*)]
    \item Presently, I will say that a triangulated category $\mT$ is Calabi--Yau four, if $\mT = H^0\big(\mD_{\lpe}\big)$ for a left CY4 dg-category $\mD$ with a pre-triangulated $\mD_{\lpe}$.
    \item Let $\mA$ be a heart of a triangulated CY4 category $H^0\big(\mD_{\lpe}\big)$ for some left CY4 dg-category $\mD$, then I will say that $\mA$ is an abelian \textit{Calabi--Yau four} category.
\end{enumerate}
\end{definition}
Consider the corresponding embeddings of the moduli stack $\mM_{\mA}$ of objects in $\mA$ into the derived stack and its classical truncation:
$$
\mM_{\mathcal{A}}\hookrightarrow \bM_{\mD}\,,\qquad  \mathcal{M}_{\mathcal{A}}\hookrightarrow\mathcal{M}_{\mD}= t_0(\bM_{D})\,.
$$ 
 I will always work in the setting when they are open. 
\begin{example}
\label{ex:bMstacks}
  \leavevmode
\vspace{-4pt}
\begin{enumerate}
    \item When $X$ is a smooth quasi-projective CY fourfold, the dg-category $\mD_{\per}(X)$ is shown to be left CY4 in \cite[Proposition 5.12]{BD1}. This implies by the above that $\bM_X$ is $-2$-shifted symplectic. The original proof of this last fact for projective $X$ appeared in \cite{PTVV}.
    \item This time, I will consider a dg-quiver $\wt{Q}^{\bullet}$ with superpotential $\mH$ in degree $-1$. I will recall this notion in §\ref{sec:dgquivers} and give some examples there. The resulting dg-path algebra $\wt{A}^\bullet$ is known to be smooth (see e.g.
\cite[Proposition 4.3.3]{Lam}).  It can be shown to be left Calabi--Yau four by applying \cite[Theorem 1.1]{WKY} because $\mH$ can be viewed as an element of the cyclic homology group $HC_1\big(\wt{A}^\bullet\big)$ determining a deformation parameter $\eta\in H_2\big(\wt{A}^\bullet\big)$.
 The moduli stack $\bM_{A^\bullet}$ from Example \ref{ex:notproperdg}.2) is then $-2$-shifted symplectic. Restricting to the dg-modules supported in degree zero, this recovers isotropy of cones for Example \ref{ex:CYquiver}.2).
\end{enumerate}
    \end{example}
 Given an algebra object $A$ among dg-categories and its left-module dg-category $C$ with the action $A\otimes C\to C$, Brav--Dyckerhoff \cite[§2.3]{BDII} introduce the bifunctor
\begin{equation}
\label{eq:Rhombif}
\RHom_A(-,-): C ^{\op}\otimes C\longrightarrow A\,.
\end{equation}
It was denoted by $\un{\Hom}_A(-,-)$ there. For the next two examples of this functor, recall that $\textnormal{Ind}\big(\Perf(\bmM_{\mD})\big)$ denotes the dg-categorical ind-completion of the dg-category of perfect complexes on $\bmM_{\mD}$.
\begin{example}
\label{ex:Rhoms}
  \leavevmode
\vspace{-4pt}
\begin{enumerate}[label=\roman*)]
\item Set $A=\textnormal{Ind}\big(\Perf(\bmM_{\mD})\big)$ and $C= \textnormal{Ind}\big(\Perf(\bmM_{\mD})\big)\otimes \mD$, then denote the resulting functor by
$$\RHom_{\bmM_{\mD}}(-,-) := \RHom_{\textnormal{Ind}\big(\Perf(\bmM_{\mD})\big)}(-,-)\,.$$
Assuming that $\mD$ is smooth, the universal object $\mE_{\mD}$ of $\bmM_{\mD}$ exists in $\textnormal{Ind}\big(\Perf(\bmM_{\mD})\big)\otimes \mD$ by \cite[Corollary 2.6, Example 3.7]{BDII}. The cotangent complex of $\bmM_{\mD}$ can be written as
$$
\LL_{\bmM_{\mD}} = \RHom_{\bmM_{\mD}}(\mE_{\mD},\mE_{\mD})^\vee[-1] 
$$
by \cite[(3.21)]{BDII}.
\item Set $A=\textnormal{Ind}\big(\Perf(\bmM_{\mD})\big)^{\otimes 2}$ and $C= \textnormal{Ind}\big(\Perf(\bmM_{\mD})\big) ^{\otimes 2}\otimes \mD$, then \eqref{eq:Rhombif} will be denoted by $\RHom_{\bmM_{\mD}\times\bmM_{\mD}}(-,-)$. I will use it to define the Ext complex 
\begin{equation}
\boldsymbol{\mExt}_{\mD} = \RHom_{\bmM_{\mD}\times\bmM_{\mD}}(\mE_{\mD},\mF_{\mD})\,.
\end{equation}
Here $\mE_{\mD}$ denotes the universal object for the first copy of $\bmM_{\mD}$ and $\mF_{\mD}$ for the second one. Both are considered as objects in $C$. Restricted to each $\CC$-point $([E],[F])\in\bmM_{\mD}\times \bmM_{\mD}$ for a pair of objects $E,F$ of $\mD$, there is a natural quasi-isomorphism of complexes
$$
\mExt_{\mD}|_{([E],[F])} \cong \Ext^\bullet\big(E,F\big)\,. 
$$
Additionally, the pull-back along the diagonal map $$\Delta: \bM_{\mathcal{D}}\rightarrow\bM_{\mathcal{D}}\times \bM_{\mathcal{D}}
$$
can be used to relate the following complexes.
\begin{equation}
\label{eq:DeltaExtLL}
\LL_{\bM_{\mathcal{D}}}\cong \Delta^*\big(\mExt_{\mD}\big)^\vee[-1]
\end{equation}
\end{enumerate}
\end{example}
 \begin{remark}
 The connected components of $\bM_{\mD}$ are in bijection with $\Ksst(\mathcal{D})$ -- the semi-topological K-theory of a dg-category as in \cite[§4.1]{B16}. Moreover, there is always a natural map from Grothendieck's group $K^0(\mD)\to \Ksst(\mD)$.
\end{remark}

\subsection{Dg-quivers and some notation}
\label{sec:dgquivers}
The conventions in this section are taken from \cite[§10.3]{vdBerg}. A more detailed account is given in  
\cite[Definition 4.1.6 and §4.4]{Lam}, though the conventions are slightly different. See also \cite[§4.2]{bojkoLec} for another short summary. Later, I was informed that Schmiermann also revisited this perspective in his PhD thesis \cite[§4.3]{Schmier}.

Calabi--Yau four dg-quivers are defined as follows. 
\begin{definition}
 \label{def:CY4quiver}  
 \leavevmode
\vspace{-4pt}
    \begin{enumerate}[label= \arabic*)]
        \item A \textit{graded quiver} $Q^\bullet = (\Ver,\Edg^\bullet)$ consists of a set of vertices $\Ver$ and a $\ZZ$-graded set of edges $\Edg^{\bullet}$ together with maps $t,h: \Edg^{\bullet}\to \Ver$ which map each edge to the vertex at its tail and head respectively. I will denote by $|-|: \Edg^{\bullet}\to \ZZ $ the degree map.  The graded path algebra $\CC Q^\bullet$ of a graded quiver is spanned by all paths in $Q^\bullet$ including the constant ones $l_{v}$, $|l_v = 0|$ at each vertex $v\in V$. The product $p\circ q$ of two paths $p$ and $q$ is given by the concatenation when $t(p) = h(q)$ and zero otherwise. A \textit{differential graded quiver} consists of a pair $(Q^{\bullet}, d)$ where $Q^{\bullet}$ is a graded quiver and $d:\CC Q^{\bullet}\to \CC Q^{\bullet +1}$ makes $\CC Q^{\bullet}$ into a differential graded algebra. 
        \item  Fix a $Q^\bullet$ satisfying $\Edg^i=0$ except when $i\in \{-1,0\}$. If $L^{-1}$ is the set of loops in degree $-1$, define $(\Edg^*)^\bullet = \Edg^\bullet\backslash L^{-1}$ at first as a set. To each $e^*\in (\Edg^*)^\bullet$ corresponding to an $e\in \Edg^\bullet$, assign the degree $|e^*|= -2-|e|$, and set $t(e^*) = h(e), h(e^*)=t(e)$. I will also set $e^*=e$ for all $e\in L^{-1}$.  Finally, one introduces $\ov{Q}^\bullet = (\Ver,\bEdg^\bullet)$ where $\bEdg^\bullet = \Edg^\bullet\sqcup (\Edg^*)^\bullet$.
\item Choose a linear combination of cyclic paths in $\ov{Q}^\bullet$ of degree $-1$. It is represented by a degree $-1$ element $\mH\in \CC \ov{Q}^{\bullet}_{\cyc}:=\CC \ov{Q}^\bullet/[\CC \ov{Q}^\bullet, \CC \ov{Q}^\bullet]$. 
        Let $\partial^{0}_f: \CC \ov{Q}^\bullet_{\cyc}\to \CC \ov{Q}^\bullet$ for $f\in \ov{E}^\bullet$ be the \textit{circular derivative} acting by 
        $$
        \partial^0_f(p) = \sum_{\begin{subarray}{c}r,q:r\circ f \circ q\\\textnormal{summand of }p
        \end{subarray}
        }(-1)^{|r|(|f|+|q|)} q\circ r
        $$
        where the sum ranges over all appearances of $f$ in $p$.
Then $\mH$ is assumed to satisfy the \textit{master equation} 
\begin{equation}
\label{eq:mastereq}
\{\mH,\mH\}:= \sum_{e} \frac{\partial^{\circ} \mH}{\partial e}\frac{\partial^{\circ} \mH}{\partial e^*} = 0
\end{equation}
in which case\,, it is called a \textit{superpotential}.
\item A \textit{Calabi--Yau four dg-quiver} $\widetilde{Q}^\bullet$ is determined by the data $(\ov{Q}^\bullet, \mH)$ from 2) and 3). Its set of edges $\wEdg^\bullet$ additionally contains a degree $-3$ loop $o_v$ attached to each $v\in V$.
The \textit{associated differential} $d: \CC \wt{Q}^{\bullet}\to \CC \wt{Q}^{\bullet+1}$ is defined by
    \begin{align*}
    \label{eq:difffromH}
    d(e) &= \frac{\partial^{\circ} \mH}{\partial e^*}\,,  \\
    d(e^*) &=(-1)^{|e|+1}\,\frac{\partial^{\circ} \mH}{\partial e}\,.
    \numberthis
    \end{align*}
    For each $v\in V$, one further sets
$$
d(o_v) = l_v\circ \sum_{e\in \Edg^{\bullet}} [e,e^*]\circ l_v\,.
$$
\item Consider the original graded quiver $Q^\bullet$ this time with the differential determined by the first line of $\eqref{eq:difffromH}$ and $L^{-1}=\emptyset$\footnote{It would be enough to assume that the cardinality of this set is even, but for simplicity, I avoid this.}. I will say that $\wt{Q}^\bullet$ is a $-2$-shifted cotangent bundle of $Q^\bullet$ when 
$$
\frac{\partial^{\circ} \mH}{\partial e} = 0\qquad \text{for all}\quad e\in \Edg^{-1}\,.
$$
    \end{enumerate}  
\end{definition}
When dealing with Calabi--Yau dg-quivers, I will from now on use different colors to specify the degrees of arrows:
\begin{enumerate}[leftmargin = 1.5em, label=\roman*)]
\item $\xrightarrow{f}$ when $|f| = 0$,
\item \BuOr{$\xrightarrow{f}$} when $|f| = \BuOr{-1}$,
\item \Ma{$\xrightarrow{f}$} when $|f| = \Ma{-2}$,
\item \FG{$\xrightarrow{e}$} when $|e| = \FG{-3}$.
\end{enumerate}
This coloring will also be used for other elements of $\CC \widetilde{Q}^\bullet$ to communicate their degree. To get used to it, I present an example here.
\begin{example}
\label{ex:CY4quiver}
Consider the quiver $\wt{Q}^\bullet$ given by
    \begin{center}
\includegraphics[scale=1]{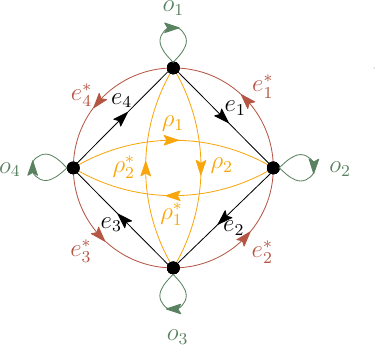}\,.
    \end{center}
Without the \FG{loops} it would be $\ov{Q}^\bullet$, and the edges not containing $(-)^*$ form the starting quiver $Q^\bullet$. The superpotential is given by
    \begin{equation}
    \label{eq:simplestCY4potential}
    \BuOr{\mH} = \BuOr{\rho^*_2}\circ e_2\circ e_1+\BuOr{\rho_1}\circ e_3\circ e_2+\BuOr{\rho_2}\circ e_4\circ e_3 - \BuOr{\rho^*_1}\circ e_1\circ e_4
    \end{equation}
from which one immediately sees that $\{\BuOr{\mH},\BuOr{\mH}\} = 0$. The differential can be computed from
$$
d(\BuOr{\rho_i}) = \frac{\partial^{\circ} \BuOr{\mH}}{\partial \BuOr{\rho_i^*}}\,,\quad d(\BuOr{\rho^*_i}) = \frac{\partial^{\circ} \BuOr{\mH}}{\partial \BuOr{\rho_i}}\,,\quad d(\Ma{e_i^*}) = -\frac{\partial^{\circ} \BuOr{\mH}}{\partial e_i}\,.
$$   
For example, one gets
$$
d(\BuOr{\rho_1}) = -e_1\circ e_4\,,\quad d(\BuOr{\rho^*_1}) = e_3\circ e_2 \,,\quad d(\Ma{e^*_1}) = e_4 \circ \BuOr{\rho^*_1} - \BuOr{\rho^*_2}\circ e_2\,. 
$$
\end{example}
       The above example is the simplest one I could find of a CY4 dg-quiver, that is not a $-2$-shifted cotangent bundle and for which the semistable degree 0 representations form non-trivial proper moduli spaces. The definition of dg-representations of a dg-quiver is going to be recalled now.
\begin{definition}
    Let $(Q^\bullet, d)$ be a dg-quiver. Its \textit{dg-representation} $M^\bullet$ is a complex of vector spaces with an action 
    $$
    \begin{tikzcd}
    \CC Q^\bullet \otimes M^\bullet \arrow[r]&M^\bullet 
    \end{tikzcd}
    $$
    which is a morphism of complexes. If  $M^\bullet$ is given by a single vector space in degree 0, I will say that it is a degree 0 representation of $(Q^{\bullet},d)$. In this case, the dimension vector $\un{d} = (d_v)_{v\in V}$ of $M$ is defined by 
    $$
    d_v= \dim_{\CC}\big(l_v\cdot M\big)\,.
    $$
    I will denote by $M$ also the induced representation of $H^0(Q^{\bullet})$ which contains the same amount of information.
\end{definition}
Another example of CY4 dg-quivers will be given in §\ref{sec:CY4dgC4}. Its finite-dimensional degree 0 representations will correspond to ideal sheaves on $\CC^4$.

       To explain why I have used the terminology of shifted cotangent bundles in Definition \ref{def:CY4quiver}, it is useful to understand the cotangent complexes on moduli stacks of representations of $\wt{Q}^\bullet$. The first step is to describe their derived refinements. Choose a dimension vector $\un{d}\in (\ZZ_{\geq 0})^{V}$ of $\wt{Q}^{\bullet}$.  Using $\SSym[S^\bullet]$ to denote the free commutative graded algebra generated by a graded vector space $S^\bullet$, set the notation
     \begin{align*}
     \label{eq:AdSdGLd}
     R^\bullet_{\un{d}} &= \SSym\big[S^\bullet_{\un{d}}\big]\qquad \textnormal{where} \quad S^\bullet_{\un{d}} = \bigoplus_{e\in \wEdg^\bullet} \Hom\big(\CC^{d_{t(e)}}, \CC^{d_{h(e)}}\big)^*\big[-|e|\big]\,,\\
     \GL(\un{d}) &= \prod_{v\in V}\GL(d_v)\,, \qquad \qquad  \,\mathfrak{gl}(\un{d}) = \textnormal{Lie}\big(\GL(\un{d})\big) \,.
     \numberthis
     \end{align*}
    Each summand in $S^{\bullet}_{\un{d}}$ could be thought of as a degree-shift of the space of coordinate functions on $\Hom\big(\CC^{d_{t(e)}}, \CC^{d_{h(e)}}\big)$. Denote by  $\{v_i\}$ the standard basis of $\CC^{v}$, so that $\{t(e)_i\}$ and $\{h(e)_j\}$ are the bases of $\CC^{d_{t(e)}}$ and $\CC^{d_{h(e)}}$ respectively. Then the standard coordinate functions on the matrices $\Hom\big(\CC^{d_{t(e)}}, \CC^{d_{h(e)}}\big)$ form the basis $\{m^e_{ji}\}$. Define the matrix $m^e := (m^e_{ji})$, and set the notation 
    \begin{equation}
    \label{eq:mpath}
    m^{e_{n}\circ \cdots \circ e_1} = m^{e_{n}} \cdots  m^{e_{1}}\,.
    \end{equation}
    Extending this notation by linearity, the differential on $R^\bullet_{\un{d}}$ is defined on the generators by
    \begin{equation}
    \label{eq:dmeji}
    d(m^e_{ji}) = m^{d_e}_{ji}\,.
    \end{equation}
     By the discussion in \cite[§4.2]{bojkoLec} (see also \cite[Theorem 2.8]{dgrepmoduli} for a more general statement), \cite[Theorem 2.1]{dgrepmoduli} implies that the derived moduli stack $\bmM_{\un{d}}$ of representations of $\wt{Q}^{\bullet}$ with the dimension vector $\un{d}$ is the stacky quotient
     \begin{equation}
     \label{eq:Mdderived}
     \bmM_{\un{d}} = \Big[\bSpec\big(R^\bullet_{\un{d}}\big)/\GL(\un{d})\Big]\,.
     \end{equation}
     The next lemma describes the induced obstruction theory on 
$$
\mM_{\un{d}}:= t_0\big(\bmM_{\un{d}}\big)\,.
$$
\begin{lemma}
\label{lem:cotangentofQ}
    Consider the complex of $H^0\big(R^{\bullet}_{\un{d}}\big)$-modules $
   \LL_{R^\bullet_{\un{d}}}\otimes^{\LL} _{R^\bullet_{\un{d}}}H^0\big(R^{\bullet}_{\un{d}}\big)\,.
    $
 It is represented by a complex with the underlying graded vector space
    \begin{equation}
    \label{eq:Dbullet}
D^\bullet := H^0\big(R^{\bullet}_{\un{d}}\big)\otimes_{\CC}S^{\bullet}_{\un{d}}\,.
    \end{equation}
    Fix an edge $|f|\in \wEdg^\bullet$ with $|f|<-1$, and consider the restriction of the differential 
    $$
\begin{tikzcd}[column sep= large]
H^0\big( R^\bullet_{\un{d}}\big)\otimes_{\CC}\Hom\big(\CC^{d_{t(f)}}, \CC^{d_{h(f)}}\big)^*\arrow[r,"\sum d_{f,g}"]& \bigoplus\limits_{g:|g| = |f|+1} H^0\big( R^\bullet_{\un{d}}\big)\otimes_{\CC}\Hom\big(\CC^{d_{t(g)}}, \CC^{d_{h(g)}}\big)^*\end{tikzcd}\,,
    $$
    where $d_{f,g}$ is the composition with the projection to each summand labeled by $g$ in the target. The maps $d_{f,g}$ are given as follows:
    \begin{itemize}
        \item Express $d(f)$ as $\sum_a p_a +R_g$, where each $p_a$ is a path containing exactly one copy of $g$ and there are no such paths as summands of $R_g$. 
        \item Each $p_a$ can be expressed as 
\begin{equation}
\label{eq:ptandg}
p_a = q_a \circ g \circ r_a
\end{equation}
where $q_a, r_a$ are degree 0 paths. 
\item Define the map 
 \begin{equation}
 \label{eq:dstarfg}
\begin{tikzcd}[column sep= large, row sep = tiny]
H^0\big( R^\bullet_{\un{d}}\big)\otimes_{\CC}\Hom\big(\CC^{d_{t(g)}}, \CC^{d_{h(g)}}\big)\arrow[r,"d^*_{f,g}"]& H^0\big( R^\bullet_{\un{d}}\big)\otimes_{\CC}\Hom\big(\CC^{d_{t(f)}}, \CC^{d_{h(f)}}\big)\\
M_g\arrow[r,mapsto]&\sum_{a}{[m^{q_a}]}\cdot M_g \cdot {[m^{r_a}]}\,,\end{tikzcd}
   \end{equation}
   where $[m^{p}]$ for a degree 0 path $p = f_n\circ \cdots \circ f_1$ is given by projecting \eqref{eq:mpath} to an $H^0(R^\bullet_{\un{d}})$-valued matrix. 
   \item The map $d_{f,g}$ is dual to $d^*_{f,g}$ as a morphism of $H^0\big(R^\bullet_{\un{d}}\big)$-modules. 
    \end{itemize}
Consider the cocone complex $E^{\bullet}$ of the natural morphism of complexes
    $$
    \begin{tikzcd}
    D^\bullet\arrow[r]& \mathfrak{gl}(\un{d})^*\otimes H^0\big(R^{\bullet}_{\un{d}}\big) \,.
    \end{tikzcd}
    $$
    Is has a natural $\GL(\un{d})$-action, thus it descends to the quotient 
    $$\mM_{\un{d}} = \Big[\Spec\Big(H^0\big( R^\bullet_{\un{d}}\big)\Big)\Big/ \GL(\un{d})\Big]\,.$$
The obstruction theory 
$$
\EE_{\un{d}} = \LL_{\bmM_{\un{d}}}|_{\mM_{\un{d}}}  
$$
is represented by $E^{\bullet}$ which satisfies $(E^{\bullet})^*[2] = E^{\bullet}$\,.
\end{lemma}
\begin{proof}
    Because $R^{\bullet}_{\un{d}}$ is cofibrant, its cotangent complex is given by the dg-module of Kaehler differentials which has the form (see, e.g., \cite[Definition 20, Exercise 21]{bojkoLec})
    $$
    \Omega^1_{R^{\bullet}_{\un{d}}} = R^{\bullet}_{\un{d}}\otimes_{\CC}S^{\bullet}_{\un{d}}\,.
    $$
Its elements are linear combinations of $s\otimes \ddr f_{ji}$ for $f\in \wEdg^{\bullet}$ and $j,i$ basis elements as in \eqref{eq:dmeji}. The action of the differential on them is determined by
$$d\big(\ddr f_{ji}\big) = \ddr \big(d(f_{ji})\big)$$
together with the Leibnitz rule. 

Because $\Omega^1_{R^\bullet_{\un{d}}}$ is a free dg-module, its restriction to $H^0(R^\bullet_{\un{d}})$ is given by \eqref{eq:Dbullet} with the induced differential. Thus, the only elements one needs to consider are $s\otimes \ddr(f)$ where $|s| = 0$. Combined with the definition of the differential, this implies that $\sum_a p_a$ will lead to the only non-zero contribution to $d_{f,g}$ for a fixed choice of $f$ and $g$. Using the matrix $\ddr m^f := \big(\ddr m^f_{ji}\big)$, I know from \eqref{eq:ptandg} that 
$$
d_{f,g}\big(\ddr m^f\big) = \sum_a  [m^{q_a}] \cdot \ddr m^g \cdot [m^{r_a}]\in D^{\bullet}
$$
where $\cdot$ represents the multiplication of matrices.
As a morphism of $H^{0}\big(A^{\bullet}_{\un{d}}\big)$-modules, this is dual to $d^*_{f,g}$ described in \eqref{eq:dstarfg}.
\end{proof}
I can now use this to prove that $\mM_{\un{d}}$ is orientable.
\begin{corollary}
\label{cor:quiveror}
   The moduli stacks $\mM_{\un{d}}$ are orientable for the obstruction theories $\EE_{\un{d}}$ given in Lemma \ref{lem:cotangentofQ}.
\end{corollary}
\begin{proof}
     It is enough to prove that the \BuOr{degree $-1$} part of $E^\bullet$ is an orientable vector bundle. 
    
    For each summand of $S^{\bullet}_{\un{d}}$ from \eqref{eq:AdSdGLd} labeled by $e\in \Edg^{-1}\backslash L^{-1}$, there is a dual term in $S^{\bullet}_{\un{d}}$ coming from $e^* \in (\Edg^*)^{-1}$ . The determinants of the corresponding terms in $E^\bullet$ cancel using the convention in Definition \ref{def:orientation}.ii). 

    I am left to show that the terms originating from some $e\in L^{-1}$ are orientable. Each contribution in $E^{-1}$ associated to $e$ takes the form $\End(\mV)\cong \mV^*\otimes \mV$ for the vector bundle $\mV$ at a vertex where the loop is. The pairing of this term with itself is given locally by
    $$  (A,B)\mapsto\Tr(A\cdot B^{\T})
    $$
where $A,B$ are matrices in $\End(\mV)$.     
    
   The determinant of $\mV^*\otimes \mV$ is trivializable because of the isomorphism
    $$
\det\big(V^*\otimes \mV\big)\cong \det(\mV^*)^n\det(\mV)^n
    $$
which does not, however, preserve degrees. Nevertheless, if $\{b_i\}^r_{i=1}$ is a Zariski local basis of $\mV$ and $\{d_i\}$ its dual basis, the element
$$
t:=d_1\otimes b_1\wedge d_1\otimes b_2\wedge\cdots \wedge d_r\otimes b_{r-1}\wedge d_r\otimes b_r
$$
glues to a global nowhere vanishing section of $\mV^*\otimes \mV$. Using the above description of the pairing, it is a straightforward computation to show that $t$ is an orientation of $\End(\mV)$. 
\end{proof}
\section{The formulation of wall-crossing via deformations of vertex algebras}
\label{sec:ingredients}
For the purpose of studying wall-crossing, Joyce introduced in \cite[§3]{Joycehall} and \cite[Def. 4.5]{JoyceWC} vertex algebras constructed on the homology of moduli stacks. After I recall the necessary assumptions on CY4 abelian categories $\mA$ to allow such a construction, I will refine Joyce's work to equivariant cohomology in §\ref{sec:equivariantVA} using the two approach discussed in Introduction.

Once the language of vertex algebras is set up, I state the general wall-crossing result that hinges on Assumptions \ref{ass:orientation}, \ref{ass:obsonflag}, and \ref{ass:welldef} being satisfied. Its application in the current form is limited due to Important Remark! \ref{impremark}, but it still holds for CY4 quivers and local CY fourfolds. 
\subsection{Assumptions on the abelian category}
\label{sec:assab}
The next definition sets some notation needed for constructing and manipulating vertex algebras a la \cite{Joycehall}. Note that all wall-crossing here will take place in abelian categories. The two natural examples that this applies to are $\Coh(X)$ for a CY fourfold $X$ and $\Rep\big(\wt{Q}^{\bullet}\big)$ -- the category of degree 0 representations of a CY4 dg-quiver $\wt{Q}^{\bullet}$. The classes $\al\in K^0(\mA)$ I will care about will always be represented by a non-zero object $E\in \mA$. The collection of such classes that additionally lie in $K^0_e(\mA)$ will be denoted by $C_e(\mA)\subset K^0_e(\mA)$.
\begin{definition}
\label{def:categoryA}
Let $\mA$ be a noetherian CY4 abelian category in the sense of Definition \ref{def:CY4abelian}. Suppose that there is an action of a torus $\T$ on $\mA$ compatible with the CY4 structure. Choose a quotient $K^0_e(\mA)\twoheadrightarrow \ov{K}(\mA)$ such that the Euler pairing on $\mA$ induces a morphisms $\chi:\ov{K}(\mA)\times\ov{K}(\mA)\to \ZZ$. The image of $C_e(\mA)$ in $\ov{K}(\mA)$ will be denoted by $\ov{C}(\mA)$.

By definition, there is a left CY4 dg-category that contains $\mA$ as a heart of $H^0(\mD)$. For its associated moduli stack $\mM_{\mD}$, I require that the natural embedding
 $$
 \begin{tikzcd}
i_{\mA}: \mM_{\mA}\arrow[r,hookrightarrow]& \mM_{\mD}
\end{tikzcd}
 $$
is open. This implies that $\mM_{\mA}$ is Artin. There is, moreover, an induced $\T$-action on $\mM_{\mA}$. 

The next few points collect some notation for structures of $\mM_{\mA}$.
\begin{enumerate}[label =\alph*)]
    \item There are maps
\begin{equation}
\label{eq:murho}
\begin{tikzcd}
\mu:\mathcal{M}_{\mA}\times \mathcal{M}_{\mA}\arrow[r] &\mathcal{M}_{\mA}\end{tikzcd}\,,\qquad\begin{tikzcd} \rho: B\GG_m\times\mathcal{M}_{\mA}\arrow[r]& \mathcal{M}_{\mA}\end{tikzcd}\end{equation}
    corresponding respectively to taking direct sums and rescaling automorphisms of objects. They are restrictions of the corresponding maps on $\mM_{\mD}$. The first map is $\T$-equivariant with respect to the diagonal action on the source. The action $\rho$ commutes with the action of $\T$.
    \item  I will denote the \textit{rigidification} (see Abramovich--Olsson--Vistoli \cite[App. A]{AOV}, Romagny \cite{Romagny}) of $\mM_{\mA}$ with respect to the action $\rho$ by
    $$
   \mM^{\rig}_{\mA} := \mathcal{M}_{\mathcal{A}}\fatslash B\GG_m\,.
    $$
    It admits the induced $\T$-action and the natural $\T$-equivariant $B\GG_m$-torsor $\Pi:\mM_{\mA}\to \mM^{\rig}_{\mA}$.
   \item Recall the complex $\mExt_{\mD}$ on $\mM_{\mD}\times \mM_{\mD}$ from Example \ref{ex:Rhoms}, and set $\mExt_{\mA} :=\mExt_{\mD}|_{\mM_{\mA}\times \mM_{\mA}}$. Consider the diagonal action of $\T$ on $\mM_{\mA}\times \mM_{\mA}$. Vertex algebras are then constructed using the $\T$-equivariant complex
   $$
\Theta_{\mA} := \mExt_{\mA}^\vee[-1]\,.
   $$
   I will make no distinction between $\Theta_{\mA}$ and its $\T$-equivariant K-theory class. When there is no chance of confusion, I will drop the subscript $\mA$.
    \item Fix a set $\msE(\mA)\subset \ov{K}(\mA)$ of \textit{emergent classes} that will satisfy assumptions specified later on. For each $\alpha\in \msE(\mA)$, I will denote by $\mM_{\alpha}$ the corresponding open and closed substack of $\mM_{\mA}$ consisting of objects of class $\alpha$. Given a perfect complex or a K-theory class on $\mM_{\mA}$, I will denote its restriction to $\mM_{\alpha}$ by appending a subscript $(-)_{\alpha}$. More generally, this will apply to products of $\mM_{\mA}$ where a restriction for each of the factors will be denoted by an additional subscript. 
 \end{enumerate}
\end{definition}
\begin{remark}
   \item From now on, I will mention $\T$-actions and equivariance only when it is strictly necessary. When applying results in this work, the reader may do so $\T$-equivariantly.
   Let $\mX$ be a moduli stack with a fixed CY4 obstruction theory $\EE$ without assuming \textit{eveness}.  I will use the notation $\vdim$ to denote the function $\Rk(\EE): \mX\to \ZZ$ constant on each connected component of $\mX$. When writing total Chern classes like $z^{\Rk(\Theta_{\mA})}c_{z^{-1}}(\Theta_{\mA})$, I will often write $z^{\Rk}$ for the leading term. 
   \item When writing total Chern classes like $z^{\Rk(\Theta_{\mA})}c_{z^{-1}}(\Theta_{\mA})$, I will often write $z^{\Rk}$ for the leading term. 
\end{remark}
The first assumption on $\mM_{\mA}$ that I will introduce concerns the existence of orientations. Restricting \eqref{eq:DeltaExtLL} along $i_{\mA}$, one obtains an obstruction theory
$$
\EE = \Delta^*\Theta_{\mA}\,.
$$
Recall that an orientation of $\mM_{\mA}$ is an isomorphism $ \un{\CC}\xrightarrow{o} \det\big(\EE\big)$ satisfying the condition of Definition \ref{def:orientation}.iii) with respect to the Serre duality induced by the CY4 structure. For each $\alpha\in \msE(\mA)$, I will denote the restriction of $o$ to $\mM_{\alpha}$ by $o_{\alpha}$. Using the notation 
\begin{equation}
\label{eq:LDdetlines}
L := \det\big(\Theta\big)\quad \textnormal{and}\quad D = \det\big(\EE\big)
\end{equation}
this means that $o_{\alpha}$ is a trivialization $\un{\CC}\xrightarrow{\sim} D_{\alpha}$. The determinant line bundles are naturally $\ZZ_2$-graded of degree 
$$
\big|L_{\alpha,\beta}\big| = -\chi(\alpha,\beta)\quad \textnormal{and}\quad \big|D_{\alpha}\big| = -\chi(\alpha,\alpha)\,.
$$

A compatibility of the orientations $o_{\alpha}$ under direct sums was observed in \cite{JTU, Joycehall}. However, the conventions used there would clash with those applied to equivariant localization in \cite{OT} and those I set in §\ref{sec:orconventions}. I correct this by following Definition \ref{def:orientation} and adding the extra sign that appeared in \eqref{eq:oNgeq} the meaning of which will be clear from the localization computation in §\ref{sec:WCflags}. For a pair of elements $\alpha,\beta\in \msE(\mA)$ such that $(\alpha+\beta)\in \msE(\mA)$, consider the equivalence
\begin{equation}
\label{eq:sumEE}
\mu^*\,\EE_{\alpha+\beta} = \EE_{\alpha}\boxplus \EE_{\be}\oplus \Theta_{\al,\beta}\oplus \sigma^*\Theta_{\beta,\al}
\end{equation}
where $\sig:\mM_{\al}\times\mM_{\be}\to \mM_{\be}\times \mM_{\al}$ switches the order of the factors. Taking determinants, it induces the isomorphism
$$
\begin{tikzcd}
\delta_{\alpha,\beta}:\mu^*D_{\alpha+\beta} \arrow[r]&\big(D_{\alpha}\boxtimes D_{\beta} \big) \otimes L_{\alpha,\beta}\otimes \sigma^*L_{\beta,\alpha}\,.
    \end{tikzcd}
$$
Due to the isomorphism $\sigma^*L_{\be,\al}\cong L^*_{\al,\be}$ that follows from
$$\sigma^*\Theta_{\beta,\alpha}= \Theta_{\alpha,\beta}^\vee[2]\,,$$ there is an induced orientation $o^{+}_{\alpha,\beta}$ of $\mu^*D_{\alpha+\beta}$ determined as the composition of the consecutive morphisms
$$
\begin{tikzcd}
\un{\CC}\arrow[r,"o_{\alpha}\boxtimes o_{\beta}"]&[0.3cm]D_{\alpha}\boxtimes D_{\beta}\arrow[r,"(-1)^{\chi(\al,\be)}o_{\Theta_{\alpha,\beta}}"]&[1.7cm]\big(D_{\alpha}\boxtimes D_{\beta}\big)\otimes L_{\alpha,\beta}\otimes \sigma^* L_{\beta,\alpha}\arrow[r,"\delta^{-1}_{\alpha,\beta}"]&\mu^*D_{\alpha + \beta}
\end{tikzcd}\,.
$$
By Lemma \ref{lem:2orientconv}.i), this is indeed an orientation. If $o_{\al+\beta}$ was already fixed, then $o^+_{\al,\be}$ can be compared to $\mu^*o_{\alpha+\beta}$ by a locally constant function $$\varepsilon_{\alpha,\beta}:\begin{tikzcd}\mM_{\alpha}\times \mM_{\beta}\arrow[r]&\{-1,+1\}\end{tikzcd}\,.$$
 
Exchanging $\alpha$ and $\beta$ in the construction of $o^+_{\alpha,\beta}$ introduces an extra sign
$$
(-1)^{\chi(\alpha,\beta)}o^+_{\alpha,\beta} = o^+_{\beta,\alpha}
$$
if one ignores which factor of $\mM_{\alpha}\times \mM_{\beta}$ appears first. 
This sign appears by Remark \ref{rem:compatibilitysigns}.i) because $o_{\sig^*\Theta_{\beta,\alpha}}$ is used instead of $o_{\Theta_{\alpha,\beta}}$. The consequence of this is the relation
\begin{equation}
\label{eq:symsign}
\varepsilon_{\alpha,\beta} =(-1)^{\chi(\alpha,\beta)}\varepsilon_{\beta,\alpha}\,.
\end{equation}
\begin{remark}
    Note that in \cite[§2.5]{JTU} and in \cite[Theorem 1.15]{CGJ}, the sign comparison is given by
\begin{equation}
\label{eq:oldsymsign}\varepsilon_{\alpha,\beta} = (-1)^{\chi(\alpha,\alpha)\chi(\beta,\beta)+\chi(\alpha,\beta)}\varepsilon_{\beta,\alpha}\,.
\end{equation}
The extra sign $(-1)^{\chi(\alpha,\alpha)\chi(\beta,\beta)}$ in the above two references is due to ignoring that the trivializations $\un{\CC}\xrightarrow{o_{\alpha}}D_{\alpha}$ originate from degree $\chi(\alpha,\alpha)$ trivial line bundles. Currently, interchanging the orientations in $o_{\alpha}\boxtimes o_{\beta}$ simultaneously swaps the two copies of $\un{\CC}$ and $D_{\al}$ with $D_{\beta}$. Thus, no additional signs appear. However, this does not play a role since I only work with $\alpha$ such that $\chi(\alpha,\alpha)\in 2\ZZ$. 
\end{remark}
The first assumption that $\mA$ needs to satisfy is orientability with further restriction on $\varepsilon_{\al,\beta}$.
\begin{assumption}
\label{ass:orientation}
I will assume that $\mM_{\alpha}$ are orientable for each $\alpha\in \msE(\mA)$.  Furthermore, there should be a uniform way of choosing orientations $o_{\alpha}$
    on each stack $\mathcal{M}_{\alpha}$ independent of the connected components of $\mM_{\alpha}$. Such a choice of $\{o_{\alpha}\}_{\alpha\in \msE(\mA)}$ should lead to $\varepsilon_{\alpha,\beta}$ that are constant for any $\alpha,\beta\in \msE(\mA)$. 
    \end{assumption}
   From the construction and the conventions in §\ref{sec:orconventions}, it is not difficult to conclude that 
      \begin{align*}
\label{eq:epsidentity}
          \varepsilon_{\alpha,0} &= \varepsilon_{0,\alpha}=1\,,\qquad
    \varepsilon_{\alpha,\beta}\varepsilon_{\alpha+\beta,\gamma}=\varepsilon_{\beta,\gamma}\varepsilon_{\alpha,\beta+\gamma}\,.
    \numberthis
\end{align*}
    \begin{example}
        \label{ex:sheavesandreps0}
    The two main examples of $\mA$ to keep in mind are $\Coh(X)$ and $\Rep\big(\wt{Q}^{\bullet}\big)$. In §\ref{sec:pairWC}, I also include different hearts $\mB$ of bounded t-structures as long as all semistable objects can be uniquely described as complexes $\{O\to F\}$ where $O,F$ lie in either of the above $\mA$ and $O$ is fixed. The openness condition for $i_{\mB}$ follows from \cite[Proposition 3.3.2]{APhearts}. Below, I will briefly discuss the two examples in more detail.
     \begin{enumerate}
         \item  For a CY fourfold $X$ the category $\Coh(X)$ embeds into $\mD_{\pe}(X)$ discussed in Example \ref{ex:smoothproperdgs}.1).  If $X$ is not proper, then one needs to restrict the description to compactly supported sheaves by replacing $\Coh(X)$ with $\Coh_{\cs}(X)$. Its moduli stack will be denoted by $\mM_X$. In both cases, orientations have been proved in \cite{CGJ,bojko} by reducing to gauge-theoretic orientations. Recently, a correction of the proof of the existence of gauge-theoretic orientations has appeared in \cite{JU}. This puts the restriction \cite[($\ast$)]{JU} on $H^3(X,\ZZ)$ in the compact case, which affects \cite{CGJ}. As shown in \cite{bojko, KT-Pardon}, the corresponding condition for a quasi-projective $X$ is expressed in terms of the compactly supported Steenrod square operation $\operatorname{Sq}^2_{\cs}: H^3_{\cs}(X,\ZZ_2)\to  H^5_{\cs}(X,\ZZ_2)$.\footnote{I have learnt of this more general formulation from Ivan Karpov.} It requires for any $\sv\in H^3_{\cs}(X,\ZZ)$ and its image $\ov{\sv}\in H^3_{\cs}(X,\ZZ_2)$ that 
         \begin{equation}
         \label{eq:H3ZZ_2vanishing}
\int_X \ov{\sv}\cup \operatorname{Sq}^2_{\cs}(\ov{\sv}) = 0\,.
         \end{equation}
         In particular, this is always satisfied when $H^3_{\cs}(X,\ZZ_2) = 0$.

        Once orientations are chosen, the compatibility required in Assumption \ref{ass:orientation} follows from \cite[Theorem 1.15.(c)]{CGJ} and  \cite[Theorem 5.4]{bojko} as long as the natural map 
         $$
         \begin{tikzcd}
         K^0\big(\Coh_{\cs,e}(X)\big)\arrow[r]& K^0_{\cs,e}(X)
         \end{tikzcd}
         $$
         factors through the chosen quotient $\ov{K}(\Coh_{\cs}(X))$.\footnote{For the non-compact case, we additionally need $H^3(X,\ZZ_2) = 0 = H^3(X,\ZZ_2)$ as explained in \cite[Theorem 1.7]{bojko}.}  Here, the additional subscript $e$ picks out the classes $\al$ with $\chi(\al,\al)$ even.
         \item Let $\wt{Q}^{\bullet}$ be a Calabi--Yau four dg-quiver as recalled in Definition \ref{def:CY4quiver}. The moduli stack of $\mA = \Rep\big(\wt{Q}^{\bullet}\big)$ embeds into $\mM_{\CC\wt{Q}^{\bullet}}$ from Example \ref{ex:notproperdg}.2). The entire stack $\mM_{\mA}$ is orientable by Corollary \ref{cor:quiveror}. Choosing $\ov{K}(\mA) =\ZZ^{\Ver}$, the Assumption \ref{ass:orientation} follows from $\mM_{\un{d}}$ being connected for each dimension vector $\un{d}$.  
     \end{enumerate} 
    \end{example}
\subsection{Additional notation}
I will introduce some notation here that will be necessary when defining equivariant homologies and deformations of vertex algebras on them. All the data given below is allowed to be $\T$-equivariant.

Let $\mX$ be a stack with a fixed CY4 obstruction theory $\EE$ without assuming \textit{eveness}.  I will use the notation 
$$
\vdim:=\Rk(\EE): \mX\to \ZZ$$
for the function constant on each connected component of $\mX$. Suppose that the connected components of $\mX$ are labeled by a set $K$, and denote them by $\mX_{\alpha}$ for $\alpha\in K$. Once I define equivariant homology $H^{\T}_*(\mX)$, I will write 
$$
H^{\T}_{*\pm \vdim}(\mX) = \bigoplus_{\alpha\in K} H^{\T}_{*\pm\vdim(\mX_{\alpha})}(\mX_{\alpha})\,.
$$
This notation will be also used in more general situations when a graded vector space is attached to each connected component.

Suppose that equivariant cohomology of $\mX$ is determined. Let $\Theta$ be a complex on $\mX$ and $z$ an additional formal variable. I will write $z^{\Rk(\Theta)}$ for the $\Rk(\Theta)$ power of $z$ on each connected component. Then the total Chern class is given by
\begin{equation}
\label{eq:totalChern}
z^{\Rk(\Theta)}c_{z^{-1}}(\Theta) = z^{\Rk(\Theta)}\sum_{i\geq 0}c_i(\Theta)z^{-i} \,.
\end{equation}
In this case, I will often write $z^{\Rk}$ for the leading term which will mean that $\Rk(-)$ is evaluated on the first K-theory class appearing to the right of it.

Suppose now that $M_*$ is a graded $R$-module (for example, $H_*(\mX)$). Recall that $R[\Ft]$ denotes the polynomial ring on $\Ft = \operatorname{Lie}(\T)$, and denote the ideal spanned by all homogeneous polynomials of degree $\geq n$ by $(\Ft)^n$. One must be careful with degrees because a homogeneous polynomial of degree $d$ on $\Ft$ will have homological degree $-2d$. Motivated by Example \ref{ex:trivialThom}, I introduce the notation 
\begin{equation}
\label{eq:Mtgr}
M_*\llbracket \Ft\rrbracket^{\gr} = \varprojlim_{n} M_*[\Ft]/(\Ft)^n
\end{equation}
where the limit is taken in graded $R$-modules and the transition morphisms are induced by the projections $R[\Ft]/(\Ft)^{n+1}\to R[\Ft]/(\Ft)^n$. Explicitly, this means that the limit is taken degree-wise:
\begin{equation}
\label{eq:Mtlimitgr}
(M_*\llbracket \Ft\rrbracket^{\gr})_i = \varprojlim_{n} (M_*[\Ft]/(\Ft)^n)_i\,.
\end{equation}
This construction gives rise to an $R[\Ft]$-module. If $\T$ is one-dimensional, I will write 
$$
M_*\llbracket u\rrbracket^{\gr}:= M_*\llbracket \Ft\rrbracket^{\gr} 
$$
where $e^u$ is a generator of characters of $\T$.

To understand the sort of elements that this limit contains, let $(m_k)_{k\geq 0}$ be a sequence of elements $m_k\in M_k$. If $u\in R[\Ft]$ has homological degree $-2$, then
\begin{equation}
\label{eq:examplepowerseries}
\sum_{k\geq 0}m_ku^k \in M_*\llbracket \Ft\rrbracket^{\gr}
\end{equation}
by definition. Most importantly, we see that one has to consider more general power-series than just polynomials. 

In what follows, I will localize only at $S$ -- the set of all homogeneous polynomials on $\Ft$. This way both
$$
\kk(\Ft)^{\gr} := S^{-1}R[\Ft]\qquad \textnormal{and}\qquad M_*\llparenthesis\Ft\rrparenthesis^{\gr}:=S^{-1}M_*\llbracket \Ft\rrbracket^{\gr}
$$
remain graded. Note that both are vector spaces over the fraction field $\kk$ of $R$, and $M_*\llparenthesis\Ft\rrparenthesis^{\gr}$ is a $\kk(\Ft)^{\gr}$-module. For a more general $R[\Ft]$-module $H_*$, I will use 
$$
(H_*)_{\loc} := S^{-1}H_*
$$
to denote the localization at $S$. 

Lastly, if a morphism $f:M_*\to N_*$ between graded $R$-modules is given, I will write
$$
\begin{tikzcd}
f\widehat{\otimes} \id: M_*\llbracket \Ft\rrbracket^{\gr}\arrow[r]&N_*\llbracket \Ft\rrbracket^{\gr}\,,
\end{tikzcd}\qquad \begin{tikzcd}
f\widehat{\otimes} \id: M_*\llparenthesis \Ft\rrparenthesis^{\gr}\arrow[r]&N_*\llparenthesis \Ft\rrparenthesis^{\gr}\,,
\end{tikzcd}
$$
for the morphisms induced by taking limits of
$$
\begin{tikzcd}
f\otimes \id_{R[\Ft]/(\Ft)^n}: M_*[\Ft]/(\Ft)^n\arrow[r]&N_*[\Ft]/(\Ft)^n
\end{tikzcd}
$$
and inverting  $S$.
\subsection{Deformations of vertex algebras}
\label{sec:equivariantVA}
I will now present two different perspectives on refining Joyce's construction of vertex algebras in \cite{Joycehall} to equivariant homology. The \textit{local approach} will be using substacks of $\T$-equivariant objects and their trivially equivariant homology, while the \textit{global approach} will work with the homology defined \cite{BB1} and summarized in Appendix \ref{app:B}. The two approaches are related by pushforward in equivariant homology, but they behave differently when interacting with invariants which I hope to revisit in future work. Examples of both equivariant homologies and their compatibility are left to §\ref{app:examples}.

In \cite{Bo26}, I will provide an axiomatization of the representation-theoretic structures obtained from both constructions. They are called \textit{$\T$-deformations of vertex algebras}, and I already explicitly formulated them in \cite[pp. 51-57]{BoHeidelberg} in a different context. It is known that deformations of vertex algebras induce Lie algebras in the usual way. 
\begin{enumerate}[wide, align=left]
\item[\textbf{Local approach}]
First, I will introduce the equivariant homology that I will work with in the local case, and I will describe the necessary operations on it.
\begin{definition}
   Using the $\T$-action on $\mM_{\mA}$, one first constructs the fixed point stack $\mM_{\mA}^{\T}$ with a natural $\T$-equivariant morphism $\iota^{\T}_{\mM_{\mA}}:\mM_{\mA}^{\T}\to \mM_{\mA}$ defined in \cite[Definition 2.3]{Romagny}. The fixed point stack is still Artin by \cite[Theorem 1]{Romagny2} because $\mM_{\mA}$ is. The underlying vector space of the vertex algebra is\footnote{From Example \ref{ex:trivialThom}, it will follow that this is the localized equivariant homology of $\mM^{\T}_{\mA}$ for the trivial $\T$-action.}
$$
V_{\loc,*-\vdim}:=H_*\big(\mM^{\T}_{\mA}\big)\llparenthesis\Ft\rrparenthesis^{\gr}\,.
$$
Here $\vdim$ is the restriction of $\Rk(\EE)$ along $\iota^{\T}_{\mM_{\mA}}$. 
\end{definition}
 The next few points discuss the necessary operations on this homology.
\begin{itemize}
    \item (\textit{Pullbacks and pushforwards}) For $\T$-equivariant maps of stacks, their $\T$-equivariant pullbacks and pushforwards are defined as pullbacks, respectively, pushforwards on the homology of fixed point stacks. This is elucidated in the next two examples. 

Since the fixed point functor $(-)^{\T}$ is functorial, one obtains the morphisms
$$
\begin{tikzcd}
\mu^{\T}:\big(\mathcal{M}_{\mA}\times \mathcal{M}_{\mA}\big)^{\T}\arrow[r] &\mathcal{M}^{\T}_{\mA}\end{tikzcd}\,,\qquad\begin{tikzcd} \rho^{\T}: B\GG_m\times\mathcal{M}^{\T}_{\mA}\arrow[r]& \mathcal{M}^{\T}_{\mA}\end{tikzcd}\,.
$$
The diagonal $T\to T\times T$ induces a morphism $\mM_{\mA}^{\T}\times \mM^{\T}_{\mA}\to \big(\mathcal{M}_{\mA}\times \mathcal{M}_{\mA}\big)^{\T}$. Composing it with $\mu^{\T}$, one gets
$$
\begin{tikzcd}
    \mu^{\T}_{\Delta}: \mM_{\mA}^{\T}\times \mM^{\T}_{\mA}\arrow[r] &\mathcal{M}^{\T}_{\mA}\,.
\end{tikzcd}
$$
I then define
$$\begin{tikzcd}\rho_*=\rho^{\T}_*\widehat{\otimes} \id:H_*\big(B\GG_m\times \mM_{\mA}^{\T}\big)\llparenthesis\Ft\rrparenthesis^{\gr}\arrow[r]&H_*\big(\mM_{\mA}^{\T}\big)\llparenthesis\Ft\rrparenthesis^{\gr}\end{tikzcd}$$
and
$$
\begin{tikzcd}
\mu_*:= \big(\mu^{\T}_{\Delta}\big)_*\widehat{\otimes} \id:H_*\big(\mM^{\T}_{\mA}\times\mM^{\T}_{\mA}\big)\llparenthesis\Ft\rrparenthesis^{\gr}\arrow[r]&H_*\big(\mM^{\T}_{\mA}\big)\otimes_{R}\kk\llparenthesis\Ft\rrparenthesis\,.
\end{tikzcd}
$$
\item (\textit{Cap product}) The cap product 
$$\begin{tikzcd}H^*\big(\mM^{\T}_{\mA}\big)\otimes_R \kk(\Ft)^{\gr}\times V_{\loc,*}\arrow[r,"\cap"]&V_{\loc,*}\end{tikzcd}$$
is determined by the usual cap product between $H^*\big(\mM^{\T}_{\mA}\big)$ and $H_*\big(\mM^{\T}_{\mA}\big)$ together with the action of $\kk(\Ft)^{\gr}$ on its module.
\item (\textit{Equivariant Künneth morphism}) I will define 
\begin{equation}
\label{eq:Kunnethlocal}
\begin{tikzcd}
\boxtimes^{\T}:=\boxtimes \otimes m : \Big(H_*\big(\mM^{\T}_{\mA}\big)\llparenthesis\Ft\rrparenthesis^{\gr}\Big)^{\otimes 2}\arrow[r]&H_*\big(\mM^{\T}_{\mA}\times \mM^{\T}_{\mA}\big)\llparenthesis\Ft\rrparenthesis^{\gr}
\end{tikzcd}
\end{equation}
constructed by first multiplying terms from $\kk(\Ft)^{\gr}$ and then applying the non-equivariant Künneth product to their coefficients in $H_*\big(\mM^{\T}_{\mA}\big)$.
\item (\textit{Equivariant total Chern classes})  For the diagonal action, let
$$\Theta^{\T}_{\mA}:= \big(\iota^{\T}_{\mM_{\mA}}\times\iota^{\T}_{\mM_{\mA}}\big)^*\Theta_{\mA}$$ be the $\T$-equivariant pull-back . It can be decomposed as 
$$
\Theta^{T}_{\mA} = \bigoplus_{u\in\chart(\Ft)}e^u\cdot \Theta^u_{\mA}\,,
$$
where $\chart(\Ft)$ is the set of $u\in R[\Ft]$ such that $e^{u}$ is an irreducible character of $\T$. This sum may be infinite, but it becomes finite when restricted to the affine schemes of an atlas of $\mM^{\T}_{\mA}\times \mM^{\T}_{\mA}$. For a homology class $w\in H_*\big(\mM^{\T}_{\mA}\times \mM^{\T}_{\mA}\big)\llparenthesis\Ft\rrparenthesis^{\gr}$, I will set
\begin{equation}
\label{eq:capwithtotalChernlocal}
\frac{w}
{z^{\Rk}c_{z^{-1}}\big(\Theta^{T}_{\mA}\big)}=w \cap z^{\Rk}c_{z^{-1}}\big(-\Theta^T_{\mA}\big) := w\prod_{u\in \chart(\Ft)}\cap(z+u)^{\Rk}c_{(z+u)^{-1}}\big(-\Theta^u_{\mA}\big)\,.
\end{equation}
in $H_*\big(\mM^{\T}_{\mA}\times \mM^{\T}_{\mA}\big)(z)\otimes_{R}\kk\llparenthesis \Ft\rrparenthesis$, where the product of operations becomes finite for each fixed $w$\footnote{This holds due to the homology of $\mM^{T}_{\mA}\times \mM^{T}_{\mA}$ being the limit of homologies of affine schemes of its atlas.}. Here, it is understood that each  factor
$$
(z+u)^{\Rk}c_{(z+u)^{-1}}\big(-\Theta^u_{\mA}\big)
$$
is expanded in $H^*\big(\mM^{\T}_{\mA}\times \mM^{\T}_{\mA}\big)(z)\llbracket \Ft\rrbracket^{\gr} $, which translates into expanding all negative powers of $(u+z)$ in $|z|>|u|$. The explicit expansion looks as follows:
\begin{equation}
\label{eq:explicitexpand}
(z + u)^{-k}= \frac{1}{z^k}\sum_{i\geq 0}{k+i-1\choose k-1 } \left(-\frac{u}{z}\right)^i\,.
\end{equation}

\end{itemize}
\begin{definition}
\label{Def:VAlocal}
    The $\T$-deformation of vertex algebras on 
    $V_{\loc,*}$ over $R\llbracket\Ft\rrbracket$ is determined by the data $\big(V_{\loc,*}, \ket{0}, T, Y\big)$ given as follows:
\begin{enumerate}
    \item Using the inclusion $0\colon*\to \mM^{\T}_{\mA}$ of the point corresponding to the zero object, set
    $$
    \ket{0}=0_*(1)\in H_0(\mM^{\T}_{\mA})\subset H_0(\mM^{\T}_{\mA})\llparenthesis\Ft\rrparenthesis^{\gr}\,.
    $$
    \item Setting  
    \begin{equation}
    \label{eq:pnofBGm}
    p^n:=[\PP^{n}]\in H_{2n}(\PP^{\infty})= H_{2n}\big(B\GG_m\big)\,,
    \end{equation}
    define the \textit{translation operator $T$} by
    $$
    e^{zT}(v) = \rho_\ast\left(\sum_{n\geq 0}p^nz^n\boxtimes v\right)\qquad\textnormal{for all}\quad v\in V_{\loc,*}\,.
    $$
    This defines a $\kk(\Ft)^{\gr}$-linear map $T: V_{\loc,*}\to V_{\loc,*+2}$ which is continuous with respect to the limit topology for \eqref{eq:Mtlimitgr}. Using the construction above, the $B\GG_m$-action on the first factor of $\mM^{\T}_{\mA}\times \mM^{\T}_{\mA}$ produces the operator $e^{zT}\otimes \id$ acting on $H_*\big(\mM^{\T}_{\mA}\times\mM^{\T}_{\mA}\big)\llparenthesis\Ft\rrparenthesis^{\gr}$.
    \item Fix $\varepsilon_{\alpha,\beta}\in \{-1,1\}$ for all $\al,\beta\in \ov{K}(\mA)$ such that it satisfies \eqref{eq:symsign} and \eqref{eq:epsidentity} . Then define an $R(\Ft)$-bilinear deformation of state-field correspondence by 
    \begin{equation}
    \label{eq:statefieldcor}
    Y(v,z)w = (-1)^{a\chi(\beta,\beta)}\varepsilon_{\alpha,\beta}\, \mu_\ast\bigg((e^{zT}\otimes \textnormal{id})\frac{v\boxtimes^{\T} w}{z^{\Rk}c_{z^{-1}}(\Theta^{\T}_{\mA})}\bigg)
    \end{equation}
   for any $v\in H_{a}(\mathcal{M}^{\T}_\alpha)\llparenthesis\Ft\rrparenthesis^{\gr}$ and $w\in H_{*}(\mathcal{M}^{\T}_{\beta})\llparenthesis\Ft\rrparenthesis^{\gr}$. 
\end{enumerate}
\end{definition}
 In \cite{Bo26}, it will be shown that the result is a $\T$-deformation of vertex algebras. The corresponding axioms are a refinement of the axioms for usual deformations of vertex algebras in \cite[§5]{HaiLi}. 
\item[\textbf{Global approach}]
Using the operations from §\ref{sec:equivariantOPs}, I extend the above construction to equivariant homology $H^{\T}_*(\mM_{\mA})$ defined in \cite{BB1} for algebraic stacks (see Appendix \ref{app:B} for a summary). In Example \ref{ex:trivialThom}, $H^{\T}_*(\mM_{\mA})$ is described completely explicitly in the case of representations of quivers.
\begin{definition}
\label{Def:VAglobal}
    The $\T$-deformation of vertex algebras on 
    $$V_{*}= H^{\T}_{*+\vdim}(\mM_{\mA})$$
    is determined by the data $\big(V_{*}, \ket{0}, T, Y\big)$ given as follows:
\begin{enumerate}
    \item Using the inclusion $0\colon*\to \mM_{\mA}$ of the point corresponding to the zero object, set
    $$
    \ket{0}=0_*(*)\in H^{\T}_0(\mM_\mA)\,.
    $$
    \item Define the translation operator $T$ by
    $$
    e^{zT}(v) = \rho_\ast\big(\sum_{n\geq 0}p^nz^n\boxtimes v\big)\qquad\textnormal{for all}\quad v\in H^{\T}_*(\mM_{\mA})\,.
    $$
    The operator $T$ is well-defined over any ring R by \cite[(3.16)]{Joycehall}. I will use the same convention as in the local approach for writing $e^{zT}\otimes \id$ acting on $H^T_*\big(\mM_{\mA}\times\mM_{\mA}\big)$. 
    \item Define an $R[\Ft]$-bilinear deformation of state-field correspondences by 
    \begin{align*}
    Y(v,z)w &= (-1)^{a\chi(\beta,\beta)}\varepsilon_{\alpha,\beta}\, \mu_\ast\bigg((e^{zT}\otimes \textnormal{id})\frac{v\boxtimes^{\T} w}{z^{\Rk}c_{z^{-1}}(\Theta_{\mA})}\bigg)
    \end{align*}
   for any $v\in H^{\T}_{*}(\mathcal{M}_\alpha)$ and $w\in H^{\T}_{*}(\mathcal{M}_{\beta})$. The total Chern class was defined in \eqref{eq:totalChern}.
\end{enumerate}
Replacing $V_*$ by the newly defined
$$
V_{\loc,*} = H^{\T}_{*+\vdim}\big(\mM_{\mA}\big)_{\loc}\,,
$$
one still recovers a $\T$-deformation vertex algebra.
\end{definition}

 \subsection{Lie algebras for both approaches}
For both the local and global approach, the quotients
$$
L_* := V_{*+2}/T(V_*) \,,\qquad L_{\loc,*} := V_{\loc,*+2}/T(V_{\loc,*})
$$
carry natural Lie brackets by the construction in \cite{Borcherds}. They are the $R[\Ft]$, respectively, $R(\Ft)$-bilinear maps acting by
\begin{equation}
\label{eq:LiefromVA}
\big[\ov{v},\ov{w}\big] = \ov{[z^{-1}]\Big\{Y(v,z)w\Big\}}\,,
\end{equation}
where $\ov{(-)}$ denotes the projection to the quotient and $[z^{-1}]\{-\}$ is picking out the residue. 

Apart from being a Lie algebra, the quotient $L_*$ is related to $H^{\T}_*\big(\mM^{\rig}_{\mA}\big)$. Let $\EE^{\rig}$ be the obstruction theory on $\mM^{\rig}_{\mA}$ induced by \eqref{eq:rigidifyingEE}. Fix $\alpha\in \ov{K}(\mA)$ such that 
\begin{equation}
\label{eq:alphabetarigcon}
\exists\, \be\in \ov{K}(\mA): \qquad \chi(\beta,\alpha)\neq 0\,,\qquad \mM^{\T}_{\beta}\neq  \emptyset\,.
\end{equation}
Then by Proposition \ref{prop:rigidisom}, there is a natural isomorphism
\begin{equation}
\label{eq:Lalphaiso}
H^{\T}_{* }\big(\mM^{\rig}_{\al}\big) \cong H^{\T}_{* +2}\big(\mM_{\al}\big)/T H^{\T}_{* }\big(\mM_{\al}\big)\,.
\end{equation}
The same applies to the homology of $\big(\mM^{\T}_{\mA}\big)^{\rig}$ after restricting $\vdim$ associated to $\EE^{\rig}$ to $\mM^{\T}_{\mA}$, . To connect everything to wall-crossing for enumerative invariants, one considers the following examples of elements in $L_{\loc,*}$.
\begin{example}
\label{ex:mMvir}
\begin{enumerate}[wide, align=left]
Some of the invariants that will appear in wall-crossing will be defined as follows, depending on the chosen approach.
\item[\textbf{Local approach}]Let $M^{\sig}_{\al}$ be a separated algebraic moduli space of $\sigma$-stable objects in $\mA$ of class $\al$ such that $\big(M^{\sig}_{\al}\big)^{\T}$ is proper and has the resolution property. 
Suppose that it comes with a fixed choice of a $\T$-equivariant open embedding
$$
\begin{tikzcd}
M^{\sig}_{\al}\arrow[r,"\iota"]&\mM^{\rig}_{\mA}\,.
\end{tikzcd}
$$
Then there is a CY4 obstruction theory on $M^{\sig}_{\al}$ obtained by pulling back the one on $\mM^{\rig}_{\mA}$. This produces the equivariant virtual cycle $\big[M^{\sig}_{\al}\big]^{\vir}_{\T}\in A^{\T}_*\big(M^{\sig}_{\al},R\big)$ where $R=\ZZ[2^{-1}]$. In general, the whole space $M^{\sig}_{\al}$ will not be proper, so one needs to work with fixed points to apply the equivariant cycle-class map. When $M^{\sig}_{\al}$ is additionally finite-type, the image of $\big[M^{\sig}_{\al}\big]^{\vir}_{\T}$ in $A^{\T}_*\big(M^{\sig}_{\al}\big)_{\loc}$ induces a cycle in 
$
A_*\Big(\big(M^{\sig}_{\al}\big)^{\T}\big)\otimes_R \kk(\Ft)^{\gr}
$
under the isomorphism of these groups proved in \cite[Theorem 6.3.5]{Kreschequiv}. Applying the cycle-class map, this produces
\begin{equation} 
\label{eq:Msigallocvir}
\big[M^{\sig}_{\al}\big]^{\vir}_{\T,\loc}\in H_*\Big(\big(M^{\sig}_{\al}\big)^{\T}\Big)\otimes_{R} \kk(\Ft)^{\gr}\,.
\end{equation}

The map $\iota$ induces the bottom horizontal morphism in the diagram
$$
\begin{tikzcd}
&\big(\mM^{\T}_{\mA}\big)^{\rig}\arrow[d]\\
\big(M^{\sig}_{\al}\big)^{\T}\arrow[ur,dashed, "\iota^{\T}"]\arrow[r]& \big(\mM^{\rig}_{\mA}\big)^{\T}
\end{tikzcd}\,.
$$
The right vertical map is the natural one, and one assumes the existence of the diagonal map $\iota^{\T}$. If $\al$ additionally satisfies \eqref{eq:alphabetarigcon}, then $$\iota^{\T}_*\big[M^{\sig}_{\al}\big]^{\vir}_{\T,\loc}\in L_{\loc,*}$$
by the local version of the isomorphism \eqref{eq:Lalphaiso}. This will be the invariant appearing in wall-crossing. I will often abuse notation and drop $\iota^{\T}_*$ from the above. When it is clear that I am working in the local equivariant approach, I will also omit $\T$ and $\loc$ from the subscript.
\item[\textbf{Global approach}]
Pushforward of \eqref{eq:Msigallocvir} along the equivariant inclusion $\big(M^{\sig}_{\al}\big)^{\T}\hookrightarrow M^{\sig}_{\al}$ leads to
$\big[M^{\sig}_{\al}\big]^{\vir}_{\T}\in H^{\T}_*\big(M^{\sig}_{\al}\big)_{\loc}.$
It is used to define
\begin{equation}
\label{eq:Msigalphaglobal}
\iota_*\big[M^{\sig}_{\al}\big]^{\vir}_{\T}\in H^{\T}_{\vdim}\big(\mM^{\rig}_{\mA}\big)_{\loc}.
\end{equation}
When $\al$ satisfies \eqref{eq:alphabetarigcon}, this induces $\iota_*\big[M^{\sig}_{\al}\big]^{\vir}_{\T}\in L_{\loc,*}$ which will be usually denoted without $\iota_*$ and the subscript $\T$, when the $\T$-action is fixed. 
\end{enumerate}
\end{example}
\end{enumerate}
\begin{remark}
    \label{rem:noncompactnondegenera}
    In the case that $X$ is a non-compact CY fourfold, the condition \eqref{eq:alphabetarigcon} would be too restrictive in the setting of Example \ref{ex:sheavesandreps0} and Example \ref{ex:mMvir}. In the case of $\CC^4$, the pairing $\chi_{\cs}$ on $K^0_{\cs}(\CC^4)$ vanishes altogether. However, there is a more suitable pairing 
    $$
    \begin{tikzcd}
    \chi:K^0(X)\times K^0_{\cs}(X)\arrow[r]&\ZZ \,.
    \end{tikzcd}
    $$
    The arguments in the proof of Proposition \ref{prop:rigidisom} still apply to $\al\in K^0_{\cs}(X)$ such that 
    $$
    \exists \,\beta\in K^0(X): \qquad \chi(\be,\al)\neq 0\,,
    $$
    and $\beta$ can be represented by a $\T$-equivariant object $B$. Thus I will replace \eqref{eq:alphabetarigcon} by this condition in such situations. 
\end{remark}
\begin{example}
\label{ex:toric}
    Here is why it will be useful to work with both approaches simultaneously. Suppose we want to study sheaves on a toric variety $X$ of dimension $d$ with its $\T:=(\CC^*)^d$ action. The moduli stack of all sheaves $\mM_X$ carries the induced $\T$-action and $\mM^{\T}_X$ parametrizes \textit{equivariant sheaves with their equivariant structure}. Roughly, we can treat such situations explicitly by
    \begin{enumerate}
        \item restricting to the toric charts while relating the local approach vertex algebra to a product of vertex algebras on the charts,
        \item fixing a chart and mapping local constructions to global ones. Global approach on a chart can be made explicit in terms of quivers as in §\ref{sec:CY4dgC4}.
    \end{enumerate}
\end{example}
\subsection{Wall-crossing theorem}
In the previous subsections, I introduced and recalled the necessary language for the general statement of wall-crossing. However, the theorem will be contingent on several assumptions that will need to be checked before applying it. These assumptions are imposed on the stability conditions, the moduli spaces and stacks of (semi)stable objects, and the obstruction theories on them. The orientability assumption was already discussed in §\ref{sec:assab}. The rest of them will be presented in §\ref{sec:assump}. Potential $\T$-actions are always included even when not mentioned. Thus, equivariant virtual fundamental classes from Example \ref{ex:mMvir} will be simply denoted by $[M]^{\vir}$ omitting $(-)_{\T}$.

In the current work, I choose to work in the largest possible generality when it comes to stability conditions. Thus, I use Joyce's \textit{weak stability conditions} from \cite{JoIII} that impose almost no restrictions. Recall that such a stability condition $\sigma$ is determined by a map 
$$
\begin{tikzcd}
\phi: \ov{C}(\mA)\arrow[r]&S
\end{tikzcd}
$$
where $(S,\leq)$ is a totally ordered set. For it to be called a weak stability condition, it should satisfy 
$$
\phi(\al_1)\leq \phi(\al)\leq \phi(\al_2)\qquad\text{or}\qquad \phi(\al_1)\geq \phi(\al)\geq \phi(\al_2)
$$
whenever $\al=\al_1 + \al_2$ for $\al_1,\al_2\in \ov{C}(\mA)$. In this case, an object $E\in \mA$ is said to be   $\sigma$-semistable if for any short exact sequence 
$$
\begin{tikzcd}
    0\arrow[r]&E_1\arrow[r]&E\arrow[r]&E_2\arrow[r]&0
\end{tikzcd}
$$
in $\mA$, one has $\phi(E_1)\leq \phi(E_2)$. If the strict inequality always holds, then $E$ is said to be $\sig$-stable. I will denote by
$$\mM^{\sig}_{\al}\subset \mM_{\mA}$$
the substacks consisting of $\sig$-semistable objects of class $\al\in \ov{K}(\mA)$. I will always fix a connected set $W$ of such weak stability conditions for a fixed $S$ with $W$ being a finite-dimensional manifold. The last condition is not necessary but makes formulating assumptions later easier. 

Let $\underline{\alpha} = (\alpha_1,\ldots,\alpha_n)$ for $\alpha_i\in \ov{C}(\mA)$ satisfy $\sum_{i=1}^n \al_i= \al$. Then I will say that it is a partition of $\alpha\in \ov{C}(\mA)$ which I will denote by $\underline{\alpha} \vdash_{\mA} \alpha$. For two weak stability conditions $\sig,\sig'\in W$, Joyce defined universal wall-crossing coefficients $\wt{U}(\un{\al};\sig,\sig')$ in \cite{JoIV} and \cite[§3.2]{JoyceWC}. The wall-crossing formulae appearing already in \eqref{eq:introWC} are expressed in terms of these coefficients. Recall, from §\ref{sec:introWCeqhom} that the precise wall-crossing statement may at first depend on some $k\in K$ for a countable index set $K$. The precise data attached to each $k$ is described in Definition \ref{Def:quiverpairs}. The most general result answering Problem (I) and (II) from §\ref{sec:introWCeqhom} is as follows.
\begin{theorem}
\label{thm:familyWC}
    If Assumption \ref{ass:orientation}, \ref{ass:stab}, and \ref{ass:obsonflag}, hold, then for any $\alpha\in \msE(\mA)$  and $k\in K$, there are connected open subsets $W_{\al,k}\subset W$ for which
    $
    \big\langle\Masi\big\rangle^k\in L_{\loc,0} 
    $
   are defined in Definition \ref{def:Masik}. These classes satisfy the following properties.
\begin{enumerate}[label=\roman*)]
    \item If there are no strictly $\sigma$-semistables of class $\al$, then 
    $$ \big\langle\Masi\big\rangle^k = \Masvir\in L_{\loc,0}$$
    where the right-hand-side was defined in Example \ref{ex:mMvir} which applies to the algebraic spaces of stable objects $M^{\sig}_{\al}$ by Assumption \ref{ass:stab}.g).
\item If $\sigma,\sigma'$ both lie in $W_{\al,k}$ for $\alpha\in \msE(\mA), k\in K$, then the wall-crossing formula
\begin{equation}
\label{eq:MasiWC}
\langle \mathcal{M}^{\sigma'}_{\alpha}\big\rangle^k = \sum_{\begin{subarray}a 
\underline{\alpha}\vdash_{\mA}\alpha 
\end{subarray}}\widetilde{U}(\underline{\alpha};\sigma,\sigma') \Big[\Big[\cdots \Big[\big\langle\mM^{\sigma}_{\alpha_1}\rangle^k,\langle \mathcal{M}^{\sigma}_{\alpha_2}\big\rangle^k\Big],\ldots\Big],\big\langle\mM^{\sigma}_{\alpha_n}\big\rangle^k\Big]
\end{equation}
holds in $L_{\loc, 0}$.
\end{enumerate}    
If Assumption \ref{ass:welldef} holds, which requires that
$$
   \big\langle\Masi\big\rangle^{k_1} =    \big\langle\Masi\big\rangle^{k_2} \qquad \text{for all}\quad k_1,k_2\in K\,,
$$
then ii) applies to any $\sig,\sig'\in W$ and $\al\in \msE(\mA)$. In this case, I write $ \big\langle\Masi\big\rangle^{k}= \big\langle\Masi\big\rangle$ for all $k\in K$.  
\end{theorem}
If $\mA = \Coh(X)$ one can consider a different heart $\mB$ of $D^b(X)$. In §\ref{sec:pairWC}, I allow situations in which $\mM^{\sig}_{\beta}$ for all $\beta\in \msE(\mB)$ and $\sig\in W$ parametrize objects in $\mB\subset D^b(X)$ that can be uniquely represented by a complex
\begin{equation}
\label{eq:VOtoF}
\begin{tikzcd}
\big\{V\otimes O\arrow[r,"s"]&F\big\}\,,
\end{tikzcd}
\end{equation}
where $V$ is a vector space, $O$ is a fixed sheaf, and $F$ may vary in $\Coh_{\cs}(X)$. Assumption \ref{ass:obsonflag} is then replaced by its slight modification in Assumption \ref{ass:pairWC}.c) and d) which, to be formulated, requires Assumption \ref{ass:pairWC}.a) and b). Consequently, the results of Theorem \ref{thm:familyWC} still hold in this scenario.
\begin{theorem}
\label{thm:BWC}
Let $\mB$ be as above and suppose that Assumption \ref{ass:orientation}, \ref{ass:stab}, and \ref{ass:pairWC} hold (see also Remark \ref{rem:afterpairs} for the non-compact case). Then there exist $\big\langle \mM^{\sig}_{\beta}\big\rangle^k$ for any $\beta\in \msE(\mB),\sig\in W$ and some $k\in K$. The analogue of Theorem \ref{thm:familyWC}.i) and ii) is satisfied by these classes. If Assumption \ref{ass:welldef} holds for $\big\langle \mM^{\sig}_{\beta}\big\rangle^k$, then ii) holds for all $\sig,\sig'\in W$ and $\be\in \msE(\mB)$.
\end{theorem}
\begin{importantremark}
\label{impremark}
    These theorems are not as general as the author originally believed. While they hold for CY4 quivers, they do not apply to $\Coh(X)$ for a general Calabi--Yau fourfold due to a gap in the proof of Assumption \ref{ass:obsonflag} explained in §\ref{sec:sheaves}. This assumption requires that CY4 obstruction theories on enhanced master spaces discussed in §\ref{sec:introproofWC} exist. Unfortunately, this appears to be too strict of a condition due to Example \ref{ex:counterexamples}. A collection of different $X$, for which Assumption \ref{ass:obsonflag} can still be shown to hold, is given in Example \ref{ex:fibrations}. It includes fibrations over smooth bases. In §\ref{sec:localCY4}, the above obstruction theories are additionally constructed in the case of any local CY fourfold.
\end{importantremark}
The major two consequences of the above remark are summarized in the next corollary. Note that its second statement implies that the part of assumptions in Theorem \ref{thm:BWC} dealing with obstruction theories can be removed for local CY fourfolds. A local CY fourfold here means the total space $\Tot(K_Y)$ of a canonical bundle of a quasi-projective smooth 3-fold $Y$ with a (potential) $\T$-action. 
\begin{corollary}
\label{cor:WCconsequences}
   Let $\mA$ be 
\begin{enumerate}
  \item the category $\Coh_{\cs}(X)$ of compactly supported sheaves on a local Calabi--Yau fourfold $X = \Tot(K_Y)$ as in Example \ref{ex:introspectral}\,,
    \item the category $\Rep\big(\wt{Q}^{\bullet}\big)$ of degree 0 representations of a CY4 dg-quiver $\wt{Q}^{\bullet}$ (see §\ref{sec:dgquivers})\,.
\end{enumerate}
Suppose that Assumption \ref{ass:stab} on stability conditions $W$ and $\msE(\mA)$ is satisfied and the Assumption \ref{ass:obsonflag}.a) holds, then Theorem \ref{thm:familyWC}.i) and ii) apply to well-defined classes $\big\langle \mM^{\sig}_{\al}\big\rangle\in L_{\loc,0}$ independent of choices $k\in K$. 
For $X$ as in 1), the conclusion of Theorem \ref{thm:BWC} holds for $B,W,\msE(\mB)$ satisfying Assumption \ref{ass:stab} and \ref{ass:pairWC}.a), b) and c) (or rather their modification in Remark \ref{rem:afterpairs}). I.e., equivariant wall-crossing holds for objects of the form \eqref{eq:VOtoF} in this situation.
\end{corollary}
\begin{proof}
    In Example \ref{ex:sheavesandquivers}, I recall that Assumption \ref{ass:orientation} is satisfied for both 1) and 2). Assumption \ref{ass:obsonflag} is checked for 1) in Proposition \ref{prop:localCY} and for 2) in Proposition \ref{prop:construct}. The former proposition already includes the obstruction theories for pair wall-crossing of Theorem \ref{thm:BWC} and that $\big\langle \mM^{\sig}_{\al}\big\rangle\in L_{\loc,0}$ are well-defined for local CY fourfolds. For dg-quivers, these classes are constructed uniquely to begin with as follows from Example \ref{ex:sheavesandquivers}.ii).
\end{proof}

\section{Assumptions}
\label{sec:assump}
From now on, fix some data $\mA, \ov{K}(\mA), \msE(\mA)$ from Definition \ref{def:categoryA}. I have already stated Assumption \ref{ass:orientation} on existence of orientations. In Theorem \ref{thm:familyWC}, I have referenced three other assumptions-
\begin{enumerate}[label=\roman*)]
    \item Assumption \ref{ass:stab} that describes how $\mM^{\sig}_{\al}$ behave for $\al\in \msE(\mA)$ while changing $\sig\in W$,
    \item Assumption \ref{ass:obsonflag} which leads to the construction of CY4 obstruction theories on enhanced master spaces,
    \item Assumption \ref{ass:welldef} that requires that $\big\langle \mM^{\sig}_{\al}\big\rangle^k\in L_{\loc,0}$ defined in §\ref{sec:InvDef} are inedpendent of $k\in K$.
\end{enumerate}
These assumptions are formulated in detail in this section. Although they are related, they differ noticably from the assumptions presented in \cite[§5.1, §5.2]{JoyceWC}. For example, I no longer use Joyce's \textit{framing functors} from \cite[Assumption 5.1 (g)]{JoyceWC}. I replace them with \textit{ample framing objects} in Definition \ref{Def:quiverpairs} leading to Assumption \ref{ass:obsonflag} and generalized in Assumption \ref{ass:pairWC}. They are the natural generalization of sufficiently positive line bundle $L$ used in \cite{JoyceSong} to define Joyce--Song stable pairs. Ample framing objects give rise to the most natural framing functors that appear in the literature. Other constructions are usually introduced ad hoc and can be replaced by ample framing objects if additional clarity of the proofs is needed.
\subsection{Assumptions on stability conditions}
\label{sec:assstability}
Recall that the connected space of stability conditions $W$ was considered with a fixed smooth structure. Moreover, the maps 
 $$
 \begin{tikzcd}
 W\ni\sigma \arrow[r,mapsto]&\phi(\al)\in S
\end{tikzcd}
$$
should be continuous with respect to the order topology on $S$. For any $\al,\beta\in \msE(\mA)$, this ensures that 
\begin{itemize}
    \item the set $W_{\phi(\al)<\phi(\be)} = \{\sig \in W \colon \phi(\al)<\phi(\be) \}$ is open,
      \item the set $W_{\phi(\al)=\phi(\be)} = \{\sig \in W \colon \phi(\al)=\phi(\be) \}$ is closed.
\end{itemize}
In the cases one usually considers, the sets $W_{\phi(\al)=\phi(\beta)}$ are finite unions of real codimension 1 loci in $W$ or the entirety of $W$ (e.g. when $\beta =n\al$).  I will make a more general assumption below. 

To write down enumerative wall-crossing formulae, one needs to make sure that there are finitely many non-zero coefficients  $U(\un{\al};\sig_0,\sig_1)$ in the sum in Theorem \ref{thm:familyWC}. For a fixed $\al\in \msE(\mA)$, this requires the set of
\begin{itemize}
    \item[$\circledast$] partitions $\un{\be}\vdash_{\mA} \al$ of size $p\geq 1$ with further partitions $\un{\al}_j\vdash_{\mA} \beta_j$ for $j=1,\ldots, p$ such that there is a pair $\sig,\sig'\in W$ with $$\mM^{\sig}_{\al_{jk}}\neq \emptyset\,,\qquad U(\un{\al}_j;\sig,\sig')\neq 0$$
    for all parts $\al_{jk}$ of $\un{\al}_j$ and all $j=1,\ldots,p$. Furthermore at $\sig'$, the equality of phases
    $$
    \phi'(\be_1)=\phi'(\be_2)=\cdots =\phi'(\be_p)
    $$
    holds.
\end{itemize}
to be finite. Note that this implies that the set of $(\al_1,\al_2)\vdash_{\mA}\al$ such that for some $\sig\in W$ one has $\phi(\al_1)=\phi(\al) = \phi(\al_2)$ and $\mM^{\sig}_{\al_{1}}\neq \emptyset\neq \mM^{\sig}_{\al_{2}}$ is finite due to $U(\al_i;\sig,\sig) = 1$. Thus there are finitely many sets $W_{\phi(\al_1)=\phi(\al_2)}$ where objects in class $\al$ can be destabilized.

The assumptions below also guarantee the existence of virtual fundamental classes $[M^{\sig}_{\al}]^{\vir}\in L_{\loc,0}$ for all $\al\in \msE(\mA)$ and $\sig$ with no strictly semistable objects. \begin{assumption}
\label{ass:stab}
Let $W$ be the manifold of weak stability conditions fixed above. 
\begin{enumerate}[label =\alph*)]
\item If $(\al_1,\al_2) \vdash_{\mA}\al\in \msE(\mA)$ and $\phi(\al_1)=\phi(\al_2)$ for some $\sig\in W$, then $\al_1,\al_2\in \msE(\mA)$.
\item For each pair of $\sig,\sig'\in W$, there is a continuous path $\gamma_{(-)}:[0,1]\to W$ between them, such that the open subset 
$$
\gamma_{\al<\beta} = \big\{t\in [0,1]\colon \phi_t(\al)<\phi_t(\beta) \textnormal{ for }\gamma_t\big\}
$$
is a finite union of connected components for each $\al,\beta\in \msE(\mA)$. 

I will say that $\gamma_{(-)}$ satisfies (P), if for each $\al\in \msE(\mA)$ and its data $\circledast$ such that $\phi_t(\beta_j)\neq \phi_t(\al)$ for some $j\in\{1,\ldots,p\}$ and $t\in [0,1]$ the set 
\begin{equation}
\label{eq:betajequalsal}
\gamma_{\be_j=\al} = \big\{s\in [0,1]\colon \phi_s(\be_j)=\phi_s(\al)\text{ for  }\gamma_s \text{ and all }j=1,\ldots,p\big\}
\end{equation}
is finite.
\item The set of $\circledast$ is finite for a fixed $\al\in \msE(\mA)$.
\item 
Suppose that (P) holds for a fixed path $\gamma_{(-)}$. Fix $\al\in \msE(\mA)$ and $t\in [0,1]$. Define the subset $B_{\al,t}\subset \msE(\mA)$ consisting of all $\beta$ such that 
\begin{equation}
\label{eq:Balt}
(\beta,\al-\be)\vdash_{\mA}\al\,,\qquad \phi_{t}(\be) = \phi_{t}(\al-\beta)\,,\quad\text{and}\quad \mM^{\gamma_t}_{\be}\neq \emptyset \neq \mM^{\gamma_t}_{\al-\be}\,.
\end{equation}
If $t'\neq t$ and $t'\notin \gamma_{\be_j=\al}$ for all \eqref{eq:betajequalsal}, there exists a group homomorphism $\lambda^{t,t'}_{\al}: \ov{K}(\mA)\to \RR$ such that $\lambda^{t,t'}_{\al}(\al)=0$ and 
\begin{equation}
\label{eq:lambdaiffphi}
\lambda^{t,t'}_{\al}(\beta)<\lambda^{t,t'}_{\al}(\al-\beta)  \qquad \iff\qquad \phi_{t'}(\beta)<\phi_{t'}(\al-\beta)\qquad \text{for all}\quad \beta\in B_{\al,t}\,.
\end{equation}

If $\ga_{(-)}$ does not satisfy $(P)$, $\lambda^{t,t'}_\al$ should exist for all $t'\neq t$.
\item For each $\sig\in W$, there exists a \textit{rank function} $\rk_{\sig}: \msE(\mA)\to \NN$ satisfying $\rk_{\sig}(\al) = \rk_{\sig}(\al_1) + \rk_{\sig}(\al_2)$ whenever $(\al_1,\al_2)\vdash_{\mA}\al \in \msE(\mA)$ and $\phi(\al_1) = \phi(\al_2)$.
\item Any $\sigma\in W$ satisfies the Harder-Narasimhan property in $\mA$.
    \item For each $\alpha\in \msE(\mA)$ and $\sigma\in W$, the substacks $\mM^{\sig}_{\al}\subset \mM_{\mA}$ are open and finite type. If additionally there are no strictly $\sigma$-semistable objects of class $\alpha$, the rigidification $$M^\sig_{\al} := \big(\mM^\sig_{\al}\big)^{\rig}$$ is a separated algebraic space with proper fixed-point locus $\big(M^{\sig}_{\al}\big)^{\T}$ that has the resolution property.
\end{enumerate}
\end{assumption}
\begin{example}
\label{ex:assstab}
To explain the motivation behind the above assumptions, I provide two examples. Note that Assumption \ref{ass:stab} is slightly different from \cite[Assumption 5.2 and 5.3]{JoyceWC}, especially Assumption \ref{ass:stab}.d). This is to accommodate example ii) below, which otherwise would not fit into the original definition. 
\begin{enumerate}[label=\roman*)]
\item Let $\mA=\Coh(X)$ and fix an ample divisor class $H\in \textnormal{NS}(X)\otimes \RR$. As one varies $H$, the associated slope and Gieseker stabilities change. The assumptions were addressed in this case in \cite[§7]{JoyceWC}. The main limitation is Assumption \ref{ass:stab}.g) which relies on boundedness of some families of pure sheaves. For this reason, the results in \cite{JoyceWC} are limited to sheaves of dimension 0, 1, and 4. 
\item Let $\{0\}=\ov{K}_{0}\subset\ov{K}_{1}\subset \cdots \subset \ov{K}_{l}  $ be a finite filtration of $\ov{K}(\mA)$ and set $\ov{G}_i = \ov{K}_i/\ov{K}_{i-1}$ for $i=1,\ldots, l$. A \textit{weak stability condition} is determined by its central charge which consists of group homomorphisms $Z_i: \ov{G}_i\to \CC$. The phase $\phi(\al) \in (0,1]$ for $\al\in \ov{K}_i\backslash \ov{K}_{i-1}$  is defined by the equality $$Z_i(\ov{\al})= m(\al)e^{i\pi\phi(\al)}$$
where $\ov{\al}$ is the projection of $\al$ to $\ov{G}_i$ and $m(\al)\in \RR$ is required to be greater than 0. 

Thus, each $\al\in \msE(\mA)$ determines a ray $\RR_{>0}\cdot e^{i\pi\phi(\al)}$. The purpose of introducing property (P) in Assumption \ref{ass:stab}.b) is for d) to be satisfied by weak stabilities. When $\phi(\al)=\phi(\be)$, the associated rays overlap and I could not find any natural way of constructing morphisms \eqref{eq:lambdaiffphi} in general. In \cite{Bo26}, I will address this issue when (P) is satisfied and apply it to stable pair wall-crossing. In §\ref{sec:localCY4}, I will only discuss the simplest case where (P) trivially holds. 
\end{enumerate}
\end{example}

\subsection{Assumptions on obstruction theories of enhanced master spaces}
\label{sec:assonFlagMS}
In \cite[Def. 5.5]{JoyceWC}, Joyce introduced certain abelian categories of representations of acyclic quivers with the vector spaces at sinks of the quiver replaced by sheaves. Fortunately, one does not need to work in this generality to prove wall-crossing. I only record the necessary quivers which lead to flag bundles over the moduli spaces of interest and enhanced master spaces. These types of spaces were originally used by Mochizuki in \cite{mochizuki} to prove his version of wall-crossing for sheaves on surfaces. However, phrasing them in the language of quivers makes everything more explicit and combinatorial. To describe the appropriate CY4 obstruction theories on enhanced master spaces, these quivers will later be completed to CY4 dg-quivers. In the case of wall-crossing for $\Rep\big(\wt{Q}^{\bullet}\big)$, this becomes the most useful formulation. 

Everything in the following definition can be stated $\T$-equivariantly in an obvious way, so I will specify it only when needed. 
\begin{definition}
\label{Def:quiverpairs}
  I continue working with $\mA$, $\mM_{\mA}$ from Definition \ref{def:categoryA} satisfying Assumption \ref{ass:orientation} and $\msE(\mA),W$ satisfying Assumption \ref{ass:stab}. Let  $\{\mA_k\}_{k\in K}$ for a countable set $K$ be a collection of exact subcategories of $\mA$ closed under direct summands. For each $\alpha\in \msE(\mA),\sig\in W$, there should exist a $k\in K$ such that the associated substack $\mM_{\mA_k}\subset \mM_{\mA}$ is open and  
\begin{equation}
\label{eq:Walk}
\mM^{\sigma}_{\alpha} \subset \mM_{\mA_k}\,.
\end{equation}
In addition, there should exist connected open subsets $W_{\al,k}\subset W$ for each $\al\in \msE(\mA), k\in K$ where \eqref{eq:Walk} holds. Their union should be the entire $W$. 

 For the above choice of $\{\mA_k\}_{k\in K}$, the following additional data should be specified:
  \begin{enumerate}[label=\alph*)]
      \item an exact fully faithful $\T$-equivariant embedding $\mA_k\hookrightarrow \mB_k$ into an exact category with a $\T$-action $\mB_k$ consisting of objects $P$ each fitting into a unique up to isomorphisms exact triple
      \begin{equation}
      \label{eq:exacttriple} 
      \begin{tikzcd}
          E\arrow[r]&P\arrow[r]&(1,0)_k\otimes V
      \end{tikzcd}
      \end{equation}
      where $(1,0)_k$ is a $\T$-equivariant object of $\mB_k$
      and $V$ is a vector space. Each morphism  $f:P_1\to 
 P_2$ fits into a unique commutative diagram 
 $$
      \begin{tikzcd}
\arrow[d,"f_{\mA}"]E_1\arrow[r]&\arrow[d,"f"]P_1\arrow[r]&\arrow[d,"f_{\Vect}"](1,0)_k\otimes V_1\\
            E_2\arrow[r]&P_2\arrow[r]&(1,0)_k\otimes V_2
      \end{tikzcd}
      $$
      after the exact triples for $P_1$ and $P_2$ are chosen. In this case, $f$ is uniquely determined by $f_{\mA}$ and $f_{\Vect}$.
   I will assume that $\mB_k$ additionally comes with a choice of an exact fully faithful $\T$-equivariant functor\footnote{This terminology implies that the class of exact triples in the image is the same as the image of the class of triples in $\mB_k$.} $\mB_k\hookrightarrow \ov{\mB_k}$ where $\ov{\mB_k}$ is a Calabi--Yau four abelian or triangulated category with a $\T$-action such that $(1,0)_k$ is a spherical object. The induced inclusions of stacks
\begin{equation}
\label{eq:MNovNk}
\mM_{\mA_k}\hookrightarrow \mM_{\mB_k}=:\mN_k\hookrightarrow\mM_{\ov{\mB_k}}=:\ov{\mN_k}
\end{equation}
are required to be open embeddings. Additionally, the data is required to satisfy
\begin{equation}
\Ext^{i}_{\ov{\mB_k}}\big((1,0)_k, E\big)=0 \qquad\textnormal{for all} \quad E\textnormal{ in }\mA_k\,,\quad i\neq 1\,.
\end{equation}
When, additionally $E\in\al$ for $\al\in \msE(\mA)$, set $V_k(E):=\Ext^{1}_{\ov{\mB_k}}\big((1,0)_k, E\big)$. Then
\begin{equation}
\label{eq:chialk}
\chi\big(\al(k)\big):=\dim\big(V_k(E)\big)>0
\end{equation}
is constant in such $E$, and there is a functorial injection 
$
\Hom_{\mA}(E,E)\hookrightarrow V^*_k(E)\otimes V_k(E)
$.
      \item an integer $r\geq 2$ and the quivers 
      \begin{center}
      \includegraphics{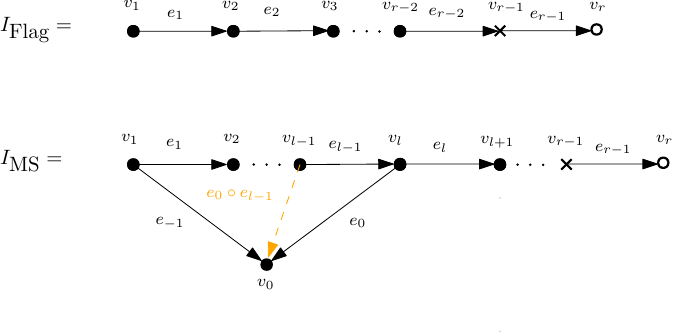}
      \end{center}
      with two types of distinguished vertices labeled by $\times$ and $\circ$. I will call the special vertex $\times$ the \textit{connecting} vertex as this will be its role later on. When working with $I_{\MS}$, I will additionally require that $r\geq 3$ and $l<r$. The quivers $\mathring{I}_{\MS}$ and $ \mathring{I}_{\Flag}$ result from erasing the vertex $\circ$ of $I_{\MS}$ and $I_{\Flag}$ and the arrow pointing to it. 
     
      For $I=I_{\MS}, I_{\Flag}$, I will simply write $\mathring{I} = (\mathring{\Ver}, \mathring{\Edg})$ to denote $\mathring{I}_{\MS}$, respectively $\mathring{I}_{\Flag}$. For a fixed choice of $I$, I will consider the exact categories $\mB_{I_k}$ whose objects are triples
      $(\un{V},\un{m},P)$ where 
      \begin{itemize}
      \item $(\un{V}, \un{m})$ is a representation of $\mathring{I}$ with the vector space at the vertex $v_{r-1}$ denoted by $V_{\times}$,
      \item $P$ is an object of $\mB_k$ determined by an exact sequence of the form
      $$
      \begin{tikzcd}
          E\arrow[r]&P\arrow[r]&(1,0)\otimes V_{\times}
      \end{tikzcd}\,.
      $$
      \end{itemize}
      The $\T$-action on $\mB_{I,k}$ is inherited from $\mB_k$, such that it acts trivially on $(\un{V},\un{m}$). The morphisms between the objects $(\un{V}_1,\un{m}_1,P_1)$ and $(\un{V}_2,\un{m}_2,P_2)$ are determined by pairs of morphisms $f_{\mathring{I}}\in \Hom_{\mathring{I}}(\un{V}_1,\un{V}_2)$ and $f\in\Hom_{\mB_k}(P_1,P_2)$ whose restrictions to morphisms $V_{\times, 1}\to V_{\times, 2}$ are equal. The stacks and the moduli spaces above are both endowed with the induced $\T$-actions.
  \end{enumerate}
Let $\sig\in W_{\al,k}$ for fixed $\al\in \msE(\mA), k\in K$ and choose $\phi\in S$. Define a new exact subcategory $\mA^{\sig}_{k,\phi}\subset \mA_k$ consisting of $\sig$-semistable objects $E$ with $\phi(E) = \phi$ together with the zero object. It is exact and closed under taking direct summands. This is also true for the subcategories $\mB^{\sig}_{I_k,\phi}\subset \mB_{I_k}$ consisting of objects $(\un{V}, \un{m}, P)$ with a short exact sequence $E\to P\to (1,0)\otimes V_{\times}$ such that $E\in \mA^{\sig}_{k,\phi}$. For the purpose of writing stability conditions on $\mB^{\sig}_{I_k,\phi}$, one replaces $\ov{K}(\mA)$  by $\ov{K}(\mB_{I_k}):= \ZZ^{\mathring{\Ver}}\times\ov{K}(\mA) $ the elements of which are denoted by $(\un{d},\al)$ with $\un{d}$ a dimension vector of $\mathring{I}$ and $\alpha\in \ov{K}(\mA)$.
  
Let $\lambda: \ov{K}(\mA)\to \RR$ be a group homomorphism together with a vector $\un{\mu}\in \RR^{\mathring{\Ver}_{\MS}}$ satisfying 
\begin{equation}
\label{eq:muicond}
1\gg \mu_1\gg\mu_2\gg\cdots \gg\mu_{r-1}>0\,,\qquad  0>\mu_0\gg -1
\end{equation}
where the precise conditions are specified in the proof 
in \cite[§10]{JoyceWC}. The stability condition $\sigma^{\lambda}_{\mu}$ on $\mB^{\sig}_{k,\phi}$ in terms of
\begin{equation}
\label{eq:barphi}
\phi^{\lambda}_{\mu}: \ov{K}(\mB_{I_k})\begin{tikzcd}[ampersand replacement=\&]\arrow[r]\&\ov{\RR}\end{tikzcd}\,,\qquad
\phi^{\lambda}_{\mu}(\un{d},\al) =\begin{cases}\frac{\lambda(\al) + \un{d}\cdot \un{\mu}}{\rk_{\sig}(\al)}&\text{if }\al\neq 0\,,\\
\infty&\text{if }\al= 0, \un{d}\cdot \un{\mu}>0\,,\\
-\infty&\text{if }\al= 0, \un{d}\cdot \un{\mu}\leq 0\,.
\end{cases}
\end{equation}
Here, the extended real numbers $\ov{\RR}$ are considered with the standard total order.

 The moduli stacks of objects in $\mB_{I_k}$ are denoted by $\mN_{I_k}$ or simply $\mN_{\Flag_k}$ and $\mN_{\MS_k}$ for the two different uivers. I will also use $\mB_{\Flag_k},\mB_{\MS_k}$ for the categories $\mB_{I_k}$.  Fixing a class $(\un{d},\alpha)\in \ov{K}(\mB_{I_k})$ with $\al\neq 0$, I will denote by $\mN_{\un{d},\al}\subset \mN_{I_k}$ the substacks of the objects of this class. When $\phi(\al) = \phi$, their substacks of $\sigma^{\lam}_{\mu}$-semistable objects in $\mB^{\sig}_{k,\phi}$ will be labeled by $\mN^{\sig^{\lam}_{\mu}}_{\un{d},\al}$. When there are no strictly $\sigma^{\lambda}_{\mu}$-semistable objects of class $(\un{d},\alpha)$, I will write
\begin{equation}
\label{eq:Nsigmalambda}
N^{\sigma^{\lambda}_{\mu}}_{\un{d},\al}\subset \mN^{\rig}_{\un{d},\al}\end{equation}
for the resulting moduli spaces. 
\end{definition}
\begin{example}
\label{ex:sheavesandquivers}
That the below two examples are instances of the data from Definition \ref{Def:quiverpairs} will be checked in §\ref{subsec:exfitsdef}.
    \begin{enumerate}[label=\roman*)]
        \item Let $\mA = \Coh(X)$ for a \textit{strict Calabi--Yau fourfold} $X$ -- one which satisfies $H^i(\mO_X) = 0$ for $i=1,2,3$. Choose a collection of ample divisors $\{D_k\}_{k\in \ZZ}$ with the associated subcategories $\mA_k$ of sheaves $E$ satisfying 
        $$
        H^i\big(E(D_k)\big) = 0 \quad \textnormal{whenever} \quad i>0\,.
        $$
        For each $k$, the corresponding $\mB_k$ is constructed as the category of pairs 
        $
       P= \big(V\otimes \mO_X(-D_k)\longrightarrow E\big)
         $
         where $E$ is in $\mA_k$ and $V$ is a vector space. The morphisms in $\mB_k$ are given by the commutative diagrams
        $$
    \begin{tikzcd}
        V\otimes \mO_X(-D_k)\arrow[d]\arrow[r]&E\arrow[d]\\
        V'\otimes \mO_X(-D_k)\arrow[r]&E'
    \end{tikzcd}\,,$$
    and exact triples are triples that are termwise exact.
    Setting
    $$(1,0) = \big( \mO_X(-D_k)\to 0\big)$$
determines uniquely the exact triple \eqref{eq:exacttriple} associated with $\big(\mO_{X}(-D_k)\to E\big)$.

The standard functor
\begin{equation}
\label{eq:functorC}
\begin{tikzcd}
C:\,\mB_k\arrow[r]&\ov{\mB_k}=D^b(X)\,,\qquad
P\arrow[r,mapsto]& P^{\bullet}
\end{tikzcd}
\end{equation}
mapping each pair to the corresponding complex in degrees $-1$ and 0 is used to define a map of stacks
$$
\Omega_C: \mN_k\longrightarrow \ov{\mN_k}:=\mM_X\,.
$$
\item When working with a Calabi--Yau dg-quiver $\wt{Q}^{\bullet}$, one uses the category $\mB_k$ constructed in the same way as in \cite[§5.5]{GJT}. Consider a new dg-quiver $\wt{Q}^{\bullet}_{\infty}$ which adds an extra vertex $\infty$ to $\wt{Q}^{\bullet}$. For each original vertex $v$ of $\wt{Q}^{\bullet}$, one draws an extra edge starting at $\infty$ and ending at $v$. To distinguish the dimension vectors of $\wt{Q}^{\bullet}$ from those of $\mathring{I}$ in Definition \ref{Def:quiverpairs}, I will label the former by $\alpha$. In the case of $\wt{Q}^{\bullet}_{\infty}$, the dimension vectors are denoted by $(d_{\infty},\al)$. The object $(1,0)$ is then set to be the unique representation of $\wt{Q}^{\bullet}_{\infty}$ with the dimension vector $(1,0)$. One sees from this definition that \eqref{eq:exacttriple} are indeed describing all the degree 0 representations of $\wt{Q}^{\bullet}_{\infty}$ which form the category $\mB_k$. 

To get the category $\ov{\mB}_k$, one additionally includes for each edge starting at $\infty$ a \Ma{degree $-2$ edge} going the opposite way. This produces a Calabi--Yau four quiver $\wt{Q}^{\bullet}_{\wt{\infty}}$ with the original superpotential as no further \BuOr{degree $-1$ cycles} have been added. Although the degree 0 representations will remain the same compared to $\wt{Q}^{\bullet}_{\infty}$ which implies $\mN_k=\ov{\mN}_k$ as classical stacks, the derived refinement of $\mN_k$ described in \eqref{eq:Mdderived} will be $-2$-shifted symplectic. This data satisfies all the conditions of Definition \ref{Def:quiverpairs}.a). As $(1,0)$ is uniquely chosen, I will drop the subscript $(-)_k$ in this case. 
    \end{enumerate}
\end{example}
\begin{remark}
\label{rem:section}
  \leavevmode
\vspace{-4pt}
\begin{enumerate}[label=\roman*)]
    \item The exact triple \eqref{eq:exacttriple} induces a distinguished triangle
$$
\begin{tikzcd}
E\arrow[r]&P\arrow[r]&(1,0)\otimes V\arrow[r,"s"]&P[1]
\end{tikzcd}
$$
in terms of $\ov{\mB_k}$ due to the exactness of $\mB_k\hookrightarrow \ov{\mB_k}$. The map $s$ will sometimes be called a \textit{section}. Since $P$ becomes the cocone of $s$, one can equivalently think of it as the complex represented by the triple $\{(1,0)[-1]\otimes V\xrightarrow{s}E$\}. In particular, there is always a \textit{universal triple} 
\begin{equation}
\label{eq:universaltriple}
\begin{tikzcd}
(1,0)[-1]\otimes \mV\arrow[r,"\mathfrak{s}"] &\mE
\end{tikzcd}
\end{equation}
on $\mN_k$, where $\mE$ is the universal object of $\mM_{\mA}$ pulled back to $\mN_k$ and $\mV$ is the universal vector bundle. I will always choose to put $\mE$ in degree $0$. This notation is motivated by Example \ref{ex:sheavesandquivers}.i),
\item The condition that $(1,0)_k$ is spherical puts restrictions on the geometry of $X$ in Example \ref{ex:sheavesandquivers}.i) as it implies that it must be a strict CY fourfold. To generalize this, one needs to work with fixed determinant obstruction theories, which will be included in a later version of the follow-up work \cite{BKLT}. 
\end{enumerate}
\end{remark}
Before I introduce the assumptions needed to construct obstruction theories on the spaces from \eqref{eq:Nsigmalambda}, I will fix the notation for the following three projections:
\begin{itemize}
    \item the natural projection   $$
\begin{tikzcd}\pi_{I}:\mN_{I_k}\arrow[r]&\mM_{\mA_k}\,,\end{tikzcd}$$
  \item in the case of $I=I_{\MS}$, the morphism $$\begin{tikzcd}\pi_{\MS/\Flag}:\mN_{I_k}\arrow[r]&\mN_{\Flag_k}\end{tikzcd}$$
  forgetting the vertex $v_0$ and the edges $e_{-1},e_0$,
  \item for $I = I_{\Flag}$, the map $$\pi_{\Flag/\JS}: \mN_{I_k}\to \mN_{k}$$ acting by composing the morphisms $m^{e_i}$ along the full sequence of edges of $I$:
\begin{equation}
\label{eq:FlagtoProj}
    \includegraphics[scale=0.75]{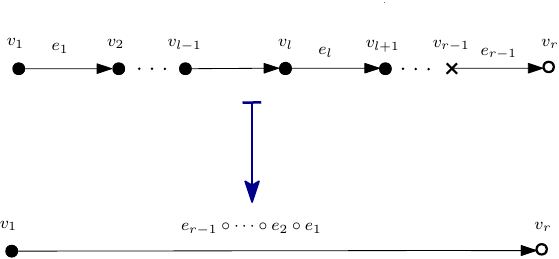}\,,
\end{equation}
where
$$
m^{e_{r-1}\circ \cdots \circ e_2\circ e_1} = m^{e_{r-1}}\circ \cdots \circ m^{e_2}\circ m^{e_1}\,.
$$
\end{itemize}
From now on, I will distinguish between the dimension vectors of $\mathring{I}_{\MS}$ and of $\mathring{I}_{\Flag}$ by writing $(d_0,\un{d})$ in the first case and simply $\un{d}$ in the second. Furthermore, when $\lambda$ and $(\mu_0,\un{\mu})$ are fixed and understood, I will simply write $N^{\sigma}_{(1,\un{d}),\alpha}$ and $N^{\sigma}_{\un{d},\al}$ for the spaces in \eqref{eq:Nsigmalambda}. A similar convention will be used for the stacks $\mN^{\sig}_{\un{d},\al}, \mN^{\sigma}_{(1,\un{d}),\alpha}$. The restrictions of the above projections to these moduli schemes and stacks will be labeled as follows:
\begin{itemize}
    \item When $I=I_{\Flag}$, the restriction of $\pi_I$ to $\mN_{\un{d},\al}$ will be denoted by 
    $$
    \begin{tikzcd}
        \pi_{\un{d},\al}:\mN_{\un{d},\al}\arrow[r]& \mM_{\alpha}\,.
    \end{tikzcd}
    $$
    For $I=I_{\MS},I_{\Flag}$, I will write
    $$
    \begin{tikzcd}
    \pi^{\sigma}_{(1,\un{d}),\alpha}: \mN^{\sigma}_{(1,\un{d}),\alpha}\arrow[r]&\mM_{\alpha}\end{tikzcd}\,,\qquad \begin{tikzcd}\pi^{\sigma}_{\un{d},\al}: \mN^{\sigma}_{\un{d},\al}\arrow[r]&\mM_{\alpha}
    \end{tikzcd}
    $$
    for the restrictions of $\pi_I$ to $\mN^{\sigma}_{(1,\un{d}),\alpha}$ and $\mN^{\sigma}_{\un{d},\al}$ respectively. I will use the same notation for the rigidification of 
 both morphisms. In the next two points, the rigidifications of such restrictions will also inherit their labels in a similar way.
    \item  Rigidifying  $\pi_{\MS/\Flag}$ one can restrict the result to $\mN^{\sig}_{(1,\un{d}),\al}$. Assuming that this restriction maps to $\mN^{\sig}_{\un{d},\al}$, I will write it as 
    $$
    \begin{tikzcd}
     \PP^{\sigma}_{\un{d},\al}:   \mN^{\sig}_{(1,\un{d}),\al}\arrow[r]& \mN^{\sigma}_{\un{d},\al}\,.
    \end{tikzcd}
    $$
    \item The restriction of $\pi_{\Flag/\JS}$ to $\mN^{\sigma}_{\un{d},\al}$ when $d_1=1$ will be denoted by 
    $$
    \begin{tikzcd}
\pi^{\sigma}_{\un{d}/1,\alpha}:\mN^{\sigma}_{\un{d},\al}\arrow[r]&\mN^{\rig}_{1,\alpha}\,.
    \end{tikzcd}
    $$
\end{itemize}
Both the obstruction theory $\EE$ of $\mM_{\mA}$ and the obstruction theory $\FF_k$ of $\mN_k$ induced by the open embedding into $\ov{\mN_k}$ satisfy self-duality and isotropy of cones from Definition \ref{def:viradmcl} by the discussion in §\ref{sec:dgsetup}. By the same argument as in \cite[(127)]{BKP} recalled in \eqref{eq:rigidifyingEE}, they induce the obstruction theories $\EE^{\rig}$ and  $\FF^{\JS}$ on $\mM^{\rig}_{\mA}$ and $\mN^{\rig}_k$ respectively .

The next observation is a generalization of Park's pair-sheaves correspondence from \cite[Proposition 4.7]{Park} to the general category $\mA$ as above. It is also stated in the language of stable infinity categories recalled in the Appendix \ref{appendix}. Thus it produces a lift of \eqref{eq:diagsym} in the sense of Definition \ref{def:lifting}. I have called this lift the \textit{$\infty$-Park's virtual pullback diagram ($\infty$-Pvp diagram)} in §\ref{Asec:inftyPvp}. Below, I will often not specify in notation the pullback along obvious maps when dealing with obstruction theories or universal objects. Thus $\EE$ may denote the pullback of the obstruction theory of $\mM_{\mA}$ to $\mN_k$ which will be clear from the context.
\begin{lemma}
\label{lem:pairshenafcorrespondence}
There is a natural self-dual homotopy commutative diagram 
\begin{equation}
\label{eq:LLrighom0}
\begin{tikzcd}
\arrow[d]\LL_{\pi^{\rig}_{1,\alpha}}[-1]\arrow[r]&\EE^{\rig}\arrow[d]\\
0\arrow[r]&\LL_{\pi^{\rig}_{1,\alpha}}^\vee[3]
\end{tikzcd}
\end{equation}
in $\mD^b\big(\mN^{\rig}_{1,\al}\big)$. Self-duality here means that it is preserved under applying $(-)^\vee[2]$. By Proposition \ref{prop:sympullback}, this produces an $\infty$-Pvp diagram for $\pi^{\rig}_{1,\alpha}$ with the resulting obstruction theory on $\mN^{\rig}_{1,\al}$ given by the restriction of $\FF^{\JS}$. 
\end{lemma}
\begin{proof}
    Let $\mM_{\Vect}$ be the moduli stack of complex vector spaces. Then there are the natural projections
\begin{equation}
\label{eq:piJS}
\pi_{\JS}:\mN_k\to\mM_{\mA_k}\,, \qquad \pi_{k}:\mN_k\to \mM_{\Vect}\times \mM_{\mA_k} \end{equation}
  induced by mapping the exact triple $E\to P\to (1,0)\otimes V$ to the object $E$ in $\mA_k$ and the vector space $V$. The complex $\FF_k$ fits into the following $3\times 3$-diagram in $\mD^b(\mN_k)$:
     \begin{equation}
   \label{eq:obsonNk}
    \begin{tikzcd}[column sep =  huge, row sep= large]
  \arrow[d]C({\delta_{k}})[-1]\arrow[r]&\arrow[d]\big(\wt{\FF}^{\JS}\big)^\vee[2]\arrow[r]&\arrow[d]\FF_k\\
  \arrow[d,"{\delta_k}"']\LL_{\pi_k}[-1]\arrow[r,"\psi_k"]&\arrow[d,"{\psi_k^\vee[2]}"]\EE\arrow[r]&\arrow[d]\wt{\FF}^{\JS}\\
  \mV\otimes \mV^*\otimes \HH\arrow[r,"{\delta_k^\vee[2]}"']&\LL_{\pi_k}^\vee[3]\arrow[r]&C\big({\delta_k}\big)^\vee[3]
    \end{tikzcd}\,,
   \end{equation}
   where $\HH = H^*(S^4)[3]$ and $C\big({\delta_k}\big)$ is the cone of $\delta_k$. 

   There is a $(\Delta^{1})^{\times 3}$ diagram 
$$
        \begin{tikzcd}[row sep=small, column sep=small]
&\B{\mV\otimes \mV^*[-1]}\arrow[dd, blue]\arrow[rr, blue]&&[0.65cm]\B{0}\arrow[dd, blue]\\
            \arrow[dd]\arrow[ur,blue] \LL_{\pi_k}[-1]\arrow[rr]&& \arrow[ur,blue]\EE\arrow[dd]&\\
            &\B{\mV\otimes \mV^*[-1]}\arrow[rr, blue]&&\B{0}
            \\
\arrow[ur,blue] \mV\otimes \mV^*\otimes \HH\arrow[rr]&&\arrow[ur,blue]\LL_{\pi_k}^\vee[3]&
        \end{tikzcd}
 $$
in $\mD^b(\mN_k)$ together with its dual. They can be combined into a $(\Delta^1)^{\times 4}$ diagram because all higher homotopies will have 0 as a target or a source. As $\LL_{\pi_{\JS}}$ is the cocone of $\LL_{\pi_k}\to \mV\otimes \mV^*$, the obstruction theory $\FF_k$ is obtained by applying Proposition \ref{prop:sympullback} to the self-dual $\Delta^1\times \Delta^1$ diagram
 \begin{equation}
 \label{eq:pairtosheaf}
 \begin{tikzcd}
     \arrow[d]\LL_{\pi_{\JS}}[-1]\arrow[r]&\EE\arrow[d]\\
     0\arrow[r]&\LL^\vee_{\pi_{\JS}}[3]\,.
 \end{tikzcd}
 \end{equation}
 To get \eqref{eq:LLrighom0}, one first restricts \eqref{eq:pairtosheaf} to $\mN_{1,\al}$ and then applies Lemma \ref{lem:rigidifyingPark} to get the rigidified version starting from $\EE^{\rig}$.
\end{proof}
\begin{example}
\label{ex:JSpairs}
    As already pointed out in \cite[Example 5.6]{JoyceWC}, there is a simple description of the stability condition in \eqref{eq:barphi} in the case that the quiver $I$ has the simple form
         \begin{center}
            \includegraphics[scale=1]{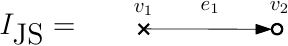}\,.
        \end{center}
For such $I$, it is enough to set $\lambda=0$ and to focus on the classes $(1,\alpha)$ for $\alpha\in \msE(\mA)$. In this case, I denote the stability \eqref{eq:barphi} by $\sigma_{\JS}$.

An object in class $(1,\alpha)$ is $\sigma_{\JS}$-semistable if and only if it is $\sigma_{\JS}$-stable, so there often exists a moduli space $$
 N^{\JS}_{k,\alpha} := N^{\sigma_{\JS}}_{1,\alpha}$$
with a virtual fundamental class $\big[ N^{\JS}_{k,\alpha}\big]^{\vir}$. The objects of $N^{\JS}_{k,\al}$ can be described as pairs of the form $P=\big((1,0)[-1]\xrightarrow{s}E\big)$ by Remark \ref{rem:section}. In this form, the $\sigma^{\JS}$-stability of $P$ is equivalent to the following three conditions:
     \begin{enumerate}
         \item $E$ is $\sigma$-semistable,
         \item $s$ is non-zero since $E$ is,
         \item for every exact sequence $0\to E'\to E\to E''\to 0$ of $\sig$-semistables in $\mA$ such that $s$ factor through $E'$, the strict inequality $\phi(E')<\phi(E'')$ holds.
     \end{enumerate}
When $\alpha$ and $\sigma$ are such that there are no strictly $\sigma$-semistable objects of class $\alpha$, then $\sigma_{\JS}$-stability only requires that the sections $s$ are non-zero. In particular, the projection 
\begin{equation}
\label{eq:piJSal}
\pi^{\JS}_{\alpha} := \pi^{\sigma}_{1,\alpha}:\begin{tikzcd} N^{\JS}_{k,\alpha}\arrow[r]&\mM^{\rig}_{\alpha}\end{tikzcd}\,,
\end{equation}
which is always smooth of dimension 
$
\chi\big(\alpha(k)\big) - 1\,,
$
becomes a projective bundle in this situation.
\end{example}
There is of course already a naïve obstruction theory $\wt{\FF}_I$ of $\mN_{I_k}$ which was constructed in \cite[Definition 5.5]{JoyceWC}.  From \cite[Definition 5.5]{JoyceWC}, it is known that the openness condition of Assumption \ref{ass:stab}.g) implies that the embeddings $\mN^{\sigma}_{\un{d},\al}\subset \mN_{\Flag_k}$ and $\mN^{\sigma}_{(1,\un{d}),\al}\subset \mN_{\MS_k}$ are open. Therefore $\wt{\FF}_{I}$ determines a naïve obstruction theory on the stacks of semistable objects. Moreover, it follows that $$\LL_{\pi^{\sig}_{\un{d},\al}}=\LL_{\pi_{\Flag}}|_{\mN^{\sigma}_{\un{d},\al}} \,,\qquad\LL_{\pi^{\sig}_{(1,\un{d}),\al}} = \LL_{\pi_{\MS}}|_{\mN^{\sigma}_{(1,\un{d}),\al}}$$
are vector bundles. I now recall the explicit form of $\wt{\FF}_I$. 

The universal vector space at each vertex $v_i$ of $\mathring{I}$ is denoted by $\mV_i$ with the notation $\mV_{\times}$ reserved for the one at the connecting vertex. Due to the definition of $\mB_{I_k}$, there is an isomorphism $\mV_{\times}\cong \mV$ on $\mN_{I_k}$ where $\mV$ is as in Remark \ref{rem:section}. For each edge $e$ of $\mathring{I}$ with the vertex at its tail $t(e)$ and the one at its head $h(e)$, I will write 
$$
\mathfrak{m}_e: \mV_{t(e)}\longrightarrow \mV_{h(e)}
$$
for the corresponding universal morphism. I will, furthermore, introduce the notation 
$$
\FG{\mV_r}:= \RHom_{\mN_{I_k}}\big((1,0)_k,\mE\big)[1]\,.
$$
The coloring is added to distinguish the \FG{\textit{r'th universal vector space}} from $\mV_i$ for $i<r$. The universal morphism 
$$
\begin{tikzcd}
\mathfrak{m}_{e_{r-1}}:\mV_{r-1}\arrow[r]&\FG{\mV_r}
\end{tikzcd}
$$
is induced by \eqref{eq:universaltriple}.

For $I=I_{\MS}$, one can write the relative obstruction theory along $\pi_{I}$ as
\begin{equation}
\label{eq:LNQk}
\begin{tikzcd}
\LL_{\pi_{I}} = \Big(\Cya{\bigoplus_{v_i\in \mathring{\Ver}}\mV_i^*\otimes \mV_i}\arrow[r,"\textnormal{ad}(\mathfrak{m})"]&[0.5cm]\mV^*_{\times}\otimes \FG{\mV_r}\oplus \bigoplus_{e\in \mathring{\Edg}}\mV_{t(e)}^*\otimes \mV_{h(e)}\arrow[r]& \BuOr{\mV_{l-1}^*\otimes \mV_0}\Big)^\vee
\end{tikzcd}
\end{equation}
with \Cya{degree $1$} terms colored in cyan, degree 0 terms in black, and \BuOr{degree $-1$} terms in orange. The first non-zero map is given by 
\begin{equation}
\label{eq:adm}
\textnormal{ad}(\Fm) = \big(\circ \Fm_e - \Fm_e\circ\big)_{e\textnormal{ of }I},
\end{equation}
while the second one is determined by Lemma \ref{lem:cotangentofQ}. The expression looks the same for $I =I_{\Flag}$, except that there is no \BuOr{degree $-1$} term. The classical obstruction theory on $\mN_{I_k}$ is determined by the natural distinguished triangle
\begin{equation}
\label{eq:wtFFI}
\begin{tikzcd}
\LL_{\pi_{I}}[-1]\arrow[r]&\EE\arrow[r]&\wt{\FF}_{I}
\end{tikzcd}
\end{equation}
in $\mD^b\big(\mN_{\un{d},\al}\big)$. 

The obstruction theory of Assumption \ref{ass:obsonflag} differs from $\wt{\FF}_{I}$ by an extra $\LL^{\vee}_{\pi_{I}}[2]$ term. One part of this assumption will be related to the additivity of the obstruction theories under the direct sum map 
$$
\mu_{\Flag_k}: \mN_{\Flag_k}\times \mN_{\Flag_k}\longrightarrow \mN_{\Flag_k}
$$
which is a generalization of $\mu$ from Definition \ref{def:categoryA}.a). For now, I just recall the simpler version applying to $\wt{F}_{\Flag}$. 
\begin{corollary}
\label{cor:sumobstructiontheory}
Writing $\mV_i$ for $i<r$ and \FG{$\mV_r$} for the universal vector spaces on the first factor of $\mN_{\Flag_k}\times \mN_{\Flag_k}$ and $\mW_j$ for $j<r$ together with \FG{$\mW_r$} for their counterparts on the second factor, I define
    \begin{equation}
\label{eq:ExtNQk}
\begin{tikzcd}
\Theta_{\mN_{\Flag_k}/\mM_{\mA}} = \Big(\Cya{\bigoplus_{i=1}^{r-1}\mV_i^*\otimes \mW_i} \arrow[r,"\textnormal{ad}(\mathfrak{m})"]& [0.5cm] \mV^*_{\times}\otimes \FG{\mW_r}\oplus\bigoplus_{i=1}^{r-2}\mV_i^*\otimes \mW_{i+1} \Big)^\vee
\end{tikzcd}
\end{equation}
as the dual of the cone of $\textnormal{ad}(\mathfrak{m})$ from \eqref{eq:adm}.  Precomposition with $\mathfrak{s}$ from Remark \ref{rem:section} induces
$$
\begin{tikzcd}
\Theta_{\mN_{\Flag_k}/\mM_{\mA}}[-1]\arrow[r]&\Theta_{\mA}
\end{tikzcd}
$$
the cone of which will be labeled by $\wt{\Theta}_{\Flag}$. Then, the following decomposition hols:
     \begin{equation}
\label{eq:rigidifiedsumobs}
\mu_{\Flag_k}^*\Big(\wt{\FF}_{\Flag}\Big) = \wt{\FF}_{\Flag}\boxplus \wt{\FF}_{\Flag}\oplus \wt{\Theta}_{\Flag}\oplus \sigma^*\wt{\Theta}_{\Flag}\,.
    \end{equation}
\end{corollary}
\begin{remark}
\label{rem:d1d2}
When taking decomposing $(\un{d},\al)$ for a fixed length $r$ of the quiver $I_{\Flag}$ into summands $(\un{d}_1,\al_1)$ and $(\un{d}_2,\al_2)$, I will always relabel the components of $\un{d}_1$ and $\un{d}_2$ after removing the sequence of zeros before the first non-zero entry. In other words, I will shorten the length of both $\un{d}_1$ and $\un{d}_2$ such that $(d_1)_1\neq 0 \neq (d_2)_1$.
\end{remark}

    The next assumption leads to the appropriate CY4 obstruction theories on $N^{\sigma}_{\un{d},\alpha}$ and $N^{\sigma}_{(1,\un{d}),\alpha}$ used in proving wall-crossing. Unfortunately, this assumption is too restrictive in the case of sheaves when the existence of such obstruction theories is not guaranteed. The follow-up work \cite{BKLT} addresses this issue by proving the assumption below on a yet larger space than $N^{\sigma}_{\un{d},\alpha}$ or $N^{\sigma}_{(\un{d},1),\alpha}$ which is constructed using a Jouanolou device.
    
    Note that the assumption itself only constructs obstruction theories on the unrigidified moduli stacks $
    \mN^{\sig}_{\un{d},\al}$ and $\mN^{\sig}_{(1,\un{d}),\al}
    $. This is to make the formulation of the analog of Corollary \ref{cor:sumobstructiontheory} simpler. By applying \eqref{eq:rigidifyingEE}, this induces obstruction theories on the rigidification.
    \begin{assumption}
\label{ass:obsonflag}
\begin{enumerate}[label=\alph*)]
Fix a dimension vector $(1,\un{d})$ such that 
\begin{equation}
\label{eq:ovdd0}
d_i \leq  d_{i+1}\leq d_{i}+1\qquad \text{for}\quad 1\leq i\leq r-2 \,,\qquad d_1 = 1\,.
\end{equation}
For any $\al\in E(\mA)$, $k\in K$,  $\sig\in W_{\al, k}$, and $(\mu_0,\un{\mu})$, $\lambda$ as in \eqref{eq:muicond}, the following should hold. 
\item When there are no strictly semistables for the stability condition \eqref{eq:barphi}, the  algebraic spaces $N^{\sigma^\lambda_{\mu}}_{\un{d},\alpha}$ are finite type and separated with $\T$-fixed point loci being proper and satisfying the resolution property. 

Let $\tau:\T\to \T\times \GG_m$ be an inclusion such that $\pi_1\circ \tau = \id_{\T} $ for the projection $\pi_1: \T\times\GG_m\to \T$. Then $\Big(N^{\sigma^\lambda_{\mu}}_{(1,\un{d}),\alpha}\Big)^{\tau(\T)}$ are proper and satisfy the $\TT$-equivariant resolution property.
 \item For the morphism $\pi^{\sigma}_{\un{d}/1,\alpha}: \mN^{\sigma}_{\un{d},\al}\to \mN_{1,\alpha}$, there is a homotopy commutative diagram 
 \begin{equation}
 \label{eq:flagobsdiag}
 \begin{tikzcd}
     \arrow[d]\LL_{\pi^{\sigma}_{\un{d}/1,\alpha}}[-1]\arrow[r]&\FF_k\arrow[d]\\
     0\arrow[r]&\LL^\vee_{\pi^{\sigma}_{\un{d}/1,\alpha}}[3]
 \end{tikzcd}\,.
 \end{equation}
 By Theorem \ref{thm:functsympull} and by \eqref{eq:pairtosheaf}, this induces a homotopy commutative diagram 
\begin{equation}
\label{eq:Ndalobstruction}
 \begin{tikzcd}
     \arrow[d]\LL_{\pi^{\sigma}_{\un{d},\al}}[-1]\arrow[r]&\EE\arrow[d]\\
     0\arrow[r]&\LL^\vee_{\pi^{\sigma}_{\un{d},\al}}[3]
 \end{tikzcd}
 \end{equation}
 producing, by Proposition \ref{prop:sympullback}, a CY4 obstruction theory $\FF^{\sigma}_{\un{d},\al}$ of $\mN^{\sigma}_{\un{d},\al}$. The homotopies from \eqref{eq:flagobsdiag} should be chosen in a compatible way with respect to $\mu_{\Fl_k}$ in the following sense.
 
 Let $\alpha = \alpha_1+\alpha_2$ and $\un{d} = \un{d}_1+\un{d}_2$ where $\alpha_i\neq 0$ and $\un{d}_i$ satisfy \eqref{eq:ovdd0} for their respective lengths explained in Remark \ref{rem:d1d2}. Furthermore, fix $\lambda_i$ for $i=1,2$ and denote the resulting stability conditions from \eqref{eq:barphi} by $\sig_1$ and $\sigma_2$ respectively. If $$\mu_{\Flag_k}\Big(\mN^{\sig_1}_{\un{d}_1,\alpha_1}\times \mN^{\sig_2}_{\un{d}_2,\alpha_2}\Big)\subset \mN^{\sig}_{\un{d},\al}\,,$$ then there should exist a homotopy commutative diagram 
 \begin{equation}
    \label{eq:extendTheta}
    \begin{tikzcd}
   \arrow[d] \Theta_{\mN_{\Flag_k}/\mM_{\mA}}[-1]\arrow[r]&\Theta_{\mA}\arrow[d]\\
   0\arrow[r]&\sigma^* \Theta_{\mN_{\Flag_k}/\mM_{\mA}}^\vee[3]
    \end{tikzcd}
    \end{equation}
 in $\mD^b\Big(\mN^{\sig_1}_{\un{d}_1,\alpha_1}\times \mN^{\sig_2}_{\un{d}_2,\alpha_2}\Big)$ producing a complex $\Theta_{\Flag_k}$ by applying the idea of Proposition \ref{prop:sympullback}. This should then satisfy
 \begin{equation}
 \label{eq:additivityofFFflag}
\mu_{\Flag_k}^*\Big(\FF^{\sig}_{\un{d},\al}\Big) = \FF^{\sig_1}_{\un{d}_1,\alpha_1}\boxplus \FF^{\sig_2}_{\un{d}_2,\alpha_2}\oplus \Theta_{\Flag}\oplus \sigma^*\Theta_{\Flag}\,.
 \end{equation}
 I omit specifying the appropriate restriction of $\Theta_{\Flag}$.
 \item  For the projection $ \PP^{\sigma}_{\un{d},\al}:   \mN^{\sig}_{(1,\un{d}),\al}\to\mN^{\sigma}_{\un{d},\al}$, there should exist a homotopy commutative diagram 
 \begin{equation}
 \label{eq:MSobsdiag}
 \begin{tikzcd}
     \arrow[d]\LL_{\PP^{\sigma}_{\un{d},\al}}[-1]\arrow[r]&\FF^{\sig}_{\un{d},\al}\arrow[d]\\
0\arrow[r]&\LL^\vee_{\PP^{\sigma}_{\un{d},\al}}[3]
 \end{tikzcd}\,.
 \end{equation}
 The result of applying Proposition \ref{prop:sympullback} to it will be denoted by $\FF^{\sig}_{(1,\un{d}),\alpha}$.
    \end{enumerate}
\end{assumption}
Using Lemma \ref{lem:rigidifyingPark}.i) produces the rigidified version of all of the $\infty$-PVP diagrams in the above assumption. Furthermore, this rigidification is compatible with applying Theorem \ref{thm:functsympull} because of Lemma \ref{lem:rigidifyingPark}.ii). In particular, there is the self-dual diagram 
\begin{equation}
 \label{eq:flagobsdiagrig}
 \begin{tikzcd}
     \arrow[d]\LL_{\pi^{\sigma}_{\un{d}/1,\alpha}}[-1]\arrow[r]&\FF^{\JS}\arrow[d]\\
     0\arrow[r]&\LL^\vee_{\pi^{\sigma}_{\un{d}/1,\alpha}}[3]
 \end{tikzcd}\,.
 \end{equation}
 on $\big(\mN^{\sig}_{\un{d},\al}\big)^{\rig}$ which together  with \eqref{eq:LLrighom0} induces
 \begin{equation}
 \label{eq:Ndalrigobstruction}
 \begin{tikzcd}
     \arrow[d]\LL_{\pi^{\sigma}_{\un{d},\al}}[-1]\arrow[r]&\EE^{\rig}\arrow[d]\\
     0\arrow[r]&\LL^\vee_{\pi^{\sigma}_{\un{d},\al}}[3]
 \end{tikzcd}\,.
 \end{equation}
 Using Proposition \ref{prop:construct} with the above diagram gives rise to the obstruction theory $\FF^{\rig}_{\un{d},\al}$ on $\big(\mN^{\sig}_{\un{d},\al}\big)^{\rig}$ which is the rigidification of $\FF^{\sig}_{\un{d},\al}$ in the sense of \eqref{eq:rigidifyingEE}.  Due to Assumption \ref{ass:obsonflag}.a) and Example \ref{ex:mMvir}, there are virtual fundamental classes 
 $$
\Big[N^{\sig}_{\un{d},\al}\Big]^{\vir}\in H^{\T}_*\Big(N^{\sig}_{\un{d},\al}\Big)_{\loc}
 $$
 in the absence of strictly semistables. The orientation used to construct them is taken from Definition \ref{def:orpullback} applied to \eqref{eq:Ndalrigobstruction}.

 A similar discussion applies to $\PP^\sig_{\un{d},\al}$ and \eqref{eq:MSobsdiag}, which gives rise to the self-dual diagram
 $$
 \begin{tikzcd}
     \arrow[d]\LL_{\pi^{\sigma}_{(1,\un{d}),\alpha}}[-1]\arrow[r]&\EE^{\rig}\arrow[d]\\ 0\arrow[r]&\LL^\vee_{\pi^{\sigma}_{(1,\un{d}),\alpha}}[3]\,.
 \end{tikzcd}
 $$
\subsection{Invariants counting semistable objects}
\label{sec:InvDef}
Here, I introduce the invariants counting semistable objects following the approach of \cite{JoyceWC} based on \cite{mochizuki}. Their construction will depend on the choice of $k\in K$, so it is of import to show that the result is independent of $k$,  especially in cases of independent interest like Gieseker or slope semistable sheaves. This was presented as Problem (II) in §\ref{sec:introWCeqhom}, while Remark \ref{rem:welldefinednotneeded} mentioned situations when addressing this problem is not necessary.

Fix an indexing set $K$ with subcategories $\{\mA_k\}_{k\in K}$ of $\mA$ as in Definition \ref{Def:quiverpairs}. The motivation behind the definition of the invariants $\langle \mM^\sig_\al\rangle$ stems from the Joyce--Song wall-crossing formula conjecture in \cite[(4.4)]{GJT}. It is formulated using the natural extension of the vertex algebra structure to the homology of $\mN_k$. In Definition \ref{Def:FlagVA} later, it will be generalized to include $\mN_{\Flag_k}$ as I prefer to avoid complicating the exposition of the present subsection. 
\begin{definition}[{\cite[Definition 2.6 and 2.7]{bojko2}}]
\label{def:NkVA}
The following set of data is defined for any $k\in K$.
\begin{itemize}
\item Let $\mu_k,\rho_k$ be the natural extensions of $\mu,\rho$ from \eqref{eq:murho}. 
\item I will use $\mV$ and $\mE$ to denote the universal vector space and the universal object of $\mA$ for the first factor of $\mN_k\times \mN_k$. The universal objects for the second copy of $\mN_k$ will be labeled by $\mW$ and $\mF$. Furthermore, set
\begin{align*}
\Theta_k =& -(\pi\times\pi)^*\Ext^\vee - 2\mV\otimes \mW^* \\&- \mV\otimes \RHom_{\mN_k}\big((1,0),\mF\big)^* - \mW^*\otimes \RHom_{\mN_k}\big((1,0),\mE\big)
\end{align*}
as a K-theory class on $\mN_k\times \mN_k$.
\item Using the expression $\chi\big(\al(k)\big)$ from \eqref{eq:chialk}, define the pairing 
$$
\begin{tikzcd}
\chi_k: \big(\ZZ\times \ov{K}(\mA)\big)^{\times 2}\arrow[r]&\ZZ
\end{tikzcd}
$$
by
$$
\chi_k\big((d,\al),(e,\be)\big) = \chi\big(\al,\be\big) -d\chi\big(\beta(k)\big) - e\chi\big(\al(k)\big) + 2de\,.
$$
\item Introduce the signs 
$$
\varepsilon^k_{(d,\alpha),(e,\beta)} = (-1)^{d\chi(\beta(k))}\varepsilon_{\al,\beta}\qquad \textnormal{for}\quad (d,\al),(e,\be)\in \ZZ\times \ov{K}(\mA)\,.
$$
Note that they satisfy the analog of \eqref{eq:epsidentity} and 
$$
\varepsilon^k_{(d,\al),(e,\beta)} = (-1)^{\chi_k((d,\al),(e,\beta))}\varepsilon_{(e,\beta),(d,\al)}
$$
because of \eqref{eq:symsign}.
\end{itemize}
The vertex algebra on
$$
W^k_* =  H_{*+\vdim}(\mN_k)
$$
and its localized version $W^k_{\loc,*}$ are constructed by following Definition \ref{Def:VAlocal} or \ref{Def:VAglobal} except that the state-field correspondence is determined by
\begin{align*}
    Y(v,z)w &= (-1)^{a\chi_k((e,\beta),(e,\beta))}\varepsilon^k_{(d,\al),(e,\beta)} \,\mu_\ast\bigg((e^{zT}\otimes \textnormal{id})\frac{v\boxtimes w}{z^{-\chi_k((d,\al),(e,\be))}c_{z^{-1}}(\Theta_{k})}\bigg)
    \end{align*}
    for $v\in H_a(\mN_{d,\al})$ and $w\in H_*(\mN_{e,\beta})$. When working $\T$-equivariantly, one should use $\boxtimes^{\T}$ from \eqref{eq:Kunnethlocal} or \eqref{eq:equivKunneth} instead.
\end{definition}
When $\al\in \msE(\mA)$ and $\sig\in W_{\al,k}$, the wall-crossing formula for Joyce--Song pairs that was used in \cite{bojko2, bojko3} for CY fourfolds takes the form
\begin{equation}
\label{eq:JSWC}
\big[N^{\JS}_{k,\al}\big]^{\vir} = \sum_{\begin{subarray}{c}
        \un{\alpha}\,\vdash_{\mA} \alpha\,, \\
          \phi(\alpha_i) =\phi(\alpha) \end{subarray}}\frac{1}{n!}\Big[\big\langle \mM^\sigma_{\alpha_n}\big\rangle, \cdots \Big[\big\langle \mM^\sigma_{\alpha_1}\big\rangle,  e^{(1,0)}_k\Big]\cdots\Big]
\end{equation}
in $W^k_{*+2}/T(W^k_*)$. Here $e^{(1,0)}_k$ is the point class of $\{(1,0)_k\}\in \mN_k$ and one considers $\mM_{\mA_k}$ as a substack of $\mN_k$. For Behrend--Fantechi obstruction theories, the conditions necessary for this formula to hold were checked in \cite[Appendix A]{bojko3}.  

Ideally, one would define the invariants appearing on the right hand side of \eqref{eq:JSWC} by induction on $\rk_{\sig}(\al)$. However, it is not self-evident that they exist.
One can use the following lemma to state the problem differently.
\begin{lemma}
\label{lem:pushJSandinj}
\leavevmode
\vspace{-4pt}
   \begin{enumerate}[label=\roman*)]
       \item  The formula \eqref{eq:JSWC} implies
  \begin{equation}
        \label{eq:Masigdef}
        \blangle \mM^{\sig}_{\alpha}\brangle = \chi\big(\al(k)\big)\Omega^{\sig}_{\alpha} - \sum_{\begin{subarray}{c}
            \un{\alpha}\,\vdash_{\mA} \alpha \,, n\geq 2\,,\\
          \bph(\alpha_i) =\bph \end{subarray}}\frac{\chi(\al_1(k))}{n!}\Big[\blangle\mM^\sig_{\al_n}\big\rangle,\Big[\cdots\Big[\big\langle\Masit\big\rangle,\langle\Masio\big\rangle \Big]\cdots \Big]\Big]\,.
    \end{equation}
    where 
\begin{equation}
\label{eq:Omegasigal}\chi\big(\alpha(k)\big)\,\Omega^{\sig,k}_{\alpha}=\big(\pi^{\JS}_{\alpha}\big)_*\Big(\Big[N^{\JS}_{k,\alpha}\Big]^{\vir}\cap c_{\chi(\alpha(k))-1}\big(T_{\pi_{\mA_k}} \big)\Big)\,.
\end{equation}
     \item The map 
     \begin{equation}
     \label{eq:brackinj}
    \begin{tikzcd} \Big[-,e^{(1,0)}_k\Big]: H_{*}\Big(\mM^{\rig}_{\mA_k}\cap \mM^{\rig}_{\al}\Big)\arrow[r]&H_*\Big(\mN^{\rig}_k\Big)\end{tikzcd}
     \end{equation}
     is injective whenever $\al\notin \Ker(\chi)$.
   \end{enumerate}
\end{lemma}
\begin{proof}
Both i) and ii) follow from 
\begin{equation}
\label{eq:bracketpushdown}
 (\pi_{\JS})_*\Big(\Big[\ov{v},e^{(1,0)}_k\Big]\cap c_{\chi(\alpha(k))-1}\big(T_{\pi_{\mA_k}} \big)\Big) = (-1)^{a\chi(\al(k))}\chi\big(\al(k)\big)\ov{v}
\end{equation}
for any $\ov{v}\in H_a(\mM^{\rig}_{\mA_k}\cap \mM^{\rig}_{\al})$ whenever $\al\notin \Ker(\chi)$. To see this, choose a lift $v\in H_a(\mM_{\mA_k})$ of $\ov{v}$. Setting $\pi:= \pi_{\JS}$ for now and 
$$
\FG{\mV_k} :=-
\RHom_{\mN_k}\big((1,0)_k,\mE\big)\,, \qquad \FG{\mW_k} :=-
\RHom_{\mN_k}\big((1,0)_k,\mF\big)\,, $$ one computes the left hand side of \eqref{eq:bracketpushdown} to be equal to $(-1)^{a\chi(\al(k))}\ov{(-)}$ of
\begin{align*}
\label{eq:bracketpushcomp}
&\,[z^{-1}]\bigg\{\pi_*\mu_\ast\bigg((e^{zT}\otimes \textnormal{id})\frac{v\boxtimes e^{(1,0)}_k}{z^{\chi(\al(k))}c_{z^{-1}}\big(\mW^*\otimes \FG{\mV_k}\big)}\bigg)\cap c_{\chi(\al(k))-1}\big(\mV^*\otimes \FG{\mV_k}\big) \bigg\}\\
\stackrel{(1)}{=}\,&[z^{-1}]\bigg\{\pi_*\mu_\ast\bigg[(e^{zT}\otimes \textnormal{id})\bigg(\frac{v\boxtimes e^{(1,0)}_k }{z^{\chi(\al(k))}c_{z^{-1}}\big(\mW^*\otimes \FG{\mV_k}\big)}\bigg)\cap c_{\chi(\al(k))-1}\big(\mW^*\otimes \FG{\mV_k}\big)\bigg] \bigg\}\\
\stackrel{(2)}{=}\,&[z^{-1}]\bigg\{\pi_*\mu_\ast\bigg[(e^{zT}\otimes \textnormal{id})\bigg(\frac{v\boxtimes e^{(1,0)}_k}{z^{\chi(\al(k))}c_{z^{-1}}\big( \FG{\mV_k}\big)}\cap\frac{d}{dz}\Big(z^{\chi(\al(k))}c_{z^{-1}}\big(\FG{\mV_k}\big)\Big)\bigg)\bigg]  \bigg\}\\
\stackrel{(3)}{=}\,&[z^{-1}]\bigg\{e^{zT}\bigg(\frac{v\cap\chi(\al(k))z^{\chi(\al(k))-1}c_{z^{-1}}\big(\FG{\mV_k}\big)}{z^{\chi(\al(k))}c_{z^{-1}}\big( \FG{\mV_k}\big)}\bigg) \bigg\}\,.
\numberthis
\end{align*}
This is clearly equal to the right hand side of \eqref{eq:bracketpushdown}. Each step in the computations holds because
\begin{enumerate}[label=(\arabic*)]
    \item $\mu^*(\mV^*\otimes \FG{\mV_k}) = \mW^*\otimes \FG{\mV_k}$ when restricted to $\mM_{\al}\times \mN_{1,0}$.
    \item $\rho^* c_{\chi(\al(k))-1}\big(\mW^*\otimes \FG{\mV_k}\big) = \frac{d}{d\tau} \Big(\tau^{\chi(\al(k))}c_{\tau^{-1}}\big(\mW^*\otimes \FG{\mV_k}\big)\Big)$ where $\tau$ is the first Chern-class of the universal line bundle on $B\GG_m$. Further, one uses $
e^{zT} =\rho_*\Big(\sum_{j\geq 0}z^jp^j\boxtimes -\Big)$ and
\begin{equation}
\label{eq:tzconvolution}
\sum_{j\geq 0}z^jp^j\cap f(\tau) =\sum_{j\geq 0}p^jz^j  \cdot f(z)\end{equation}
for any formal power-series $f(\tau)$. The latter follows from $p^i\cap \tau^n = t^{i-n}$ for $i\geq n\geq 0$.
\item $\pi_*\mu_*(-\boxtimes e^{(1,0)}_k) = \id_{H_*(\mM_{\mA_k})}$ holds. Additionally, any power of $zT$ leads to 0 after projecting to the quotient, so the term containing $\frac{d}{dz}c_{z^{-1}}(\FG{\mV_k})$ only contains powers of $z$  less than -1.
\end{enumerate}
By \eqref{eq:bracketpushdown}, there is a left inverse of \eqref{eq:brackinj} as $\chi\big(\al(k)\big)> 0$. This proves i). 

To show ii), cap \eqref{eq:JSWC} with $c_{\chi(\al(k))-1}\big(\mV^*\otimes \FG{\mV_k}\big)=e(T_{\pi})$ and push it forward along $\pi$. I claim that this produces 
\eqref{eq:Masigdef}, which could be misconstrued as an application of \cite[§2.5]{GJT} that would require working with $c_{\chi(\al(k))}\big(\mV^*\otimes \FG{\mV_k}\big)$ instead. Because this is similar to the computation in \eqref{eq:bracketpushcomp}, I will only sketch the argument. For a K-theory class $\mK$, the notation $c_{\Rk-a}(\mK)$ will be used to represent the degree $\Rk(\mK)-a$ Chern class of $\mK$, where $\Rk(-)$ computes the rank of a K-theory class.

 Commuting $\mu_*$ appearing in the definition of the outermost bracket of each summand in \eqref{eq:JSWC} with $\cap c_{\chi(\al(k))-1}\big(\mV^*\otimes \FG{\mV_k}\big)$ produces
$$
\cap \,c_{\Rk-1}\big(\mW^*\otimes \FG{\mW_k}\big)\cdot c_{\Rk}\big(\mW^*\otimes \FG{\mV_k}\big) + \cap \,c_{\Rk}\big(\mW^*\otimes \FG{\mW_k}\big)\cdot c_{\Rk-1}\big(\mW^*\otimes \FG{\mV_k}\big)\,.
$$
I claim that the term interacting with the second summand vanishes after applying $\pi_*$. For this, choose lifts $\blangle \mM^{\sig}_{\al_i}\brangle\in H_*(\mM_{\mA_k})$, and represent 
\begin{equation}
\label{eq:pullbackofTpi}
\Big[\big\langle \mM^\sigma_{\alpha_{n-1}}\big\rangle, \cdots \Big[\big\langle \mM^\sigma_{\alpha_1}\big\rangle,  e^{(1,0)}_k\Big]\cdots\Big]\,.
\end{equation}
by a class $\mathscr{L}\in H_*(\mN_k)$ computed in terms of these lifts using \eqref{eq:LiefromVA} but without projecting to the quotient by $T$ in each step. This clearly satisfies $\mathscr{L}\cap c_1(\mW) = 0$, so $c_{\Rk-1}\big(\mW^*\otimes \FG{\mW_k}\big)$ can be replaced by $ c_{\Rk-1}\big(\FG{\mW_k}\big)$. This allows me to use the push-pull formula in (co)homology to rewrite the pushforward along $\pi$ of the term associated to the second summand in \eqref{eq:pullbackofTpi} as an expression containing 
$$
\pi_*\Big(\mL\cap c_{\Rk}\big(\mW_*\otimes \FG{\mW_k}\big)\Big)\,.
$$
To this, one may already apply \cite[§2.5]{GJT} (recalled in more detail in the proof of Theorem \ref{thm:independence!}), so it becomes 
$$\Big[\big\langle \mM^\sigma_{\alpha_{n-1}}\big\rangle, \cdots \Big[\big\langle \mM^\sigma_{\alpha_1}\big\rangle,  \pi_*e^{(1,0)}_k\Big]\cdots\Big] = 0$$
where the Lie bracket is on $L_*$. For the vanishing, I used that $\pi_*e^{(1,0)}_k = \ket{0}$. 

As such, I am left with the term coming from $\cap \,c_{\Rk-1}\big(\mW^*\otimes \FG{\mW_k}\big)\cdot c_{\Rk}\big(\mW^*\otimes \FG{\mV_k}\big)$ which can be rewritten by induction and the above vanishing as an iterated Lie bracket 
$$
\Big[\big\langle \mM^\sigma_{\alpha_{n}}\big\rangle, \cdots \Big[\big\langle \mM^\sigma_{\alpha_2}\big\rangle,  \pi_*\Big(\Big[\big\langle \mM^\sigma_{\alpha_1}\big\rangle,  e^{(1,0)}_k\Big]\cap c_{\Rk-1}\big(\mV^*\otimes \FG{\mV_k}\big)\Big)\Big]\cdots\Big]\,.
$$
The induction step uses a computation just like the one in \eqref{eq:bracketpushcomp} to deal with the factor $c_{\Rk}\big(\mW^*\otimes \FG{\mV_k}\big)$. Applying \eqref{eq:bracketpushdown}  and noting that $a$ is even, since it is the virtual dimension, completes the proof.
\end{proof}

This is why \cite{JoyceWC} uses \eqref{eq:Masigdef} to define the invariants as follows.

\begin{definition}
\label{def:Masik}
Fix $\sig\in W$ and $k\in K$. Let $\msE^{\sig}(\mA_k)\subset \msE(\mA)$ be the set of emergent classes $\alpha$ such that $\sig\in W_{\al,k}$. For a fixed $\phi\in S$ and $\alpha\in \msE^{\sig}(\mA_k)$ such that $\phi(\alpha) = \phi$, define 
    \begin{enumerate}
        \item   the classes
        $$
       \big\langle \Masi\big\rangle^k = \big[M^{\sigma}_{\alpha}\big]^{\vir}\in L_0\,,
        $$
        using Example \ref{ex:mMvir}
        whenever there are no strictly $\sigma$-semistable objects of class $\al$.
        \item the invariants $\Omega^{\sig,k}_{\alpha}\in L_0$ by 
        $$ \chi\big(\alpha(k)\big)\,\Omega^{\sig,k}_{\alpha}=(\pi^{\JS}_{\alpha})_*\Big(\Big[N^{\JS}_{k,\alpha}\Big]^{\vir}\cap c_{\chi(\alpha(k))-1}\big(T_{\pi^{\JS}_{\al}} \big)\Big)$$
         using Example \ref{ex:JSpairs}.
        \item the  classes $\blangle \mM^{\sig}_{\alpha}\brangle^k\in L_0$ for all $\al\in \msE^{\sig}(\mA_k)$ with $\phi(\al) =\phi$ by an induction on $\rk_{\sig}(\al)$ imposing that \eqref{eq:Masigdef} holds. If there are no strictly $\sig$-semistables of class $\al$, this is equivalent to 1) by Theorem \ref{thm:virtualpullpush} and 2). In general, the right-hand side of \eqref{eq:Masigdef} contains only $\blangle\mM^{\sig}_{\alpha_i}\brangle^k$ for $\al_i$ of strictly smaller rank due to Assumption \ref{ass:stab}.e). Thus, every term on the right-hand side is defined by the induction assumption.
    \end{enumerate}
\end{definition}
A few comments about the equivariant version are necessary. The object $(1,0)_k$ was already specified to be $\T$-equivariant. As it is simple, there is a $\ZZ$-choice of its equivariant structure which I always fix to be trivial. This fixes the classes $\Big[N^{\JS}_{k,\alpha}\Big]^{\vir}$ and thus $\Omega^{\sig,k}_{\alpha}$ in a way that the resulting wall-crossing formulae \eqref{eq:MasiWC} contain expressions with compatible equivariant weights. Without the extra requirement, these equations would not hold. Note also that even if we are working with the \textit{local approach}, the relative tangent bundle $T_{\pi^{\JS}_{\alpha}}$ is taken for the usual projection $\pi^{\JS}_{\al}$ from \eqref{eq:piJSal}. The relative tangent bundle for 
$$\big(\pi^{\JS}_{\al}\big)^{\T}: \begin{tikzcd} \big(N^{\JS}_{k,\alpha}\big)^{\T}\arrow[r]&\big(\mM^{\T}_{\alpha}\big)^{\rig}\end{tikzcd}$$
is different and not the correct one, so we restrict $T_{\pi^{\JS}_{\alpha}}$ to the $\T$-fixed locus of $N^{\JS}_{k,\alpha}$ where $\Big[N^{\JS}_{k,\alpha}\Big]^{\vir}_{\T,\loc}$ is defined by equivariant localization as in \eqref{eq:Msigallocvir}. The definition of $\Omega^{\sig,k}_{\al}$ then becomes
$$ \chi\big(\alpha(k)\big)\,\Omega^{\sig,k}_{\alpha}=(\pi^{\JS}_{\alpha})^{\T}_*\Big(\Big[N^{\JS}_{k,\alpha}\Big]^{\vir}_{\T, \loc}\cap c_{\Rk}\big(T_{\pi^{\JS}_{\al}} \big)\Big)\,.$$
This is still compatible with the classes in point 1) if the map $\iota_{\T}$ from Example \ref{ex:mMvir} determines the equivariant structure of fixed points compatible with $(1,0)_k$ having trivial weight in $\big(N^{\JS}_{k,\alpha}\big)^{\T}$.
\begin{remark}
\label{rem:invdef}
\leavevmode
\vspace{-4pt}
\begin{enumerate}[label=\roman*)]

  \item  There is a simple way to think about the above definition. In step 2), one acts as if $N^{\JS}_{k,\alpha}$ were a projective bundle over $(\mM^{\sig}_{\al})^{\rig}$, and defines the enumerative invariant by applying this perspective. In 3), one corrects this assumption by removing from $\chi\big(\al(k)\big)\Omega^{\sig,k}_\al$ all strictly $\sig$-semistable contributions. When there are no strictly semistable objects, then one gets a projective bundle by Example \ref{ex:JSpairs} and there are no corrections. In particular, one recovers the case in Definition \ref{def:Masik}.1). In the first step of the induction, one needs to remove the loci where the Harder--Narasimhan associated graded of an object $E$ becomes the direct sum of $E_1$ and $E_2$ with classes $\alpha_1$ and $\alpha_2$ respectively. Assuming that the wall-crossing holds, the resulting contribution should correspond to
    $$
\Big[\big\langle\Masit\big\rangle,\langle\Masio\big\rangle \Big]\,.
    $$
    Repeating this idea motivates \eqref{eq:Masigdef}.
    \end{enumerate}
\end{remark}
The last assumption I make requires that the above construction is independent of different choices of $k\in K$.
\begin{assumption}
    \label{ass:welldef}
    For any $\al\in E(\mA)$, $\sig\in W$, and $k_1,k_2\in K$ such that $\sig\in W_{\al,k_1}\cap W_{\al,k_2}$ the equality
    $$
\blangle\mM^{\sig}_{\alpha}\brangle^{k_1}=\blangle\mM^{\sig}_{\alpha}\brangle^{k_2}
    $$
    holds.
\end{assumption}
\section{Checking assumptions}
After showing that Example \ref{ex:sheavesandquivers}.i) provides a collection of data required by Definition \ref{Def:quiverpairs}, this section explains why Assumption \ref{ass:obsonflag} holds for Calabi--Yau dg-quivers. On the other hand, I point out where it goes wrong for sheaves on Calabi--Yau fourfolds. The last subsection addresses Assumption \ref{ass:welldef} in the case of semistable torsion-free sheaves by reducing it to the Joyce--Song pair wall-crossing formula \eqref{eq:JSWC}. This approach is entirely different from the one that appeared in \cite{JoyceWC} and \cite{mochizuki} for this purpose. It relies on a quantum Lefschetz-type argument for the Joyce--Song pairs and is meant to give more insight into why Assumption \ref{ass:welldef} should hold. 
\subsection{Example \ref{ex:sheavesandquivers}.i) fits into Definition \ref{Def:quiverpairs}}
\label{subsec:exfitsdef}
I check that the framework of Definition \ref{Def:quiverpairs} is suitable for sheaf wall-crossing although I will not prove it in full generality in this work. The language introduced in Definition \ref{Def:quiverpairs} will be used in the subsequent work extending the current results, so the next lemma paves the way for future developments. 
\begin{lemma}
\label{lem:sheavesandquivers}
The data from Example \ref{ex:sheavesandquivers}.i) satisfies the conditions of Definition \ref{Def:quiverpairs}.a).
\end{lemma}
\begin{proof}
Firstly, I will show that the functor $C$ from \eqref{eq:functorC} is exact and fully faithful. Let
$$P^\bullet = [\mO_X(-D_k)\to E]$$
be the complex $C(P)$, then there exists a distinguished triangle 
$$
\begin{tikzcd}
   E\arrow[r]&P^\bullet \arrow[r]&V\otimes \mO_X(-D_k)[1]\arrow[r]&E[1] \,.
\end{tikzcd}
$$
For any $E_1\in \mA_k$ and a vector space $V_2$ one sees that
$$
\Ext^i\big(F_1,V_2\otimes \mO_X(-D_k)\big)= V_2\otimes H^{4-i}\big( F_1(D_k)\big)^* = 0 \quad \textnormal{for}\quad i=0,1\,.
$$
Starting from the diagram of full arrows
\begin{equation}
\label{eq:fillinginf}
\begin{tikzcd}
   E_1\arrow[d, dashed]\arrow[r]&P_1^\bullet\arrow[d,"f"] \arrow[r]&V_1\otimes \mO_X(-D_k)[1]\arrow[d,dashed]\arrow[r]&E_1[1]\arrow[d,dashed]\\
    E_2\arrow[r]&P_2^\bullet \arrow[r]&V_2\otimes \mO_X(-D_k)[1]\arrow[r]&E_2[1]
\end{tikzcd}
\end{equation}
with both rows being distinguished triangles, one can use the above vanishings to fill in uniquely the dashed arrows and obtain a commutative diagram. This proves fully faithfulness of $C$. The functor clearly maps exact triples to distinguished triangles, but to see that every distinguished triangle of the form 
\begin{equation}
\label{eq:Pbullettriangle}
\begin{tikzcd}
    P^{\bullet}_1\arrow[r]&P^{\bullet}_2\arrow[r]&P^{\bullet}_3\arrow[r]&P^{\bullet}_1[1]
\end{tikzcd}
\end{equation}
where $P^{\bullet}_i = C(P_i)$ can be represented in this way requires additional work. By the above, the first two morphisms in \eqref{eq:Pbullettriangle} can be uniquely represented by morphisms $P_1\xrightarrow{f_1}P_2\xrightarrow{f_2}P_3$ in $\mB_k$. By similar arguments as above, one can show that the cokernel of $f_1$ is a direct summand of $P_3$. If its complement $\Delta_3$ were nontrivial, the object $P^{\bullet }_1$ would split as $C(P'_1)\oplus C(\Delta_3)[-1]$ where $P'_1$ is the kernel of $f_2$. This contradicts the claim that $P^{\bullet}_1\in C(\mB_k)$.

Further, I need to prove that $\Omega_C: \mN_k\to \mM_X$ is an open embedding, which I will do by comparing their obstruction theories. Setting $\ov{\FF}$ to be the natural obstruction theory on $\mM_X$ (inherited from the derived stack $\bM_X$ explained in Example \ref{ex:bMstacks}), its pullback $\FF = \Omega_C^*\ov{\FF}$ is given by 
$$
\FF = \RHom_{\mN_k}\big(\mP_k,\mP_k\big)
$$
where $\mP_k = \big\{\mV\otimes\mO_X(-D_k)\to \mE\}$ is the universal two term complex on $X\times\mN_k$. Applying the $\infty$-bifunctor $\RHom_{\mN_k}(-,-)$ to the fiber sequence
$$
\begin{tikzcd}
\mV\otimes \mO_X(-D_k)\arrow[r]& \mE\arrow[r]& \mP_k
\end{tikzcd}
$$
in $\mD^b(X\times\mN_k)$ produces the homotopy commutative diagram
$$
\begin{tikzcd}
\RHom\big(\mV\otimes \mO_X(-D_k),\mP_k\big)&\arrow[l]\RHom\big(\mE,\mP_k\big)&\arrow[l]\RHom\big(\mP_k,\mP_k\big)\\
\arrow[u]\RHom\big(\mV\otimes \mO_X(-D_k),\mE\big)&\arrow[l]\arrow[u]\RHom\big(\mE,\mE\big)&\arrow[l]\arrow[u]\RHom\big(\mP_k,\mE\big)
\\
\arrow[u]\mV^*\otimes \mV\otimes H^\bullet(\mO_X)&\arrow[l]\arrow[u]\RHom\big(\mE,\mV\otimes\mO_X(-D_k)\big)&\arrow[l]\arrow[u]\RHom\big(\mP_k,\mV\otimes\mO_X(-D_k)\big)\\
\end{tikzcd}
$$
 in $\mD^b(\mN_k)$ where each row and column is a distinguished triangle. I omitted the subscripts $(-)_{\mN_k}$, here. Setting 
 $$
 \EE = \RHom_{\mN_k}\big(\mE,\mE)^\vee[-1]\,,
 $$
 which is the pullback of the obstruction theory complex on $\mM_{\mA_k}$, this produces the diagram \eqref{eq:obsonNk} after dualizing and shifting by $[-1]$. For this, I have used that
 $$
 \LL_{\pi_k}\simeq\RHom_{\mN_k}\big(\mV\otimes \mO_X(-D_k),\mE\big)^\vee\,.
 $$
In \cite[Definition 5.5]{JoyceWC}, Joyce constructed derived enrichments of $\mN_k$ which I will label by $\bmN_k$ here. The induced obstruction theory of $\mN_k$ associated with this derived stack is the naïve one denoted by $\wt{\FF}_k$. It was shown in \cite[(8.38)]{JoyceWC} that it fits into the cofiber sequence
$$
\begin{tikzcd}[column sep=huge]
    \LL_{\pi_k}[-1]\arrow[r,"{\begin{pmatrix}\psi_k\\ -\delta_k\end{pmatrix}}"]&\EE\oplus  \mV\otimes \mV^*[-1]\arrow[r]&\wt{\FF}_{k}\,.
\end{tikzcd}
$$
To obtain a diagram for $\FF_k$ from this one, one changes \eqref{eq:obsonNk} into 
$$
\begin{tikzcd}[column sep=huge, row sep=large]
&&\FF_k\arrow[d]\\
    \LL_{\pi_k}[-1]\arrow[d]\arrow[r,"{\begin{pmatrix}\psi_k\\ -\delta_k\end{pmatrix}}"]&\arrow[d]\EE\oplus  \mV\otimes \mV^*[-1]\arrow[r]&\arrow[d]\wt{\FF}_{k}\\
    0\arrow[r]&\LL_{\pi_{\JS}}^\vee[3]\arrow[r,equal]&\LL_{\pi_{\JS}}^\vee[3]
\end{tikzcd}
$$
where $\pi_{\JS}$ was defined in \eqref{eq:piJS}. The first two terms in the first row were left out. The vertical distinguished triangle on the right shows that the map $\FF_k\to \wt{\FF}_{k}$ induces an isomorphism on cohomologies $h^1,h^0$ and $h^{-1}$ by the smoothness of $\pi_{\JS}$. This proves that $\Omega_C$ is étale, and as it is a bijection on its image, it is an open embedding. 

\end{proof}
\begin{remark}
  \leavevmode
\vspace{-4pt}
\begin{enumerate}[label=\roman*)]
\item All of the arguments generalize to the $\T$-equivariant setting.  
\item For quasi-projective $X$, the above does not work because $H^{\bullet}(\mO_X)$ is not well behaved. However, proving that $\Omega_C:\mN_k\to \mM_X$ is an open embedding is not strictly necessary for the argument in the proof of wall-crossing. One simply needs appropriate obstruction theories on all $N^{\sig}_{\un{d},\al}$ and $N^{\sig}_{(1,\un{d}),\al}$ which can be sometimes constructed using the fixed-determinant obstruction theory for $D^b(X)$. This situation is explored in §\ref{sec:pairWC} and §\ref{sec:localCY4}.
\end{enumerate}
\end{remark}

\subsection{Checking Assumption \ref{ass:obsonflag} for quivers}
\label{sec:quivers}
In this subsection, I focus on the situation from Example \ref{ex:sheavesandquivers}.ii). 
Because the underlying classical moduli stacks $\mM_{\mA}$, $\mN_{k}$, and $\mN_I$ are just moduli stacks of representations of quivers with relations, one can use \cite[Proposition 4.3]{King} to prove their projectivity whenever there are no cycles whose composition of morphisms is non-zero after including relations. This applies to Example \ref{ex:CY4quiver}. When cycles are present, one can find a $\T$-action rescaling morphisms of these cycles such that the $\T$-fixed point loci are projective. An example corresponding to $\Hilb^n(\CC^4)$ is given in §\ref{sec:CY4dgC4}. This addresses Assumption \ref{ass:obsonflag}.a).

From now on, I will fix $\wt{Q}^\bullet$, and I will aim to construct the obstruction theories on $N^\sigma_{(1,\un{d}),\alpha}$.  Then I will show that they satisfy the conditions in Assumption \ref{ass:obsonflag}.b) and c). I will continue using the notation introduced in §\ref{sec:dgquivers} except that I will also use $\Cya{\bullet}$ to denote the degree \Cya{$1$} contributions in the cotangent complex from Lemma \ref{lem:cotangentofQ}. These terms correspond to the endomorphisms of vector spaces at the vertex.  Here, I will use Lemma \ref{lem:cotangentofQ} as a recipe rather than as a result, as it will prescribe the terms in the cotangent complex and the morphisms between them. Then the relative obstruction theory $\LL_{\pi_I}$ for the map $\pi_I:\mN_{I}\to \mM_{\mA}$ can be represented by the diagram
\begin{equation}
\label{eq:Lrela}
    \includegraphics{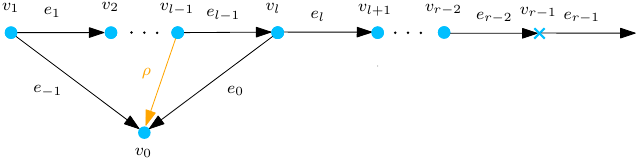}
\end{equation}
where $d(\BuOr{\rho}) = e_{l-1}\circ e_0$. 
Note that to get an actual dg-quiver, one would need to keep the vertex $v_{r}$, in which case, I will label the quiver by $I^{\bullet}_{\MS}$. The above pictorial notation will be used throughout the rest of this work.

Keep in mind that $$\begin{tikzcd}\Cya{\times}\arrow[r,"e_{r-1}"]&\,\end{tikzcd}$$ represents the relative obstruction theory $\LL_{\mN/\mM_{\mA}}$ in Example \ref{ex:sheavesandquivers}.ii) which means that it includes arrows from $\Cya{\times}$ to all original vertices of $\wt{Q}^\bullet$ (see Example \ref{ex:simplestCY4QMS}).

Using \FG{$\stackrel{v_r}{\circ}$} to represent the obstruction theory on $\mM_{\mA}$, a homotopy commutative diagram
\begin{equation}
\label{eq:NQtoMAsymmpull}
 \begin{tikzcd}
     \arrow[d]\LL_{\pi_{\MS}}[-1]\arrow[r]&\EE\arrow[d]\\
     0\arrow[r]&\LL^\vee_{\pi_{\MS}}[3]
 \end{tikzcd}
 \end{equation}
 on $\mN_{I_{\MS}}$ can be expressed as
\begin{equation}
\label{eq:constructQMSCY4}
      \includegraphics[scale=0.85]{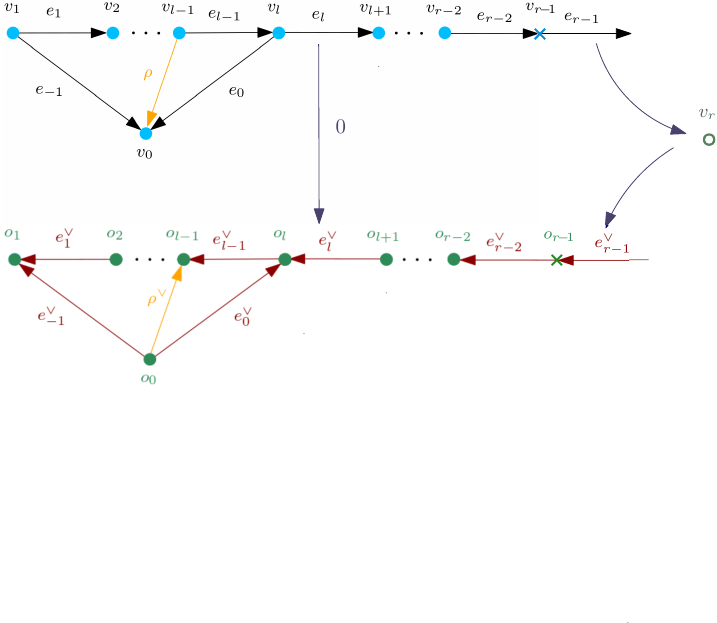} 
\end{equation}

\vspace{-100pt}
in this notation. Here, I used \FG{$\bullet$} and \FG{$\times$} to represent the \FG{loops} of degree \FG{$-3$} to avoid having to draw them. Altogether, the resulting obstruction theory $\FF_{\MS}$ on  $\mN_{I_{\MS}}$ can be represented by the CY4 dg-quiver $ \wt{I}^{\bullet}_{\MS}$
\begin{equation}
\label{eq:fullobsQMS}  
    \includegraphics{QMSCY4completed.pdf}
\end{equation}
where the superpotential that determines the differential outside of the original quiver $\wt{Q}^\bullet$ takes the form 
$$
\BuOr{\mH_{\MS}}= \BuOr{\rho^*}\circ e_0\circ e_{l-1}\,.
$$
This describes how to construct a new CY4 dg-quiver containing $\wt{Q}^{\bullet}$ and induces the required obstruction theory. I make it clearer through the next example.  
\begin{example}
\label{ex:simplestCY4QMS}
Consider the CY4 dg-quiver from Example \ref{ex:CY4quiver}, then the new CY4 dg-quiver explained above takes the form
\begin{equation}
\label{eq:QMStowtQ}
\includegraphics[scale = 0.8]{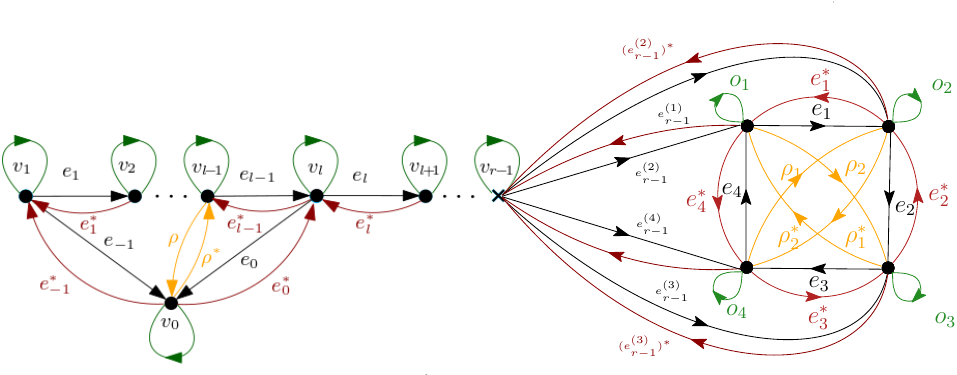}
\end{equation}
with superpotential
$$\BuOr{\mI_{\MS}}=\BuOr{\mH_{\MS}} + \BuOr{\mH}$$
for $\BuOr{\mH}$ from \eqref{eq:simplestCY4potential}. I will denote this dg-quiver by $\wt{I}^\bullet_{\MS}\cup_{\times} \wt{Q}^\bullet$. Note that there are $\Ver$ times the number of copies of $e_{r-1}$, each labeled $e^{(v)}_{r-1}$ for a vertex $v\in \Ver$ at which it ends. I also use this notation for their dual edges. 
\end{example}
The next lemma provides an explicit construction of \eqref{eq:NQtoMAsymmpull}.
\begin{lemma}
\label{lem:FFMSrig}
There is a natural homotopy commutative diagram
    \begin{equation}
 \label{eq:NrigtoMrigpullback}
 \begin{tikzcd}
     \arrow[d]\LL_{\pi_{\MS}}[-1]\arrow[r]&\EE\arrow[d]\\
     0\arrow[r]&\LL^\vee_{\pi_{\MS}}[3]
 \end{tikzcd}
 \end{equation}
 producing an obstruction theory $\FF_{\MS}$ on $\mN_{\MS}$ by Proposition \ref{prop:sympullback}.
\end{lemma}
\begin{proof}
Replacing $\wt{I}^{\bullet}_{\MS}$ by just $I^{\bullet}_{\MS}$, I will write $I^{\bullet}_{\MS}\cup_{\times} \wt{Q}^\bullet$ for the dg-quiver resulting from the construction explained in Example \ref{ex:simplestCY4QMS}. More explicitly, one obtains $I^{\bullet}_{\MS}\cup_{\times} \wt{Q}^\bullet$ from $\wt{I}^{\bullet}_{\MS}\cup_{\times} \wt{Q}^\bullet$ by removing $\BuOr{\rho^*}$ and all the \Ma{red edges} and \FG{green loops} starting or ending at a vertex not contained in $\wt{Q}^{\bullet}$. The new differential is 0 outside of the edges of $\wt{Q}^{\bullet}$ except for $d(\BuOr{\rho}) = e_0\circ e_{l-1}$. It is not difficult to check that forgetting the extra edges is compatible with the differential and thus induces the left morphism of dg-algebras in the diagram
$$
\begin{tikzcd}
\CC \big(\wt{I}^\bullet_{\MS}\cup_{\times} \wt{Q}^\bullet\big)\arrow[r]\arrow[r]&\CC \big(I^\bullet_{\MS}\cup_{\times} \wt{Q}^\bullet\big) &\arrow[l]\CC \wt{Q}^{\bullet}
\end{tikzcd}\,.
$$
The second morphism corresponds to adding edges in degree 0 and \BuOr{$-1$}. Together they lead to the following diagram of categories of degree 0 modules:
\begin{equation}
\label{eq:tildeIMSIMSmodu}
\begin{tikzcd}
\Rep\Big(\wt{I}^{\bullet}_{\MS}\cup_{\times}\wt{Q}^{\bullet}\Big)&\arrow[l]\Rep\Big(I^{\bullet}_{\MS}\cup_{\times}\wt{Q}^{\bullet}\Big)\arrow[r]&\Rep\Big(\wt{Q}^{\bullet}\Big)
\end{tikzcd}\,.
\end{equation}

Consider the following derived stacks of modules of degree 0:
\begin{itemize}
    \item the stack $\bmN_{\wt{I}_{\MS}}$ of degree-0 representations of $\wt{I}^\bullet_{\MS}\cup_{\times} \wt{Q}^\bullet$,
    \item the stack $\bmN_{I_{\MS}}$ of degree-0 representations of $I^{\bullet}_{\MS}\cup_{\times} \wt{Q}^\bullet$,
    \item the stack $\bmM$ of degree-0 representations of $\wt{Q}^{\bullet}$.
\end{itemize}
Then \eqref{eq:tildeIMSIMSmodu} determines the diagram of derived stacks
\begin{equation}
\label{eq:quiverlagcor}
\begin{tikzcd}
  & \arrow[dl] \bmN_{I_{\MS}}\arrow[dr]&\\
  \bmN_{\wt{I}_{\MS}}&&\bmM
\end{tikzcd}\,.
\end{equation}
Note that this is an example of a $-2$-shifted Lagrangian correspondence, but I will not pursue this perspective further here. Taking the distinguished triangle of cotangent complexes induced by the morphism on the right produces the bottom row of 
$$
    \begin{tikzcd}[column sep=large]   
   0\arrow[d]\arrow[r]&\LL^{\vee}_{\pi_{\MS}}[2]\arrow[d]\arrow[r,equal]&\LL^{\vee}_{\pi_{\MS}}[2]\arrow[d]\arrow[r]&0\arrow[d]\\
    \LL_{\pi_{\MS}}[-1]\arrow[d,equal]\arrow[r]&\arrow[d]\wt{\FF}_{\MS}^\vee[2]\arrow[r]&\FF_{\MS}\arrow[d]\arrow[r]&\LL_{\pi_{\MS}}\arrow[d,equal]\\
  \LL_{\pi_{\MS}}[-1]\arrow[r]&\EE\arrow[r]&\wt{\FF}_{\MS}\arrow[r]&\LL_{\pi_{\MS}}
    \end{tikzcd}\,.
$$
The left morphism of \eqref{eq:quiverlagcor} recovers the third column of the diagram. Using Lemma \ref{lem:cotangentofQ}, it is not difficult to see that this leads to an $\infty$-PVP diagram along $\pi_{\MS}$. In particular, it implies \eqref{eq:NQtoMAsymmpull} and is recovered from it by applying Proposition \ref{prop:sympullback}. 
\end{proof}
The next Proposition shows that the obstruction theory $\FF_{\MS}$ is the appropriate one for wall-crossing. 
\begin{proposition}
\label{prop:construct}
Assumption \ref{ass:obsonflag}.b) and c) are satisfied by the restriction to $\mN^{\sig}_{(1,\un{d}),\al}$ of the obstruction theory $\FF_{\MS}$ constructed from \eqref{eq:NrigtoMrigpullback}. More generally, there is a  self-dual homotopy commutative diagram 
\begin{equation}
\label{eq:JStoFLpullbackrig}
\begin{tikzcd}    \arrow[d]\LL_{\pi_{\Flag/\JS}}[-1]\arrow[r]&\FF\arrow[d]\\
     0\arrow[r]&\LL^\vee_{\pi_{\Flag/\JS}}[3]
 \end{tikzcd}
\end{equation}
inducing an obstruction theory $\FF_{\Flag}$ of $\mN_{\Flag}$ from $\FF$ on $\mN$.
Another one
\begin{equation}
\label{eq:FltoMSpullbackrig}
\begin{tikzcd}    \arrow[d]\LL_{\pi_{\MS/\Flag}}[-1]\arrow[r]&\FF_{\Flag}\arrow[d]\\
     0\arrow[r]&\LL^\vee_{\pi_{\MS/\Flag}}[3]
 \end{tikzcd}
\end{equation}
recovers the obstruction theory $\FF_{\MS}$ . Furthermore, the complex $\FF_{\Flag}$ satisfies 
$$
\mu_{\Flag}^*\big(\FF_{\Flag}\big) = \FF_{\Flag}\boxplus \FF_{\Flag}\oplus \Theta_{\Flag}\oplus \sigma^*\Theta_{\Flag}
$$
where $\Theta_{\Flag}$ is determined by \eqref{eq:extendTheta}.
\end{proposition}
\begin{proof}
Consider the unique CY4 quivers $\wt{I}^\bullet_{\Flag}$ and $\wt{I}^\bullet_{\JS}$ 
$$
\includegraphics[scale=0.8]{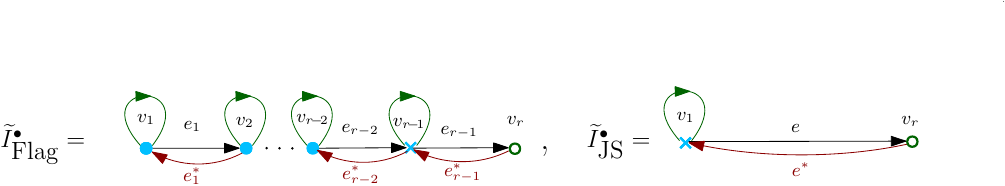}
$$
constructed from $I_{\Flag}$ and $I_{\JS}$ respectively. As represented in Example \ref{ex:simplestCY4QMS}, I attach the original quiver $\wt{Q}^\bullet$ to the vertex $\times$ in all four cases (see also the proof of Lemma \ref{lem:FFMSrig}). The new quivers constructed this way will be labeled by 
$$
\wt{I}^{\bullet}_{\Flag}\cup_{\times}\wt{Q}^{\bullet},\quad I_{\Flag}\cup_{\times}\wt{Q}^{\bullet}, \quad \wt{I}^{\bullet}_{\JS}\cup_{\times}\wt{Q}^{\bullet},\quad \text{and } I_{\JS}\cup_{\times}\wt{Q}^{\bullet}\,.
$$
Furthermore, there will be the quivers
\begin{center}
    \includegraphics[scale=0.9]{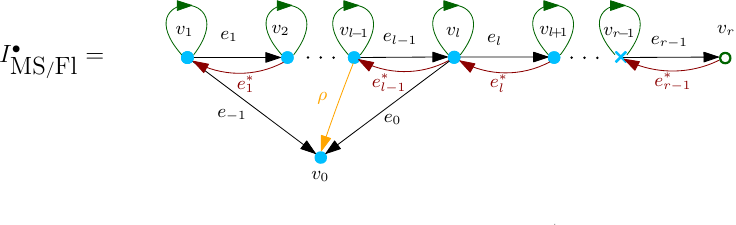}
    \includegraphics[scale=0.9]{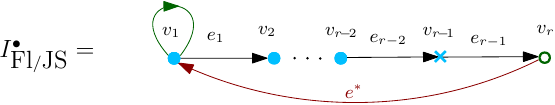}
\end{center}
interpolating between $\wt{I}^{\bullet}_{\MS}, \wt{I}^{\bullet}_{\Flag}$, and $\wt{I}^{\bullet}_{\JS}$. For $I^{\bullet}_{\Flag/\JS}$, the differential of the \FG{loop} at $v_1$ is equal to $-\Ma{e^*}\circ e$ where $e_=e_{r-1}\circ\cdots\circ e_1$. These dg-quivers also come with the combined $I^{\bullet}_{\MS/\Flag}\cup_{\times}\wt{Q}^{\bullet}$ and $I^{\bullet}_{\Flag/\JS}\cup_{\times}\wt{Q}^{\bullet}$, where the latter quiver takes the following form in the case of Example \ref{ex:CY4quiver}:
\begin{center}
    \includegraphics[scale = 0.85]{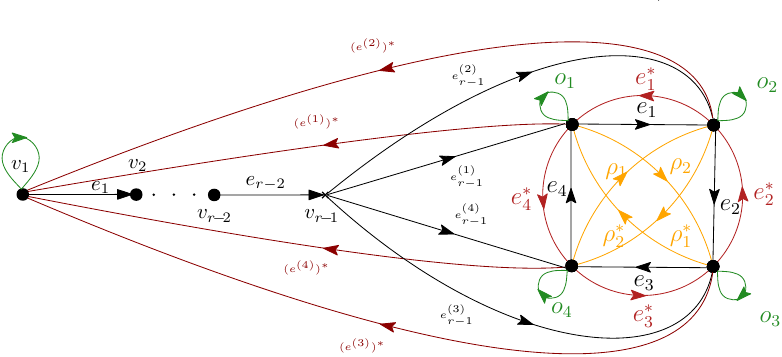}
\end{center}
Following the same notation convention as in \eqref{eq:quiverlagcor} and using $\mbN = \mbN_{I_{\JS}}$, there is then the following diagram of derived stacks refining it:
\begin{equation}
\label{eq:diagofroofs}
\begin{tikzcd}[column sep={4.5em,between origins},
  row sep={4.2em,between origins}]
&&&\arrow[ddll]\bmN_{I_{\MS}}\arrow[dr]&&\\
&&&&\arrow[dl]\bmN_{I_{\Flag}}\arrow[dr, blue, "\boldsymbol{\pi}_{\Flag/\JS}"]&\\
&\arrow[dl]\arrow[dr]\bmN_{I_{\MS/\Flag}}&&\arrow[dl, blue]\bmN_{I_{\Flag/\JS}}\arrow[dr, blue] &&\arrow[dl]\bmN\arrow[dr]& \\
\bmN_{\wt{I}_{\MS}}&&\bmN_{\wt{I}_{\Flag}}&&\bmN_{\wt{I}_{\JS}}&&\bmM
\end{tikzcd}
\end{equation}
Each morphism follows from a morphism of the dg-path algebras of the underlying quivers. The black arrows represent morphisms which just like the ones in \eqref{eq:quiverlagcor} originate from forgetting or adding edges. The \B{blue morphisms} are obtained by composing with degree 0 edges. As \B{$\boldsymbol{\pi}_{\Flag/\JS}$} is just the derived refinement of $\pi_{\Flag/\JS}$ from \eqref{eq:FlagtoProj}, I will only focus on the left-over roof diagram of \B{blue arrows}. The map 
$$
\begin{tikzcd}
\CC\big(\wt{I}^{\bullet}_{\JS}\cup_{\times} \wt{Q}^{\bullet}\big)\arrow[r]&\CC\big(I^{\bullet}_{\Flag/\JS}\cup_{\times} \wt{Q}^{\bullet}\big)
\end{tikzcd}
$$
is obvious, as one maps the edges $e^{(v)}$ to $e^{(v)}_{(r-1)}\circ e_{r-2}\circ \cdots \circ e_{1}$ for each $v\in \Ver$. The morphism 
$$
\begin{tikzcd}
\CC\big(\wt{I}^{\bullet}_{\Flag}\cup_{\times} \wt{Q}^{\bullet}\big)\arrow[r]&\CC\big(I^{\bullet}_{\Flag/\JS}\cup_{\times} \wt{Q}^{\bullet}\big)
\end{tikzcd}
$$
is determined by 
$$
\begin{tikzcd}
    \Ma{(e^{(v)}_{r-1})^*}\arrow[r,mapsto] &e_{r-2}\circ\cdots  \circ e_{1}\circ\Ma{(e^{(v)})^*} &\textnormal{for}\quad v\in\Ver\,,\\
   \Ma{e^*_{i}}\arrow[r,mapsto]&\begin{array}{c}\sum\limits_{v\in\Ver} e_{i-1}\circ \cdots \circ e_{1}\circ \Ma{(e^{(v)})^*}\\
   \circ e^{(v)}_{r-1}\circ e_{r-2}\circ \cdots \circ e_{i+1}\end{array}&\textnormal{for}\quad i=1,2,\ldots, r-2\,.
\end{tikzcd}
$$
Recall here that $\FF$ is the obstruction theory of $\mN$. Using Lemma \ref{lem:cotangentofQ}, one sees that this roof diagram induces a self-dual homotopy commutative diagram
\begin{equation}
\label{eq:JStoFLpullback}
\begin{tikzcd}    \arrow[d]\LL_{\pi_{\Flag/\JS}}[-1]\arrow[r]&\FF\arrow[d]\\
     0\arrow[r]&\LL^\vee_{\pi_{\Flag/\JS}}[3]
 \end{tikzcd}
\end{equation}
in the same way as \eqref{eq:quiverlagcor} gave \eqref{eq:NQtoMAsymmpull}. 

In fact, all the roofs whose leftmost and rightmost terms are in the same row should be $-2$-shifted Lagrangian correspondences, and they give rise to $\infty$-PVP diagrams. Furthermore, the roof between $\bmN_{\wt{I}_{\MS}}$ and $\bmN_{\wt{I}_{\Flag}}$ and the roof between $\bmN_{\wt{I}_{\JS}}$ and $\bmM$ produce  
    \begin{equation}
\label{eq:FltoMSpullback}
\begin{tikzcd}    \arrow[d]\LL_{\pi_{\MS/\Flag}}[-1]\arrow[r]&\FF_{\Flag}\arrow[d]\\
     0\arrow[r]&\LL^\vee_{\pi_{\MS/\Flag}}[3]
 \end{tikzcd}
\end{equation}
and \eqref{eq:pairtosheaf} respectively. Due to the diagram \eqref{eq:diagofroofs}, one knows that composing \eqref{eq:pairtosheaf}, \eqref{eq:JStoFLpullback}, and \eqref{eq:FltoMSpullback} in the sense of Theorem \ref{thm:functsympull} gives the diagram \eqref{eq:NQtoMAsymmpull} constructed from the roof \eqref{eq:quiverlagcor} because it is equal to the largest roof in \eqref{eq:diagofroofs}.

I am now left to prove the last sentence of the lemma. Because $\wt{I}^{\bullet}_{\Flag}\cup_{\times}\wt{Q}^{\bullet}$ is again a CY4 quiver, I can use \eqref{eq:sumEE} together with \eqref{eq:DeltaExtLL}. Choose  $$\Theta_{\Flag} = \mExt_{\mD}^\vee[-1]|_{\mN_{\Flag}\times \mN_{\Flag}}$$
for $\mD$ the dg-category of  $\wt{I}^{\bullet}_{\Flag}\cup_{\times}\wt{Q}^{\bullet}$ dg-modules. Then, this complex can be described using \eqref{eq:extendTheta} due to Lemma \ref{lem:cotangentofQ}\footnote{More precisely its generalization to the $\mExt$-complex which can also be easily derived.}.
\end{proof}
\subsection{Absence of $\FF_{\Flag}$ for sheaves in general and examples}
\label{sec:sheaves}
I will now fix $\mA = \Coh(X)$ for a projective Calabi--Yau fourfold $X$ and the data from Example \ref{ex:sheavesandquivers}.i). The question of the existence of obstruction theories from Assumption \ref{ass:obsonflag} is more intricate in this case. Lemma \ref{lem:pairshenafcorrespondence} and \eqref{eq:pairtosheaf} imply that $\FF^{\sig}_{1,\al}$ and $\FF^{\rig}_{1,\al}$ can both be constructed from self-dual homotopy commutative diagrams. Thus, one has the required diagrams \eqref{eq:Ndalobstruction} and \eqref{eq:Ndalrigobstruction} for the projective bundle $\pi^{\JS}_{\al}: N^{\JS}_{k,\al}\to M^{\sig}_{\al}$ in the absence of strictly $\sigma$-semistables. The semistable objects for $I=I_{\Flag}$ and the stability \eqref{eq:barphi} form flag bundles over $M^{\sig}_{\al}$, so the moduli spaces $N^{\sig}_{\un{d},\al}$ can be described as iterated projective bundles over $M^{\sig}_{\al}$ in this case. One may hope, therefore, that there would be an analogue of \eqref{eq:pairtosheaf} and Lemma \ref{lem:pairshenafcorrespondence} for $Q_{\Flag}$. 

One way I was hoping to approach the construction of $\FF^{\sig}_{\un{d},\al}$ was by trying to mimic the situation in \eqref{eq:constructQMSCY4} which would have been too difficult to do directly because 
$$
\EE = \RHom_{\mM_{\mA}}(\mE,\mE)^\vee[-1]
$$
is too complicated to understand in full generality. Instead, one can already start from 
\begin{center}
\includegraphics[]{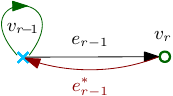}
\end{center}
and attach the rest of the obstruction theory at $\stackrel{v_{r-1}}{\FG{\times}}$ rather than at $\stackrel{v_{r}}{\FG{\circ}}$. Explicitly, this translates to constructing an $\infty$-Pvp diagram along the natural projection 
$$
\begin{tikzcd}
\pi_{\Flag/\mN_k}: \mN_{\Flag_k}\arrow[r]&\mN_{k}
\end{tikzcd}
$$
forgetting all the arrows to the left of $\stackrel{v_{r-1}}{\times}$. Such a diagram would be determined by a self-dual homotopy commutative diagram of the form
\begin{equation}
 \begin{tikzcd}
     \arrow[d]\LL_{\pi_{\Flag/\mN_k}}[-1]\arrow[r]&\FF_k\arrow[d]\\
     0\arrow[r]&\LL^\vee_{\pi_{\Flag/\mN_k}}[3]
 \end{tikzcd}\,.
 \end{equation}
Using \eqref{eq:obsonNk} describing $\FF_k$, one concludes that this is equivalent to requiring that both diagrams
$$
\begin{tikzcd}
     \arrow[d]\LL_{\pi_{\Flag/\mN_k}}[-1]\arrow[r]&\mV\otimes \mV^*\otimes\HH\arrow[d]\\
     0\arrow[r]&\LL^\vee_{\pi_{\Flag/\mN_k}}[3]
 \end{tikzcd}\qquad\qquad \begin{tikzcd}
     \arrow[d]\LL_{\pi_{\Flag/\mN_k}}[-1]\arrow[r]&\mV\otimes \mV^*\otimes\HH\arrow[d,"{\delta^\vee_k[2]}"]\\
     0\arrow[r]&\LL^\vee_{\pi_{k}}[3]
 \end{tikzcd}
 $$
 can be made homotopy commutative. The first diagram is satisfied immediately because one can write 
 $$\mV\otimes \mV^*\otimes\HH = \mV\otimes\mV^*[-1]\oplus \mV\otimes\mV^*[3]\,,$$
 and the maps factor through either of the factors on the right-hand side. I mistakenly believed that I had a proof of the same factorization for the morphisms in the second diagram which would have implies homotopy commutativity.. In fact, it would have been enough to prove that the composition \begin{equation}
 \label{eq:compvanishwrong}
 \begin{tikzcd}\mV\otimes \mV^*[-1]\arrow[r]&\mV\otimes \mV^*\otimes\HH\arrow[r]& \LL^{\vee}_{\pi_{\Flag/\mN_k}}[3]\end{tikzcd}\end{equation}
 is null-homotopic. Below, I will provide two counterexamples to this statement. Alternative constructions of $\FF^{\sig}_{\un{d},\al}$ that I have attempted always run into the same issue. This convinced me that $\FF^{\sig}_{\un{d},\al}$ does not exist in general due to the obstruction in \eqref{eq:compvanishwrong}. I would like to thank Sasha who provided the first example in \cite{MOsasha}.
 \begin{example}
 \label{ex:counterexamples}
 All the situations below satisfy $H^0\big((F(D_k)\big) \cong \CC$ for all sheaves $F$ considered and $D_k$ sufficiently positive as in Example \ref{ex:sheavesandquivers}.i). This would make it enough to work with $Q_{\JS}$, so Lemma \ref{lem:pairshenafcorrespondence} would provide the obstruction theory of the necessary form. However, suppose that $\alpha$ is the class of such $F$'s and the condition \eqref{eq:compvanishwrong} is satisfied for $n\alpha$. As there is a direct sum map
 $$
 \begin{tikzcd}
 \mu_k: \mN_k\times \mN_k\arrow[r]&\mN_k
 \end{tikzcd}
 $$
 inducing $\prod_{i=1}^n\mN_{1,\al}\to \mN_{n,n\al}$, one can restrict the composed morphism along the latter. The vanishing of \eqref{eq:compvanishwrong} on $\mN_{n,n\al}$ would imply it for $\mN_{1,\al}$. Therefore, the cases considered below are valid counter-examples.
     \begin{enumerate}
         \item Consider the moduli space $M_{p}$ parametrizing sky-scrapper sheaves (as $p\in K^0(X)$ here stands for the K-theory class of such sheaves). In this case, the divisor $D_k$ can be chosen to be empty. The Joyce--Song pair moduli space is given by
         $$N^{\JS}_{k,\al} = \Hilb^1(X) = X\,.$$
         The universal sheaf on $X\times X$ is the structure sheaf of the diagonal $\mE = \mO_{\Delta(X)}$. The first map in \eqref{eq:compvanishwrong} gives rise to the element 
         $$
         \tau\in H^4\big(H^\bullet(\mO_X)\otimes \mO_X\big) \cong  H^4(\mO_X)\otimes \CC\oplus \CC\otimes H^4(\mO_X)
         $$
         corresponding to $(1,0)$ under the isomorphisms $H^4(\mO_X)\cong H^0(\mO_X)^*\cong \CC$. Pushing the universal pair $\mO_{X\times X}\to \mO_{\Delta(X)}$ forward along the projection to the second factor gives 
         $$
         \begin{tikzcd}
             H^{\bullet}(\mO_X)\otimes \mO_X\arrow[r]&\mO_X\,.
         \end{tikzcd}
         $$
         Applying $H^4(-)$ to the above morphism gives the direct sum map $H^4(\mO_X)\oplus H^4(\mO_X)\to H^4(\mO_X)$.
         Acting with it on $\tau$ recovers the composition \eqref{eq:compvanishwrong} and shows that it is non-zero.  
         \item To show that the above considerations are not limited to 0-dimensional sheaves, I will consider line bundles here. I also reduce the dimension by working with elliptic curves $E$ where the analogous problem can be formulated. To get back to four dimensions, one can take products of elliptic curves. 

         Consider the Jacobian of $E$ which is the moduli space of degree 0 line bundles. It is naturally isomorphic to $E$ and the universal line bundle on $E\times E$ is given in terms of $E_2 = \{\pt\}\times E$ by $\mO_{E\times E}(\Delta - E_2)$. I used $\pt$ to denote a point of $E$ and $\Delta = \Delta(E)$ for simplicity. It is now sufficient to choose $D_k = E_2$ which leads to $N^{\JS}_{k,\al} = E$ with the universal pair \eqref{eq:universaltriple} after tensoring with $\mO(E_2)$ given by 
         \begin{equation}
         \label{eq:universalpairPicE}
         \begin{tikzcd}
         \mO_{E\times E}\arrow[r]&\mO_{E\times E}(\Delta)\,.
         \end{tikzcd}
         \end{equation}
         The first part of the argument describing $\tau\in H^1(\mO_E)\otimes \CC\oplus \CC\otimes H^1(\mO_E)$ is identical. Taking $H^1(-)$ of \eqref{eq:universalpairPicE} again induces the direct sum $H^1(\mO_E)\oplus H^1(\mO_E)\to H^1(\mO_E)$ showing that the analogue of \eqref{eq:compvanishwrong} is non-zero. 
         \item The last example continues along the same direction as I look at elliptic fibrations $\phi: X\to B$ where $X$ is CY4 and $B$ is smooth. Consider the moduli space $M_E$ of 1-dimensional sheaves with Chern character $(0,0,0,E,0)$ for the fiber class $E$. If the Picard rank of $B$ is 1, then the argument in \cite[Lemma 2.1]{CMT} implies that $M_E\cong X$. Suppose that $\phi$ admits a section with the image $H_{\phi}$, and choose $D_k=H_{\phi}$ leading to 
         $N^{\JS}_{D,\al} = X$. Taking the universal pair on $X\times X$ tensored by $\mO(H_\phi)$ on the first factor and projecting to the second factor gives 
         $$
         \begin{tikzcd}
         H^\bullet(\mO_X)\otimes \mO_X\arrow[r]&\mO_X
         \end{tikzcd}\,.
         $$
    In the same manner as in i), one concludes that \eqref{eq:compvanishwrong} does not vanish. 
     \end{enumerate}
 \end{example}
 The above counterexamples include dimension 0, dimension 1 and torsion-free sheaves on Calabi--Yau fourfolds. As such, they should provide compelling evidence that the obstruction theories $\FF_{\Flag}$ could exist only in special cases. The following example provides ones such situation.
 \begin{example}
 \label{ex:fibrations}
     Let $\phi: X\to B$ be a flat surjective morphism with $B$ a smooth base of strictly lower dimension than 4. Let $\mM^{\sig}_{\al}$ be a moduli stack of $\sigma$-semistable sheaves of class $\al$. Suppose that each such sheaf has the form $E = \phi^*E_B$, i.e., one can identify $\mM^{\sig}_{\al}$ with a moduli stack of sheaves on $B$. Let $L$ be a sufficiently ample line bundle on $X$ such that $R\phi_*(L)=T$ is a vector bundle in degree $0$.  In the definition of $\mN_{k}$ set $\mO_X(D_k)=L$. 
     
     Denote by $\phi_{\mM}: X\times \mM^{\sig}_{\al}\to B\times \mM^{\sig}_{\al}$ the action of $\phi$ on the first factor. The second morphism in \eqref{eq:compvanishwrong} takes the form
     $$
     \begin{tikzcd}
    \mV^*\otimes \mV \otimes H^{\bullet}(\mO_X)\arrow[r]&\RHom_{\mM^{\sig}_{\al}}\big(\mO_{X\times \mM^{\sig}_{\al}}, \phi^*_{\mM}(\mE_B)\otimes L\big)
    \end{tikzcd}
     $$
     for the universal sheaf $\mE_B$ on $B\times \mM^{\sig}_{\al}$. Using \cite{Kollar2}, one knows that $$R\phi_* \mO_X= \mO_B \oplus \bigoplus_{i>0}R^i\phi_*(\mO_X)[-i]\,,$$ which implies that the above can be factored as 
     $$
     \begin{tikzcd}
    \mV^*\otimes \mV\otimes H^{\bullet}(\mO_X)\arrow[r]& \mV^*\otimes \mV\otimes  H^{\bullet}(\mO_B)\arrow[r]&\RHom_{\mM^{\sig}_{\al}}\big(\mO_{B\times \mM^{\sig}_{\al}}, \mE_B\otimes T\big)
    \end{tikzcd}\,.
     $$
     Consequently, the composition \eqref{eq:compvanishwrong} is zero. 
 \end{example}
 This situation is realized for example when $\phi:X\to B$ is an elliptic fibration. In this case the moduli space of PT stable pairs $\mO_X\to F$ with $\ch(F)=(0,0,0,dE,0)$ can be identified with $\Hilb^d(B)$ via pullback along $\phi$. Here the flags would be constructed for $H^0\big(F(k)\big)$ as was already done in \cite{mochizuki}, so the discussion in Example \ref{ex:fibrations} still applies. More generally, this works for the surface counting theories of \cite{GJL, BKP} under some restrictions on geometry. This will be the focus of one of the future works.

 Another example utilizes the spectral correspondence when $X$ is a total space of a canonical bundle over a three-fold. I will postpone the precise discussion of this set up to §\ref{sec:localCY4}, where I will also give the first complete wall-crossing formula between stable pair invariants on local Calabi--Yau fourfolds. 
 \subsection{Minor adaptation for pairs}
 \label{sec:pairWC}
In many situation, one wants to work with different hearts $\mB\subset D^b(X)$ other than $\mA=\Coh(X)$. Disappointingly, it seems difficult to find objects $(1,0)$ from Definition \ref{Def:quiverpairs} in such generality. Instead, a semi-stable object in $\mB$ is represented by an explicit complex if such a description is unique. Due to \cite{PT, todacurve} and \cite{GJL, BKP}, this is the case for all the situations that will be considered in the sequels. 
\begin{example}
\label{ex:DTPT}
For a fixed $X$, set $\Coh_{\leq d}(X)$ to be the full subcategory of sheaves supported in dimension $\leq d$ and $\Coh_{> d}(X)$ to be the full subcategory of sheaves with no torsion in dimension $\leq d$. Both \cite{Bayer} and \cite{todacurve} consider the 3-fold analogue of the heart
$$
\mB = \big\langle \Coh_{>1}(X)[1], \Coh_{\leq 1}(X)\big\rangle  
$$
and describe a family of stability conditions $t\mapsto \sigma_{t}$ for $t=[0,1]$ such that all $\sigma_{t}$-semistable objects $P^\bullet \in \mB$ with Chern character
$$
\ch(P^\bullet) = (-1,0,0,[C],n)\in H^*(X)\qquad \text{for}\quad [C] \in H^6(X)
$$
have the form $P^\bullet = \{\mO_X\xrightarrow{s} F\}$ for a one dimensional sheaf $F$. Here, I use \cite{todacurve} where weak stability conditions of Example \ref{ex:assstab}.ii) have been introduced and used for this purpose. When
\begin{itemize}
    \item $t=0$, the $\sig_1$-semistable objects of this form are those that satisfy $\coker(s)\in \Coh_{\leq 0}(X)$ and $F\in \Coh_{>0}(X)$. They are called PT-stable pairs. 
    \item $t=1$, the $\sig_0$-semistable objects of this form are equivalently all those for which $s$ is surjective, so $P^\bullet \cong P$ for an ideal sheaf $P$.
\end{itemize}
\end{example}
When working with similar pairs, I will assume here that $F$ is torsion -- a detail that will be removed in the sequel. This fixes their determinants, so that the moduli spaces are contained in $\mM_{\mB,L}$ -- the moduli stack of perfect complexes in $\mB$ with determinants given by some fixed line bundle $L$. When $X$ is not compact, such a stack may not be contained in $\mM_X$ from Example \ref{ex:notproperdg}.1). In this case, one can either take a smooth compactification of $X$ or work directly with the moduli stacks of semistable objects with fixed determinats as explained in Remark \ref{rem:afterpairs}.

For now, I will assume that $X$ is compact. I will remove this assumption in Remark \ref{rem:afterpairs}. Fix a heart $\mB$ in $D^b(X)$ with the data of Definition \ref{def:categoryA} satisfying Assumption \ref{ass:orientation}. Choose a set of stability conditions $W$ for which Assumption \ref{ass:stab}.a) holds. Here, we reformulate the rest of Assumption \ref{ass:stab} so that it can be applied to wall-crossing in $\mB$ for stable pairs $O\xrightarrow{s}F$ where $O\in \Coh(X)$ and $F$ has positive codimension so lies in $\Coh_{<4}(X)$. Hence, we will use the fixed determinant moduli stack $\mM_{\mB,\det(O)}$ and its substacks which will always be labeled by attaching subscript $\det(O)$.   
\begin{assumption}
\label{ass:pairWC}
    Suppose that there is an object $O\in \Coh(X)$ and consider the abelian category $\mB_O$ of triples $(V,F,s)$ where $V\in \Vect$, $F\in\Coh_{<4}(X)$ and $s: V\otimes O\to F$ is a morphism of sheaves. 
    Let $W^P$ be a space of weak stability conditions on $\mB_O$ with a bijection $(-)^P: W\to W^P$. Denote by $\mN^{\sig^P}_{d,\al}\subset \mN_O$ the moduli substack of $\sig^P$-semistable $(V,F,s)$ with $\dim(V) = d$ and $\llbracket F\rrbracket = \al\in \ov{C}\big(\mA\big)$. The above data is chosen such that
    \begin{enumerate}[label=\alph*)]
          \item for all $\beta\in \msE(\mB),\sig\in W$, there exists $d \in\{0,1\}$ and $\al\in\ov{C}\big(\mA\big)$ such that there are isomorphisms of stacks
        \begin{equation}
        \label{eq:pairtocomplexiso}
       \big(\mN^{\sig ^P}_{d,\al}\big)^{\rig}\cong \begin{cases}
         \mM^{\sig}_{\be,\det(O)}  &\text{if }d=1\,,\\
         \\
    \big(\mM^{\sig}_{\be}\big)^{\rig}  &\text{if }d=0 \,,
        \end{cases}
        \end{equation}
   induced by the functor $C: \mB_O\to \mB$ mapping each $V\otimes O\xrightarrow{s}F$ to the corresponding complex in degrees $[-1,0]$. We denote the set of such $(d,\al)$ by $\msE(\mB_O)$. 
    \item For a fixed $\tau\in W$ the restriction of $C$ to:
\begin{equation}
\label{eq:equivalence-cat-pairs}
\left\{\begin{array}{c}
   V\otimes O\to F: \tau^P\textnormal{-semistable}    \\
      \textnormal{of class }(d,\al)\in \msE(\mB_O)
\end{array}\right\}\xlongrightarrow{C}\left\{\begin{array}{c}
I: \tau\textnormal{-semistable}    \\
      \textnormal{of class }\beta\in \msE(\mB)\\
       \textnormal{with determinant }\det(O) \textnormal{ if }d=1  
\end{array}\right\}
\end{equation}
is an equivalence of categories identifying the classes of exact triples contained in them.
        \item there exists a set of sufficiently positive ample divisor $\{D_k\}_{k\in K}$ in $X$ satisfying the analogue of Definition \ref{Def:quiverpairs} for $(1,0) = \mO_X(-D_k)[-1]$. Explicitly, this means the following:
        \begin{itemize}
            \item Let $\mB_{O,k}\subset \mB_O$ be the subcategory of objects $(V,F,s)$ satisfying 
            \begin{equation}
            \label{eq:pair-higher-coh-vanishing}
            H^i\big( O(D_k)\oplus F(D_k)\big) = 0\qquad \text{for}\quad i>0\,,
            \end{equation}
           and $\mN_{O,k}$ its moduli stack. For each $\be\in\msE(\mB)$ and $k\in K$, there exists a connected open subset $W_{\beta,k}\subset W$ such that $\mN^{\sig^P}_{d,\al}\subset \mN_{O,k}$ for any $\sig\in W_{\beta,k}$ and $(d,\al)$ as in \eqref{eq:pairtocomplexiso}. The union of $W_{\beta,k}$ over all $k\in K$ for a fixed $\beta$ is equal to $W$. 
            \item Define the category $\mbB_{O,k}$ of objects $(V,F,s, V_{\times},f)$ where $(V,F,s)\in \mB_{O,k}$, $V_{\times}\in \Vect$, and 
            \begin{equation}
            \label{eq:pair-framing}\begin{tikzcd}
            f: V_{\times}\otimes \mO_X(-D_k)\arrow[r]& V\otimes O\oplus F
            \end{tikzcd}
            \end{equation}
            is a morphism. The moduli stack of these objects is denoted by $\mbN_{O,k}$. 
\item For fixed $\sig\in W,\phi\in S$ and the  quivers $I=I_{\JS}, I_{\Flag}$, and $I_{\MS}$ denote by $\mbB^{\sig,\phi}_{O,I_k}$ the categories whose objects are again pairs consisting of a representation of $\mathring{I}$ and an object $(V,F,s, V_{\times},f)\in \mbB_{O,k}$ with $I=C(V,F,s)$ in \eqref{eq:equivalence-cat-pairs} and $\phi(I)=\phi$. As before, the vector spaces $V_{\times}$ at the connecting vertex $\times$ are identified. The associated moduli stacks will be denoted by $\mbN^{\sig,\phi}_{O,I_k}$. Fixing $\un{\mu}$ and $\lambda$ as in \eqref{eq:muicond}, reintroduce the stability condition $\sig^{\lambda}_{\mu}$ from \eqref{eq:barphi} constructed this time from $\sig$ on $\mbB^{\sig,\phi}_{O,I_k}$. The moduli stacks of $\sig^{\lam}_{\mu}$-semistable objects for a dimension vector $\un{d}$ of $\mathring{I}$ and $\beta\in \msE(\mB)$ will be denoted by $\mbN^{\sig}_{\un{d},\be}$.  Their rigidifications need to be proper algebraic spaces whenever there are no strictly semistables.
        \end{itemize} 
        \item Just as in §\ref{sec:assonFlagMS}, there are natural projections  
            \begin{equation}
            \label{eq:OprojmN}
            \begin{tikzcd}[column sep=large]
\mbN^{\sig,\phi}_{O,\MS_k}\arrow[r,"\pi_{\MS/\Flag}"]&\mbN^{\sig,\phi}_{O,\Flag_k}\arrow[r,"\pi_{\Flag/ \JS}"]&\mbN^{\sig,\phi}_{O,\JS_k}\arrow[r,"\pi_{\JS}"]&\mN_{O,k}\
            \end{tikzcd}
            \end{equation}
            and their rigidifications.
Due to the isomorphism \eqref{eq:pairtocomplexiso} and the open embeddings $$\mM^{\sig}_{\be,\det(O)}\subset \mM_{X,\det(O)}\,,\qquad \mM^{\sig}_{\be}\subset \mM_X\,,$$ there are CY4 obstruction theories $\FF^{\sig}_{d,\al}$ on $\big(\mN^{\sig^P}_{d,\al}\big)^{\rig}$. One assumes that there exist $\infty$-PVP diagrams for $\FF^{\sig}_{d,\al}$ along the compositions of $\pi^{\rig}_{\JS}$, $\pi^{\rig}_{\Flag/\JS}$ and $\pi^{\rig}_{\MS/\Flag}$ when restricted to the appropriate moduli substacks of semistables. This condition is the pair version of \eqref{eq:pairtosheaf} and Assumption \ref{ass:obsonflag}.b) and c). As such, one also requires that the obvious analogue of the additivity of obstruction theories in \eqref{eq:additivityofFFflag} holds.
    \end{enumerate}
\end{assumption}
\begin{remark}
\label{rem:afterpairs}
  \leavevmode
\vspace{-4pt}
\begin{enumerate}[label=\roman*)]
\item One needs to change the formulation slightly when $X$ is not compact. In this case, the moduli stack $\mN_O$ is constructed for some $O$ in $\Coh(X)$, but $F$ has to be in $\Coh_{\cs}(X)$. The classes $\al$ one considers live therefore in $\ov{K}_{\cs}:=\ov{K}\big(\Coh_{\cs}(X)\big)$ from Example \ref{ex:sheavesandquivers}. The category $\mB_O$ is now the starting object instead of $\mB\subset D^b\big(\Coh(X)\big)$. Thus, it is the set $W^P$ that is determined, and it needs to satisfy Assumption \ref{ass:stab} with respect to $\mB_O$. The admissible classes are now replaced by $(d,\al)$ for $d=0,1$ such that they form a subset $\msE(\mB_O)\subset \ZZ\times \ov{K}_{\cs}$.

In this case, the most general approach is to choose a smooth compactification $\ov{X}$ with an inclusion $X\xhookrightarrow{\iota}\ov{X}$ and $\ov{O}\in \Coh(\ov{X})$ such that $\ov{O}|_X=O$. For each triple $(V,F,s)\in \mB_O$, there exists, by adjunction, the map $\ov{s}: V\otimes \ov{O}\to F$ and a corresponding 2-term complex $\ov{P}^\bullet\in D^b(\ov{X})$ in degrees $[-1,0]$. Let $\ov{\Omega}^{\rig}_C: \mN^{\rig}_O\to \mM_{\ov{X},\det(\ov{O})},\mM^{\rig}_{\ov{X}}$ be the induced map of stacks. Assumption \ref{ass:pairWC}.a) now becomes the requirement that the restriction of $\ov{\Omega}^{\rig}_C$ to $(\mN^{\sig^P}_{d,\al})^{\rig}$ for $\al\in \msE_O$ is
$$
\begin{cases}
   \textnormal{an open embedding into }\mM_{\ov{X},\det(\ov{O})}&\textnormal{if } d=1\,,\\
   &\\
    \textnormal{an open embedding into }\mM^{\rig}_{\ov{X}}&\textnormal{if } d=0\,.
\end{cases}
$$
This determines an obstruction theory on such $(\mN^{\sig^P}_{d,\al})^{\rig}$ by pullback. It is assumed to be CY4.

Assumption \ref{ass:pairWC}.b) is skipped and c), d) remain unchanged except that one may weaken the properness condition as in the next point. 
\item The Assumption \ref{ass:pairWC} and its version for a non-compact $X$ in i) both make sense in the presence of a $\T$-action if one adds the condition that $O$, $X\xhookrightarrow{\iota}\ov{X}$, and $\ov{O}$ are $\T$-equivariant. Additionally, it is only required that the $\T$ and $\TT$-fixed point loci of the finite-type separated algebraic spaces $\big(\mbN^{\sig^P}_{\un{d},\be}\big)^{\rig}$ are proper whenever there are no strictly $(\sig^P)^{\lambda}_{\mu}$-semistables and have the $\TT$-equivariant resolution property.
\item 
    For Assumption \ref{ass:pairWC}.a) to be satisfied, one already needs to put restrictions on $O$. For example in \cite{bojko3}, I have assumed O to be locally-free, rigid, and simple (see \cite{Ricolfi} for the case of Calabi--Yau threefolds). 
    \end{enumerate}
\end{remark}
 Clearly, Assumption \ref{ass:pairWC}.d) is even more restrictive than the the 
 condition from \eqref{eq:compvanishwrong} for sheaves. Due to Example \ref{ex:counterexamples}, it will not hold in full generality, but I will consider one situation when it is satisfied as is.

In §\ref{sec:localCY4}, I will work with Calabi--Yau fourfolds obtained as total spaces of canonical bundles. In this case, one can use spectral correspondence to prove Assumption \ref{ass:pairWC}.d) and thus also wall-crossing. In fact, this situation will also cover the Joyce--Song wall-crossing formula \eqref{eq:JSWC} for compactly supported sheaves. For this purpose, it is more natural to work directly with the stack $\mN_{k}$ from Example \ref{ex:sheavesandquivers}.i) for a fixed divisor $D_k$. The consequent small deviation from Assumption \ref{ass:pairWC} is detailed in the next example.
\begin{example}
 Fix a $\sig\in W$ from Assumption \ref{ass:stab}. Consider the categories $\mB_k$ from Example \ref{ex:sheavesandquivers}.i). However, this time I will restrict them to pairs $V\otimes \mO_X(-D_k)\to F$ where $F$ is $\sigma$-semistable and $\phi(F)=\phi$ for a fixed value $\phi\in S$. I will also impose either $\dim(F)<4$ or that $X$ is a strict CY fourfold. This smaller category will be denoted by $\mB^{\phi}_k$ with its moduli stack $\mN^{\phi}_k$. For simplicity, I assume here that for each $\al\in \msE(\mA)$, $\phi(\al) = \phi$, and $\sig\in W_{\al,k}$.
    
   Note that $H^i\big(F(D_k)\big)=0$ already holds for all $(V\otimes \mO_X(-D_k)\to F)\in \mB^{\phi}_{k}$. However, I need some analog of \eqref{eq:pair-higher-coh-vanishing} so I choose another divisor $D^+_{k}$ such that $H^i\big(\mO_X(D^+_k - D_k)\big) = 0$ for $i>0$. Then $\mbB_{k, D^+_k}$ is defined by setting $O=\mO_X(-D_k)$  in the construction of $\mbB_{O, k}$ in Assumption \ref{ass:pairWC}.c) and replacing  the pairs \eqref{eq:pair-framing} by the data
   \begin{equation}
   \label{eq:Joyce-Song-framing}
    f: V_{\times}\otimes \mO_X(-D_k)\to F\,, \qquad g:V_{\times} \otimes \mO_X(-D^+_k)\to \mO_X(-D_k)\,.
   \end{equation}
Consider the family of stability conditions $\{\tau_t\}_{t\in [-1,1]}$ on $\mB^{\phi}_k$ with phases
    $$
  \psi_{t}(d,\al) = \begin{cases}
     t&\text{if}\quad d\neq 0\,,\\
     0&\text{if}\quad d=0\,.
    \end{cases}
    $$
    Note that the $\tau_{t}$ semistable objects in $\mB^{\phi}_k$ are given as follows:
\begin{itemize}
    \item for $t>0$ and the class $(1,\al)$, they are precisely the Joyce--Song stable pairs from Example \ref{ex:JSpairs},
    \item for $t=0$, all objects of $\mB^{\phi}_k$ are semistable,
    \item for $t<0$, the only semistable objects have class $(d,0)$ or $(0,\al)$.
\end{itemize}
   As in Assumption \ref{ass:pairWC}.c), one can thus define the categories $\mbB^{\sig,\phi}_{k,I^+_{k}}$ in the present situation using \eqref{eq:Joyce-Song-framing} and replacing \eqref{eq:equivalence-cat-pairs} by $\mB^{\phi}_k$. Due to the arguments in \cite[§12.6, §12.7]{JoyceSong}, there is a CY4 obstruction theory $\FF^{\rig}_k$ on $\big(\mN^{\phi}_k\big)^{\rig}$. Thus, one can still formulate Assumption \ref{ass:pairWC}.d).
   
The above, describes a family of stability conditions that leads to Joyce--Song wall-crossing. It was already discussed in \cite[Appendix A]{bojko3} where I explained briefly why Assumption \ref{ass:stab} is satisfied. The question of properness from Assumption \ref{ass:pairWC}.c) was also addressed there. This just leaves Assumption \ref{ass:pairWC}.d) to be checked. Having done so, one obtains the appropriate virtual fundamental classes of enhanced master spaces used in the proof of the wall-crossing in \eqref{eq:JSWC} for sheaves. One example of this is discussed in §\ref{sec:localCY4}. 
\end{example}
\subsection{Well-defined invariants counting torsion-free sheaves}
\label{sec:proof}
In this subsection, I move away from Assumption \ref{ass:obsonflag} and focus on Problem (II) from the introduction which corresponds to proving Assumption \ref{ass:welldef}. For this purpose, I focus on torsion-free semistable sheaves on a projective CY fourfold $X$. 

I now fix the heart $\mA=\Coh(X)$ and its stability condition $\sigma$ such that all $\sigma$-semistable $E$ of positive rank are torsion-free. Moreover, I will require that the rank function from Assumption \ref{ass:stab}(e) is determined by the usual rank of sheaves. I also set $\msE(\mA)$ to be the set of all $\alpha$ with $\Rk(\al)>0$. This example includes slope and Gieseker stability. For an ample divisor $D$, apply the construction from Example \ref{ex:sheavesandquivers}.i) to produce the moduli stack $\mN_{D}$. If $D_1, D_2$ are two such divisors, the resulting moduli spaces $N^{\JS}_{D_1,\al}$ and $N^{\JS}_{D_2,\al}$ give rise to the invariants
\begin{equation}
\label{eq:D1D2invs}
\blangle \mM^{\sig}_{\al}\brangle^{D_1}, \blangle \mM^{\sig}_{\al}\brangle^{D_2} \in L_*
\end{equation}
as in Definition \ref{def:Masik}. In Theorem \ref{thm:independence!} below, I reduce the equality of these two invariants to \eqref{eq:JSWC}, which lifts the defining formula \eqref{eq:Masigdef} and has now been proved for sheaves in  \cite{BKLT}.

Recall that for each point  $I_D=\big[\mO_X(-D)\to F\big]$ of $N^{\JS}_{D, \al}$\footnote{As I am working with sheaves here and no longer need to consider the quivers from Definition \ref{Def:quiverpairs}, I will reserve the letter $I$ for pairs instead of $P$.}, the usual definition of a trace map  
$$
\begin{tikzcd}
\RHom(I_D,I_D)\arrow[r, shift left=0.2em, "\Tr"]&\arrow[l,shift left=0.2em,"\id"]\mO_X
\end{tikzcd}
$$
satisfies $\Tr\circ\, \id = \Rk(\al)-1$. 
To have a well-defined \textit{normalized trace map} $\tr = \Tr/\big(\Rk(\al)-1\big)$, I require $\Rk(\al)\neq 1$ in the theorem below. As usual, this gives the traceless $\RHom$ complex
$$
\RHom(I_D,I_D)_0
$$
which describes the obstruction theory of $N^{\JS}_{D, \al}$ at the point $I_D$. When $\Rk(\al) = 1$, there are no strictly semistables, so the equality of \eqref{eq:D1D2invs} holds immediately. 
 \begin{theorem}
 \label{thm:independence!}
     Fix $\sigma$ and $\al$ as above. For $D_1,D_2$ ample divisors, assume that $\Masi\subset \mM_{\mA_{D_i}}$ and \eqref{eq:JSWC} holds for $i=1,2$. In this case, the equality
     $$
     \Masig^{D_1} =\Masig^{D_2}
     $$
     holds for any Calabi--Yau fourfold.
     \end{theorem}
     \begin{proof}
     The main idea is to compare the moduli spaces $N^{\JS}_{D_i, \al}$ for $i=1,2$ by embedding them into a bigger one, which is just $N^{\JS}_{D, \al}$ for $D=D_1+D_2$\footnote{The potential issue that $D$ is not sufficiently positive for the fixed $\sigma$ and $\al$ can be resolved by replacing $D_i$ by their powers.}. Since $\langle M^{\sig}_{\al}\rangle^{D_i}$ are defined in terms of $\big[N^{\JS}_{D_i, \al}\big]^{\vir}$ via \eqref{eq:Masigdef}. this already describes a relation between \eqref{eq:D1D2invs}.  To prove the equality, however, one needs \eqref{eq:JSWC}.
     
     Denoting by $\mF$ the universal sheaf on $X\times N^{\JS}_{D, \al}$,  it fits together with the universal pair $\mI_D$ into the distinguished triangle
 $$
 \begin{tikzcd}
 \mI_D\arrow[r,"\mathfrak{l}"]&\mO(-D)\arrow[r, "\mathfrak{f}_D"]&\mF\arrow[r,"\mathfrak{u}"]&\mI_D[1]
 \end{tikzcd}\,.
 $$ 
 The next lemma, describes $N^{\JS}_{D_i, \al}$ as closed subschemes of $N^{\JS}_{D,\al}$. I will always focus on $i=1$ from now on, but the same works for $i=2$. Without loss of generality, the divisor $D_2$ is assumed to be the smooth vanishing locus of a section $s_2\in H^0\big(\mO_X(D_2)\big)$. For a very ample $\mO_X(D_2)$, such a choice can be made by Bertini's theorem. Otherwise, one may replace $\mO_X(D_2)$ by its sufficiently large power. 
 \begin{lemma}
The map
 $$\big(\mO_X(-D_1)\xrightarrow{f_{D_1}} F\big)\mapsto \big(\mO_X(-D)\xrightarrow{f_{D_1}\circ s_1} F\big)$$
 induces a closed embedding
 \begin{equation}
\label{eq:D1embed}
N^{\JS}_{D_1, \al}\xlongrightarrow{\iota_1} N^{\JS}_{D, \al}\end{equation}
as it maps stable pairs to stable ones. Letting $p:X\times N^{\JS}_{D,\al}\to N^{\JS}_{D,\al}$ denote the projection to the second factor, the subscheme $N^{\JS}_{D_1, \al}$ can be expressed as the vanishing locus of the natural section  $
\mathfrak{v}_1: \mO\to \VV_1
$ where
\begin{equation}
\label{eq:VVbundle}
\VV_1 = Rp_{*}\Big(\mF(D)|_{D_2}\Big) \,.
 \end{equation}
is a vector-bundle on $N^{\JS}_{D, \al}$. Here, I used $|_{D_2}$ to denote the restriction to $D_2\times N^{\JS}_{D, \al}$. The map $\mathfrak{v}_1$ is the pushforward along $p$ of the composition of $\Ff_D$ with the restriction to $\mF(D)|_{D_2}$.
\end{lemma}
\begin{proof}
 To see that \eqref{eq:D1embed} is well-defined, consider a subsheaf $F'\subset F$ with its cokernel $Q$. Then we have the following double complex
 $$
 \begin{tikzcd}
     &\arrow[d]0&\arrow[d]0&\arrow[d]0\\
0\arrow[r]&\Hom\big(\mO_{D_2}(-D_1),F'\big)\arrow[d]\arrow[r]&\Hom\big(\mO_{D_2}(-D_1),F\big)\arrow[d]\arrow[r]&\Hom\big(\mO_{D_2}(-D_1),Q\big)\arrow[d]\\
   0\arrow[r]&\Hom\big(\mO_X(-D_1),F'\big)\arrow[d]\arrow[r]&\Hom\big(\mO_X(-D_1),F\big)\arrow[d]\arrow[r]&\Hom\big(\mO_X(-D_1),Q\big)\arrow[d]\\
    0\arrow[r]&\Hom\big(\mO_X(-D),F'\big)\arrow[r]&\Hom\big(\mO_X(-D),F\big)\arrow[r]&\Hom\big(\mO_X(-D),Q\big)\\
 \end{tikzcd}
 $$
 which has exact columns and rows. Due to $F'$ and $F$ being torsion-free, the second and third term in the second row also vanish. In particular,  the map $f_D = f_{D_1}\circ s_1$ is 0, if and only if $f_{D_1}$ is. This implies that condition 2 from Example \ref{ex:JSpairs} holds for $f_D$. Next, suppose that $f_D = f_{D_1}\circ s_1$ factors through $f'_D: \mO_X(-D)\to F'$, so that $f_D$ is both an image of $f'_D$ and $f_{D_1}$ under the maps in the above diagram. Since the total complex of the double complex is exact, this implies that there is a unique $f'_{D_1}\in \Hom(\mO_X(-D_1),F')$ through which the original morphism $f_{D_1}$ factors. Therefore, the second condition is satisfied for $\mO_X(-D)\to F$ if and only if it is satisfied for $\mO_X(-D_1)\to F$.

 Next, I need to show that the perfect complexes constructed in \eqref{eq:VVbundle} are vector bundles. For this, we take the long exact sequence obtained by acting with the $\Hom(-,F)$ functor on $0\to \mO_X(-D)\to \mO_X(-D_1)\to \mO_{D_2}(-D_1)\to 0$: 
\begin{equation}
\label{eq:shortexactvb}
\begin{tikzpicture}[descr/.style={fill=white,inner sep=1.5pt}]
        \matrix (m) [
            matrix of math nodes,
            row sep=1em,
            column sep=2.5em,
            text height=1.5ex, text depth=0.25ex
        ]
   {  &  0& \Hom\big(\mO_X(-D_1), F\big) & \Hom\big(\mO_X(-D), F\big)& \\& \Ext^1\big(\mO_{D_2}(-D_1), F\big) &0 &0& \\
             &  \Ext^2\big(\mO_{D_2}(-D_1), F\big) &0 &\cdots& \\
        };

        \path[overlay,->, font=\scriptsize,>=latex]
        (m-1-2) edge (m-1-3)
        (m-1-3) edge (m-1-4)
        (m-1-4) edge[out=355,in=175] (m-2-2)
        (m-2-2) edge (m-2-3)
        (m-2-3) edge (m-2-4)
        (m-2-4) edge[out=355,in=175]   (m-3-2)
        (m-3-2) edge (m-3-3)
        (m-3-3) edge (m-3-4);
\end{tikzpicture}
\end{equation}
I used that $F$ is torsion-free to show that the first term vanishes. As $D_2$  is very ample, one can choose a smooth projective representative and apply the Grothendieck--Serre duality along the inclusion $D_2\times N^{\JS}_{D, \al}\hookrightarrow X\times N^{\JS}_{D, \al} $ which shows that $R\mHom_{N^{\JS}_{D, \al}}\big(\mO_{D_2}(-D_1), \mF\big) = \VV_i[-1]$. This is then a vector bundle in degree 1.

To show that $N^{\JS}_{D_1, \al} = \Fv_1^{-1}(0)$, use \eqref{eq:shortexactvb} to conclude that the map $\mO_X\to F(D)$ factors uniquely through $F(D_1)$ if and only if the induced map $\mO_X\to F(D)|_{D_2}$ vanishes.
\end{proof}
 The following observation will lead to the proof of the theorem. 
\begin{proposition}
\label{lem:relateDD1}
    The virtual fundamental classes of $N^{\JS}_{D_1,\al}$ and $N^{\JS}_{D, \al}$ are related by 
    \begin{equation}
    \label{eq:NJSD1Dvir}
(\iota_{1})_*\Big(\big[N^{\JS}_{D_1, \al}\big]^\vir\Big) = \big[N^{\JS}_{D, \al}\big]^\vir\cap c_{\Rk}(\VV_1) \,.
    \end{equation}
\end{proposition}
Before, I prove this, I will explain how it can be immediately used to reduce the proof of Assumption \ref{ass:welldef} to \eqref{eq:JSWC} which now takes the form  
\begin{equation}
\label{eq:DJSWC}
\big[N^{\JS}_{D,\al}\big]^{\vir} = \sum_{\begin{subarray}{c}
        \un{\alpha}\vdash_{\mA} \alpha\,, \\
          \phi(\alpha_i) =\phi(\alpha) \end{subarray}}\frac{1}{n!}\Big[\langle \mM^\sigma_{\alpha_n}\rangle_D, \cdots \Big[\langle \mM^\sigma_{\alpha_1}\rangle_D,  e^{(1,0)}_D\Big]\cdots\Big]\,.
\end{equation}     
with an identical formula for $D_1$. 
\begin{proof}[Proof of Theorem \ref{thm:independence!}]
My approach here is similar to what I used in \cite[§3.1, §3.2]{bojko2}. For this reason, I will be brief and concise here. Firstly, one can extend the vector bundle $\VV_1$ that appears to a vector bundle $\WW_1$ on $\mN_D\times \mN_D$. Continuing to use the notation introduced in Definition \ref{def:NkVA}, it is defined by 
$$
\WW_1 = \mV^*\otimes p_*\big(\mF(D)|_{D_2}\big)
$$
where I am now using $\mF$ to denote the universal sheaf on the second factor as I did in Definition \ref{def:NkVA}.
It satisfies $\Delta^*(\WW_1)|_{N^{\JS}_{D,\al}} = \VV_1$, so I will instead write $\VV_1=\Delta^*(\WW_1)$ from now on. I will then consider three different vertex algebras. One on the homology of $\mN_{D_1}$ and another two constructed from $H_{*}(\mN_{D,D_1})$. Here $\mN_{D,D_1}\subset \mN_{D}$ is the substack of pairs $V\otimes \mO_X(-D)\to F$ such that additionally $H^i\big(F(D_1)\big) = 0$ for $i>0$. The vertex algebras are
\begin{enumerate}[label= \arabic*)]
     \item the vertex algebra $W^{D_1}_{*} = H_{*+\vdim}(\mN_{D_1})$ constructed for $\mN_{D_1}$ in Definition \ref{def:NkVA},
     \item the vertex algebra $W^{D,D_1}_*$ which is constructed from $H_{*+\vdim}(\mN_{D,D_1})$ using $\Theta_D,\chi_D$, and $\varepsilon_D$ in the same way as one would for $W^D_*$,
     \item the vertex algebra $W^{\iota_1(D_1)}_*$ constructed on $H_*(\mN_{D,D_1})$ (shifted by an appropriate degree) in the same way as $W^{D,D_1}_*$ except that one uses 
     \begin{align*}
     \Theta_{\iota_1(D_1)}= -(\pi\times\pi)^*\Ext^\vee - 2\mV\otimes \mW^* + \mV\otimes p_*\big(\mF(D_1)\big)^* + \mW^*\otimes p_*\big(\mE(D_1)\big)
     \end{align*}
     on $\mN_{D,D_1}$ instead of $\Theta_D$, $\chi_{D_1}$ instead of $\chi_D$, and $\varepsilon^{D_1}$ instead of $\varepsilon^D$. Here I used the same notation convention for $\mV,\mW,\mE$, and $\mF$ as I did in Definition \ref{def:NkVA}. 
\end{enumerate}
The diagram \eqref{eq:shortexactvb} implies that there is a short exact sequence of vector bundles 
$$
\begin{tikzcd}
  0\arrow[r]&  p_*\big(\mF(D_1)\big)\arrow[r]& p_*\big(\mF(D)\big)\arrow[r]&p_*\big(\mF(D)|_{D_2}\big)\arrow[r]&0
\end{tikzcd}
$$
 on $\mN_{D,D_1}$, because it is still a stack of torsion-free sheaves such that higher cohomologies vanish for both $D_1$ and $D$. As such, one sees that
 \begin{equation}
 \label{eq:ThetaDD1}
 \Theta_{\iota_1(D_1)}=\Theta_D - \WW_1^* - \sig^*\WW_1\,.
 \end{equation}
I will use the simpler version of \cite[Definition 2.11, Theorem 2.12]{GJT} which I noted down in \cite[Definition 2.15]{bojko2}. Compared to \cite{bojko2}, my present convention for $\Theta_{(\cdots)}$ differs by a total sign which one needs to pay attention to. The formula \eqref{eq:ThetaDD1} implies that the most crucial condition of \cite[Definition 2.15]{bojko2} is satisfied. The other conditions of the definition are immediate so \cite[Proposition 2.16]{bojko2} using \cite[Theorem 2.12]{GJT} implies that the vertical arrow of
\begin{equation}
\label{eq:WWscomparison}
\begin{tikzcd}[column sep=huge, row sep = huge]
&W^{D,D_1}_*\arrow[d,"\cap c_{\Rk}(\VV_1)"]\\
W^{D_1}_*\arrow[r,"(\iota_1)_*"]&W^{\iota_1(D_1)}_*
\end{tikzcd}
\end{equation}
is a morphism of vertex algebras. The horizontal arrow $(\iota_1)_*$ is induced by extending \eqref{eq:D1embed} to a morphism of stacks $\mN_{D_1}\to \mN_{D,D_1}$ and then taking pushforward in homology. To finish the argument, note that there is no difference between formulating \eqref{eq:DJSWC} using $W^D_*$ and $W^{D,D_1}_*$. Due to \eqref{eq:NJSD1Dvir}, one obtaints
$$
(\iota_1)_*\Big(\textnormal{RHS of }\eqref{eq:DJSWC}\textnormal{ for }D_1\Big) =\Big(\textnormal{RHS of }\eqref{eq:DJSWC}\textnormal{ for }D\Big)\cap c_{\Rk}\big(\VV_1\big)\,.
$$
which becomes
\begin{align*}
&\sum_{\begin{subarray}{c}
        \un{\alpha}\vdash_{\mA} \alpha\,, \\
          \phi(\alpha_i) =\phi(\alpha) \end{subarray}}\frac{1}{n!}\Big[\langle \mM^\sigma_{\alpha_n}\rangle_D, \cdots \Big[\langle \mM^\sigma_{\alpha_1}\rangle_D,  e^{(1,0)}_D\Big]\cdots\Big]\\
          &=\sum_{\begin{subarray}{c}
        \un{\alpha}\vdash_{\mA} \alpha\,, \\
          \phi(\alpha_i) =\phi(\alpha) \end{subarray}}\frac{1}{n!}\Big[\langle \mM^\sigma_{\alpha_n}\rangle_{D_1}, \cdots \Big[\langle \mM^\sigma_{\alpha_1}\rangle_{D_1},  e^{(1,0)}_D\Big]\cdots\Big]
\end{align*}
in terms of the Lie bracket on $W^{\iota_1(D_1)}_{*+2}/T\big(W^{\iota_1(D_1)}_{*+2}\big)$. The same argument as was used to prove Lemma \ref{lem:pushJSandinj}.ii) implies that the above $\big[-,e^{(1,0)}_D\big]$ restricted to $\al\notin \Ker(\chi)$ is injective. By induction on $\Rk(\al)$, one concludes that
$$
 \langle \mM^{\sig}_\al\rangle^{D} = \langle \mM^{\sig}_\al\rangle^{D_1}\,.
$$
\end{proof}
The rest of the section will be concerned with constructing Park's diagram \eqref{eq:introJSdiag} leading to 
$$
\big[N^{\JS}_{D_1, \al}\big]^\vir=\iota^{!}_1\Big(\big[N^{\JS}_{D, \al}\big]^\vir\Big)
$$
which implies \eqref{eq:NJSD1Dvir}.
\begin{proof}
  Throughout the proof, I will denote the universal pairs by 
  $$
  \mI_{D_1} = \big(\mO(-D_1)\longrightarrow \mF\big)\,, \quad   \mI_{D} = \big(\mO(-D)\longrightarrow \mF \big)\,.
  $$
  \begin{importantremark}
      Unlike the rest of the work, the convention I use here states that pairs determine complexes in degrees $[0,1]$. This is due to this section being written first.
  \end{importantremark}
    To prove the first statement, I need to show that the obstruction theories 
    \begin{align*}
    \label{eq:FF1FF1+2}
    \FF_1 &= \RHom_{N^{\JS}_{D_1, \al}}\big(\mI_{D_1},\mI_{D_1}\big)^\vee_0[1]\,,\\
    \FF_{1+2} &= \RHom_{N^{\JS}_{D, \al}}\big(\mI_{D},\mI_{D}\big)^\vee_0[1]
    \numberthis
    \end{align*}
    fit into the commutative diagram of distinguished triangles 
    \begin{equation}
    \label{eq:JScomparisondiag}
    \begin{tikzcd}[column sep=large] \arrow[d,"\ov{\kappa}"]\widetilde{\FF}^\vee[2]\arrow[r,"{\ov{\mu}}"]&\FF_1\arrow[d,"\mu"]\arrow[r,Maroon,"\nu"]&\VV_1^\vee[1]\arrow[d,equal]\arrow[r]&\widetilde{\FF}^\vee[3]\arrow[d]\\
   \arrow[d,"\iota_1^*\psi"] \iota_1^*\big(\FF_{1+2}\big)\arrow[r,"{\kappa}"]&\widetilde{\FF}\arrow[r]\arrow[d,"\phi"]&\VV_1^\vee[1]\arrow[d,equal]\arrow[r]&\iota_1^*\big(\FF_{1+2}\big)[1]\arrow[d]\\   \iota_1^*\big(\mathbb{L}_{N^{\JS}_{D, \al}}\big)\arrow[r]&\mathbb{L}_{N^{\JS}_{D_1, \al}}\arrow[r]&\LL_{\iota_1}\arrow[r]&\iota_1^*\big(\mathbb{L}_{N^{\JS}_{D, \al}}\big)[1]
    \end{tikzcd}\,.
\end{equation}
There are multiple ways to show this. The most insightful one is summarized in Remark \ref{rem:dagJScomp}, but here I describe an alternative approach relying purely on diagram chasing in triangulated categories. This proof was written before I learnt to use stable $\infty$-categories so they are absent from it. First recall some basic results about completing a morphism of distinguished triangles 
\begin{equation}
\label{eq:mortrian}
\begin{tikzcd}
[execute at end picture={
\foreach \Nombre in  {A_1,A_2,B_1,B_2}
  {\coordinate (\Nombre) at (\Nombre.center);}
\fill[blue,opacity=0.3] 
(A_1) -- (A_2) -- (B_2) -- (B_1) -- cycle;}]
|[alias =A_1]|A_1\arrow[d]\arrow[r]&|[alias =A_2]|A_2\arrow[d]\arrow[r]&\arrow[d,"g"]A_3\arrow[r]&\arrow[d]A_1[1]\\
|[alias =B_1]|B_1\arrow[r]&|[alias =B_2]|B_2\arrow[r]&B_3\arrow[r]&B_1[1]
\end{tikzcd}
\end{equation}
to the \textit{$3\times3$ diagram}
$$
\begin{tikzcd}
[execute at end picture={
\foreach \Nombre in  {A_1,A_2,B_1,B_2}
  {\coordinate (\Nombre) at (\Nombre.center);}
\fill[blue,opacity=0.3] 
(A_1) -- (A_2) -- (B_2) -- (B_1) -- cycle;}]
|[alias =A_1]|A_1\arrow[d]\arrow[r]&|[alias =A_2]|A_2\arrow[d]\arrow[r]&\arrow[d,"g"]A_3\arrow[r]&\arrow[d]A_1[1]\\
|[alias =B_1]|B_1\arrow[d]\arrow[r]&|[alias =B_2]|B_2\arrow[d]\arrow[r]&\arrow[d]B_3\arrow[r]&\arrow[d]B_1[1]\\
C_1\arrow[d]\arrow[r]&C_2\arrow[d]\arrow[r]&\arrow[d]\arrow[d]C_3\arrow[r]&\arrow[d]C_1[1]\\
A_1[1]\arrow[r]&A_2[1]\arrow[r]&A_3[1]\arrow[r]\arrow[r,phantom, shift left=4.0ex, "{\{-,-\}}" marking]&A_1[2]
\end{tikzcd}
$$
where each column and row are distinguished triangles, each square except the one labelled by $\{-,-\}$ is commutative, and this last square is anti-commutative.

For a fixed commutative square
\begin{equation}
\label{eq:ABcomdiag}
\begin{tikzcd}
[execute at end picture={
\foreach \Nombre in  {A_1,A_2,B_1,B_2}
  {\coordinate (\Nombre) at (\Nombre.center);}
\fill[blue,opacity=0.3] 
(A_1) -- (A_2) -- (B_2) -- (B_1) -- cycle;}]
    |[alias =A_1]|A_1\arrow[d]\arrow[r]&|[alias =A_2]|A_2\arrow[d]\\
|[alias =B_1]|B_1\arrow[r]&|[alias =B_2]|B_2
\end{tikzcd}
\end{equation}
the morphism $g$ is called \textit{good} when a completion to a 3x3 diagram exists. Note that this is slightly different from the original definition given by Neeman \cite[Def. 1.9]{neeman}. Using his \cite[Thm. 2.3]{neeman}, the original definition implies the one that I am using, so I will not distinguish between them. 

The lemma below summarizes the cases when $g$ is good appearing in the rest of the proof.
\begin{lemma}
\label{lem:goodtrian}
Suppose that in \eqref{eq:mortrian}
\begin{enumerate}[label=\roman*)]
    \item the bottom row splits as 
    $$
    \begin{tikzcd}
        B_3[-1]\oplus B_2\arrow[r,shift left=0.2em]&\arrow[l,shift left=0.2em]B_2\arrow[r,"0"]&B_3\arrow[r, shift left =0.1 em ]&\arrow[l, shift left =0.1 em]B_3\oplus B_2[1] 
    \end{tikzcd}
    $$
    so $g$ is unique and given by the composition $A_3\to A_1[1]\to B_3\oplus B_2[1]\to B_3$
    \item the set of morphisms $\Hom\big(A_1[1], B_3\big)$ vanishes so $g$ is unique
\end{enumerate}
then $g$ is good.
\end{lemma} 
\begin{proof}
    By Theorem \cite[Prop. 1.1.11]{BBD}, there always exists a good morphism $g$ completing the commutative diagram \eqref{eq:ABcomdiag} to a morphism of distinguished triangles \eqref{eq:mortrian}.  In all of the situations above, there is a unique morphism $g$, so this is always the case. 
\end{proof}

To avoid complicating the next few expressions and diagrams with extra symbols, I will omit writing $\RHom_S$ where $S$ is clear from the context. For example, I will write $(\mI_{D},\mI_D)$, $(\mI_D,\mI_D)_0$ to denote $\RHom_{X\times N^{\JS}_{D, \al}}(\mI_D,\mI_D)$, $\RHom_{X\times N^{\JS}_{D, \al}}(\mI_D,\mI_D)_0$ or $\RHom_{N^{\JS}_{D, \al}}(\mI_D,\mI_D)$, $\RHom_{N^{\JS}_{D, \al}}(\mI_D,\mI_D)_0$ respectively. I will also not specify the pullbacks of $\RHom$ complexes along $\id_X\times \iota_1$ and $\iota_1$ as they are constructed in the same way using universal sheaves $\mF$. We start by constructing the following $3\times 3$ diagram of complexes on $X\times N^{\JS}_{D_1, \al}$:
\begin{equation}
\label{eq:3x3D2}
\begin{tikzcd}
[execute at end picture={
\foreach \Nombre in  {A_1,A_2,B_1,B_2}
  {\coordinate (\Nombre) at (\Nombre.center);}
\fill[blue,opacity=0.3] 
(A_1) -- (A_2) -- (B_2) -- (B_1) -- cycle;}]
 |[alias = A_1]| \big(\mO_{D_2}(-D_1),\mI_{D_1}\big)\arrow[r,Maroon,"\circ \mathfrak{r}_I"]\arrow[d]&|[alias = A_2]|\big(\mI_{D_1},\mI_{D_1}\big)\arrow[r]\arrow[d, shift left = 0.2em, "\tr"]& \big(\mI_D,\mI_{D_1}\big)\arrow[r]\arrow[d,"\tr_D"]&\big(\mO_{D_2}(-D_1),\mI_{D_1}\big)[1]\arrow[d]\\
|[alias = B_1]| (\mO_{D_2},\mO)\arrow[d]\arrow[r] &|[alias = B_2]|\arrow[u, shift left = 0.2em, "\id_\mI"]\mO\arrow[d,"0"]\arrow[r]&\mO(D_2)\arrow[d]\arrow[r]&\big(\mO_{D_2},\mO\big)[1]\arrow[d]\\
  \big(\mO_{D_2}(-D_1),\mF\big)
  \arrow[r,Maroon,"\mathfrak{r}_I^{0}"]\arrow[d,Maroon, shift left = 0.2em,"\mathfrak{u}\circ "]&\big(\mI_{D_1},\mI_{D_1}\big)_0[1]\arrow[r]\arrow[d, shift left = 0.2em, "i"]& \big(\mI_D,\mI_{D_1}\big)_0[1]\arrow[r]\arrow[d]&\big(\mO_{D_2}(-D_1),\mF\big)[1]\arrow[d]\\
   \big(\mO_{D_2}(-D_1),\mI_{D_1}\big)[1]\arrow[r, Maroon ,"{(\circ \mathfrak{r}_I)[1]}"] &\big(\mI_{D_1},\mI_{D_1}\big)[1]\arrow[r]\arrow[u, Maroon ,shift left = 0.2em,"p"]& \big(\mI_D,\mI_{D_1}\big)[1]\arrow[r]\arrow[r,phantom, shift left=4.0ex, "{\{-,-\}}" marking]&\big(\mO_{D_2}(-D_1),\mI_{D_1}\big)[2]
\end{tikzcd}
\end{equation}
where the doubled arrows going both ways are meant to represent split distinguished triangles. Note that the distinguished triangle in the first column can also be written using Grothendieck--Serre duality in \cite[Thm. 3.4.4]{hartshorneresdual} as 
$$
\begin{tikzcd}
\mI_{D_1}(D)|_{D_2}[-1]\arrow[r]&\mO_{D_2}(D_2)[-1]\arrow[r]&\mF(D)|_{D_2}[-1]\arrow[r]&\mI_{D_1}(D)|_{D_2}\,.
\end{tikzcd}
$$
The map $\tr_D$ in the third column is constructed such that after taking its cone $(\mI_D,\mI_{D_1})_0[1]$ the resulting triangle can be completed to a $3\times 3$ diagram.  To see that this is possible, I apply Lemma \ref{lem:goodtrian} i) to the first two columns on the left (instead of rows) and conclude that $\color{Maroon}\mathfrak{r}^0_I$ is just the composition of the other {\color{Maroon}red} arrows in the diagram.

The morphism $\Ma{\circ \mathfrak{r}_I}$ in the first row is the result of applying $(-,\mI_{D_1})$ to the distinguished triangle
$$
\begin{tikzcd}
\mI_D\arrow[r," s_{\mI}"]&\mI_{D_1}\arrow[r, Maroon, "\mathfrak{r}_I"]&\mO_{D_2}(-D_1)\arrow[r]&\mI_D[1]
\end{tikzcd}
$$
\sloppy and the second row is itself the natural distinguished triangle associated with $s_1\in H^0\big(\mO_X(D_2)\big)$. Taking the dual of the pushforward to $N^{\JS}_{D_1, \al}$ of the third row in \eqref{eq:3x3D2} is going to be the first row of the diagram \eqref{eq:JScomparisondiag} with 
\begin{equation}
\label{eq:nuFFtilde}
{\color{Maroon}\nu =(\mathfrak{r}^0_I)^\vee}\,,\qquad \wt{\FF} = \RHom_{N^{\JS}_{D_1, \al}}\big(\mI_D,\mI_{D_1}\big)_0[3]\end{equation}
after using the same notation for the pushed forward maps. 

In fact, I can already construct the diagram 
 \begin{equation}
 \label{eq:JScompareEE}
    \begin{tikzcd}[column sep=large] 
0\arrow[d]\arrow[r]&\VV_1[1]\arrow[d,"\xi"]\arrow[r, equals]&\VV_1[1]\arrow[d,Maroon,"{\nu^\vee[2]}"]\arrow[r]&0\arrow[d]\\
   \VV^\vee_1\arrow[d, equals]\arrow[r]& \arrow[d,"\ov{\kappa}"]\widetilde{\FF}^\vee[2]\arrow[r,"{\ov{\mu}}"]&\FF_1\arrow[d,"\mu"]\arrow[r,Maroon,"\nu"]&\VV_1^\vee[1]\arrow[d,equal]\\
      \VV^\vee_1\arrow[d]\arrow[r]& \arrow[d] \EE\arrow[r,"{\kappa}"]&\widetilde{\FF}\arrow[r]\arrow[d]&\VV_1^\vee[1]\arrow[d]\\
   0\arrow[r]&\VV_1[2]\arrow[r, equals]&\VV_1[2]\arrow[r]&0
    \end{tikzcd}\,.
\end{equation}
by using Park's \cite[Lemma C.2]{Park} which guarantees that $\EE^\vee[2]\cong \EE$ is compatible the self-duality of the diagram. The map $\xi$ exists because the top right square is commutative and it is the unique map such that the resulting diagram is commutative due to $\Hom_{N^{\JS}_{D_1, \al}}\big(\VV_1[1], \VV_1^\vee) = 0$. There are multiple ways to conclude the commutativity of the top right square which will become apparent from the rest of the proof below dedicated to showing that 
\begin{equation}
\label{eq:EEFF12}
\EE\cong\iota_1^*(\FF_{1+2})
\end{equation}
and all the other maps $\kappa,\mu,\xi$ are the natural ones.

The computation below shows that the map $\xi$ which I constructed using {\color{Maroon}$\mathfrak{r}^0_I$} from \eqref{eq:3x3D2} can be identified with the map constructed using the dual of \eqref{eq:3x3D2} after replacing the second $\mI_{D_1}$ by $\mI_D$ in each term. This requires giving an alternative definition of the map {\color{Maroon}$\mathfrak{r}^0_I$} relying on the diagram
\begin{equation}
\label{eq:trlessoctahedral}
\begin{tikzcd}[ column sep={5.5em,between origins},
  row sep={4.2em,between origins}, ampersand replacement=\&]
    \eqmathbox{\mO}\arrow[dr, shift left=0.2 em ,"\id_{I}"]\arrow[rr, bend left = 40]\&\& \eqmathbox{\big(\mI_{D_1},\mO_X(-D_1)\big)}\arrow[dr] \arrow[rr, bend left =40, "\circ\mathfrak{f}_{D_1}"]\&\&\eqmathbox{\big(\mI_{D_1},F\big)}\arrow[dr]\arrow[rr,bend left = 40]\&\&\eqmathbox{\big(\mI_{D_1},\mI_{D_1}\big)_0[1]}\\
\&\arrow[ul, shift left = 0.2em]\eqmathbox{\big(\mI_{D_1},\mI_{D_1}\big)}\arrow[ur,"\circ\mathfrak{l}"]\arrow[dr, shift left = 0.2em]\&\&\eqmathbox{\big(\mF,\mO(-D_1)\big)[1]}\arrow[ur,WildStrawberry]\arrow[dr]\&\&\eqmathbox{\big(\mI_{D_1},\mI_{D_1}\big)[1]}\arrow[ur]\&\\
\&\&\arrow[ul, shift left = 0.2 em]\eqmathbox{\big(\mI_{D_1},\mI_{D_1}\big)_0}\arrow[rr, bend right = 40]\arrow[ur, WildStrawberry]\&\&\eqmathbox{\mO}\arrow[ur]\&\&
\end{tikzcd}\,.
\end{equation}
obtained by applying the octahedral axiom to the left-most commutative triangle. The distinguished triangle containing the {\color{WildStrawberry} colored arrows} is used to construct the commutative diagram
$$
\begin{tikzcd}
0\arrow[r]&\big(\mF,\mO_{D_2}(-D_1)\big)[1]\arrow[r,Cyan, equals]&\big(\mF,\mO_{D_2}(-D_1)\big)[1]\arrow[r]&0\\
\arrow[u]\big(\mI_{D_1},\mF\big)[-1]\arrow[r, Purple]&\arrow[u,Cyan,"^0\widetilde{\mathfrak{r}}_I"]\big(\mI_{D_1}, \mI_{D_1}\big)_0\arrow[r,WildStrawberry]&\arrow[u, Cyan]\big(\mF,\mO(-D_1)\big)[1]\arrow[r, WildStrawberry]&\arrow[u]\big(\mI_{D_1},\mF\big)\\
\big(\mO_{D_2}(-D_1),\mF\big)[-1]\arrow[u,Maroon,"\circ \mathfrak{r}_I"]\arrow[r,Purple, equals]&\big(\mO_{D_2}(-D_1),\mF\big)[-1]\arrow[u,Purple,"\widetilde{\mathfrak{r}}^0_I"]\arrow[r]&0\arrow[u]\arrow[r]&\big(\mO_{D_2}(-D_1),\mF\big)\arrow[u]
\end{tikzcd}
$$
where the middle column composes to 0. One can equivalently construct $\Pur{\widetilde{\mathfrak{r}}^0_I}$ by using the dual distinguished triangle
$$
\begin{tikzcd}
\big(\mO(-D_1),\mF\big)[-1]\arrow[r]&\big(\mI_{D_1}, \mI_{D_1}\big)_0\arrow[r]&\big(\mF,\mI_{D_1}\big)[1]\arrow[r]&\big(\mI_{D_1},\mF\big)[1]
\end{tikzcd}
$$
which shows that $\Pur{^0\widetilde{\mathfrak{r}}_I} = (\Pur{\widetilde{\mathfrak{r}}^0_I})^\vee[-4]$\,. 

The idea is now to establish that $\Pur{\widetilde{\mathfrak{r}}^0_I} = \Pur{\mathfrak{r}^0_I}$ followed by comparing ${\color{Cyan} ^0\widetilde{\mathfrak{r}}_I}$ with ${\color{Blue}^0\mathfrak{r}_{I_D}}$ which is constructed in a comparable way just for $\big(\mI_{D},\mI_{D}\big)_0[1]$. This is then used to show \eqref{eq:EEFF12}.

To do the first step, we use the diagram 
\begin{center}
       \includegraphics[scale = 0.475]{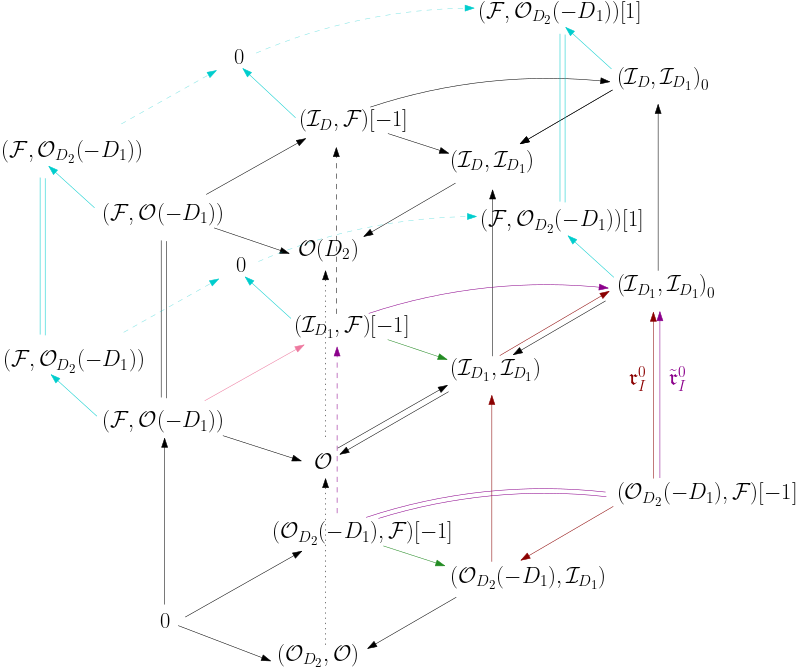}
\end{center}
 
with each row being a distinguished triangle. As \Ma{$\mathfrak{r}^0_I$} is equivalent to tracing the other \Ma{red path} in the diagram, one can use  commutativity of the bottom level and the {\color{Purple}purple equality} $\big(\mO_{D_2}(-D_1),\mF\big)[-1]\begingroup\color{Purple} =\endgroup \big(\mO_{D_2}(-D_1),\mF\big)[-1]$ to change the first arrow of the \Ma{red path} to the bottom {\color{Green}green arrow}. I then use the commutativity of the bottom half of the diagram to replace the path starting with the \Gre{green arrow} by the rest of the \Pur{purple path}. This shows that $\begingroup\color{Purple}\widetilde{\mathfrak{r}}^0_I\endgroup =  \begingroup\color{Maroon}\mathfrak{r}^0_I\endgroup$. The \Cya{cyan part} of the diagram describes the Serre dual construction. In particular, we now know that
$
\Cya{^0\widetilde{\mathfrak{r}}_I} = (\Pur{\widetilde{\mathfrak{r}}^0_I})^\vee[-4] = (\Ma{\mathfrak{r}^0_I})^\vee[-4]
$
so the definition of the morphism $\big(\mI_{D},\mI_{D_1}\big)_0\begingroup \color{Cyan}\longrightarrow \endgroup\big(\mF,\mO_{D_2}(-D_1)\big)[1]$ is independent of the two choices. 

The second step relies on the following diagram 
\begin{equation}
\label{eq:bigdiag2}
    \includegraphics[scale=0.32]{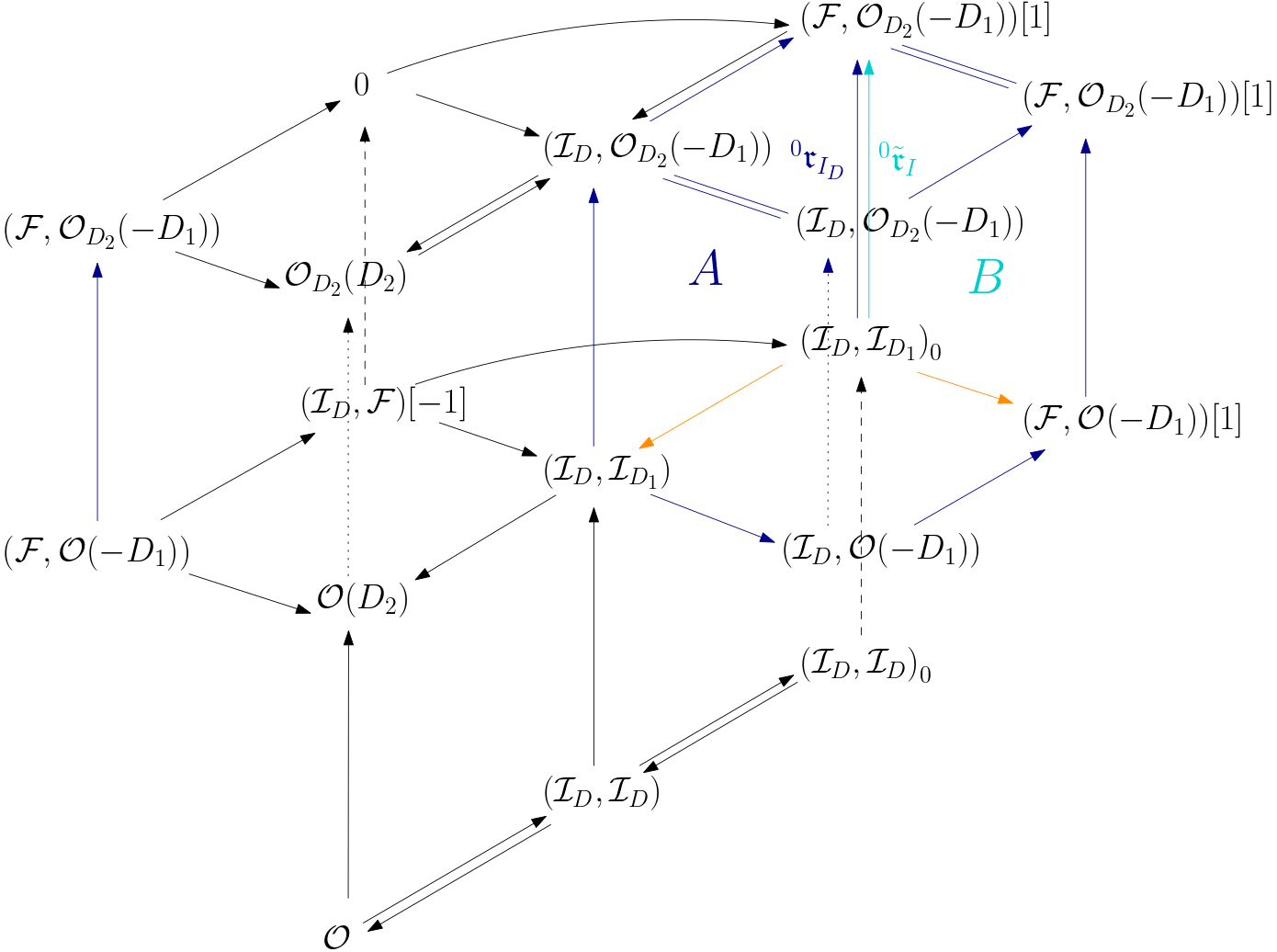}
\end{equation}
where each row is formed by distinguished triangles. 

The fact that the middle vertical plane exists as a commutative $3\times 3$ diagram of distinguished triangles is concluded using the commutativity of the right square in 
$$
\begin{tikzcd}
[execute at end picture={
\foreach \Nombre in  {A_1,A_2,B_1,B_2}
  {\coordinate (\Nombre) at (\Nombre.center);}
\fill[blue,opacity=0.3] 
(A_1) -- (A_2) -- (B_2) -- (B_1) -- cycle;}]
\big(\mI_D,\mI_D\big)[1]\arrow[r, Maroon, shift left=0.2em]&\arrow[l,shift left=0.2em]\big(\mI_D,\mI_D\big)_0[1]&|[alias=A_1]|\arrow[l,"0"]\mO&|[alias=A_2]|\arrow[l]\big(\mI_D,\mI_D\big)\\
\big(\mI_D,\mO_{D_2}(-D_1)\big)\arrow[u, Maroon]&\arrow[l, Maroon]\arrow[u,Maroon,"\mathfrak{r}^0_{I_D}"]\big(\mF,\mO_{D_2}(-D_1)\big)[1]&|[alias=B_1]|\arrow[l,"0"]\mO_{D_2}(D_2)[-1]\arrow[u]&|[alias=B_2]|\arrow[l]\big(\mI_D,\mO_{D_2}(-D_1)\big)[-1]\arrow[u]
\end{tikzcd}
$$
and Lemma \ref{lem:goodtrian}. The bottom triangle splits because the section $s_1$ vanishes on $D_2$. Note that for the commutativity of the resulting $3\times 3$ holds only for the arrows going left.

Pushing this diagram down from $X\times N^{\JS}_{D_1, \al}$ to $N^{\JS}_{D_1, \al}$, the cube consisting of \B{blue} and \FG{green arrows} is commutative as can be seen immediately for the \B{blue part}. To conclude commutativity also for the base of the cube, I use that the morphism $\big(\mI_D,\mI_{D_1}\big)_0 \begingroup\color{ForestGreen}\longrightarrow\endgroup\big(\mF,\mO(-D_1)\big)[1]$ is the unique one inducing a morphism of the vertical distinguished triangles in 
\begin{center}
    \includegraphics[scale = 0.273]{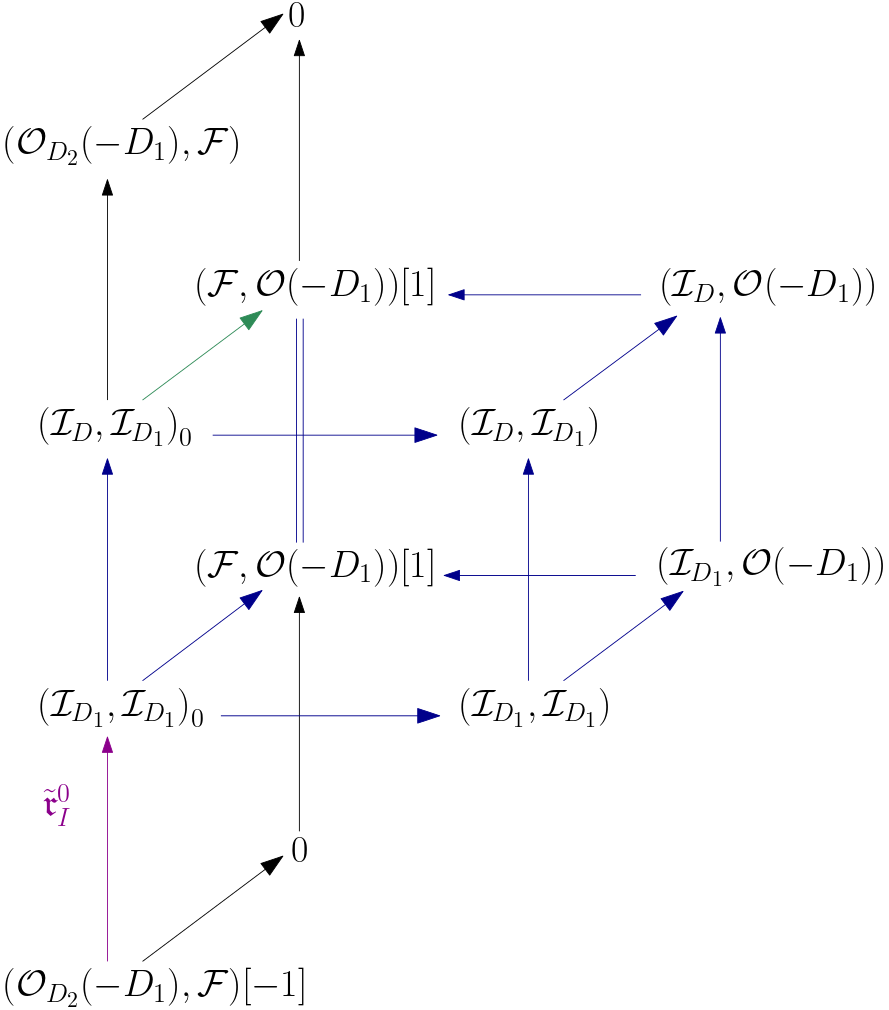}
\end{center}
because the set of morphisms from $\big(\mO_{D_2}(-D_1),\mF\big)$ (in degree 0) to $\big(\mF,\mO(-D_1)\big)[1]$ (in degree 3) being zero. I thus need to show that the roof of the cube is a commutative square. We do so by tracing along the \B{blue arrows} that span it and precomposing with $\big(\mI_{D_1},\mI_{D_1}\big)_0\begingroup \color{Blue} \longrightarrow \endgroup \big(\mI_{D},\mI_{D_1}\big)_0$. Using the commutativity of the full diagram consisting of the \B{blue arrows} which follows from \eqref{eq:trlessoctahedral}, we show that the composition is equal to the original map $\big(\mI_{D_1},\mI_{D_1}\big)_0\begingroup \color{Blue} \longrightarrow \endgroup \big(\mF,\mO(-D_1)\big)[1]$. As this is the only condition on the morphism $(\mI_D,\mI_{D_1})_0\begingroup\color{Green}\longrightarrow\endgroup \big(\mF,\mO(-D_1)\big)[1]$ to induce a morphism of distinguished triangles, we conclude by the previously mentioned uniqueness that the entire diagram is commutative. 

In conclusion, I showed that the two maps $\big(\mI_D,\mI_{D_1}\big)_0
\substack{\begingroup\color{Cyan} \longrightarrow\endgroup\\[-1em] \begingroup\color{Blue} \longrightarrow\endgroup }
\big(\mF,\mO_{D_2}(-D_1)\big)[1]$ induced by the commutativity of the squares ${\color{Blue}A}$ and ${\color{Cyan}B}$ are equal. Therefore, the cocone of ${\color{Cyan}^0\widetilde{\mathfrak{r}}_{I_D}}$ is $\big(\mI_D,\mI_{D}\big)_0$. Equivalently, I have shown that $\EE$ in \eqref{eq:JScompareEE} is given by $\iota^*_1\big(\FF_{1+2}\big)$. To obtain the complete diagram \eqref{eq:JScomparisondiag}, I am left to construct the map $\phi$ that induces a morphism of the lower 2 distinguished triangles. 

Using the map $s_{\mI}:\mI_{D}\to \mI_{D_1}$, the functoriality of Atiyah classes implies that the diagram
$$
\begin{tikzcd}
    \RHom_{N^{\JS}_{D_1, \al}}\big(\mI_{D_1},\mI_{D}\big)[3]\arrow[d,"\circ s_{\mI}"]\arrow[r,"s_{\mI}\circ"]& \RHom_{N^{\JS}_{D_1, \al}}\big(\mI_{D_1},\mI_{D_1}\big)[3]\arrow[d,"\At(\mI_{D_1})"]\\
      \RHom_{N^{\JS}_{D_1, \al}}\big(\mI_{D_1},\mI_{D}\big)[3]\arrow[r,"\At(\mI_{D})"]&\LL_{N^{\JS}_{D_1,\al}}\mI_{D_1}
\end{tikzcd}
$$
commutes. The diagram \eqref{eq:3x3D2}, the middle vertical plane of \eqref{eq:bigdiag2}, and the vanishing of the composition of $H^{\bullet}(\mO_X)\otimes \mO_{N^{\JS}_{D_1, \al}}\to    \RHom_{N^{\JS}_{D_1, \al}}\big(\mI_{D_1},\mI_{D_1}\big)\to \LL_{N^{\JS}_{D_1,\al}}$ and the composition of $H^{\bullet}(\mO_X)\otimes \mO_{N^{\JS}_{D_1, \al}}\to    \RHom_{N^{\JS}_{D_1, \al}}\big(\mI_{D},\mI_{D}\big)\to \LL_{N^{\JS}_{D_1,\al}}$ induce the following commutative diagram:
\begin{equation}
\label{eq:Atcommutativity}
\begin{tikzcd}[row sep= large, column sep=huge]
   \wt{\FF}^\vee[2]\arrow[d,"\ov{\kappa}"]\arrow[r,"\ov{\mu}"]& \FF_1\arrow[d,"\At_0(\mI_{D_1})"]\\
     \iota^*_1\big(\FF_{1+2}\big)\arrow[r,"\At_0(\mI_{D})"]&\LL_{N^{\JS}_{D_1,\al}}
\end{tikzcd}\,.
\end{equation}
Here, I used \eqref{eq:FF1FF1+2}, \eqref{eq:nuFFtilde}, and $\At_0(-)$ to denote the traceless part of the Atiyah class. The composition of 
$$
\begin{tikzcd}
\VV_1[1]\arrow[r, Maroon, "{\nu^\vee[2]}"]&\FF_1\arrow[r,"\At_0(\mI_{D_1})"]&[1.5cm]\LL_{N^{\JS}_{D_1,\al}}
\end{tikzcd}
$$
vanishes by \eqref{eq:JScompareEE} and \eqref{eq:Atcommutativity}, which induces the appropriate map $\wt{\FF}\xrightarrow{\phi}\LL_{N^{\JS}_{D_1,\al}}$. It is a unique such map because $\Hom\big(\VV_1[2],\LL_{N^{\JS}_{D_1,\al}}\big) = 0$. By commutativity of \eqref{eq:Atcommutativity} and the same vanishing argument, the left bottom square of \eqref{eq:JScomparisondiag} commutes. Commutativity of 
$$
\begin{tikzcd}
    \FF_1\arrow[d,"\At_0(\mI_{D_1})"']\arrow[r,"\nu", Maroon]&\VV_1^\vee[1]\arrow[d,equal]\\
    \LL_{N^{\JS}_{D_1,\al}}\arrow[r]&\LL_{\iota_1}
\end{tikzcd}
$$
and the vanishing of $\Hom\big(\VV_1[2],\VV^\vee_1[1]\big)$ imply that the bottom half of \eqref{eq:JScomparisondiag} is a morphism of distinguished triangles. This concludes the proof that \eqref{eq:JScomparisondiag} exists.

\end{proof}
 \begin{remark}
     \label{rem:dagJScomp}
     To give an interpretation of \eqref{eq:JScomparisondiag} in terms of derived schemes rather than just obstruction theories, one can replace the schemes $N^{\JS}_{D_1, \al}, N^{\JS}_{D, \al}$ by their $-2$-shifted symplectic derived enrichments $\bN^{\JS}_{D_1, \al}, \bN^{\JS}_{D, \al}$. Note that there is no map of \textit{derived schemes} 
     $$
   \bN^{\JS}_{D_1, \al}\xlongrightarrow{  \boldsymbol{\iota}_1} \bN^{\JS}_{D, \al}$$
   lifting \eqref{eq:D1embed} because this would imply that there is a distinguished triangle 
   $$
   \begin{tikzcd}
        \MM[-1]\arrow[r]&\iota_1^*(\FF_{1+2})\arrow[r]&\FF_1\arrow[r]&\MM
   \end{tikzcd}
   $$
   for some complex $\MM$. 
     
Instead, the construction \eqref{eq:VVbundle} applied to the universal sheaf on $\bN^{\JS}_{D, \al}$ leads to a derived vector bundle $\textbf{V}_1$ with the section $\boldsymbol{\mathfrak{v}}_1: \mO\to \textbf{V}_1$ and the \textit{derived vanishing locus}
     $$
    \boldsymbol{\mathfrak{v}}^{-1}_1(0) =: \widetilde{\bN}_{D_1,\al}\xlongrightarrow{\boldsymbol{\iota}_1}\bN^{\JS}_{D, \al}\,.
     $$
     The cotangent complexes fit into the distinguished triangle 
     $$
     \begin{tikzcd}
            \bs{\iota}_1^* \LL_{\bN^{\JS}_{D, \al}} \arrow[r]&\LL_{\widetilde{\bN}_{D_1,\al}}\arrow[r]&\textbf{V}^\vee_1[1]\arrow[r]&              \LL_{\bN^{\JS}_{D, \al}}|_{\boldsymbol{\mathfrak{v}}^{-1}_1(0)} [1]\,.
     \end{tikzcd}
     $$
     The relation to $\bN^{\JS}_{D_1, \al}$ is more subtle, because there is only a morphism
$$
 \widetilde{\bN}_{D_1,\al} \xlongrightarrow{\bs{\pi}} \bN^{\JS}_{D_1, \al}\,,
$$
which induces the distinguished triangle
\begin{equation}
\label{eq:etale}
\begin{tikzcd}
\bs{\pi}^*\LL_{\bN^{\JS}_{D_1, \al}}\arrow[r]&\LL_{\widetilde{\bN}_{D_1,\al}}\arrow[r]&\VV_1[2]\arrow[r]&\bs{\pi}^*\LL_{\bN^{\JS}_{D_1, \al}}[1]\,.
\end{tikzcd}
\end{equation}
Combining the distinguished triangles leads to the diagram
  \begin{equation}
    \label{eq:dagJScomp}
    \begin{tikzcd}[column sep=large]  \LL^\vee_{\widetilde{\bN}_{D_1,\al}}[2]\arrow[d]\arrow[r]&\bs{\pi}^*\LL_{\bN^{\JS}_{D_1, \al}}\arrow[d]\arrow[r]&\VV_1^\vee[1]\arrow[d,equal]\arrow[r]& \LL^\vee_{\widetilde{\bN}_{D_1,\al}}[3]\arrow[d]\\
 \bs{\iota}_1^* \LL_{\bN^{\JS}_{D, \al}} \arrow[r]&\LL_{\widetilde{\bN}_{D_1,\al}}\arrow[r]&\VV_1^\vee[1]\arrow[r]& \bs{\iota}_1^* \LL_{\bN^{\JS}_{D, \al}} [1]
    \end{tikzcd}\,,
    \end{equation}
    on $\widetilde{\bN}_{\nv,D_1}$ instead of $\bN^{\JS}_{D_1, \al}$. However, the truncation 
    $$
\pi = t_0(\bs{\pi}): t_0\big(\widetilde{\bN}_{ D_1,\al}\big)\longrightarrow N^{\JS}_{D_1, \al} 
    $$
    is a bijection and additionally étale by the distinguished diagram \eqref{eq:etale}. As such, it is an isomorphism, so after restricting \eqref{eq:dagJScomp} to $t_0\big(\widetilde{\bN}_{D_1,\al}) = N^{\JS}_{D_1, \al}$, one recovers \eqref{eq:JScomparisondiag}. That it satisfies the necessary self-duality and commutativity would follow from the claim that
    $$
    \begin{tikzcd}[column sep=small]
&\arrow[dl,"\bs{\pi}"]\widetilde{\bN}_{D_1,\al}\arrow[dr,"\boldsymbol{\iota}_1"]&\\
    \bN^{\JS}_{D_1, \al}&&\bN^{\JS}_{D, \al}
    \end{tikzcd}
    $$
    is a \textit{shifted Lagrangian correspondence}. This derived-geometric formulation of Pvp-digrams was observed by Park and stated independently in \cite{LagCor} by Schürg.
 \end{remark}
     \end{proof}

     \section{Applications}
     The second situation of Corollary \ref{cor:WCconsequences} concerning quivers was proved in §\ref{sec:quivers}. I will begin this section by applying it to a quiver that reproduces Hilbert schemes of points on $\CC^4$. The resulting wall-crossing formula is the equivariant version of the one used in \cite{bojko2}.

     The second subsection deals with the claim of Corollary \ref{cor:WCconsequences} that addresses local CY fourfolds. It describes the obstruction theories required by Assumption \ref{ass:obsonflag} and Assumption \ref{ass:pairWC}.d). I then discuss the situations where the rest of the assumptions apply, which, in particular, leads to the proof of Corollary \ref{cor:introDTPT}.    
\subsection{Wall-crossing into $\Hilb^n(\CC^4)$}
\label{sec:CY4dgC4}
Consider the CY4 dg-quiver
\begin{equation}
\label{eq:C4dgquiver}
 \includegraphics[scale=1.2]{C4quiver.pdf}
\end{equation}
with the superpotential 
$$
\BuOr{\mH} =\frac{1}{4} \sum_{\sig\in S_4}(-1)^{\sig}\BuOr{c_{\sig(1)\sig(2)}}\circ[x_{\sig(3)},x_{\sig(4)}]\,.
$$
Here I used $(-1)^{\sig}$ to denote the sign of the the permutation $\sig$ and $\BuOr{c_{ji}}:=-\BuOr{c_{ij}}$ for $i<j$. For any $\sig\in S_4$ and the edge $\BuOr{c_{\sig(1)\sig(2)}}$, I further set $\BuOr{c_{\sig(1)\sig(2)}^*} =(-1)^{\sig}\BuOr{c_{\sig(3)\sig(4)}}$ which is well-defined.

The action of the differential on degree \BuOr{$-1$} edges becomes
$$
d(\BuOr{c_{\sig(1)\sig(2)}}) = (-1)^{\sig}\frac{\partial^{\circ} \BuOr{\mH}}{\partial \BuOr{c_{\sig(3)\sig(4)}}} = [x_{\sig(1)},x_{\sig(2)}] \,.
$$
From this, one sees that the master equation \eqref{eq:mastereq} holds. 

I have used a different description as compared to Definition \ref{def:CY4quiver} because the degree \BuOr{$-1$} loops are not paired with themselves. The original definition can be recovered by introducing $$\BuOr{e_{\sig(1)\sig(2)}} = \BuOr{c_{\sig(1)\sig(2)}}+(-1)^\sig\BuOr{c_{\sig(3)\sig(4)}}$$
for any $\sig\in S_4$. 

The category $\mA$ of degree 0 representations of 
 $\wt{C}_4^{\bullet}$ is equivalent to the category of representations of the quiver $C_4:=H^0(\wt{C}_4^{\bullet})$. This quiver is given by forgetting the edges of degrees $<0$ in \eqref{eq:C4dgquiver} and imposing the relations
 $$
 [x_i,x_j] =0 \qquad\text{whenever}\quad 1\leq i<j\leq 4\,.
 $$ 
This is the quiver used in \cite{KRdraft} to prove Nekrasov's conjecture. Moreover, the obstruction theory described by Lemma \ref{lem:cotangentofQ} and \eqref{eq:C4dgquiver} coincides with the one appearing in this reference. 

Here, I will consider a simple family of stability conditions on the category $\mA$ that is determined by slope stabilities of $C_4$. It is given by 
$$
\begin{tikzcd}
{[-1,1]}\ni t\mapsto \mu_{t} = (t,0)
\end{tikzcd}
$$
acting by $\mu_{t}(d_{\infty},d_0) = d_{\infty}\cdot t$. This family has two stability chambers $\{t>0\}$, $\{t<0\}$, and one wall $\{t=0\}$. I will consider wall-crossing for representations 
$$
\begin{tikzcd}
&v=m_e(1)\arrow[d, phantom, "{\rotatebox[origin=c]{270}{$\in$}}"]\\
V_\infty\arrow[r,"m_e"]&V_0\arrow[loop,  in=-35,out=35, looseness=14, "X_i"]
\end{tikzcd}
$$
 of dimension vectors $(d_{\infty},d_0)$ with $d_{\infty}=0,1$ and $0\leq d_0<\infty$. Here, $X_i = m_{x_i}$ are endomorphisms of $V_0$ for $i\in [4]$, and when $d_{\infty}=1$, I identify $V_{\infty}=\CC$ and set $v:=m_e(1)$. With this notation, the semistable objects of class $(d_{\infty},d_0)$ are the following representations:
 \begin{table}[h]
 \centering
\begin{tabular}{c|c|c}
&$d_{\infty}=1$&$d_{\infty}=0$\\
\hline
$\{t>0\}$&$\Big\{p(X_i)(v)\colon p\in \CC[x_1,x_2,x_3,x_4]\Big\}=V_{0}$&all\\
$\{t=0\}$&all&all\\
$\{t<0\}$&$V_0=0$&all
\end{tabular}
\end{table}

It is well-known (see for example \cite[§2.2]{Szendroi} and \cite[Example 6.3.1]{Lam}) that  for $d_{\infty}=1$ and $\{t>0\}$, this description implies that 
\begin{equation}
\label{eq:Hilbident}
M^{\mu_{t}}_{(1,n)}\cong \Hilb^n(\CC^4)\qquad \text{and}\qquad \Big[M^{\mu_{t}}_{(1,n)}\Big]^{\vir}= \Big[\Hilb^n(\CC^4)\Big]^{\vir}\,.
\end{equation}
 Let $\mM_{mp}$ denote the moduli stack of 0-dimensional sheaves of characteristic $m$ on $\CC^4$. For $(d_{\infty},d_0) = (0,m)$ and $\{t>0\}$, the identification \begin{equation}\label{eq:Mnpident}\mM^{\mu_{t}}_{(0,m)}\cong \mM_{mp}
\end{equation}
 also holds. 

After an easy check of Assumption \ref{ass:stab} and Assumption \ref{ass:obsonflag}.a), Corollary \ref{cor:WCconsequences} applies to the family of stability conditions $\{\mu_t\}_{t\in [-1,1]}$. Explicitly, this means the following.
\begin{corollary}
\label{cor:hilbWC}
Using the identifications \eqref{eq:Hilbident} and \eqref{eq:Mnpident}, the formula
$$
\sum_{n>0}\big[\Hilb^n(\CC^4)\big]^{\vir}q^n= \exp\bigg\{\sum_{n>0}\Big[\langle \mM_{np}\rangle,- \Big]q^n\bigg\}e^{(1,0)}
$$
holds for the corresponding classes in $L_{\loc,0}$ constructed for $\mA$. Here $e^{(1,0)}$ is the point-class of the connected component $\mM_{(1,0)}$.
\end{corollary} 
The proof of Assumption \ref{ass:stab} will be addressed in a larger generality in the next subsection. 
\subsection{Stable pair wall-crossing on local Calabi--Yau fourfolds}
\label{sec:localCY4}
Due to the limitations explained in §\ref{sec:sheaves}, one needs to impose restrictions on the geometry to construct the obstruction theories of Assumption \ref{ass:obsonflag}. In this subsection, I will work with a fixed 3-fold $Y$ with a $\GG_m$-action. The total space $X$ of its canonical bundle $K_Y$ admits the induced $\GG_m$-action such that the natural projection $\pi: X\to Y$ is $\GG_m$-equivariant. Thus one obtains a short exact sequence of $\GG_m$-equivariant vector bundles
$$
\begin{tikzcd}
0\arrow[r]&\pi^*K_Y\arrow[r]&T_X\arrow[r]&\pi^*T_Y\arrow[r]&0\,.\end{tikzcd}
$$
Taking determinants and duals, one constructs an isomorphism of $\GG_m$-equivariant line bundles $K_X\cong \pi^*K_Y ^* \otimes \pi^*K_Y\cong \mO_X$. Thus $X$ has a $\GG_m$-equivariant Calabi--Yau form and \cite{OT} define equivariant virtual fundamental cycles of perfect complexes on $X$.

I now briefly recall the spectral correspondence argument which was used for curves in \cite{BNR}. I will follow the formulation presented in \cite{TT1}. Thus, let $\pi: X\to Y$ be the above projection with $Y$ not necessarily compact. For each compactly supported sheaf $F$ on $X$, consider the pair $(F,\eta)$ where $\eta:F\to F\otimes \pi^*K_Y$ is the tautological section. Projecting it to $Y$ produces a Higgs pair $(\pi_*F,\phi)$ where
$$
\begin{tikzcd}
\phi:=\pi_*(\eta): \pi_*(F)\arrow[r]&\pi_*(F)\otimes K_Y\,.
\end{tikzcd}
$$
This induces an equivalence between the categories 
$$
\left\{
\begin{array}{c}
F\in\Coh(X) \textnormal{ compactly supported}
\end{array}\right\}\begin{tikzcd}\, \arrow[r, leftrightarrow]&\,\end{tikzcd}\left\{
\begin{array}{c}(E,\psi)\textnormal{ compactly supported}\\
\textnormal{Higgs pairs on }Y
\end{array}\right\}
$$
Here, a compactly supported Higgs pair $(E,\psi)$ consists of a compactly supported sheaf $E$ on $Y$ and a morphism $\psi: E\to E\otimes K_Y$.

The classical stack $\mM^{\sig}_{\al}$ of $\sig$-semistable sheaves of class $\al\in K^0_{\cs,e}(X)$ can be identified with an open substack of compactly supported Higgs pairs. This also holds on the level of derived refinements. To describe its obstruction theory, consider the map $\mathscr{P}: \mM^{\sig}_{\al}\to \mM_Y$ which forgets the morphism $\psi$. Set
\begin{equation}
\label{eq:GGEE}
\GG :=\mathscr{P}^* \RHom_{\mM_Y}(\mE,\mE)^\vee[-1]\,,\qquad \EE := \RHom_{\mM^{\sig}_{\al}}(\mF,\mF)^\vee[-1]
\end{equation}
for the universal objects $\mE$ on $Y\times\mM_Y$ and $\mF$ on $X\times\mM^{\sig}_{\al}$. Then, there exists the fiber sequence (in black)
\begin{equation}
\label{eq:spectralobs}
\begin{tikzcd}
&\B{\MM_O[-1]}\arrow[dl,blue]\arrow[r, blue,equal]&\B{\MM_O[-1]}\arrow[dl,blue]&\\
\GG\arrow[r]&\EE\arrow[dl,blue]\arrow[r]&{\GG^{\vee}[2]}\arrow[dl,blue]&\\
\B{\MM_O^\vee[3]}\arrow[r, blue,equal]&\B{\MM_O^\vee[3]}&&
\end{tikzcd}
\end{equation}
with connecting morphism $\GG^\vee[2]\to \GG[1]$, which is the dual of
$$
\begin{tikzcd}
{[-,\Psi]}: \mathscr{P}^* \RHom_{\mM_Y}(\mE,\mE)\arrow[r]&\mathscr{P}^* \RHom_{\mM_Y}(\mE,\mE\otimes K_Y)
\end{tikzcd}
$$
for the universal morphisms $\Psi:\mE\to \mE\otimes K_Y$. 

Let us now consider pairs of sheaves and the setup described in Assumption \ref{ass:pairWC} and modified for the non-compact setting in Remark \ref{rem:afterpairs}. In this case, I choose the object 
$$
O=\pi^*(L_Y)
$$
for a line bundle $L_Y$ on $Y$. Consider the obstruction theory on $\big(\mN^{\sig^P}_{1,\al}\big)^{\rig}$ for $(1,\al)\in \msE(\mB_O)$ induced by its open embedding into $\mM_{\ov{X},\ov{O}}$ where $\ov{X} = \PP_{Y}\big(\mO_Y\oplus K_Y\big)$ and $\ov{O}$ is the pullback of $L_Y$ along the projection from $\ov{X}$. To describe this obstruction theory in the form of \eqref{eq:spectralobs}, consider the map
\begin{equation}
\label{eq:dequals1}
\begin{tikzcd}
\mathscr{P}_O: \mN^{\sig^P}_{1,\al}\arrow[r]&\mN_{1,\pi_*(\al)}
\end{tikzcd}
\end{equation}
where the latter stack parametrizes pairs of the form $L_Y\to E$ for $\llbracket E\rrbracket =\pi_*(\al)$. This map is induced by the adjunction between $\pi^*$ and $\pi_*$.

Continuing to use $\mF$, respectively $\mE$, for the universal sheaves on $X\times \big(\mN^{\sig^P}_{1,\al}\big)^{\rig}$ and $Y\times \big(\mN_{1,\pi_*(\al)}\big)^{\rig}$, the obstruction theory for the pairs is recovered by taking cones and cocones of the \B{blue arrows} meeting $\EE$ in \eqref{eq:spectralobs}  in the sense of Proposition \ref{prop:sympullback}. This is possible because their composition is naturally homotopic to 0. Here, it is understood that \eqref{eq:GGEE} is used with $\mathscr{P}^{\rig}_O$ replacing $\mathscr{P}$, and I set 
\B{$$\MM_O :=\RHom_{\big(\mN^{\sig^P}_{1,\al}\big)^{\rig}}\big(\pi^*(L_Y),\mF\big) =\big(\mathscr{P}^{\rig}_O\big)^{*}\RHom_{\mN^{\rig}_{1,\pi_*(\al)}}\big(L_Y,\mE\big)\,.$$}
 The null-homotopy of the composition follows because the morphism $\B{\MM_O[-1]\to }\FF$ is given by the composition of $\B{\MM_O[-1]\to }\GG\to \FF$, and the dual statement also holds. The resulting obstruction theory is given by
$$
\FF^{\sig^P}_{1,\al} = \RHom_{\mN^{\sig^P}_{1,\al}}\big(O\to \mF,O\to \mF\big)^\vee_0[-1]
$$
where $(-)_0$ denotes the traceless part.

To prove Corollary \ref{cor:WCconsequences}.2) and its generalization to stable pairs, it is necessary to show that the obstruction theories of Assumption \ref{ass:pairWC}.d) exist.
\begin{proposition}
\label{prop:localCY}
Continue working in the situation above with a fixed $O$, a set of weak stability conditions $W^P$ on $\mB_O$, and $\msE(\mB_O)$ satisfying Assumption \ref{ass:stab}, \ref{ass:pairWC}.a), and c) (or rather their replacements from Remark \ref{rem:afterpairs}). If $D_k$ are pullbacks of divisors in $Y$, then Assumption \ref{ass:pairWC}.d) and Assumption \ref{ass:welldef} hold. 
\end{proposition}
\begin{proof}
As before, fix a dimension vector $\un{d}$ of $\mathring{I}$ such that $d_1=1$. I will also assume that $d=1$ as in \eqref{eq:dequals1} and denote the corresponding moduli stack of objects in $\mbB_{O,I_k}$ by $\mbN^{\rig}_{O,(\un{d},1,\al)}$.  Let $\widehat{\pi}_{\un{d},\al}: \mbN^{\rig}_{O,(\un{d},1,\al)}\to \mbN_{O,k}$ be a projection induced by the composition of the arrows in \eqref{eq:OprojmN}. For the rest of the proof, I will omit specifying pullbacks of complexes when they are clear. Setting $\NN_{\un{d},\al}:=T_{\widehat{\pi}_{\un{d},\al}}$, one can construct a CY4 obstruction theory on the appropriate semistable locus of $\mbN^{\rig}_{O,(\un{d},1,\al)}$ from 
$$
\begin{tikzcd}
&\B{\NN_{\un{d},\al}[-1]\oplus\MM_O[-1]}\arrow[dl,blue]\arrow[r, blue,equal]&\B{\NN_{\un{d},\al}[-1]\oplus\MM_O[-1]}\arrow[dl,blue]&\\
\GG\arrow[r]&\EE\arrow[dl,blue]\arrow[r]&{\GG^{\vee}[2]}\arrow[dl,blue]&\\
\B{\NN^\vee_{\un{d},\al}[3]\oplus\MM_O^\vee[3]}\arrow[r, blue,equal]&\B{\NN^\vee_{\un{d},\al}[3]\oplus\MM_O^\vee[3]}&&
\end{tikzcd}
$$
by the same arguments as in the case of \eqref{eq:spectralobs} and $\FF^{\sig}_{1,\beta}$. This addresses Assumption \ref{ass:pairWC}.d). 

The proof of Assumption \ref{ass:welldef} in the case of compactly supported codimension 1\footnote{In this case, one needs to assume that $Y$ is projective.} Gieseker and slope-stable sheaves is addressed in Corollary \ref{cor:JSWC} to showcase the quantum Lefschetz-type argument used in §\ref{sec:proof}. The general statement follows by the arguments developed in \cite[§9]{JoyceWC} based on \cite[§7.3]{mochizuki}. In combination with spectral correspondence, the latter approach was also used in the case of CY threefolds and dimension 2 sheaves in \cite{LiuVW}.
\end{proof}
The next table summarizes the references that prove the different parts of Assumption \ref{ass:stab}, Assumption \ref{ass:pairWC}.a), b) and c). Wall-crossing in these situations follows from the above.
\setlength{\tabcolsep}{15pt}
\renewcommand{\arraystretch}{1.5}
\begin{center}
\begin{tabular}{c|c c}
&Assumption \ref{ass:pairWC}.a) and b) & Assumption \ref{ass:pairWC}.c)\\
\hline
\begin{tabular}{c}DT/PT wall-crossing\\
(Example \ref{ex:DTPT})
\end{tabular}&\cite[Lemma 3.15]{todacurve}&\cite[Proposition 6.1.5]{KLT}\\
\begin{tabular}{c}JS wall-crossing\\
\eqref{eq:JSWC}
\end{tabular}&\cite[Appendix A]{bojko3} &\begin{tabular}{c}\cite[Appendix A]{bojko3},\\
\cite[Proposition 6.1.5]{KLT}
\end{tabular}
\end{tabular}
\end{center}
Note that all of these references are for lower-dimensional cases, but they can be adapted in a straightforward way to the present setting. The first consequence is Joyce--Song pair wall-crossing for 3-dimensional Gieseker/slope semistable sheaves and projective $Y$. One needs to modify the references cited by observing the following lemma securing properness of the fixed-point loci.
\begin{lemma}
\label{lem:3dim0section}
    Consider $\pi: X\to Y$ as above for a smooth projective $Y$ with Kodaira dimension $\kappa(Y)=-\infty$, then every compactly supported 3-dimensional sheaf $F$ on $X$ is set-theoretically supported on the zero section. In particular, this holds when $Y$ is a projective Fano 3-fold or admits a non-trivial $\T$-action.
\end{lemma}
\begin{proof}
The assumption $\kappa(Y)=-\infty$ implies that
$$
  H^0(Y, K^{\otimes m}_Y) = 0 \qquad \textnormal{for}\quad m>0\,.
  $$
  Let $\eta\in H^0(\pi^*K_Y)$ be the tautological section with $\eta^{-1}(0) = Y$. If the scheme-theoretic support of $F$ is a divisor $D\subset X$ such that $\pi|_D: D\to X$ is degree d, then $D$ is the locus where 
  $$
  \eta^d + \pi^*(a_1)\eta^{d-1}+\cdots +\pi^* (a_d) = 0
  $$
  for some $a_m\in H^0(K^{\otimes m}_Y)$. From the above, we see that this is equivalent to $\eta^d=0$, which proves the first statement.

  The presence of a non-trivial $\T$-action on $Y$ implies that it is uniruled. This, in turn, says that $\kappa(Y)=-\infty$. The same is known for $Y$ being Fano.
\end{proof}
By the arguments of §\ref{sec:proof}, this implies that the invariants defined by the procedure from Definition \ref{def:Masik} are independent of $k$.
\begin{corollary}
\label{cor:JSWC}
Let $X$ be as in Lemma \ref{lem:3dim0section} and $\sig$ Gieseker or slope stability for $X$ defined relative to $\pi^*H$ where $H$ is an ample divisor class of $Y$. Then Joyce--Song pair wall-crossing \eqref{eq:JSWC} holds for any appropriate $D_l$  and $\al\in K^0_{\cs,e}(X)$ such that $\pi_*(\al)$ has positive rank. All classes $\al_i$ appearing in the wall-crossing formula also satisfy this condition. 
\end{corollary}
To obtain the diagram \eqref{eq:JScomparisondiag}, one does not even need to go through the full proof of Theorem \ref{thm:independence!}. This is due to \eqref{eq:spectralobs} and the vector bundle $\VV_1$ from \eqref{eq:VVbundle} only modifying \B{$\MM_O[-1]$} independently of \B{$\MM^\vee_O[3]$}. The divisors $D_k$ are still required to be pullbacks of ones in $Y$. By applying the rest of the proof of Theorem \ref{thm:independence!} I conclude the following.
\begin{corollary}
\label{cor:spectralindep}
Let $\sig$ and $\al\in K^0_{cs}(X)$ be as in Corollary \ref{cor:JSWC}, then 
$$
\big\langle\mM^{\sig}_{\al}\big\rangle^{k_1}=\big\langle\mM^{\sig}_{\al}\big\rangle^{k_2}\,.
$$
\end{corollary}
I wrap up this section by addressing the proof of Corollary \ref{cor:introDTPT} in the next example. 
\begin{example}
    Consider the weak stability conditions $\{\sig_t\}_{t\in [0,1]}$ on $\mB$ from Example \ref{ex:DTPT}. In this case, one sets $O = \mO_X$, $\ov{K}_{\cs} = H^*_{\cs}(X)$ and $\msE_O = \{(\beta,n)\in H^{6}_{\cs}(X)\oplus H^{8}_{\cs}(X) : \beta \geq 0 \}$. Thus, one only allows 1-dimensional $F$ with $\ch(F) \in \msE_O$. The set of stability conditions $W^P$ is $[0,1]$. Each $\sig^P_t$ acts by assigning the following phases in $[-\frac{1}{2},1]$:
$$
\big(d,(\beta,n)\big)\mapsto \begin{cases}
    \frac{1}{2}&\textnormal{if }d>0\,,\\
      -\frac{1}{2}&\textnormal{if }d=0,\beta>0\,,\\
        1-t&\textnormal{if }d,\beta=0, n>0\,.\\
\end{cases}
$$
For $\big(1,(\beta,n)\big)$ with $\beta>0$ the associated group homomorphism $\lambda^{\frac{1}{2},t}$ from \eqref{eq:lambdaiffphi} for $t\neq \frac{1}{2}$ is defined by
$$
\lambda^{\frac{1}{2},t}\big(d,(\beta,n)\big) = \begin{cases}
0&\textnormal{if }d>0\textnormal{ or }\beta>0\,,\\
n(\frac{1}{2}-t) &\textnormal{if }d,\beta=0, n>0\,.
\end{cases}
$$
We can describe all the semistable objects for $t=\frac{1}{2}$ as pairs $\mO_X\xrightarrow{s}F$ where $F$ is a compactly supported sheaf with $\ch(F) =(\beta,n)$ and $s$ has 0-dimensional cokernel. 
\end{example}
The rest of the assumptions follow from the table above combined with Lemma \ref{lem:compactcurve} below, which determines when $\T$-equivariant $\sig^P_{\frac{1}{2}}$-semistable pairs $(F,s)$ are supported on a projective subvariety of $X$ -- in the sense that $F$ are.

For simplicity, I set $\T = \CC^*$ as the discussion below can be easily generalized to any dimension of $\T$. Let $C=\overline{\CC^*\cdot y}$ be a rational orbit closure of some $y\in Y$, meaning that its normalization is isomorphic to $\PP^1$. Let $p_{0}(y),p_{\infty}(y)\in C$ be the two fixed endpoints of $C$ given by $\lim_{t\to 0} t\cdot y$ and  $\lim_{t\to \infty} t\cdot y$, respectively. Denote the $\CC^*$-weights of $K_Y|_C$ restricted to the endpoints by 
$$
e^{\lambda_0}= K_Y|_{p_0(y)}\,,\qquad  e^{\lambda_{\infty}}= K_Y|_{p_{\infty}(y)}\,.
$$
\begin{lemma}
\label{lem:compactcurve}
 \begin{enumerate}
  \leavevmode
  Let $Y$ be a quasi-projective smooth 3-fold with a $\CC^*$-action such that $Y^{\CC^*}$ is projective. Moreover, the reduced closure $\msC\mS$ of the union of all rational orbit closures $C=\ov{\CC^*\cdot y}$ ought to be projective as well.
Assume then that the following two further conditions hold:
\vspace{-4pt}
    \item[(0-dimensional)] The restriction $K_Y|_{Y^{\CC^*}}$ has nowhere-vanishing weights.
    \item[(1-dimensional)] For a rational orbit closure $C$ as above, assume that whenever $\lambda_0\geq \lambda_{\infty}$, then $\lambda_0<0$ or $\lambda_{\infty}>0$.
 \end{enumerate}
 Fix a class $(\beta,n)\in \msE_{O}$. Then there is a closed projective subscheme $\msS\subset X$ such that all $\CC^*$-equivariant 0-dimensional sheaves and $\sig^P_{\frac{1}{2}}$ semistable pairs are scheme-theoretically supported on $\msS$.
\end{lemma}
\begin{proof}
Because the case when $\beta = 0$ is simpler, I will assume $\beta>0$. A $\T$-equivariant semistable pair $(F,s)$ will then admit a $\T$-equivariant short exact sequence 
$$
\begin{tikzcd}
0\arrow[r]&\mO_D\arrow[r]& F\arrow[r]&Q\arrow[r]&0
\end{tikzcd}
$$
where $D$ is a proper $\T$-invariant 1-dimensional subvariety of $X$. Moreover, we can take the maximal 0-dimensional subsheaf $P\subset \mO_D$ and focus on the quotient $\mO_D/ P$. Since $\beta$ is fixed and $\mO_D/ P$ is pure, the lower bound on the slope $\mu(\mO_D/ P)$ proved in \cite[Corollary 2.13]{lepotier} introduces a lower bound on $\chi(\mO_D/ P)$ which in turn bounds $N_0 = \chi(P) + \chi(Q)$. Since both of them are 0-dimensional they must be scheme-theoretically supported on a thickening of $Y^{\CC^*}$. The degree of the thickening is bounded by $N_0$.

By the above, I can focus on pure 1-dimensional sheaves $F$ with scheme-theoretic support $D$ on $X$. Without loss of generality, I assume that $D$ is irreducible. Let $C$ be its reduced scheme-theoretic image under $\pi$. Then it is $\CC^*$-invariant under the action on $Y$. Since $C$ is irreducible and reduced, it must either lie in $Y^{\CC^*}$ or be isomorphic to a rational orbit closure lying in $\msC\mS$. In the former case, one already concludes that $F$ is set-theoretically supported on the zero section due to the positivity of the weights of $K_Y|_{Y^{\CC^*}}$. Thus, we may assume that $C$ admits a normalization $\nu:\widetilde{C}\to C$ such that $\wt{C}\cong \PP^1$. It is easy to check that 
    $$
    \deg\big(\nu^*K_Y|_C\big) <0 \qquad\iff  \qquad \lambda_0  <\lambda_{\infty}  $$
    which would already imply that there are no sections of $\nu^*K^i_Y|_C$, so $F$ would be set-theoretically supported on the zero section. When the degree is non-negative, one needs to show that 
    $$
    H^0\big(\nu^*K^i_Y|_C\big)^{\CC^*} = 0\,.
    $$
    Since $\nu^*K_Y|_C \cong \mO_{\widetilde{C}}(d)$ for $d \geq 0$, its sections have weights 
    $$
    (\lambda_0,\lambda_0-\tau,\cdots,\lambda_0-a\tau)
    $$
    where $\tau>0$ is the tangent weight at $0$ and $\lambda_0-a\tau = \lambda_{\infty}$. Thus, there are no invariant sections whenever $\lambda_0<0$ or $\lambda_{\infty}>0$. The same argument applies to the $i$'th powers, so $D$ lies set-theoretically in the zero-section. Since $\beta$ was fixed, we obtain an upper bound $N_1$ on the possible thickening of $0(C)$ where $F$ will be supported scheme-theoretically. 

Combining the two paragraphs, one may take $N = \max\{N_0,N_1\}$ and the $N$'th thickening of $Y^{\CC^*}\cup \msC\mS$ to construct the required $\msS$.
\end{proof}
To include a more general torus $\T$, one would simply restrict to 1-dimensional rational orbit closures $\ov{\T\cdot y}$ and the weights $\lambda_0,\lambda_\infty$ would be taken with respect to $\T$ modulo the stabilizer of $y$.
\begin{example}
There are many examples that this set-up includes, so I will not attempt to enumerate them. For example, any projective Fano 3-fold with a non-trivial $\CC^*$-action such that the simple (0-dimensional) condition is satisfied is automatically included.

One could also, for example, consider taking projective bundles. So let $Z$ be a $<3$-dimensional smooth projective variety, $E$ a vector-bundle on it so that $\dim(Z) +\Rk(E)= 3$. Then, one can take $Y$ to be the total space of $\PP(\mO_Z\oplus E)\to Z$ where $\CC^*$ acts along the fibers of $\mO_Z$ with non-zero weight. To remove the projectivity assumption of $Z$, one could allow some action of an extra torus $\T$, so that $E$ becomes $\T$-equivariant and both $Y^{\CC^*\times T}$ and $\msC\mS$ become proper.

It is also not difficult to come up with examples obtained by taking total spaces of equivariant vector bundles, but I will leave it to the reader to consider further cases.
\end{example}
\section{Proving wall-crossing}
Fix an abelian category $\mA$ with the data as in Definition \ref{def:categoryA}. Here, I will consider the general situation of Theorem \ref{thm:familyWC}. Therefore, the statements in this subsection work under the condition that Assumptions \ref{ass:orientation}, \ref{ass:stab}, and \ref{ass:obsonflag} hold. 
\subsection{Summary of the main steps}
\label{sec:summaryproof}
I will now explain the core ideas of Joyce's proof of wall-crossing from \cite[§10]{JoyceWC}, which relies on two key formulae. These need to be proved separately as is done in §\ref{sec:WCflags} and §\ref{sec:Flagtosheaves}. 

\begin{enumerate}
    \item \textbf{Divide wall-crossing into a finite number of steps}
    By Assumption \ref{ass:stab}.b), there exists a well-behaved path $\gamma_{(-)}$ between any two stability conditions $\sig,\sig'\in W$. Fix $\al\in \msE(\mA)$, for which the wall-crossing formula \eqref{eq:MasiWC} is to be proved. Because Assumption \ref{ass:welldef} is not enforced yet, one needs to choose $k\in K$ and replace $W$ by its subset $W_{\al,k}$ assuming that the latter is non-empty.
    
    Consider the finite set of partitions $\circledast$ from Assumption \ref{ass:stab}.c). Because $\gamma_{\beta_j=\al}$ from \eqref{eq:betajequalsal} is a finite union of closed intervals by Assumption \ref{ass:stab}.b), one can choose $\{t_a\}_{a \in A}$ for a finite set $A$ to be the set of all the boundary values of these intervals for all $\circledast$ where $\beta_j$ are not pairwise collinear. If condition $(P)$ from Assumption \ref{ass:stab}.b) holds, then the above intervals are isolated points. All non-trivial wall-crossing contributions will appear when passing through some $t_a$ for $a\in A$.
    \item \textbf{Prove wall-crossing enhanced by flags in the neighbourhood of each $t_a$.} This forms the core of the argument and is where wall-crossing truly happens. For any $a\in A$ and a small $\delta$ such that $(t_a-\delta,t_a+\delta)$ does not contain any other $t_b$, set
    \begin{equation}
    \label{eq:sigprclosetosig} 
    \sig = \sig_{t_a}\,,\qquad \sig'= \sig_{t}\qquad \text{for some}\quad  t\in (t_a-\delta,t_a)\cup (t_a,t_a+\delta)\,.\end{equation}
For the above chosen $\al\in \msE(\mA)$, fix the vector $\un{d}$ with $d_i=i$ for all $i\in \big\{1,2,\ldots,\chi\big(\al(k)\big)\big\}$, and the subset $B:= B_{\al,t_a}\subset \msE(\mA)$ defined in \eqref{eq:Balt}. Assumption \ref{ass:stab}.d) gives a group homomorphism $\lambda' =\lambda^{t_a,t}_{\al}: \ov{K}(\mA)\to \RR$ that can be used to compare phases $\phi'(\beta)$ and $\phi'(\al)$ for all $\beta\in B$ (see \eqref{eq:lambdaiffphi}). Consider the family of stability conditions
$$\sig^{s\lambda'}_{\mu}\qquad\text{for}\quad s\in [0,1]$$
which satisfies
\begin{equation}
\label{eq:Nsig'0}
N^{\sig^{\lambda'}}_{\un{d},\al} = N^{(\sig')^{0}}_{\un{d},\al}\,.
\end{equation}
 To conclude the above equality, one observes that $\sig^{\lambda'}_\mu$-semistable objects $(\un{V},\un{m},E)$ in $\mB^{\sig}_{\Flag_k,\phi(\al)}$ of class $(\un{d},\al)$ have $\sig'$-semistable $E$ due to $E$ being already $\sig$-semistable and $\mu_i\ll \lambda'(\beta)$ for all $\beta\in A$. The detailed proof follows by using \cite[Lemma 11.4]{JoyceWC} to reduce the situation \eqref{eq:sigprclosetosig} to the one set up in Definition \cite[Definition 10.1]{JoyceWC}. Then, \eqref{eq:Nsig'0} follows from \cite[Proposition 10.2 and 10.5]{JoyceWC}. Additionally, \cite[Proposition 10.3]{JoyceWC} shows that all morphisms $m^{e_i}:V_i\to V_{i+1}$ are injective for $i<r-1$ and $m^{e_{r-1}}:V_{\times}\to V_k(E)$ is an isomorphism for semistable objects $(\ov{V},\ov{f},E)$. This explains the choice of notation $(-)_{\Flag}$.
 
The main claim of this step is a wall-crossing formula satisfied by $\sig^{s\lambda'}_{\mu}$-semistable objects in $\mB^{\sig}_{\Flag_k,\phi(\al)}$ as one varies $s\in [0,1]$. This appears as the top horizontal line in the following graphic representing the full argument of the proof:
\begin{equation}
\label{eq:proofWCgraphic}
\includegraphics[scale=1.2]{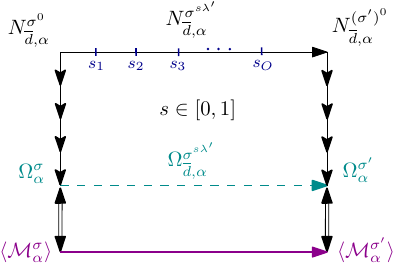}
\end{equation}
Here $s_o$ for $o\in \{1,\ldots,O\}$ are the points in $[0,1]$  where there are strictly  $\sig^{s\lambda'}_{\mu}$-semistable objects of class $(\un{d},\al)$. Crossing these walls leads to the formula \eqref{eq:gammaFlagWC}, which needs to be proved. That there are finitely many such $s_o$ was proved in \cite[Proposition 10.16]{JoyceWC}. There it is also shown that the decompositions of $(\un{d},\al)$ into classes of the same phase contain exactly two terms $(\un{d}^j_1,\al^j_1)$ and $(\un{d}^j_2,\al^j_2)$. The label $j$ distinguishes different such pairs and  takes values in finite sets $J_o$ for all $o$. Each $\un{d}^j_i$ satisfies \eqref{eq:ovdd0} with the last maximal coefficient equal to $\chi\big(\al ^j_i(k)\big)!$.

The computations in §\ref{sec:WCflags} shows that the wall-crossing formula takes the explicit form presented in Theorem \ref{thm:FlagWC} in terms of the natural vertex algebra on $H_*\big(\mN_{\Flag_k}\big)$ and its associated Lie algebra (see Definition \ref{Def:FlagVA}). 
\item \textbf{Conclude wall-crossing in $\mA$ in the neighborhood of each $t_a$}
Continuing with the above set up, one now projects down to formulae in $L_{\loc,*}$ along the vertical arrows $\to\!\to\!\to\!\to$ of the graphic. This will show that  \eqref{eq:MasiWC} holds for $\sig,\sig'$ from \eqref{eq:sigprclosetosig}.

Let $(\un{d},\al)$ be as in step 2, and recall the projection
$$\begin{tikzcd}\pi^{\sig}_{\un{d},\al}: N^{\sig^{s\lambda'}}_{\un{d},\al}\arrow[r]& \mM^{\rig}_{\al}\end{tikzcd}\,.$$ 
Starting from the classes
$$\Big[N^{\sig^{s\lam'}}_{\un{d},\al}\Big]^{\vir}\in H_*\Big(\mN^{\rig}_{\un{d},\al}\Big)$$
whenever there are no strictly semistables, one constructs 
\begin{equation}
\label{eq:Omega-slambda}
\chi(\al(k))!\cdot \Omega^{\sigma^{s\lambda'}}_{\un{d},\alpha}:=\big(\pi^{\sig}_{\un{d},\al}\big)_*\bigg(\Big[N^{\sig^{s\lam'}}_{\un{d},\al}\Big]^{\vir}\cap c_{\Rk}\big(T_{\pi^{\sig}_{\un{d},\al}}\big)\bigg)\,.
\end{equation}
When
\begin{itemize}
\item $s=0$, these classes satisfy $\Omega^{\sigma^{0\cdot\lambda'}}_{\un{d},\alpha}=\Omega^{\sigma,k}_{\alpha}$,
\item $s=1$, they satisfy $\Omega^{\sigma^{\lambda'}}_{\un{d},\alpha}=\Omega^{\sigma',k}_{\alpha}$
\end{itemize}
where the right hand side is defined by \eqref{eq:Omegasigal} in both cases. 

Since $T_{\pi^{\rig}_{\un{d},\al}}$ is a vector bundle that describes (half) of the difference between the obstruction theories of $\mM^{\rig}_{\al}$ and $\mN^{\rig}_{\un{d},\al}$, the results in \cite[§2.5]{GJT} imply that 
$$
\begin{tikzcd}
\big(\pi^{\rig}_{\Flag_k}\big)_*\Big(-\cap c_{\Rk}\big(T_{\pi^{\rig}_{\Flag_k}} \big)\Big): H_*\big(\mN^{\rig}_{\Flag_k}\big)\arrow[r]& H_*\big(\mM^{\rig}_{\mA}\big)
\end{tikzcd}
$$
induces a morphism of Lie algebras. The main conlusion of step 2 stated in Theorem \ref{thm:FlagWC} is transformed under the above morphism into Proposition \ref{prop:OmegaWC} below. Due to the proof of the results in \cite[§2.5]{GJT} not being publicly available yet and the unofficial version being complicated, I present my version of the proof in §\ref{sec:Flagtosheaves} paying special attention to the correctness of signs.
\begin{proposition}
\label{prop:OmegaWC}
Fix an $o\in \{1,\ldots,O\}$ and $s^-_o\in (s_{o-1},s_o), s^+_o\in (s_o,s_{o+1})$, then 
\begin{equation}
\label{eq:OmegaWC}
\Omega^{\sigma^{s^+_{o}\lambda}}_{\un{d},\alpha} - \Omega^{\sigma^{s^-_o\lambda}}_{\un{d},\alpha} = \mathlarger{\mathlarger{\sum}}_{j\in J_o}{\chi(\al(k))\choose \chi(\al^j_1(k))}^{-1}\Big[\Omega^{\sigma^{s_o\lambda}}_{\un{d}_1^j,\alpha^j_1},\Omega^{\sigma^{s_o\lambda}}_{\un{d}^j_2,\alpha^j_2}\Big]\,.
\end{equation}
holds in $L_{\loc,*}$.
\end{proposition}
	Applying this for all $o$ proves a wall-crossing formula between the classes $\Omega^{\sig,k}_{\al}$ and $\Omega^{\sig',k}_{\al}$. The relation \eqref{eq:Masigdef} establishes a one-to-one correspondence between the sets of invariants $\langle \mM^{\sig}_{\al} \rangle^k$ and $\Omega^{\sig}_{\al,k}$ as represented by the vertical arrows $\iff$ in the graphic \eqref{eq:proofWCgraphic}. Consequently, Proposition \ref{prop:OmegaWC} can be used to deduce a concrete relation between $\langle \mM^{\sig}_{\al}\rangle^k$ and $\langle \mM^{\sig'}_{\al}\rangle^k$. The combinatorics behind the conversion into \eqref{eq:MasiWC}  are the focus of \cite[§10.5]{JoyceWC}. The conclusion depends solely on the input provided by \eqref{eq:OmegaWC}.
    \item \textbf{Piece together the wall-crossing formulae from step 2} 
   Using that wall-crossing formulae \eqref{eq:MasiWC} compose and can be inverted, one assembles the segments addressed in step 3 into the complete interval $[0,1]$. This uses \cite[Lemma 11.5]{JoyceWC} which proves for any $a\in A$ that 
   \begin{itemize}
   \item wall-crossing \eqref{eq:MasiWC} for $\al\in \msE(\mA)$ holds between $\sig=\gamma_0$ and $\sig'=\gamma_{t}$ for $t\in (t_{a-1},t_{a+1})$ if and only if this is true for $t=t_a$. 
   \end{itemize}
  While the present version of the finiteness assumption is slightly weaker than the one in \cite[Assumption 5.3]{JoyceWC}, it is sufficient for the argument.
   
   In other words, one first inverts the conclusion of step 3 to include the segment \Pur{$(-)$} and then composes with the wall-crossing for the segment \Pur{$(+)$}. After finitely many steps, one recovers the interval $[0,1]$.
   $$\includegraphics[scale=1.25]{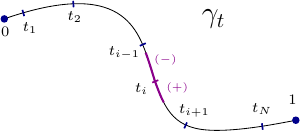}
   $$
    \end{enumerate}
    \begin{remark}
        Note that the only geometric wall-crossing happens in step 2) which takes place inside of $\mB^{\sig}_{\Flag_k,\phi(\al)}$. When it comes to wall-crossing for stable pairs with fixed-determinant obstruction theories, Assumption \ref{ass:pairWC} was set up in such a way that the corresponding replacement category $\mbB^{\sig,\phi}_{O,I_k}$ does not distinguish between the left and the right-hand side of \eqref{eq:equivalence-cat-pairs}. In particular, one can still wall-cross with the stabilities $\sig$ in the heart $\mB\subset D^b(X)$ and use the fixed-determinant obstruction theories, but one also has the necessary moduli spaces of flags constructed using \eqref{eq:pair-framing}. In particular, the same arguments as above still apply directly.
    \end{remark}
\subsection{Orientations and vertex algebras for $\mN_{\Flag_k}$}
\label{sec:orVAFlag}
One of the major differences compared to \cite{JoyceWC} is the flexibility in choosing orientations on moduli spaces. In the computation of the wall-crossing formula in §\ref{sec:WCflags}, one therefore needs to pay attention to signs comparing orientations on different connected components of $\mN_{\Flag_k}$ under taking direct sums. This subsection describes the correct choices of signs fixed by Definition \ref{def:orpullback}. Later it will be shown that they are appearing in the wall-crossing formulae for $N^{\sig}_{\un{d},\al}$ that are stated in terms of a natural vertex algebra structure on $H_*\big(\mN_{\un{d},\al}\big)$. Towards the end of the subsection, I will describe this vertex algebra and how the new signs appear in its definition.

By \eqref{eq:Ndalobstruction}, I know that $\FF^{\sig}_{\un{d},\al}$ is constructed from $\EE$ by applying \eqref{eq:pairtosheaf} and then \eqref{eq:flagobsdiag}. Lemma \ref{lem:functoror} guarantees that the orientations induced by Definition \ref{def:orpullback} are the same. While Assumption \ref{ass:obsonflag} implies only existence of $\FF^{\sig}_{\un{d},\al}$ on $\mN^{\sig}_{\un{d},\al}$, below I will argue as if I was given $\FF_{\Flag}$ on the full $\mN_{\Flag,k}$. Everything still makes sense as long as one restricts back to $\mN^{\sig}_{\un{d},\al}$ for $\al\in \msE(\mA)$.

Set the notation
 $$
M_{\Flag}:=\det\big(\LL_{\pi_{\Flag}}\big)
$$
and use the isomorphism $\det\big(\LL_{\pi_{\Flag}}^\vee[2]\big)\cong M^*_{\Flag}$ without mentioning it. Recall that the induced orientation of $F_{\Flag} := \det(\FF_{\Flag})$ fixed in Definition \ref{def:orpullback} uses the natural isomorphism 
 \begin{equation}
 \label{eq:FtoDiso}
 \begin{tikzcd}
 F_{\Flag}\cong M_{\Flag}^*\cdot
M_{\Flag}\cdot D\arrow[r,"o^{-1}_{\LL_{\pi_{\Flag}}}\otimes \id"]&[1.5cm]D
\end{tikzcd}
 \end{equation}
 and the trivialization of the line bundle \eqref{eq:LDdetlines} on the right.
 
Applying $\det(-)$ to \eqref{eq:additivityofFFflag} induces the isomorphism 
$$
\delta_{\Flag}: \begin{tikzcd}
    \mu_{\Flag_k}^* \big(F_{\Flag}\big) \arrow[r]&(F_{\Flag}\boxtimes F_{\Flag})\cdot N_{\Flag}\cdot \sigma^* N_{\Flag}\,.
\end{tikzcd}
$$
where I used
$$
N_{\Flag} := \det\Big(\Theta_{\mN_{\Flag_k}}\Big)\,.
$$
From \eqref{eq:extendTheta}, one sees that $\sigma^*\Theta_{\Flag_k}\cong \Theta_{\Flag_k}^\vee[2]$, so the last two factors on the right can be given an orientation $(-1)^{\Rk(\Theta_{\Flag})}o_{\Theta_{\Flag}}$. 
\begin{definition}
\label{def:varepsflag}
  Choose $(\un{d}_i,\alpha_i)\in \ov{K}(\mB_{\Flag_k})$ for $i=1,2$ and let $(\un{d},\al)$ be their sum. Let $o_{\un{d}_i,\alpha_i}$ be the orientations \eqref{eq:FtoDiso} of $\mN_{\un{d}_i,\al_i}$ for $i=1,2$, then they induce an orientation of $\mN_{\un{d},\al}$ as the composition of the consecutive morphisms
    \begin{equation}
    \label{eq:inducedNQkor}
    \begin{tikzcd}
        \un{\CC}\arrow[r,"o_{\un{d}_1,\alpha_1}\boxtimes o_{\un{d}_2,\alpha_2}"]&[1.3cm] F_{\un{d}_1,\alpha_1}\boxtimes F_{\un{d}_2,\alpha_2}\arrow[r,"\id\otimes o_{\sig^*\Theta_{\Flag}}"]&[0.7cm] F_{\un{d}_1,\alpha_1}\boxtimes F_{\un{d}_2,\alpha_2}\cdot N_{\Flag}\cdot \sigma^*N_{\Flag}\arrow[r,"\delta_{\Flag}^{-1}"]&[-0.2cm]\mu_{\Flag_k}^*F_{\un{d},\al}\,.
    \end{tikzcd}
    \end{equation}
    By writing $N_{\Flag}$ and $\sigma^*N_{\Flag}$, I mean their appropriate restrictions to $\mN_{\un{d}_1,\alpha_1}\times\mN_{\un{d}_2,\alpha_2}$. The sign 
    $
    \varepsilon^{\un{d_1},\alpha_1}_{\un{d_2},\alpha_2}
    $ is defined as the difference between the orientations $\mu^*_{\Flag_k}(o_{\un{d},\al})$ and \eqref{eq:inducedNQkor}.
\end{definition}
The next lemma shows that the signs are well-defined due to Assumption \ref{ass:orientation} and that they can be expressed in terms of $\varepsilon_{\al_1,\al_2}$.
\begin{lemma}
\label{lem:epsilonQFlag}
Let 
$$
\xi\big((\un{d}_1,\alpha_1),(\un{d}_2,\alpha_2)\big):= \Rk\Big(\Theta_{\mN_{\Flag_k}/\mM}|_{\mN_{\un{d}_1,\alpha_1}\times\mN_{\un{d}_2,\alpha_2}}\Big)\,,
$$
which is constant if $\al_1,\al_2\in \msE(\mA)$. The diagram 
\begin{equation}
\label{eq:sumFtosumD}
\begin{tikzcd}[column sep = 3cm, row sep = large]
\mu_{\Flag_k}^*\big(F_{\un{d},\al}\big)\arrow[d,"\Big(\id\otimes o^{-1}_{\sig^*\Theta_{\Flag}}\Big)\circ \delta_{\Flag}"']\arrow[r,"\eqref{eq:FtoDiso}"]&\mu^*D_{\alpha}\arrow[d,"\Big(\id\otimes o^{-1}_{\sig^*\Theta}\Big)\circ \delta_{\alpha_1,\alpha_2}"]\\
F_{\un{d}_1,\alpha_1}\boxtimes F_{\un{d}_2,\alpha_2}\arrow[r,"\eqref{eq:FtoDiso}\otimes \eqref{eq:FtoDiso}"]&D_{\alpha_1}\boxtimes D_{\alpha_2}
\end{tikzcd}
\end{equation}
commutes up to the sign $(-1)^{\xi\big((\un{d}_1,\alpha_1),(\un{d}_2,\alpha_2)\big)}$. This implies that 
\begin{equation}
\label{eq:epscomp}
 \varepsilon^{\un{d_1},\alpha_1}_{\un{d_2},\alpha_2} = (-1)^{\xi\big((\un{d}_1,\alpha_1),(\un{d}_2,\alpha_2)\big)}\varepsilon_{\alpha_1,\alpha_2}\,.
\end{equation}
\end{lemma}
\begin{proof}
    The proof of this is a direct application of the conventions in §\ref{sec:orconventions}. After expanding each of the isomorphisms in \eqref{eq:sumFtosumD} carefully, one notices that the only difference between the upper right path $\gamma_R$ from $\mu^*_{\Flag_k}F_{\un{d},\al}$ to $D_{\alpha_1}\boxtimes D_{\alpha_2}$ in \eqref{eq:sumFtosumD} and the lower left path $\gamma_L$ is in the choice of the trivialization of
    $$
    \begin{tikzcd}N_{\Flag}\cdot \sigma^*N_{\Flag}\arrow[r]&\un{\CC}\end{tikzcd}\,.
    $$
    \begin{enumerate}[label=\roman*)]
        \item For $\gamma_R$, it amounts to first using the isomorphism 
        \begin{equation}
        \label{eq:NQfactors}
       \begin{tikzcd}
           N_{\Flag}\arrow[r,"\sim"]& L_{\alpha_1,\alpha_2} \det\big(\Theta_{\mN_{\Flag_k}/\mM}\big)\det\big(\sigma^*\Theta_{\mN_{\Flag_k}/\mM}^\vee[2]\big)
       \end{tikzcd} 
        \end{equation}
        that follows from \eqref{eq:extendTheta}. Doing so for $\sigma^*N_{\Flag}$ and collecting the factors of the two terms, the trivialization is obtained by first applying 
        \begin{equation}
        \label{eq:detThetaMNQk}
\begin{tikzcd}
\det\big(\Theta_{\mN_{\Flag_k}/\mM}\big)\det\big(\Theta_{\mN_{\Flag_k}/\mM}^\vee[2]\big)\arrow[r,"{o^{-1}_{\Theta_{\mN_{\Flag,k}/\mM}}}"]&[1cm]\un{\CC}
\end{tikzcd}
        \end{equation}
        and $\sigma^*$ of it. The remaining two terms are trivialized via
        \begin{equation}
        \label{eq:Lal1al2}
          \begin{tikzcd}
        L_{\alpha_1,\alpha_2}\otimes \sigma^* L_{\alpha_2,\alpha_1}\arrow[r,"o^{-1}_{\sig^*\Theta}"]&[0.4cm]\un{\CC}\,.
        \end{tikzcd}   
        \end{equation}
        \item For $\gamma_L$, the trivialization is given directly by $o^{-1}_{\sig^*\Theta_{\Flag}}$ as follows from \eqref{eq:inducedNQkor}. To compare with the above point, I determine how this isomorphism acts on each factor after applying \eqref{eq:NQfactors}. Firstly, \eqref{eq:Lal1al2} remains the same. The only difference is that $\gamma_L$ uses
        $$
\begin{tikzcd}\det\big(\Theta_{\mN_{\Flag_k}/\mM}^\vee[2]\big)\det\big(\Theta_{\mN_{\Flag_k}/\mM}\big)\arrow[r,"o^{-1}_{\Theta^\vee_{\mN_{\Flag,k}/\mM}[2]}"]&[1cm]\un{\CC}\end{tikzcd}
        $$
        because the other pair of terms is cancelled the same way as in $\gamma_R$ by applying $\sigma^*$ to \eqref{eq:detThetaMNQk}. This introduces the sign $(-1)^{\xi\big((\un{d}_1,\alpha_1),(\un{d}_2,\alpha_2)\big)}$ when comparing $\gamma_R$ and $\gamma_L$. 
    \end{enumerate}
 The induced orientation $o_{\alpha_1,\alpha_2}$ of $D_{\alpha}$ is obtained from $o_{\alpha_1}\boxtimes o_{\alpha_2}$ by using the right vertical arrow. The signs $\varepsilon_{\alpha_1,\alpha_2}$ were fixed in Assumption \ref{ass:orientation} by comparing $o_{\alpha_1,\alpha_2}$ with $o_{\alpha_1+\alpha_2}$. Combined with the sign that makes \eqref{eq:sumFtosumD} commute, the comparison of $\varepsilon$'s is concluded, since the orientation of $F_{\Flag}$ is determined by the horizontal arrows.
\end{proof}
Before I describe the explicit vertex algebra structure on the shifted homology of $\mN_{\Flag_k}$, I set the notation
$$
\begin{tikzcd}
\rho_{\Flag_k}:B\GG_m\times \mN_{\Flag_k}\arrow[r]&\mN_{\Flag_k}
\end{tikzcd}
$$
for the $B\GG_m$-action rescaling automorphisms of each object.
\begin{definition}
\label{Def:FlagVA}
For each $(\un{d}_1,\al_1), (\un{d}_2,\al_2)\in \ov{K}\big(\mB_{\Flag_k}\big)$, consider the restriction 
$$
\Theta_{(\un{d}_1,\al_1),(\un{d}_2,\al_2)}:=\Theta_{\Flag_k}\big|_{\mN_{\un{d}_1,\al_1}\times \mN_{\un{d}_2,\al_2}}
$$
of the complex constructed in \eqref{eq:extendTheta}. Set
\begin{align*}
\chi_{\Flag_k}\big((\un{d}_1,\al_1), (\un{d}_2,\al_2)\big):=&-\Rk\Big(\Theta_{(\un{d}_1,\al_1),(\un{d}_2,\al_2)}\Big) \\
=&\chi(\al_1,\al_2) + \xi\big((\un{d}_1,\al_1),(\un{d}_2,\al_2)\big)+\xi\big((\un{d}_2,\al_2),(\un{d}_1,\al_1)\big)\,.
\end{align*} 
Continuing to use $(-)_{*+\vdim}$ for the shift of degrees by the rank of the obstruction theory, the underlying graded vector space of the vertex algebra will be
 $$V\iFlag_*:= H_{*+\vdim}\big(\mN_{\Flag_k}\big)\,.$$

The vacuum vector $\ket{0}$ and the translation operator are defined in terms of the 0-object and the action $\rho_{\Flag_k}$, respectively, in the same way as in Definition \ref{Def:VAlocal} or \ref{Def:VAglobal}. The state-field correspondence takes the form
\begin{align*}
    Y_{\Flag_k}(v,z)w &= (-1)^{a\chi_{\Flag_k}((\ov{e},\be),(\ov{e},\beta))}\varepsilon^{\un{d},\al}_{\ov{e},\be} \,(\mu_{\Flag_k})_\ast\bigg((e^{zT}\otimes \textnormal{id})\frac{v\boxtimes w}{z^{\Rk(\Theta_{\Flag_k})}c_{z^{-1}}(\Theta_{\Flag_k})}\bigg)
    \end{align*}
   for $v\in H_a(\mN_{\un{d},\al})$ and $w\in H_*(\mN_{\ov{e},\be})$.
\end{definition}
This can be defined $\T$-equivariantly as in Definition \ref{Def:VAlocal} or Definition \ref{Def:VAglobal} by using the equivariant Künneth morphism $\boxtimes^T$ from \eqref{eq:Kunnethlocal} or \eqref{eq:equivKunneth}, respectively. In particular, one gets the  Lie algebra 
$$
L\iFlag_*:= V\iFlag_{*+2}/T\big(V\iFlag_*\big)
$$
and its localized version 
$
L\iFlag_{\,\loc,*}.
$
The difference between the vertex algebra $V\iFlag_*$ and its associated Lie algebra is milder than in the case of $\mM_{\mA}$ because the $B\GG_m$ torsor $\Pi_{\un{d},\al}:\mN_{\un{d},\al}\to \mN^{\rig}_{\un{d},\al}$ is trivial whenever $d_v = 1$ for any $v\in \mathring{\Ver}$. To construct a natural section, fix $\un{d}$ such that $d_1=1$. Then, there is a universal object of $\mN^{\rig}_{\un{d},\al}$ constructed by letting the weight-zero expressions
$$
\mV^{\rig}_{i} = \mV^*_1\otimes\mV_i\qquad \textnormal{for}\quad 1\leq i\leq r-1\,,\qquad \mE^{\rig} = \mV_1^*\otimes \mE\,.
$$
descend along $\Pi_{\un{d},\al}$. The morphisms $\mathfrak{m}^e$ need to be tensored by $\id_{\mV^*_1}$. This universal object induces a section
\begin{equation}
\label{eq:srig}
s_{\un{d},\al}:\begin{tikzcd}\mN^{\rig}_{\un{d},\al}\arrow[r]&\mN_{\un{d},\al}\end{tikzcd}
\end{equation}
of $\Pi_{\un{d},\al}$.
\subsection{Wall-crossing for flags}
\label{sec:WCflags}
In this section, I will apply equivariant localization to the virtual fundamental classes $\Big[N^{\sig}_{(1,\un{d}),\al}\Big]^{\vir}$ when $I=I_{\MS}$. This will produce wall-crossing formulae in the category $\mB_{\Flag_k}$. After projecting to the moduli stack $\mM_{\mA}$, this ends up proving Theorem \ref{thm:familyWC}.  For usual perfect obstruction theories, this localization computation was carried out in \cite{JoyceWC}. In the present situation, one needs to pay attention to signs. For these reasons, I set all the conventions in §\ref{sec:orconventions}, Theorem \ref{thm:eqvirloc}, Definition \ref{def:orpullback}, §\ref{sec:assab}, and §\ref{sec:orVAFlag} in exactly the way that will be used in the computation below. Without correcting the Oh--Thomas localization formula in Theorem \ref{thm:eqvirloc}, vertex algebras would not have emerged from the wall-crossing. Lastly, this section considers only the case when $\T = \{1\}$. The equivariant refinement is discussed in §\ref{sec:equivariantWC}.

Unlike \cite[§10.6]{JoyceWC}, I do not use derived geometry as the computation relies only on the compatibilities of obstruction theories from Assumption \ref{ass:obsonflag}. As such, it will translate directly to the computation in the sequel where it is done on a further auxiliary space to prove the general case. 

Fix $\al\in \msE(\mA), k\in K$ and stability conditions $\sig,\sig'$ as in \eqref{eq:sigprclosetosig}. Let $r$ in the definition of $I_{\Flag}$ be equal to $\chi\big(\al(k)\big)+1$ and fix the dimension vector $\un{d}$ of $\mathring{I}_{\Flag}$ as in step 2) of §\ref{sec:summaryproof}. With it, I will associate the dimension vector $(1,\un{d})$ of $\mathring{I}_{\MS}$ in the case that this quiver is obtained by adding $v_0$ to the above choice of $\mathring{I}_{\Flag}$. Using $\lambda'$ as in §\ref{sec:summaryproof}, there is a virtual fundamental class $\Big[N^{\sig^{\lambda'}}_{(1,\un{d}),\al}\Big]^{\vir}
$ by Assumption \ref{ass:obsonflag}. For $o\in \{1,\ldots,O\}$ and $s_o\in (0,1)$ from \eqref{eq:proofWCgraphic}, consider the  $\GG_m$-action on the moduli space $N^{\sig^{s_o\lam'}}_{(1,\un{d}),\al}$ determined by rescaling the morphisms of representations at the edge $e_{-1}$:
\begin{equation}
\label{eq:QMSaction}
\includegraphics[scale=0.24]{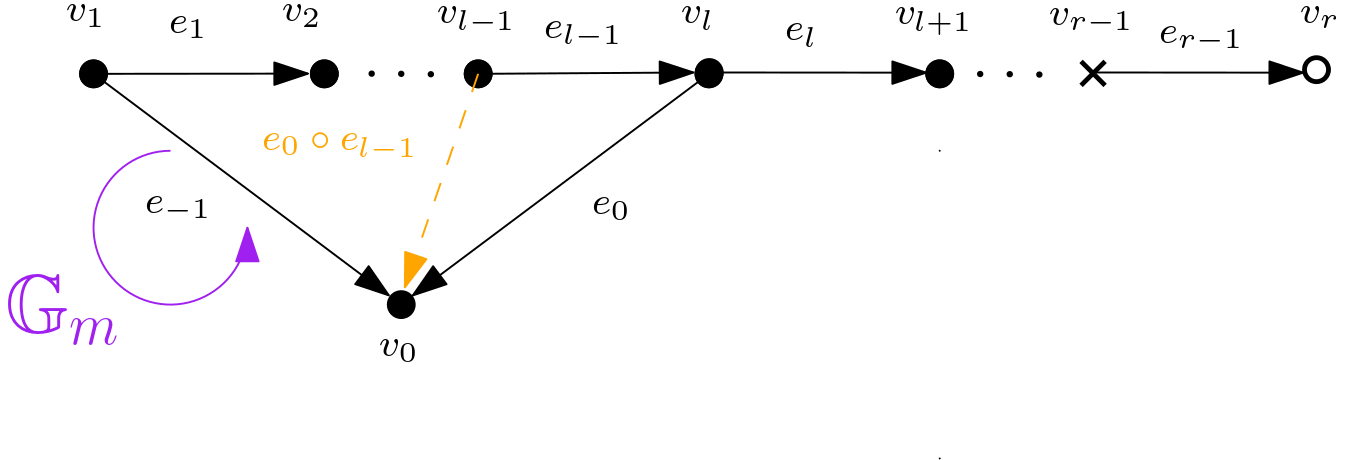}
\end{equation}
I will denote its weight 1 equivariant trivial line bundle by $w=e^z$ with the equivariant first Chern class $c_1(t) = z$.  

The next proposition describes the fixed point loci of this action, their virtual fundamental classes and the virtual normal bundles. The first step is already proved in \cite[Proposition 10.20]{JoyceWC}. I recall the arguments in a more pictorial way using quiver diagrams as this is needed for the explicit description of virtual normal bundles. The pictures of quivers appearing in the statement of the proposition below include the effects of the \Pur{$\GG_m$-action} on the vertices which are explained in the proof and should be ignored on the first read. Because the $\GG_m$-action commutes with the induced $\T$-action on $\mN_{\MS_k}$ from Definition \ref{Def:quiverpairs}, everything below works also $\T$-equivariantly. 
\begin{proposition}[{\cite[Proposition 10.20]{JoyceWC}}]
\label{prop:fixedpoints}
For the above data, let $(\un{d}^j_i,\al^j_i)\in \ZZ^{\chi(\al(k))}\times \msE(\mA)$ for $j\in J_o$ and $i=1,2$ be the  classes that $(\un{d},\al)$ splits into at $s=s_o$ as recalled in step 3) of §\ref{sec:summaryproof}.  The fixed points locus $\Big(N_{(1,\un{d}),\al}^{\sigma^{s_o\lambda'}}\Big)^{\GG_m}$ can be expressed as a disjoint union of the following three types of subsets:
\begin{enumerate}
\item The subscheme $$N_{(1,\un{d}),\alpha}^{\sigma^{s_o\lambda'}}\Big|_{e_{-1}=0}\subset N_{(1,\un{d}),\alpha}^{\sigma^{s_o\lambda'}}$$ consisting of objects $(\un{V},\un{m}, E)$ such that $m^{e_{-1}} = 0$.
\begin{equation}
\label{eq:e-1locus}
  \includegraphics[scale = 0.4]{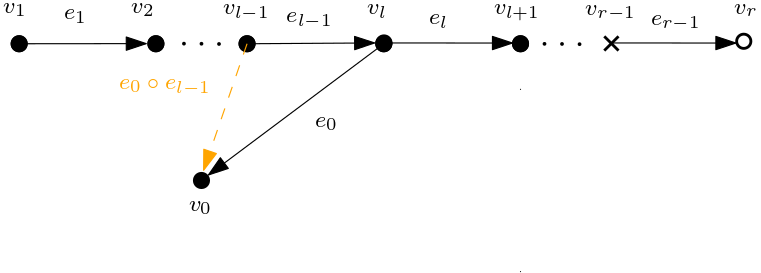}  
\end{equation}
It is canonically isomorphic to $N_{\un{d},\al}^{\sigma^{s^-_{o}\lambda'}}$ for $s^-_{o}\in (s_{o-1}, s_o)$, and its virtual fundamental class is identified up to a sign with 
\begin{equation}
\label{eq:FPvir1}
\Big[N_{\un{d},\al}^{\sigma^{s^-_{o}\lambda'}}\Big]^{\vir}\in H_*\Big(N_{\un{d},\al}^{\sigma^{s^-_{o}\lambda'}}\Big)
\end{equation}
under this isomorphism. Its virtual normal bundle is 
\begin{equation}
\label{eq:Nvir1}
N^{\vir}_{e_{-1}=0}=t \cdot \mV^*_1\otimes\mV_{0} \oplus t^{-1}\cdot \mV_0^*\otimes \mV_1[-2]\,.
\end{equation}
\item The subscheme $$N_{(1,\un{d}),\alpha}^{\sigma^{s_o\lambda'}}\Big|_{e_{0}=0}\subset N_{\alpha,(1,\un{d})}^{\sigma^{s_o\lambda'}}$$ consisting of objects $(\un{V},\un{m}, E)$ such that $m^{e_{0}} = 0$.
\begin{equation}
\label{eq:e0locus}
  \includegraphics[scale = 0.23]{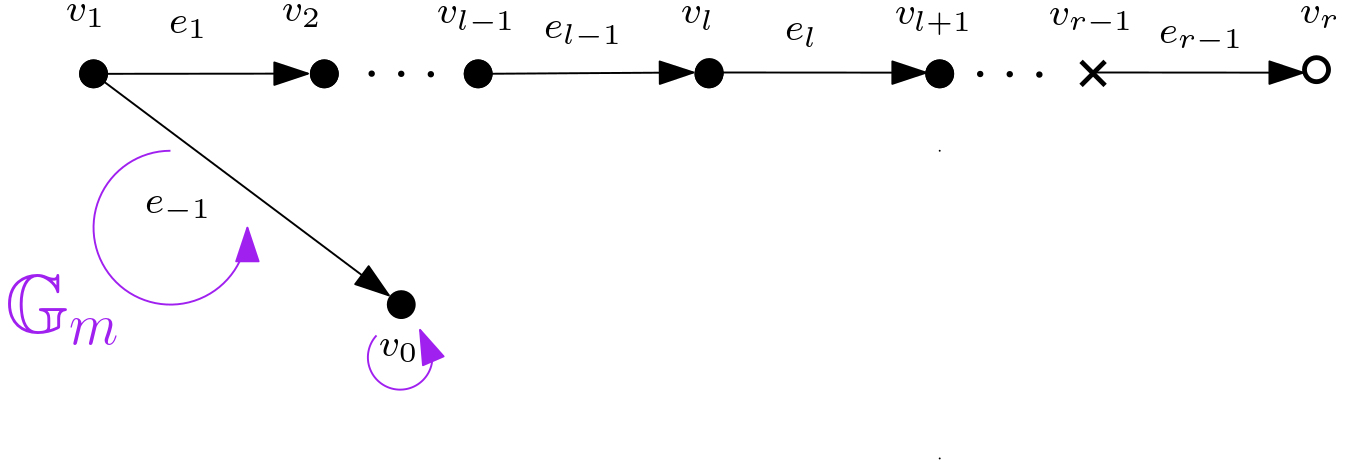}  
\end{equation}
It is canonically isomorphic to $N_{\un{d},\al}^{\sigma^{s^+_o\lambda'}}$ for $s^+_o\in (s_o,s_{o+1})$, and its virtual fundamental class is identified up to a sign with 
\begin{equation}
\label{eq:FPvir2}
\Big[N_{\un{d},\al}^{\sigma^{s^+_o\lambda'}}\Big]^{\vir}\in H_*\Big(N_{\un{d},\al}^{\sigma^{s^+_o\lambda'}}\Big)
\end{equation}
under this isomorphism. Its virtual normal bundle is 
\begin{equation}
\label{eq:Nvir2}
N^{\vir}_{e_{0}=0}=t^{-1}\cdot \big(\mV_{l}/\mV_{l-1}\big)^*\otimes \mV_0\oplus t\cdot \mV_0^*\otimes \big(\mV_l/\mV_{l-1}\big)[-2]\,.
\end{equation}
\item I fix $j\in J_o$ and omit it from the superscripts when specifying the splitting $(\un{d_i},\alpha_i)$ for $i=1,2$ discussed above to improve readability. The associated fixed point locus
$$N_{(1,\un{d}),\alpha}^{\sigma^{s_o\lambda'}}\Big|^{(\un{d}_1,\alpha_1)}_{(\un{d}_2,\alpha_2)}\subset N_{(1,\un{d}),\alpha}^{\sigma^{s_o\lambda'}}$$
contributes for each $j$ and it was described in \cite[Proposition 10.16]{JoyceWC}. This locus is represented by the following quiver:
\begin{equation}
    \label{eq:sumlocus}
    \includegraphics[scale = 1]{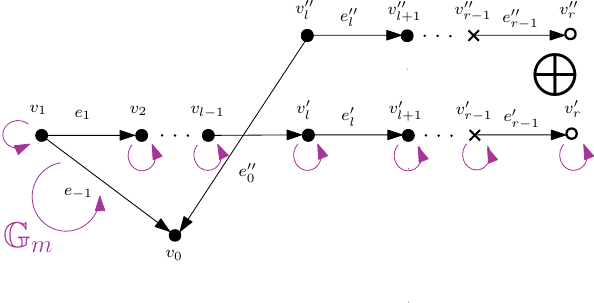}
\end{equation}
The map 
$$
\mu^{(\un{d}_1,\alpha_1)}_{(\un{d}_2,\alpha_2)}: \begin{tikzcd}
    \mN_{\un{d}_1,\alpha_1}^{\rig}\times \mN_{\un{d}_2,\alpha_2}^{\rig}\arrow[r]& \mN^{\rig}_{(1,\un{d}),\alpha}
\end{tikzcd}
$$
constructed in \eqref{eq:sumtoMSmap} induces an isomorphism 
\begin{equation}
\label{eq:Sigmatofixedsumlocus}
\mu^{(\un{d}_1,\alpha_1)}_{(\un{d}_2,\alpha_2)}: \begin{tikzcd}N_{\un{d}_1,\alpha_1}^{\sigma^{s_o\lambda'}}\times N_{\un{d}_2,\alpha_2}^{\sigma^{s_o\lambda'}}\arrow[r,"\sim"]&N_{(1,\un{d}),\alpha}^{\sigma^{s_o\lambda'}}\Big|^{(\un{d}_1,\alpha_1)}_{(\un{d}_2,\alpha_2)} \end{tikzcd}
\end{equation}
which identifies the virtual fundamental class of the latter with 
\begin{equation}
\label{eq:FPvir3}
\Big[N_{\un{d}_1,\alpha_1}^{\sigma^{s_o\lambda'}}\Big]^{\vir}\boxtimes \Big[N_{\un{d}_2,\alpha_2}^{\sigma^{s_o\lambda'}}\Big]^{\vir}\in H_*\Big(N_{\un{d}_1,\alpha_1}^{\sigma^{s_o\lambda'}}\times N_{\un{d}_2,\alpha_2}^{\sigma^{s_o\lambda'}}\Big)
\end{equation}
up to a sign. Pulling back the virtual normal bundle $\Big(N^{(\un{d}_1,\alpha_1)}_{(\un{d}_2,\alpha_2)}\Big)^{\vir}$ along $\mu^{(\un{d}_1,\alpha_1)}_{(\un{d}_2,\alpha_2)}$ yields
\begin{equation}
\label{eq:Nvirmiddle}
\Big(\mu^{(\un{d}_1,\alpha_1)}_{(\un{d}_2,\alpha_2)}\Big)^*\Big(N^{(\un{d}_1,\alpha_1)}_{(\un{d}_2,\alpha_2)}\Big)^{\vir} = t\cdot\Theta_{Q_{\Flag_k}}[-2]\oplus t^{-1}\cdot \sigma^*\Theta_{Q_{\Flag_k}}[-2]\,,
\end{equation}
where, by writing $\Theta_{\mN_{Q_{\Flag_k}}}$, I understand its restriction to $N_{\un{d}_1,\alpha_1}^{\sigma^{s_o\lambda'}}\times N_{\un{d}_2,\alpha_2}^{\sigma^{s_o\lambda'}}$.
\end{enumerate}    
\end{proposition}
\begin{proof}
It was shown in \cite[Proposition 10.20]{JoyceWC}, that these are the only types of fixed point loci one obtains. The rest of the proof will describe how they appear and what their fixed obstruction theories and virtual normal bundles are. This must be done carefully because it will affect the signs in Theorem \ref{thm:FlagWC} below. 
\begin{enumerate}
    \item The first fixed-point locus is clear, as it sets the morphisms rescaled by $\GG_m$ to zero. In \cite[Proposition 10.20 (a)]{JoyceWC}, it is shown that $\sigma^{s_o\lambda'}$-semistability of $(\un{V},\un{f},E)$ implies injectivity of $m^{e_{l-1}}:V_{l-1}\to V_l$. On this fixed point locus, the morphism $m_0$ additionally factors through the isomorphism 
    $$
    \begin{tikzcd}
        \coker\big({m^{e_{l-1}}}\big)\arrow[r,"\sim"]&V_0\,.
    \end{tikzcd}
    $$
    Therefore, $m_0$ presents no additional information, and one can naturally identify $N_{(1,\un{d}),\alpha}^{\sigma^{s_o\lambda'}}\Big|_{e_{-1}=0}$ with a substack of $\mN^{\rig}_{\un{d},\al}$. Comparing stability conditions, \cite[Proposition 10.20 (a)]{JoyceWC} proves the identification with $N_{\un{d},\al}^{\sigma^{s^-_{o}\lambda'}}$.

    The obstruction theory on $\mN^{\rig}_{(1, \un{d}),\alpha}$ is explicitly described by applying the description in Lemma  \eqref{eq:fullobsQMS} to \eqref{eq:constructQMSCY4} where \FG{$\circ$} now more generally represents the obstruction theory $\EE$ on $\mM_{\mA}$. From this, it becomes clear that the moving part of the tangent complex  restricted to $N_{(1,\un{d}),\alpha}^{\sigma^{s_o\lambda'}}\Big|_{e_{-1}=0}$ is represented by 
    $$
    \begin{tikzcd}
        \arrow[dr,"e_{-1}"]&\\
        &\arrow[ul, Maroon, bend left, "e^*_{-1}"]
    \end{tikzcd}\,.
    $$
The weights of $\begin{tikzcd}\,
        \arrow[r,"e_{-1}"]&\,\end{tikzcd}$ and $\begin{tikzcd}\,&\arrow[l, Maroon ,"e^*_{-1}"']\,
        \end{tikzcd}$ are $t$
 and $t^{-1}$ respectively, which is equivalent to \eqref{eq:Nvir1}.
        
        The left-over fixed part of the obstruction theory has an extra
        $$
        \begin{tikzcd}[column sep=small]&[-10pt]\,\arrow[dl, BurntOrange,"\rho"']&\,&\,\arrow[dlll,"e_0"]\\
\stackrel{\Cya{\bullet}}{v_0} &&\end{tikzcd}
        $$
        and its shifted dual when compared to $\FF^{\rig}_{\un{d},\al}$. However, this term is simply 
        $$
\begin{tikzcd}
\Big(\Cya{\mV_0^*\otimes \mV_0}\arrow[r,"\circ m^{e_{0}}"]&\mV_l^*\otimes \mV_0\arrow[r,"\circ m^{e_{l-1}}"]&\BuOr{\mV_{l-1}^*\otimes \mV_0}\Big)^\vee
\end{tikzcd}
        $$
which is acyclic as a consequence of the choice of stability. I leave the question of comparing orientations for later, so the above only implies that the virtual fundamental classes are equal up to signs.
\item This situation is slightly more complicated. Consider an object $(\un{V}, \un{f}, E)$ corresponding to a point in $N_{(1,\un{d}),\alpha}^{\sigma^{s_o\lambda'}}\Big|_{e_{0}=0}$. Because this scheme parametrizes such objects up to isomorphisms, one may compensate the action of \Pur{$x\in \GG_m$} on $m^{e_{-1}}$ by scaling the identity of $V_{v_0}$ by \Pur{$x^{-1}$}:
$$
\begin{tikzcd}[column sep= large]
    \arrow[d,"\id_{V_1}"']V_{1}\arrow[r,"\Pur{x}\cdot m^{e_{-1}}"]&V_0\arrow[d,"\Pur{x^{-1}}\cdot \id_{V_0}"]\\
    V_1\arrow[r,"m^{e_{-1}}"']&V_0
\end{tikzcd}\,.
$$
Because $m^{e_0}=0$ this extends to an isomorphism of objects, so $(\un{V}, \un{f}, E)$ indeed represents a fixed point of the $\GG_m$-action. It is shown in \cite[Proposition 10.20 (b)]{JoyceWC} that the stability condition implies that $m^{e_{-1}}$ is an isomorphism. Therefore, the data of $m^{e_{-1}}$ and $V_0$ can be neglected, and $N_{(1,\un{d}),\alpha}^{\sigma^{s_o\lambda}}\Big|_{e_{0}=0}$ can be identified with an open substack of $\mN^{\rig}_{\un{d},\al}$. After comparing stability conditions, \cite[Proposition 10.20 (b)]{JoyceWC} proves that this substack is $N_{\un{d},\al}^{\sigma^{s^+_{o}\lambda'}}$.

From the description of the fixed point locus and from looking at \eqref{eq:fullobsQMS}, it follows that the moving part of the obstruction theory restricted to $N_{(1,\un{d}),\alpha}^{\sigma^{s_o\lambda'}}\Big|_{e_{0}=0}$ is represented by 
$$
\begin{tikzcd}[row sep=huge]&[-10pt]\,\arrow[dl, bend right, BurntOrange, "\rho"']&\,&\,\arrow[llld,"e_0"']\\
\,\arrow[urrr,Maroon, bend right, "e^{*}_0"']\arrow[ur, bend right, BurntOrange, "\rho^*" {yshift=-1ex}]&&&\end{tikzcd}
$$
The weights on its dual come from the action rescaling $\mV_0$, so they are $t$ for the terms corresponding to arrows ending at $v_0$ and $t^{-1}$ for the terms corresponding to arrows starting at $v_0$. Explicitly this adds up to
$$
t\cdot\Big(\begin{tikzcd}[column sep= large]\mV_l^*\otimes \mV_0\arrow[r,"\circ m^{e_{l-1}}"]&\BuOr{\mV_{l-1}^*\otimes \mV_0}\end{tikzcd}\Big)
$$
and its dual. To conclude \eqref{eq:Nvir2}, I simply recall that $m^{e_{l-1}}$ is injective. 

The difference between $\FF^{\rig}_{\un{d},\al}$ restricted to $N_{\un{d},\al}^{\sigma^{s_{o}\lambda'}}$ and the fixed part of $\eqref{eq:fullobsQMS}$ is given by 
$$
\begin{tikzcd}
\arrow[dr,"e_{-1}"]&\\
&\stackrel{\bullet}{v_0}
\end{tikzcd}
$$
and its shiftedí dual. The corresponding complex is clearly acyclic because $m^{e_{-1}}$ is an isomorphism.
\item
First of all, I will introduce a stack $\mN^{(\un{d}_1,\alpha_1)}_{(\un{d}_2,\alpha_2),1}$ that parametrizes objects of the form \eqref{eq:sumlocus}. They consist of a pair of objects $P',P''$ in $\mB_k$ of classes $(\alpha_i,d_{r-1,i})$ for $i=1,2$ respectively together with a representation of the quiver obtained from \eqref{eq:sumlocus} after removing both $\stackrel{v'_r}{\circ}$ and $\stackrel{v''_r}{\circ}$. As before, I require that the vector spaces in the definition of $P'$ and $P''$ in \eqref{eq:exacttriple} are identified with the vector spaces $V'_{r-1}$ at $v'_{r-1}$ and $V''_{r-1}$ at $v''_{r-1}$. Note that I use the convention explained in Remark \ref{rem:d1d2} to shorten the dimension vectors $\un{d}_i$. Lastly, I take the stack $\mN^{(\un{d}_1,\alpha_1)}_{(\un{d}_2,\alpha_2),1}$ to be the rigidified one, meaning that the objects it parametrizes carry a fixed isomorphism $V_1\cong \CC$. 

There is then a natural map 
$$
\Lambda^{(\un{d}_1,\alpha_1)}_{(\un{d}_2,\alpha_2)}: \begin{tikzcd}N_{\un{d}_1,\alpha_1}^{\sigma^{s_o\lambda'}}\times N_{\un{d}_2,\alpha_2}^{\sigma^{s_o\lambda'}}\arrow[r]&\mN^{(\un{d}_1,\alpha_1)}_{(\un{d}_2,\alpha_2),1} \end{tikzcd}
$$
which maps a pair of objects represented by the two quivers 
\begin{equation}
\label{eq:twoFlags2}
\includegraphics[scale=0.9]{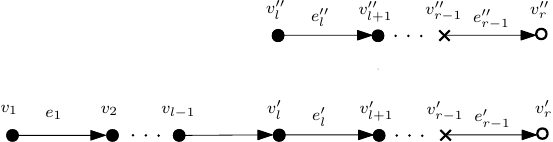}
\end{equation}
with classes $(\un{d}_i,\alpha_i)$ for $i=1,2$ (labelling from the bottom) to an object described by \eqref{eq:sumlocus} by setting $V_0 = \CC$ and $m^{e_{-1}} = \id_{\CC}=m^{e''_0}$ while keeping the rest unchanged. Due to $N_{\un{d}_1,\alpha_1}^{\sigma^{s_o\lambda'}}$ being rigidified, there are fixed isomorphisms $V''_l\cong \CC, V_1 \cong \CC$. This is an open embedding because the contributions of $e_{-1}$ and $e''_0$ for the target cancel with the ones of $v_0$ and $v''_{l}$ respectively. Next, there is a map 
$$
\widetilde{\mu}^{(\un{d}_1,\alpha_1)}_{(\un{d}_2,\alpha_2)}: \begin{tikzcd}\mN^{(\un{d}_1,\alpha_1)}_{(\un{d}_2,\alpha_2),1}\arrow[r]&\mN^{\rig}_{(1,\un{d}),\alpha}\end{tikzcd}
$$
that acts by preserving the vector spaces $V_j$ for $0\geq j\leq l-1$ and by taking pairwise sums of the rest:
$$
V_k = V'_k\oplus V''_k\qquad\textnormal{for}\quad k\geq l\,,\qquad E = E'\oplus E''\,.
$$
Following \cite[Proposition 10.20 (c)]{JoyceWC}, I claim that 
\begin{equation}
\label{eq:sumtoMSmap}
\mu^{\un{d}_1,\alpha_1}_{\un{d}_2,\alpha_2}:=\widetilde{\mu}^{(\un{d}_1,\alpha_1)}_{(\un{d}_2,\alpha_2)}\circ \Lambda^{(\un{d}_1,\alpha_1)}_{(\un{d}_2,\alpha_2)}
\end{equation}
is an isomorphism on its image which is equal to a union of some connected components of the fixed point locus. I will also denote the obvious extension of this map to $ \mN_{\un{d}_1,\alpha_1}^{\rig}\times \mN_{\un{d}_2,\alpha_2}^{\rig}$ in the same way by $\mu^{(\un{d}_1,\alpha_1)}_{(\un{d}_2,\alpha_2)}$.

To get this fixed point locus, one uses the same argument that introduced the \Pur{weight $-1$} $\GG_m$-action at $v_0$ in \eqref{eq:e0locus}. Going in the opposite direction, one gets the trivial $\GG_m$-action on $V_{0}$ and the \Pur{weight 1} action on $V_{1}$. Because of this, there needs to be a non-trivial action on $V_j$ for $j>1$ and on $E$. The weight decomposition of these terms will have only weight 0 and weight 1 summands due to stability. Thus \cite[Proposition 10.20 (c)]{JoyceWC} concludes that each fixed point of this type is of the form \eqref{eq:sumlocus} for $j\in J_o$ where the action is of \Pur{weight $1$} at the vertices of the longer horizontal sequence. It is trivial for the upper sequence. Moreover, it is shown there that $m^{e_0}$ can be decomposed as a sum of 
$$m^{e'_0}:\begin{tikzcd} V'_{l}\arrow[r,"0"]& V_0\end{tikzcd}\quad \textnormal{and}\quad m^{e''_0}: \begin{tikzcd} V''_{l}\arrow[r,"\sim"]& V_0\end{tikzcd}\,,$$
and that $m^{e_{-1}}$ and $m^{e'_{l-1}}: V_{l-1}\xrightarrow{\sim} V'_l$ are isomorphisms. This is then used to conclude that \eqref{eq:Sigmatofixedsumlocus} indeed identifies the two schemes.

Consider the projection 
\begin{equation}
\label{eq:PiV}
\Pi_{\mathsf{V}}:\begin{tikzcd}N_{(1,\un{d}),\alpha}^{\sigma^{s_{o}\lambda'}}\arrow[r]& \mN_{\un{d},\al}\end{tikzcd}
\end{equation}
that corresponds to forgetting the vertex $v_0$ in \eqref{eq:QMSaction} and all the arrows pointing towards it. This map is defined even though the latter stack is not rigidified because $d_1=d_0=1$. The restriction of $\FF_{(1,\un{d}),\alpha}$ to $N_{(1,\un{d}),\alpha}^{\sigma^{s_o\lambda'}}\Big|^{(\un{d}_1,\alpha_1)}_{(\un{d}_2,\alpha_2)}$ differs from the pull-back of $\FF_{\un{d},\al}$ along $\Pi_{\mathsf{V}}$ by the term corresponding to 
\begin{equation}
\label{eq:PiVobsdiff}
\includegraphics[scale=0.9]{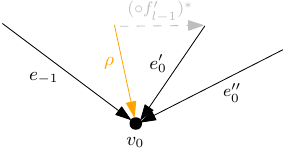}
\end{equation}
and its dual. Since $m^{e_{-1}},m_{e''_0}$ are isomorphisms, the terms $\begin{tikzcd}
    \,\arrow[r,"e_{-1}"]&\,
\end{tikzcd}$ and $\begin{tikzcd}
    \, &\arrow[l,"e''_{0}"']\,
\end{tikzcd}$ cancel with $\stackrel{v_0}{\bullet}$ and $\stackrel{v''_{l-1}}{\bullet}$. Finally, the contribution of $\begin{tikzcd}\,\arrow[r,"\rho",BurntOrange]& \,\end{tikzcd}$ is mapped to the one of $\begin{tikzcd}\,&\arrow[l,"e'_0"']\,\end{tikzcd}$ by the isomorphism $m^{e_{l-1}}\circ $, so that after including $\stackrel{v''_{l-1}}{\bullet}$ into the above diagram, the combined complex is acyclic. This implies that one needs to remove the term $\begin{tikzcd}
    \stackrel{v''_{l-1}}{\bullet}
\end{tikzcd}$ and its dual from 
$\Big(\mu^{(\un{d}_1,\al_1)}_{(\un{d}_2,\al_2)}\Big)^*\circ \Pi_{\mathsf{V}}^*\FF_{\un{d},\al}$ to get $\Big(\mu^{(\un{d}_1,\al_1)}_{(\un{d}_2,\al_2)}\Big)^*\FF_{(1,\un{d}),\al}$. Note that 
$$
\Pi_{\mathsf{V}}\circ \mu^{(\un{d}_1,\al_1)}_{(\un{d}_2,\al_2)} = \mu_{\Flag} \circ \big(s_{\un{d}_1,\al_1}\times s_{\un{d}_2,\al_2}\big)
$$
so 
\begin{align*}
\Big(\mu^{(\un{d}_1,\alpha_1)}_{(\un{d}_2,\alpha_2)}\Big)^*\FF_{(1,\un{d}),\alpha}&\cong \Big(\mu^{(\un{d}_1,\al_1)}_{(\un{d}_2,\al_2)}\Big)^*\circ \Pi_{\mathsf{V}}^*\FF_{\un{d},\al}+(\textnormal{remove }
\begin{tikzcd}[ampersand replacement = \&]
    \Cya{\bullet}\arrow[loop, ForestGreen, in=55,out=115, looseness=10,"v''_{l-1}"]
\end{tikzcd})\\
&\cong  \mu_{\Flag}^*\FF_{\un{d},\al}+(\textnormal{remove }
\begin{tikzcd}[ampersand replacement = \&]
    \Cya{\bullet}\arrow[loop, ForestGreen, in=55,out=115, looseness=10,"v''_{l-1}"]
\end{tikzcd})\\
&\cong \FF_{\un{d}_1,\al_1} \boxplus \FF_{\un{d}_2,\alpha_2}\oplus \Theta_{\Flag}\oplus \sigma^*\Theta_{\Flag}+(\textnormal{remove }
\begin{tikzcd}[ampersand replacement = \&]
    \Cya{\bullet}\arrow[loop, ForestGreen, in=55,out=115, looseness=10,"v''_{l-1}"]
\end{tikzcd})
\end{align*}
by \eqref{eq:additivityofFFflag}. Furthermore, another copy of $\mO$ and $\mO[-4]$ are removed by starting from $\FF^{\rig}_{(1,\un{d}),\al}$ instead. Because this only affects the diagonal terms, I will focus on them. For the classical obstruction theory \eqref{eq:wtFFI}, rigidifying and removing $\stackrel{v''_{l-1}}{\Cya{\bullet}}$ would produce 
\begin{equation}
\label{eq:rigidifiedsumFFna}
\begin{tikzcd}\Big(\wt{\FF}_{\un{d}_1,\al_1}\arrow[r]&\mO[-1]\Big) \boxplus \Big(
\wt{\FF}_{\un{d}_2,\alpha_2}\arrow[r]&\mO[-1]\Big) 
\end{tikzcd}
\end{equation}
which are the classical rigidified obstruction theories. Symmetrizing this observation, one gets 
$$
\Big(\mu^{(\un{d}_1,\alpha_1)}_{(\un{d}_2,\alpha_2)}\Big)^*\FF_{(1,\un{d}),\alpha}\cong \FF^{\rig}_{\un{d}_1,\al_1}\boxplus \FF^{\rig}_{\un{d}_2,\al_2}\oplus \Theta_{\Flag}\oplus \sigma^*\Theta_{\Flag}\,.
$$
 To understand the weights of each term, observe that \eqref{eq:sumlocus} implies that the objects on first line of \eqref{eq:twoFlags2} have weight 0, while the ones on the second line are weight $1$. Combined with the description of $\Theta_{\Fl}$ from \eqref{eq:extendTheta} using Corollary \ref{cor:sumobstructiontheory}, this implies that
\begin{equation}
\label{eq:sum-locus-normal}
\Big(\mu^{(\un{d}_1,\alpha_1)}_{(\un{d}_2,\alpha_2)}\Big)^*\FF^{\rig}_{(1,\un{d}),\alpha}\cong \FF^{\rig}_{\un{d}_1,\alpha_1}\boxplus \FF^{\rig}_{\un{d}_2,\alpha_2}\oplus t\cdot \Theta_{\Flag}\oplus t^{-1}\cdot\sigma^*\Theta_{\Flag}
\end{equation}
 in $D^b_{\GG_m}\Big(N_{\un{d}_1,\alpha_1}^{\sigma^{s_o\lambda'}}\times N_{\un{d}_2,\alpha_2}^{\sigma^{s_o\lambda'}}\Big)$.
\end{enumerate}
\end{proof}
The wall-crossing formula will be a consequence of the virtual equivariant localization in Theorem \ref{thm:eqvirloc} applied to
$$\Big[N_{(1,\un{d}),\alpha}^{\sigma^{s_o\lambda'}}\Big]^{\vir}\in H^{\GG_m}_*\Big(N_{(1,\un{d}),\alpha}^{\sigma^{s_o\lambda'}}\Big)\,.$$
I choose to state the wall-crossing formula for objects in $\mB_{Q_{\Flag}}$ first and then deduce Theorem \ref{thm:familyWC} as a consequence of the more general result. This will allow a simpler derivation of alternative wall-crossing formulae in the future, in which case one can directly use Theorem \ref{thm:FlagWC} below. The computation is heavily based on \cite[Proposition 10.23 and 10.24]{JoyceWC}. Still, because of the additional attention one needs to pay to orientations, and due to the present result being stated for $\mB_{\Flag}$ rather than for $\mA$, I choose to present complete arguments within reason. I will denote by $\big[N^{\sigma^{\lambda}}_{\un{d},\al}\big]^{\vir}$ both the virtual fundamental class in $H_*\big(N^{\sigma^{\lambda}}_{\un{d},\al}\big)$ and its pushforward to $H_*\big(\mN^{\rig}_{\un{d},\al}\big)$ along the open embedding \eqref{eq:Nsigmalambda}. 
\begin{theorem}
\label{thm:FlagWC}
For $(\un{d},\al)$ and each $o\in O$ as in \eqref{eq:proofWCgraphic}, the formula
\begin{equation}
\label{eq:FlagWC}
\Big[N^{\sig^{s^+_o\lam'}}_{\un{d},\al}\Big]^{\vir}-\Big[N^{\sig^{s^-_o\lam'}}_{\un{d},\al}\Big]^{\vir} = \sum_{j\in J_o}\bigg[\Big[N^{\sig^{s_o\lam'}}_{\un{d}^j_1,\al^j_1}\Big]^{\vir},\Big[N^{\sig^{s_o\lam'}}_{\un{d}^j_2,\al^j_2}\Big]^{\vir}\bigg]
\end{equation}
holds in $L\iFlag_{*}$.
\end{theorem} 
\begin{proof}
After applying equivariant localization to
$$
\Big[N_{(1,\un{d}),\alpha}^{\sigma^{s_o\lambda'}}\Big]^{\vir}\in H^{\GG_m}_*\Big(N_{(1,\un{d}),\alpha}^{\sigma^{s_o\lambda'}}\Big)
$$
one needs to take the pushforward along
\begin{equation}
\label{eq:NVspacetoNstack}
\Pi_{\mathsf{V}}:\begin{tikzcd}
    N_{(1,\un{d}),\alpha}^{\sigma^{s_o\lambda'}}\arrow[r] & \mN^{\rig}_{\un{d},\al}\,.
\end{tikzcd}
\end{equation}
and extract the residue. The last operation uses that the $\GG_m$-action on $\mN^{\rig}_{\un{d},\al}$ is trivial, so 
$$
H^{\GG_m}_*\big(\mN^{\rig}_{\un{d},\al}\big)\otimes_{R[z]}R[z^{-1}] = H_*\big(\mN^{\rig}_{\un{d},\al}\big)[z,z^{-1}]\,.
$$
 Because $N_{(1,\un{d}),\alpha}^{\sigma^{s_o\lambda'}}$ has proper $\T$-fixed points by Assumption \ref{ass:obsonflag}, the residue in $z$ of the pushforward of its (localized) equivariant virtual fundamental class vanishes. Therefore, the residue of the sum of the contributions of the fixed point loci described in Proposition \ref{prop:fixedpoints} is zero. 
 
I will compute each summand separately, labelling them in the same way as in Proposition \ref{prop:fixedpoints}. In the process, I will always implicitly push forward along \eqref{eq:NVspacetoNstack} without mentioning it to avoid cluttering the computation with lengthy notation.
\begin{enumerate}
    \item The orientation from \eqref{eq:oNgeq} applied to \eqref{eq:Nvir1} makes $t^{-1}\cdot \mV^*_0\otimes \mV_1$ positive. Because this is visibly the same as the orientation induced by Definition \ref{def:orpullback} (see also \eqref{eq:FtoDiso}) on this piece of $\det\big(\FF^{\rig}_{Q_{\MS}}\big)$, I conclude that the correct sign of the induced virtual fundamental class \eqref{eq:FPvir1} is 
    $$
    \RX{+}\Big[N_{\un{d},\al}^{\sigma^{s^-_{o}\lambda'}}\Big]^{\vir}\,.
    $$
    Consequently, the term assigned to this fixed point locus becomes
    \begin{equation}
    \label{eq:virlocus1}
    [z^{-1}]\left\{ \frac{\Big[N_{\un{d},\al}^{\sigma^{s^-_{o}\lambda'}}\Big]^{\vir}}{z\big(1+z^{-1}\mV^*_1\otimes \mV_0\big)}\right\} = \Big[N_{\un{d},\al}^{\sigma^{s^-_{o}\lambda'}}\Big]^{\vir}\in H_*\big(\mN^{\rig}_{\un{d},\al}\big)\,.
    \end{equation}
    \item
    In this case, the orientation from \eqref{eq:oNgeq} makes $t^{-1}\cdot  \big(\mV_{l}/\mV_{l-1}\big)^*\otimes\mV_0$ positive. This time, this differs from the orientation induced by Definition \ref{def:orpullback}, for which $t\cdot \mV_0^*\otimes (\mV_{l}/\mV_{l-1}\big) $ is positive. This gives rise to the additional sign
    $$
    \RX{-}\Big[N_{\un{d},\al}^{\sigma^{s^+_{o}\lambda'}}\Big]^{\vir}
    $$
    and leads to the expression
\begin{equation}
    \label{eq:virlocus2}
    [z^{-1}]\left\{ \frac{-\Big[N_{\un{d},\al}^{\sigma^{s^+_{o}\lambda'}}\Big]^{\vir}}{z\big(1+z^{-1}\mV^*_0\otimes (\mV_l/\mV_{l-1})\big)}\right\} =- \Big[N_{\un{d},\al}^{\sigma^{s^+_{o}\lambda'}}\Big]^{\vir}\in H_*\big(\mN^{\rig}_{\un{d},\al}\big)\,.
\end{equation}
    \item In the third and last situation, the orientation on the virtual normal bundle is given by $o_{\sig^*\Theta_{Q_{\Flag}}}$. First note that $F_{\un{d},\al} \cong F_{(1,\un{d}),\alpha}$ because of the term \eqref{eq:PiVobsdiff} being acyclic once $\stackrel{v_{1}}{\bullet}$ is included. This implies that one needs to compare the orientations of $F_{\un{d},\al}$ and $F_{\un{d}_1,\alpha_1}\boxtimes F_{\un{d}_2,\alpha_2}$ in the same way as in Definition \ref{def:varepsflag}. Thus, one needs to modify \eqref{eq:FPvir3} by the sign 
    $$
\RX{\varepsilon^{\un{d}_1,\alpha_1}_{\un{d}_2,\alpha_2}\cdot }\Big[N_{\un{d}_1,\alpha_1}^{\sigma^{s_o\lambda'}}\Big]^{\vir}\boxtimes \Big[N_{\un{d}_2,\alpha_2}^{\sigma^{s_o\lambda'}}\Big]^{\vir}\,,
    $$
    where I again omit the superscripts $j$. The corresponding localization term is expressed as
    \begin{equation}
    \label{eq:residueFlag}
    [z^{-1}]\left\{ \varepsilon^{\un{d}_1,\alpha_1}_{\un{d}_2,\alpha_2}\Big(\Pi_{\mathsf{V}}\circ\mu^{(\un{d}_1,\alpha_1)}_{(\un{d}_2,\alpha_2)}\Big)^{\GG_m}_*\left(\frac{\Big[N_{\un{d}_1,\alpha_1}^{\sigma^{s_{o}\lambda'}}\Big]^{\vir}\boxtimes \Big[N_{\un{d}_2,\alpha_2}^{\sigma^{s_{o}\lambda'}}\Big]^{\vir}}{z^{\Rk\big(\Theta_{Q_{\Flag}}\big)}c_{z^{-1}}\Big(\Theta_{Q_{\Flag}}\Big)}\right)\right\}\,,
    \end{equation}
    where $(-)^{\GG_m}_*$  denotes the $\GG_m$-equivariant pushforward. To bring it into the form appearing in \eqref{eq:gammaFlagWC}, I follow the observation from \cite[(9.50)]{JoyceWC} and adapt it to $\mN_{\Flag_k}$. 
    
    First note that the map $\mu^{(\un{d}_1,\alpha_1)}_{(\un{d}_2,\alpha_2)}$ is $\GG_m$-equivariant only up to a 2-morphism $\mathfrak{a}$ in the sense of \cite[(2)]{Romagny}. For each \Pur{$x\in \GG_m$}, this 2-morphism is determined by a natural transformation that assigns to each object of the form \eqref{eq:sumlocus} with $m^{e_{-1}}$ rescaled by \Pur{$x$} a morphism to the same object without the scaling at the edge $e_{-1}$. This morphism acts as represented in \eqref{eq:sumlocus} by \Pur{$x\cdot \id$} on each vertex of the longer horizontal $A_r$ quiver. Taking the stacky quotient $ (-)/\GG_m$ of such a weakly $\GG_m$-equivariant morphism, the natural transformation $\mathfrak{a}$ contributes to the resulting morphism between the quotient stacks described in the proof of \cite[Proposition 2.6]{Romagny}. Carefully following the cited construction, one arrives at the following commutative diagram involving $\Big(\Pi_{\mathsf{V}}\circ \mu^{(\un{d}_1,\alpha_1)}_{(\un{d}_2,\alpha_2)}\Big)/\GG_m$:
    $$
    \begin{tikzcd}[column sep = 3cm, row sep= 1.3cm]
          \arrow[d,"\Big(\Pi_V\circ \mu^{(\un{d}_1,\alpha_1)}_{(\un{d}_2,\alpha_2)}\Big)/\GG_m"']  B\GG_m\times \mN^{\rig}_{\un{d}_1,\alpha_1}\times \mN^{\rig}_{\un{d}_2,\alpha_2}\arrow[r,"\id\times s_{\un{d}_1,\al_1}\times s_{\un{d}_2,\al_2}"]&B\GG_m\times \mN_{\un{d}_1,\alpha_1}\times \mN_{\un{d}_2,\alpha_2}\arrow[d,"\id\times \rho_{12}\times \id"]\\
     B\GG_m\times \mN^{\rig}_{\un{d},\al}&\arrow[l,"\id\times (\Pi_{\un{d},\al} \circ \mu_{\Flag})"] B\GG_m\times \mN_{\un{d}_1,\alpha_1}\times \mN_{\un{d}_2,\alpha_2}\,.
    \end{tikzcd}
    $$
    The morphism $\id\times \rho_{12}: B\GG_m\times \mN_{\un{d}_1,\alpha_1}\to B\GG_m\times\mN_{\un{d}_1,\alpha_1}$ is the identity on $B\GG_m$, but also lets $B\GG_m$ act on the second factor.  The factor $\rho_{12}$ is precisely the correction corresponding to the 2-morphism $\mathfrak{a}$.    

    Taking $H^*(-)$ of the above diagram, expresses the $\GG_m$-equivariant pull-back $\Big(\Pi_V\circ \mu^{(\un{d}_1,\alpha_1)}_{(\un{d}_2,\alpha_2)}\Big)^*_{\GG_m}$ in terms of the composition of the other three arrows. Because $H^*_{\GG_m}(-)$ is dual as an $R[z] = H^*(B\GG_m)$ module to $H_*^{\GG_m}(-)$ in all the above cases, this shows that
    \begin{align*}
    \label{eq:ezTorigin}
    \Big(\Pi_V\circ \mu^{(\un{d}_1,\alpha_1)}_{(\un{d}_2,\alpha_2)}\Big)^{\GG_m}_* &= \big(\Pi_{\un{d},\al} \circ \mu_{\Flag}\big)_*\circ  (\rho_{12}\otimes \id)_* \Big(\sum_{i\geq 0}z^ip^i\boxtimes \big(s_{\un{d}_1,\al_1}\times s_{\un{d}_2,\al_2}\big)_*\Big)\\
    &=\big(\Pi_{\un{d},\al} \circ \mu_{\Flag}\big)_*\circ (e^{zT}\otimes \id)  \circ \big(s_{\un{d}_1,\al_1}\times s_{\un{d}_2,\al_2}\big)_*\,.
    \numberthis
    \end{align*}
     Here $(\rho_{12})_* \Big(\sum_{i\geq 0}z^ip^i\boxtimes -\Big):H_*^{\GG_m}\big(\mN_{\un{d}_1,\alpha_1}\big)\to H_*^{\GG_m}\big(\mN_{\un{d}_1,\alpha_1}\big)$ can be directly seen to be dual to $(\id\times \rho_{12})^*: H^*\big(B\GG_m\times\mN_{\un{d}_1,\alpha_1}\big)\to H^*\big(B\GG_m\times\mN_{\un{d}_1,\alpha_1}\big)$. Using \eqref{eq:srig}, I introduce the notation 
$$
\big[N^{\sigma^{\lambda'}}_{\un{d},\al}\big]^{\svir} = (s_{\un{d},\al})_*\big[N^{\sigma^{\lambda'}}_{\un{d},\al}\big]^{\vir}\in H_*(\mN_{\un{d},\al})\,.
$$
Applying \eqref{eq:ezTorigin} to \eqref{eq:residueFlag} produces \begin{align*}
\label{eq:gammaFlagWC}\big(\Pi_{\un{d},\al}\big)_*  \big[z^{-1}\big]\left\{ \varepsilon^{\un{d}_1,\alpha_1}_{\un{d}_2,\alpha_2}\big(\mu_{\Flag}\big)_*\left(\big(e^{zT}\otimes \id\big)\frac{\Big[N_{\un{d}_1,\alpha_1}^{\sigma^{s_{o}\lambda'}}\Big]^{\svir}\boxtimes \Big[N_{\un{d}_2,\alpha_2}^{\sigma^{s_{o}\lambda'}}\Big]^{\svir}}{z^{\Rk\big(\Theta_{\Flag_k}\big)}c_{z^{-1}}\Big(\Theta_{\Flag_k}\Big)}\right)\right\}
\numberthis
\end{align*}
   in $H_*\big(\mN^{\rig}_{\un{d},\al}\big)$.
    \end{enumerate}
Setting the sum of all 3 terms to be zero proves \eqref{eq:FlagWC} after comparing with the vertex algebra structure from Definition \ref{Def:FlagVA}.
\end{proof}
\begin{remark}
    Note that in the setting of Assumption \ref{ass:pairWC} we do not have an obstruction theory on $\mN_{d,\al}^{\tau^p}$ when $d=1$ as it is only provided on its rigidification. However, we still have the direct sum morphism 
    $$
    \mu:\mN_{1,\al_1}^{\tau^p,\rig}\times \mN_{0,\al_2}^{\tau^p}\to \mN_{1,\al_1+\al_2}^{\tau^p,\rig}
    $$
    constructed using the map \eqref{eq:srig}. Let $I$ be an object of $\mN_{1,\al_1}^{\tau^p,\rig}$ and $F$ and object of $\mN_{0,\al_2}^{\tau^p}$. Then we see that
    $$
    \operatorname{Rhom}_X(I\oplus G,I\oplus G)_0 = \operatorname{Rhom}_X(I,I)_0\oplus \operatorname{Rhom}_X(I,G)\oplus \operatorname{Rhom}_X(G,I)\oplus \operatorname{Rhom}_X(G,G)
    $$
    which will produce the analog of \eqref{eq:sum-locus-normal}.
  So the above computations together with their equivariant refinements below also apply in the case of Assumption \ref{ass:pairWC}.
\end{remark}
\subsection{The $\T$-equivariant case}
\label{sec:equivariantWC}
In this subsection, I refine the arguments in Theorem \ref{thm:FlagWC}. That this requires special attention was pointed out to me by Henry Liu. I will first focus on the local approach from §\ref{sec:equivariantVA}. It is simpler and, as far as I know, has not appeared anywhere else. 

From now on, I will fix the notation 
$$
\TT = \T\times \GG_m\,.
$$
I will be working with the rigidified stack $\big(\mN^{\T}_{\un{d}, \al}\big)^{\rig}$ which is identified with $\big(\mN^{\tau(\T)}_{\un{d}, \al}\big)^{\rig}$ due to Assumption \ref{ass:obsonflag}.a). Choosing the weight at the vertex $v_1$ to always be trivial, one obtains the diagram
\begin{equation}
\label{eq:PiVT}
\begin{tikzcd}
\Big(N_{(1,\un{d}),\alpha}^{\sigma^{s_o\lambda'}}\Big)^{\TT}\arrow[rd,"\Pi^{\TT}_{\V}"]\arrow[r,"\iota^{u}"] \arrow[rr, bend left=30, "\iota^{\TT}"]& \Big(N_{(1,\un{d}),\alpha}^{\sigma^{s_o\lambda'}}\Big)^{\tau(\T)}\arrow[r,"\iota^{\tau(\T)}"]\arrow[d,"\Pi^{\tau(\T)}_{\V}"]&N_{(1,\un{d}),\alpha}^{\sigma^{s_o\lambda'}}\arrow[d,"\Pi_{\V}"]\\
&\big(\mN^{\T}_{\un{d},\al}\big)^{\rig}\arrow[r]&\big(\mN_{\un{d},\al}\big)^{\rig}
\end{tikzcd}
\end{equation}
induced by $\Pi_{\V}$. All the above maps are $\TT$-equivariant for the trivial action of $\GG_m$ on the bottom row.

Assumption \ref{ass:obsonflag}.a) and c) constructs virtual fundamental classes
\begin{align*}
\label{eq:mastervir}
\Big[\big(N^{\sig}_{(1,\un{d}),\al}\big)^{\TT}\Big]^{\vir}\in H^{\TT}_*\Big(\big(N^{\sig}_{(1,\un{d}),\al}\big)^{\TT}\Big)\,,\\
\Big[\big(N^{\sig}_{(1,\un{d}),\al}\big)^{\tau(\T)}\Big]^{\vir}\in H^{\TT}_*\Big(\big(N^{\sig}_{(1,\un{d}),\al}\big)^{\tau(\T)}\Big)\,,
\numberthis
\end{align*}
I now fix an identification 
$
\TT/\tau(\T)\xrightarrow{\sim}\CC^*
$
and denote the resulting one-dimension torus by $\CC^*_{u}$. Here $u$ is a cohomological weight of $\T$ such that the above induces a map 
$$
\begin{tikzcd}
e^{z+u}: \TT\arrow[r]&\CC^*
\end{tikzcd}\,.
$$
Let $\NN_{(-)}$ be the virtual normal bundles of $\iota^{(-)}$ in the upper row of \eqref{eq:PiVT}, then we see that
$$
\NN_{\TT} = \NN_{u} +(\iota^{u})^*\NN_{\tau(\T)}\,.
$$
Using the isomorphism 
$$
H^{\TT}_*\Big(\big(N^{\sig}_{(1,\un{d}),\al}\big)^{\tau(\T)}\Big)\cong H^{\TT}_*\Big(N^{\sig}_{(1,\un{d}),\al}\Big)
$$
one can identify
\begin{align*}
\big[N^{\sig}_{(1,\un{d}),\al}\big]^{\vir}_{\TT,\loc}&=\iota^{u}_*\,\frac{\Big[\big(N^{\sig}_{(1,\un{d}),\al}\big)^{\TT}\Big]^{\vir}}{e_{\TT}(\NN_{\TT})}\\
&=\frac{1}{e_{\TT}(\NN_{\tau(\T)})}\cdot \iota^{u}_*\,\frac{\Big[\big(N^{\sig}_{(1,\un{d}),\al}\big)^{\TT}\Big]^{\vir}}{e_{\CC^*_{u}}(\NN_{u})}\\
&=\frac{\Big[\big(N^{\sig}_{(1,\un{d}),\al}\big)^{\tau(\T)}\Big]^{\vir}}{e_{\TT}(\NN_{\tau(\T)})}\,.
\end{align*}
In particular, applying $\big(\Pi^{\tau(\T)}_{\V}\big)_*$ we get an expression in
$$
H^{\TT}_*\Big(\big(\mN^{\T}_{\un{d},\al}\big)^{\rig}\Big) = H_*\Big(\big(\mN^{\T}_{\un{d},\al}\big)^{\rig}\Big)\llparenthesis \Ft\oplus \GG_a\rrparenthesis^{\gr}
$$
that does not have poles at $(z+u) = 0$ because $\NN_{\tau(\T)}$ does not contain this weight. 

Let us now fix $\tau = \id_{\T}\times 1$ setting $u=0$. We reverse the order of inclusions giving 
$$
\begin{tikzcd}
\Big(N_{(1,\un{d}),\alpha}^{\sigma}\Big)^{\TT}\arrow[r,"\iota^{\TT/\GG_m}"] \arrow[rr, bend left=30, "\iota^{\TT}"]& \Big(N_{(1,\un{d}),\alpha}^{\sigma}\Big)^{\GG_m}\arrow[r,"\iota^{\GG_m}"]&N_{(1,\un{d}),\alpha}^{\sigma}
\end{tikzcd}
$$
This induces the splitting
$$
\NN_{\TT} = \NN_{\GG_m} + \NN_{\TT/\GG_m}\,.
$$
The first term $\NN_{\GG_m}$ was computed already in Proposition \ref{prop:fixedpoints} but it now contains weights of the full torus $\TT$. The second term $\NN_{\TT/\GG_m}$ only contains $\T$-weights under the identification $\TT/\GG_m =\T$. Thus we may write
\begin{equation}
\label{eq:PTVMSvir}
\big(\Pi^{\TT}_{\V}\big)_*\big[N^{\sig}_{(1,\un{d}),\al}\big]^{\vir}_{\TT,\loc} = \big(\Pi^{\TT}_{\V}\big)_*\frac{\Big[\big(N^{\sig}_{(1,\un{d}),\al}\big)^{\TT}\Big]^{\vir}}{e_{\TT}(\NN_{\TT})}  
\end{equation}
as a sum of the expressions \eqref{eq:virlocus1}, \eqref{eq:virlocus2}, and \eqref{eq:residueFlag} where each $$\Big[N_{\un{d},\al}^{\sigma^{s\lambda'}}\Big]^{\vir}_{\T,\loc}\in H^{\T}_*\Big(\big(\mN^{\T}_{\un{d},\al}\big)^{\rig}\Big)_{\loc}$$ is now defined as in Example \ref{ex:mMvir} and the weights at $v_1$ are still trivial.

Because of the vanishing of the residues at each $(z+u)=0$, the expansion of \eqref{eq:PTVMSvir} using \eqref{eq:explicitexpand} has vanishing total residue at $z=0$. By the same arguments as in the proof of Theorem \ref{thm:FlagWC}, this leads to the $\T$-equivariant version of \eqref{eq:FlagWC}.
\subsection{Projecting flags to pairs and sheaves}
\label{sec:Flagtosheaves}
Here, I will detail the argument of step 3 from §\ref{sec:summaryproof} represented by the vertical arrows $\to\!\to\!\to\!\to$ in \eqref{eq:proofWCgraphic}. The section culminates in the proof of Proposition \ref{prop:OmegaWC} which implies the wall-crossing formula in $\mA$. I go through the argument only for $\T = \{1\}$ because the equivariant generalization is straigh-forward.

From \eqref{eq:Nsig'0}, recall that $N^{\sig^{\lambda'}}_{\un{d},\al} = N^{(\sig')^{0}}_{\un{d},\al}$ holds implying that $\Omega^{\sig^{\lambda'}}_{\un{d},\al} = \Omega^{\sig',k}_{\al}$. For both $\zeta = \sig,\sig'$, there is a commutative diagram
\begin{equation}
\label{eq:NrigstoMrig}
\begin{tikzcd}[column sep=large, row sep= large]
\arrow[dr,"\pi^{\zeta}_{\un{d},\al}"']N^{\zeta^0}_{\un{d},\al}\arrow[r,"\pi^{\zeta}_{\un{d}/1,\al}"]& N^{\zeta_{\JS}}_{1,\al}\arrow[d,"\pi^{\JS}_{\al}"]\\
    &\mM^{\rig}_{\al}
\end{tikzcd}\,.
\end{equation}
For the local equivariant approach, the bottom term becomes $(\mM^{\T}_{\al})^{\rig}$. The maps are still commutative because the weights at $v_1$ were always fixed to be trivial. The obstruction theories of the spaces in the upper row are related by the corresponding $\infty$-Pvp diagrams from Assumption \ref{ass:obsonflag}. They are compatible in the sense of Theorem \ref{thm:functsympull} explained in §\ref{Sec:pullback}.

Using that 
$$
\chi\big(\pi^{\zeta}_{\un{d}/1,\al}\big) = \Big(\chi\big(\al(k)\big)-1\Big)!\,,
$$
Theorem \ref{thm:virtualpullpush} implies that
\begin{align*}
\chi(\al(k))!\cdot \Omega^{\zeta^0}_{\un{d},\al} &= 
\big(\pi^{\zeta}_{\un{d},\alpha}\big)_*\Big(\big[N_{\un{d},\alpha}^{\zeta^{0}}\big]^{\vir}\cap c_{\Rk}\big(T_{\pi^{\zeta}_{\un{d},\alpha}}\big)\Big)\\
&= \big(\pi^{\JS}_{\al}\big)_*\bigg(\big(\pi^{\zeta}_{\un{d}/1,\alpha}\big)_*\Big(\big[N_{\un{d},\alpha}^{\zeta^{0}}\big]^{\vir}\cap c_{\Rk}\big(T_{\pi^{\zeta}_{\un{d}/1,\alpha}}\big)\Big)\cap c_{\Rk}(T_{\pi^{\JS}_{\al}})\bigg)\\  &= \big(\chi(\al(k))-1\big)!\cdot\big(\pi^{\JS}_{\al}\big)_*\Big(\big[N^{\zeta_{\JS}}_{1,\alpha}\big]^{\vir}\cap c_{\Rk}\big(T_{\pi^{\JS}_{\alpha}}\big)\Big)\\
&=\chi(\al(k))! \cdot \Omega^{\zeta,k}_{\al} \,.
\end{align*}
Thus, Proposition \ref{prop:OmegaWC} relates $\Omega^{\sig,k}_{\al}$ and $\Omega^{\sig',k}_{\al}$ as claimed in step 3 of §\ref{sec:summaryproof}.

\begin{proof}[Proof of Proposition \ref{prop:OmegaWC}]
To get the wall-crossing formula \eqref{eq:OmegaWC} one needs to cap \eqref{eq:FlagWC} with $c_{\Rk}\Big(T_{\pi^{\rig}_{\un{d},\alpha}}\Big)$
and push it forward along $\pi^{\rig}_{\un{d},\alpha}$. I also divide by $\chi\big(\al(k)\big)$ afterwards. The following computation does so explicitly while paying attention to the signs. This is where the sign comparison between flags and sheaves from Lemma \ref{lem:epsilonQFlag} comes into play. The proof could also be directly handled by \cite[§2.5]{GJT} because $T_{\pi^{\rig}_{\un{d},\alpha}}$ is a vector bundle when restricted to the appropriate stable loci.

    I will again compute each of the terms separately following the proof of Theorem \ref{thm:FlagWC}. The left-hand side of \eqref{eq:FlagWC} becomes 
$$
\Omega^{\sigma^{s^+_{o}\lambda'}}_{\un{d},\alpha} -\, \Omega^{\sigma^{s^-_o\lambda'}}_{\un{d},\alpha} 
$$ 
using \eqref{eq:Omega-slambda}. To compute the right-hand side, I will fix a single $j\in J$ and omit it from the notation here. Note that $T_{{\pi_{\un{d},\alpha}}}$ is weight 0 with respect to $\rho_{\Flag_k}$, and one may identify it with $\Pi_{\Flag_k}^*T_{{\pi^{\rig}_{\un{d},\alpha}}}$. From now on, I will lift the entire computation to $\mN_{\un{d},\al}$ and $\mM_{\al}$. Using \eqref{eq:additivityofFFflag}, one shows that 
    \begin{align*}
    \label{eq:rhomugamma}
    \big(\mu_{\Flag_k}\circ (\rho_{12}\times \id)\big)^*\Big(T_{{\pi_{\un{d},\alpha}}}\Big) = \Big(\pi_1^*T_{{\pi_{\un{d}_1,\alpha_1}}}+\pi_2^*T_{{\pi_{\un{d}_2,\alpha_2}}}+& \,\Ma{e^{-\tau}\cdot \Theta^\vee_{\mN_{\Flag_k}/\mM}} + \Ma{e^{\tau}\cdot \sigma^*\Theta^\vee_{\mN_{\Flag_k}/\mM}} \Big)\,.
    \numberthis
    \end{align*}
  Using \eqref{eq:tzconvolution}, one replaces $e^\tau$ and $e^{-\tau}$ in \eqref{eq:rhomugamma} by $u=e^z$ and $u^{-1} =e^{-z}$ respectively. One can also expand the denominator in \eqref{eq:gammaFlagWC} as
    $$
    u\cdot \Theta_{\mN_{\Flag}} = u\cdot \Theta_{\mA} + \Ma{u\cdot \Theta_{\mN_{\Flag_k}/\mM_{\mA}}} + \Ma{u\cdot\sigma^*\Theta^\vee_{\mN_{\Flag_k}/\mM_{\mA}}[2]}
    $$
    using \eqref{eq:extendTheta}. The \Ma{maroon terms} can be paired in the order they appear in each equation and cancelled. Because the first \Ma{terms} are dual to each other, one gets an additional sign \RX{$(-1)^{\xi((\un{d}_1,\alpha_1),(\un{d}_2,\alpha_2))}$}. Due to \eqref{eq:epscomp}, the expression in the large curly brackets in \eqref{eq:gammaFlagWC} capped with $c_{\Rk}\big(T_{{\pi_{\un{d},\alpha}}}\big)$ becomes
    \begin{equation}
    \label{eq:bigWCcomp}
\varepsilon_{\alpha_1,\alpha_2}\big(\mu_{\Flag}\big)_*
 \left((e^{zT}\otimes \id)\frac{ \Big(\Big[N_{\un{d}_1,\alpha_1}^{\sigma^{s_{o}\lambda'}}\Big]^{\svir}\cap c_{\Rk}\big(T_{\pi_{\un{d}_1,\alpha_1}}\big)\Big) \boxtimes \Big(\Big[N_{\un{d}_2,\alpha_2}^{\sigma^{s_{o}\lambda'}}\Big]^{\svir}\cap c_{\Rk}\big(T_{\pi_{\un{d}_2,\alpha_2}}\big)\Big)}{z^{-\chi(\al_1,\al_2)}c_{z^{-1}}\big(\Theta_{\alpha_1,\alpha_2}\big)}\right) \,.   
    \end{equation} 
Here, I have used that \eqref{eq:rhomugamma} restricted to $N_{\un{d}_1,\alpha_1}^{\sigma^{s_o\lambda'}}\times N_{\un{d}_2,\alpha_2}^{\sigma^{s_o\lambda'}}$ is a vector bundle, so each of its summands are. Thus Chern classes of higher degree than the rank of each summand vanish.

For the last step, observe that the diagram
$$
\begin{tikzcd}[column sep=3cm, row sep= 1cm]
H_*\big(\mN_{\un{d}_1,\alpha_1}\big)\otimes H_*\big(\mN_{\un{d}_2,\alpha_2}\big)\arrow[d,"(\mu_{\Flag})_*\circ (e^{zT}\otimes \id)"']\arrow[r,"(\pi_{\un{d}_1,\al_1})_*\otimes (\pi_{\un{d}_2,\al_2})_*"]&H_*(\mM_{\al_1})\otimes H_*(\mM_{\al_2}) \arrow[d,"\mu_*\circ (e^{zT}\otimes \id)"]\\
H_*\big(\mN_{\un{d},\al}\big)\arrow[r,"(\pi_{\un{d},\al})_*"]&H_*(\mM_{\al})
\end{tikzcd}
$$
is commutative, which implies that the pushforward of \eqref{eq:bigWCcomp} along $\pi_{\un{d},\al}$ becomes
$$
\chi\big(\al_1(k)\big)!\,\chi\big(\al_2(k)\big)!\,\varepsilon_{\alpha_1,\alpha_2}\,\mu_*
 \left((e^{zT}\otimes \id )\frac{ \widehat{\Omega}^{\sigma^{s_o\lambda'}}_{\un{d}_1,\alpha_1} \boxtimes  \widehat{\Omega}^{\sigma^{s_o\lambda'}}_{\un{d}_2,\alpha_2}}{z^{-\chi(\al_1,\al_2)}c_{z^{-1}}\big(\Theta_{\alpha_1,\alpha_2}\big)}\right) \,,  
$$
    where $ \widehat{\Omega}^{\sigma^{s_o\lambda'}}_{\un{d}_i,\alpha_i}$ are lifts of $ \Omega^{\sigma^{s_o\lambda'}}_{\un{d}_i,\alpha_i}$ to $V_{\loc,*}$. The result of taking the residue at $z=0$ and projecting along $\Pi_{\mM_{\mA}}:\mM_{\mA}\to \mM^{\rig}_{\mA}$ can be expressed using Definition \ref{Def:VAlocal} or \ref{Def:VAglobal} and \eqref{eq:LiefromVA} as
    $$
\chi\big(\al_1(k)\big)!\,\chi\big(\al_2(k)\big)! \Big[\Omega^{\sigma^{s_o\lambda'}}_{\un{d}_1,\alpha_1}, \Omega^{\sigma^{s_o\lambda'}}_{\un{d}_2,\alpha_2}\Big]\,.
    $$
    The formula \eqref{eq:OmegaWC} follows immediately from this result.
\end{proof}
\appendix

\section{Stable $\infty$-Pvp diagrams}
\label{appendix}
\subsection{Stable $\infty$-categories}
\label{sec:stableinfty}
Provided one is familiar with higher topos theory, stable $\infty$-categories are easy to use, yet offer a lot of mileage in making constructions independent of choices. Additionally, due to cones of morphisms being $\infty$-functorial, computations become simpler than in corresponding triangulated categories. These two points are why this theory is used in the present work. The classical statements are recovered by taking homotopy categories of the stable $\infty$-categories.

In collecting some results from \cite[Chapter 1]{LurieHA} about stable $\infty$-categories, I will choose to focus on the example of complexes of sheaves on moduli spaces and stacks. Lurie's Higher Topos Theory \cite{LurieHT} is the only prerequisite for this section as I work with his definition of $\infty$-categories. Let $\mM$ be an Artin stack. Then consider its category $\mC^{+}(\mM)$ of injective quasi-coherent complexes 
$$
\cdots\longrightarrow 0\longrightarrow I^{-n}\longrightarrow I^{-n+1}\longrightarrow I^{-n+2}\longrightarrow  \cdots
$$
satisfying $I^{-k} = 0$ for $k\gg 0$.    The category $\mC^+(\mM)$ is enriched in the category of complexes of complex vector spaces $\mC\big(\Spec(\CC)\big)$. For any two $I^\bullet_1,I^\bullet_2\in \mC^+(\mM)$, one sets
$$
\Hom^\bullet_{\mM}(I^\bullet_1,I^\bullet_2) = \Tot^\bullet\big(\Hom_{\mM}(I^\bullet_1,I^\bullet_2))
$$
where $\Tot^\bullet(-)$ denotes the total complex constructed from the double complex $\Hom(I^\bullet_1,I^\bullet_2)$ with terms $\Hom(I^p_1,I^{q}_2)$ in double-degree $(-p,q)$. This means that the action of the differential on $f\in \Hom^{n}_{\mM}(I^{\bullet}_1, I^{\bullet}_2)$ is given by
$$
d(f) = d\circ f-(-1)^nf
$$
 This makes $\mC^+(\mM)$ into a dg-category. 

In \cite[Definition 1.1.2.4]{LurieHT}, $\infty$-categories are defined to be simplicial sets satisfying the additional properties of \textit{weak Kan complexes}. In \cite{LurieHA}, Lurie provides two equivalent constructions of $\infty$-categories from dg-categories. I choose to follow the one that will, in his language, produce the stable $\infty$-category $\mD^+(\mM)=N_{\dg}\big(\mC^+(\mM)\big)$. Throughout the appendix, I will use the notation 
$$
[n] = \{0,1,2,\ldots ,n\}
$$
for $n\in\ZZ_{\geq 0}$.
\begin{definition}[{\cite[Construction 1.3.1.6]{LurieHA}}]
\label{def:dgnerve}
The \textit{differential graded nerve} of a differential graded category $\mD$ is a simplicial set, such that the set of its $n$-simplices $N_{\dg}(\mD)_n$ consists of the collections of data $\Big(\{X_i\}^n_{i=0}, \{f_I\}_{I\subset [n]}\Big)$ such that
\begin{enumerate}
    \item $X_i$ are objects of $\mD$ for $0\leq i\leq n$,
    \item for $I =\{i_0<i_1<\cdots <i_{m}<i_{m+1}\}\subset [n]$, the element $f_I$ is a degree $-m$ morphism in $\Hom^{-m}_{\mD}\big(X_{i_0},X_{i_{m+1}}\big)$ satisfying 
    $$
    d f_I = (-1)^{m+1}\sum_{j=1}^{m} (-1)^{j}\big(f_{I\backslash\{i_j\}}-f_{I_{\geq j}}\circ f_{I_{\leq j}}\big)\,.
    $$
    Here, I used $I_{\leq j} = \{i_0,i_1,\ldots,i_{j-1}, i_{j}\}$ and $I_{\geq j} = \{i_j,i_{j+1},\ldots,i_m,i_{m+1}\}$.
\end{enumerate}
The morphisms $s_{\alpha}: N_{\dg}(\mD)_{n_2}\to N_{\dg}(\mD)_{n_1}$ for a non-decreasing $\alpha: [n_1]\to [n_2]$ are described explicitly in \cite[Construction 1.3.1.6]{LurieHA}. I will write $\mD^+(\mM) = N_{\dg}(\mC^+)$ and $\mD^b(\mM)$ for its full $\infty$-subcategory on the objects $I^\bullet$ in $\mC^+(\mM)$ satisfying $H^k(I^\bullet) =0$ for $k\gg 0$. When there is no difference between using $\mD^+(\mM)$ or $\mD^b(\mM)$, I will write $\mD^{\#}(\mM)$.
\end{definition}
\begin{example}
\label{ex:simplices}
    \begin{enumerate}[label=\roman*)]
        \item A 0, 1, or 2-simplex in the category $\mC^+(\mM)$ correspond respectively to a single complex $I^\bullet_0$, a morphism $f_{01}:I_0^\bullet\to I_1^\bullet$ of complexes, and a diagram of morphisms of complexes 
        \begin{equation}
        \label{eq:2simplex}
        \begin{tikzcd}[row sep =large]
        &I^{\bullet}_1\arrow[dr,"f_{12}"]&\\
I^{\bullet}_0\arrow[ur,"f_{01}"]\arrow[rr,""{name=A}]{}&&I^{\bullet}_2
\arrow[Rightarrow, from=A, to=1-2, "f_{123}", Blue]
        \end{tikzcd}
        \end{equation}
        with a homotopy $f_{012}$ between $f_{02}$ and $f_{12}\circ f_{01}$:
        $$
        d\circ f_{012} + f_{012}\circ d = f_{12}\circ f_{01}-f_{02}\,.
        $$
        \item Starting from the diagram 
        \begin{equation}
        \label{eq:IxIhomotopy}
        \begin{tikzcd}
\arrow[d,"g_0"]I_0^\bullet\arrow[r,"f_{01}"]&I_1^\bullet \arrow[d,"f_{12}"]\\
I_3^\bullet\arrow[r,"g_1"]&I_2^\bullet
        \end{tikzcd},
        \end{equation}
        one might construct a 2-simplex \eqref{eq:2simplex} by setting $f_{02} = g_1 \circ g_{0}$ and finding an appropriate homotopy $f_{012}$. Given such data, I will say that \eqref{eq:IxIhomotopy} is homotopy commutative. This is somewhat different from the usual definition of a $\Delta^1\times \Delta^1$ diagram for which the map $f_{02}$ is given independently and there is an extra homotopy from $f_{02}$ to $g_1\circ g_0$. I will ignore this minor deviation because it is clear how to go from one picture to the other,
        \item Though 4-simpleces are mentioned later on, the highest dimension of a simplex I will write down explicitly is 3. The data of a 3-simplex is determined by the two diagrams 
        \begin{equation}
        \label{eq:3simplex}
        \includegraphics{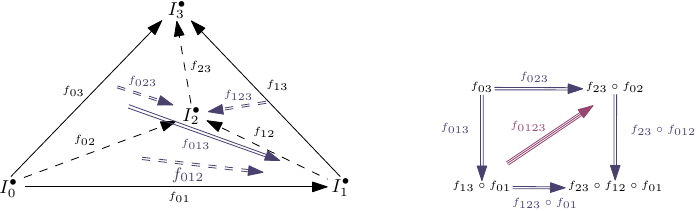}
        \end{equation}
        where the black arrows are morphisms between complexes, the \B{blue doubled arrows} are homotopies assigned to each face of the 3-simplex as in i), and the \Pur{purple tripled arrow} is $f_{0123}$ from Definition \ref{def:dgnerve}.2. Explicitly, this means that $f_{0123}\in \Hom^{-2}_{\mM}(I^{\bullet}_0,I^{\bullet}_3)$ satisfies 
        $$
        d\circ f_{0123} - f_{0123}\circ d = f_{023} + f_{23}\circ f_{012} - (f_{013} +f_{123}\circ f_{01})\,.
        $$
        In other words, the degree $-2$ morphism $f_{0123}$ is a higher homotopy that makes the diagram of homotopies on the right commute. 
        \item Consider a box diagram 
        \begin{equation}
        \label{eq:boxdiagram}
        \begin{tikzcd}
        &I_4^\bullet\arrow[dd]\arrow[rr]&&I_5^\bullet\arrow[dd, red]\\
            \arrow[dd]\arrow[ur] I^\bullet_0\arrow[rrrd, red]\arrow[rr, red]\arrow[rrru,red]&& \arrow[ur,red]I_1^\bullet \arrow[dd]\arrow[dr, red]&\\
            &I^\bullet_6\arrow[rr]&&I_7^{\bullet}
            \\
            \arrow[ur]I^\bullet_2\arrow[rr]&&\arrow[ur]I_3^\bullet&
        \end{tikzcd}
        \end{equation}
        where each face of the box is homotopy commutative. It
        contains a 2-skeleton of a 3-simplex that is spanned by the vertices $\{I_0^\bullet, I_1^\bullet,I^\bullet_5, I^\bullet_7\}$. Its morphisms are compositions of morphisms of the box and are denoted by \RX{red arrows} in the diagram. The homotopies of its faces are sums of homotopies of the faces of the box in an obvious way.  If there exists an $f_{0157}\in \Hom^{-2}_{\mM}(I_0^\bullet,I_7^\bullet)$ that turns the \RX{skeleton} into a 3-simplex, then I will say that the box is \textit{2-homotopy commutative}. Usually, one defines this notion using the 6 different 3-simpleces that a box can be decomposed into. I leave it to the reader to convince themselves that there is a correspondence between the two definitions.
    \end{enumerate}
\end{example}
After presenting Lurie's definition of stable $\infty$-categories, I will recall that $\mD^{\#}(\mM)$ is an example. 
\begin{definition}[{\cite[Definition 1.1.1.9, Proposition 1.1.3.4]{LurieHA}}]
\label{def:stableinfty}
    An $\infty$-category $\mD$ is \textit{stable} if 
    \begin{enumerate}[label=\alph*)]
        \item it containts a \textit{zero object}, i.e., an object that is simultaneously initial and final in $\mC$,
        \item it admits finite homotopy limits and finite homotopy colimits,
        \item a homotopy commutative square 
        $$
        \begin{tikzcd}
            \arrow[d]X_1\arrow[r]&X_2\arrow[d]\\
            X_3\arrow[r]&X_4
        \end{tikzcd}
        $$
        in $\mD$ is a homotopy pushout if and only if it is a homotopy pullback. 
    \end{enumerate}
\end{definition}
This definition shows that one can define stable $\infty$-categories as a special class of $\infty$-categories satisfying some natural conditions. This is contrary to the definition of triangulated categories, which require additional data in the form of an additive structure, a suspension functor, and a class of distinguished triangles. Nevertheless, I will recall in Proposition \ref{prop:stabletotriang} that the homotopy category $\Ho(\mD)$ consisting of the same objects as $\mD$ with morphisms 
$$
\Hom_{\Ho(\mD)}\big(X,Y\big) = \pi_0\big(\Map_{\mD}(X,Y)\big)\qquad \textnormal{for}\quad X,Y\textnormal{ in } \mD
$$
can be given a natural triangulated structure. 

In $\mD$, the role of distinguished triangles is played by \textit{fiber} and \textit{cofiber sequences} that are homotopy pullback, respectively pushout diagrams of the form
\begin{equation}
\label{eq:cofiber}
\begin{tikzcd}
    \arrow[d]X_1\arrow[r]&X_2\arrow[d]\\
    0\arrow[r]&X_3
\end{tikzcd}\,.
\end{equation}
The conditions in Definition \ref{def:stableinfty} imply that
\begin{enumerate}[label=\roman*)]
    \item for every morphism $X_1\to X_2$, there exists a cofiber sequence \eqref{eq:cofiber} where $X_3$ is the \textit{cone},
    \item for every morphism $X_2\to X_3$, there exists a fiber sequence \eqref{eq:cofiber} where $X_1$ it the \textit{cocone}.
\end{enumerate}
Further, the classes of fiber and cofiber sequences are identical. I will use the shortened notation $X_1\to X_2\to X_3$ to denote the (co)fiber sequence \eqref{eq:cofiber}.

The \textit{suspension functor} $\Sigma:\mD\to \mD$ is constructed by taking the cofiber sequences of morphisms $X\to 0$. These sequence are given up to equivalences by $X\to 0\to \Sigma(X)$. Because fiber sequences are the same as cofiber sequences, there is an inverse of $\Sigma$ called the \textit{loop functor} $\Omega$ that acts on $Y$ by taking a cocone of $0\to Y$. In particular, both $\Sigma$ and $\Omega$ are autoequivalences of $\mD$.

Let us return to the example $\mD^{\#}(\mM)$ and discuss why it is a stable $\infty$-categories. I will not write out the proof of this as it can be found in \cite[§1.3.2]{LurieHT}. Instead, I will describe the zero object, the (co)fiber sequences, and the suspension functor. This is done to later explain the connection with the usual structure of corresponding triangulated categories. 
\begin{proposition}[{\cite[Proposition 1.3.2.10, Corollary 1.3.2.18]{LurieHT}}]
\label{prop:mDstable}
    Both $\mD^+(\mM)$ and $\mD^b(\mM)$ are stable infinity categories. 
\end{proposition}
\begin{proof}
    As promised, I will only describe the data that will be important later on. In fact, the existence of (co)fiber sequences and their equivalence can replace ii) and iii) in Definition \ref{def:stableinfty} by \cite[Proposition 1.1.3.4]{LurieHT}.
    \begin{enumerate}[label=\alph*)]
        \item Zero objects are the acyclic complexes.
        \item Starting from the morphism $            E_1^\bullet\xrightarrow{f}E^\bullet_2$,
        one can construct its cofiber sequence as the diagram
        $$
        \begin{tikzcd}
        \arrow[d]E_1^\bullet\arrow[r,"f"]& E^\bullet_2\arrow[d] \\
           C^\bullet(\id_{E^\bullet_1})\arrow[r] &C^\bullet(f)
        \end{tikzcd}
        $$
        where $C^\bullet(-)$ denotes the cone of the corresponding morphism of complexes, and all arrows are the natural ones. We can also construct the fiber sequences of $f$ as  
          $$
        \begin{tikzcd}
       C^\bullet(f)[-1] \arrow[d]\arrow[r]& E_1^\bullet\arrow[d,"f"] \\
           C^\bullet(\id_{E^\bullet_2})[-1]\arrow[r] &E^\bullet_2
        \end{tikzcd}\,.
        $$
        In both cases, one uses that $C^\bullet(\id_{E^\bullet})$ is homotopy equivalent to $0$. 
        \item As a special case of the above fiber sequence, one has
                $$
        \begin{tikzcd}
        \arrow[d]E^\bullet\arrow[r,]& 0\arrow[d] \\
           C^\bullet(\id_{E^\bullet})\arrow[r] &E^\bullet[1]
        \end{tikzcd}
        $$
        which shows that the suspension functor is the usual shift functor.
    \end{enumerate}
    Because these constructions respect boundedness conditions on cohomologies of complexes, they apply to both $\mD^+(\mM)$ and $\mD^b(\mM)$.
\end{proof}
\subsection{Functoriality and uniqueness of cones}
\label{Asec:functoriality}
One benefit of working with stable $\infty$-categories is that cones and cocones can be constructed uniquely up to contractible choices and $\infty$-functorially. The uniqueness is a consequence of Lurie's \cite[Proposition 4.3.2.15]{LurieHT} that is stated for a class of $\infty$-functors between some $\infty$-categories $\mC$ and $\mD$.

The set of vertices of the simplicial set $\Fun(\mC,\mD)$ introduced in \cite[Notation 1.2.7.2]{LurieHT} is formed by the $\infty$-functors from $\mC$ to $\mD$, and it is itself an $\infty$-category. If $\mC$ is a small simplicial set and $\mD$ is stable, then $\Fun(\mC,\mD)$ is stable by \cite[Proposition 1.1.3.1]{LurieHA}. This is what eventually implies functoriality of cones, because we can express diagrams of objects and morphisms in $\mD$ as such functors, and natural transformations between these functors admit (co)cones that can be constructed object-wise by \cite[Proposition 7.1.7.2]{kerodon}.
\begin{example}
\label{ex:functomD}
    \begin{enumerate}[label=\roman*)]
        \item If $\mC =\Delta^0$ is just the zero-simplex, then $\Fun(\mC,\mD)=\mD$.
        \item The vertices of $\Fun(\Delta^1,\mD)$ are morphisms $X_1\to X_2$ in $\mD$. I will denote the full $\infty$-subcategory of equivalences $X_1\xrightarrow{\sim}X_2$ by $\Fun^{\sim}(\Delta^1,\mD)$.
        \item The 1-simplices of $\Fun(\Delta^1,\mD)$ are given by functors in $\Fun(\Delta^1\times \Delta^1,\mD)$ which are in turn represented by homotopy-commutative diagrams
        \begin{equation}
        \label{eq:IxI}
        \begin{tikzcd}
            \arrow[d,"g_1"]X_1\arrow[r,"f_1"]&X_2\arrow[d,"f_2"]\\
            Y_1\arrow[r,"g_2"]&Y_2
        \end{tikzcd}
        \end{equation}
        in $\mD$. 
        When $\mD$ is a stable $\infty$-category, I will write $\Fib(\mD)\subset \Fun(\Delta^1\times \Delta^1,\mD)$ for the full $\infty$-subcategory of (co)fiber sequences in $\mD$ with $Y_1\simeq 0$.
        \item I will work with the full $\infty$-subcategory $\Fun\big(\Delta^1,\Fib(\mD)\big)$ of $\Fun(\Delta^1\times\Delta^1\times\Delta^1,\mD)$ such that the last  pair $\Delta^1\times\Delta^1$ corresponds to (co)fiber sequences. In other words, the vertices of this simplex are diagrams
        \begin{equation}
        \label{eq:IxIxI}
        \begin{tikzcd}
X_1\arrow[d]\arrow[r]&X_2\arrow[d]\arrow[r]&\arrow[d]X_3\\
            Y_1\arrow[r]&Y_2\arrow[r]&Y_3
        \end{tikzcd}
        \end{equation}
       in $\mD$ where each horizontal sequence of arrows is a (co)fiber sequence. Note that this diagram is only 2-homotopy commutative by Example \ref{ex:simplices}.iv).
       \item The most important type of diagrams for this work are the \textit{$3\times 3$ diagrams}
       \begin{equation}
       \label{eq:IxIxIxI}
        \begin{tikzcd}
X_1\arrow[d]\arrow[r]&X_2\arrow[d]\arrow[r]&\arrow[d]X_3\\
            \arrow[d]Y_1\arrow[r]&\arrow[d]Y_2\arrow[r]&\arrow[d]Y_3\\
Z_1\arrow[r]&Z_2\arrow[r]&Z_3
        \end{tikzcd}\,.
        \end{equation}
        where each horizontal and vertical sequence is a (co)fiber sequence in the stable $\infty$-category $\mD$. The appropriate homotopies of the diagram are described by it being a full $\infty$-subcategory of $\Fun\big((\Delta^1\times\Delta^1)^{\times 2},\mD\big)$, which I denote by $\Diag^{3\times 3}(\mD)$. 
        \item Set $\mD^{\Delta^1} := \Fun(\Delta^1,\mD)$
         and  $\mD^{\Delta^1\times \Delta^1} := \Fun(\Delta^1\times \Delta^1,\mD)$. The above discussion of examples of diagrams also applies to these stable $\infty$-categories. 
    \end{enumerate}
   One can start from \eqref{eq:IxI} and take cones in $\Fun(\Delta^1,\mD)$ along the horizontal arrows to obtain \eqref{eq:IxIxI}. The resulting diagram will lie in $\Fun\big(\Delta^1,\Fib(\mD)\big)$, because (co)limits in the $\infty$-category of functors are constructed objects-wise by \cite[Proposition 7.1.7.2]{kerodon}. Taking cones in $\Fun(\Delta^1\times \Delta^1\times \Delta^1, \mD)$ along the vertical direction in the diagram \eqref{eq:IxIxI} produces \eqref{eq:IxIxIxI}. To make sure that the bottom row is also a (co)fiber sequence, one uses \cite[Chapter 7, Corollary 7.3.8.20]{kerodon}, which implies that taking homotopy (co)limits is independent of order. Therefore, $Z_1\to Z_2$ is the cone of \eqref{eq:IxI} along the vertical direction, and $Z_1\to Z_2\to Z_3$ is the associated cofiber sequence. 
\end{example}
Equipped with the above terminology, one can formulate the precise uniqueness and functoriality statements for cones. This also shows that the procedure for constructing diagrams \eqref{eq:IxIxIxI} out of \eqref{eq:IxI} only depends on choices that form contractible simplicial sets. Note that the proof is standard and well-known, but I explain it in one case, to make the statement itself easier to digest.
\begin{proposition}
\label{prop:contractiblesec}
    Let $\mD$ be a stable $\infty$-category, then the following maps of simplicial sets have a contractible, therefore non-empty, space of sections:
    \begin{enumerate}[label=\roman*)]
    \item the $\infty$-functor $\Fun^{\sim}(\Delta^1,\mD)\to \mD$ mapping each equivalence $X_1\xrightarrow{\sim}X_2$ to $X_1$,
        \item the $\infty$-functor $\Fib(\mD)\to \Fun(\Delta^1,\mD)$ mapping $X_1\to X_2\to X_3$ to $X_1\to X_2$,
           \item the $\infty$-functor $\Fib(\mD)\to \Fun(\Delta^1,\mD)$ mapping $X_1\to X_2\to X_3$ to $X_2\to X_3$,
        \item the $\infty$-functor $\Fun\big(\Delta^1, \Fib(\mD)\big)\to \Fun(\Delta^1\times \Delta^1,\mD)$ mapping \eqref{eq:IxIxI} to $\eqref{eq:IxI}$,
        \item the $\infty$-functor $\Diag^{3\times 3}(\mD)\to \Fun\big(\Delta^1, \Fib(\mD)\big)$ mapping \eqref{eq:IxIxIxI} to \eqref{eq:IxIxI} or any other choice of a $2\times 3$ sub-diagram. 
    \end{enumerate}
This applies also to $\mD^{\Delta^1}$ and $\mD^{\Delta^1\times \Delta^1}$ introduced in Example \ref{ex:functomD}.vi).
\end{proposition}
\begin{proof}
    I will use Lurie's terminology from \cite[Example 2.0.0.1, 2.0.0.2]{LurieHT} which defines morphisms between simplicial sets called \textit{trivial Kan fibrations}. If $\mC\to \mD$ is a trivial Kan fibration with non-empty fibers, then its space of sections is contractible. The rest is just interpreting {\cite[Proposition 4.3.2.15]{LurieHT}} correctly in the above settings. I will only discuss ii) here because the other cases follow analogously.

    Using the notation of \cite[Proposition 4.3.2.15]{LurieHT}, set 
    $$
    \mC = \Delta^1\times \Delta^1\,,\quad\mC^0 = \Delta^1\times\{0\}\cup \{0\}\times \Delta^1\subset \mC\,,\quad  \mD' = \{0\}\,.
    $$
Then for a fixed $\mD$, the functors $\Fun(\mC,\mD)$ are the diagrams \eqref{eq:IxI}, while $\Fun(\mC^0,\mD)$ are the diagrams
\begin{equation}
\label{eq:I+I}
\begin{tikzcd}
    \arrow[d]X_1\arrow[r]&X_2\\
    Y_1&
\end{tikzcd}\,.
\end{equation}
The assumptions of the cited proposition relate \eqref{eq:IxI} to \eqref{eq:I+I} as its homotopy pushout. Consider the restriction map $\Fun(\mC,\mD)\to \Fun(\mC^0,\mD)$, then \cite[Proposition 4.3.2.15]{LurieHT} states that its restriction to the full $\infty$-subcategory consisting of homotopy pushout diagrams is a trivial Kan fibration. 

 Consider the embedding $\textnormal{emb}:\Fun(\mC^0,\mD)_0\hookrightarrow \Fun(\mC^0,\mD)$ of the full $\infty$-subcategory consisting of diagrams \eqref{eq:I+I} with $Y_1=0$. The homotopy pull-back of a trivial Kan fibration along emb is again a trivial Kan fibration, so the statement follows.
\end{proof}

\begin{remark}
\label{rem:choosingresols}
    Fix an isomorphism class of an object $[X]$ in $\Ho(\mD)$, then one may wonder whether a construction using a lift of $X$ to $\mD$ depends on the choice made. In the case of $\mD=\mD^{\#}(\mM)$, this is equivalent to finding an injective resolution. Fixing one lift $X$, all other lifts $Y$ admit an equivalence $X\xrightarrow{\sim} Y$ in $\mD$.  Such objects $Y$ correspond to the vertices of a simplicial set which is the fiber over $X$ of the $\infty$-functor $\Fun^{\sim}(\Delta^1,\mD)\to \mD$ from Proposition \ref{prop:contractiblesec}.i). Therefore, it is a contractible choice. 
\end{remark}
Lastly, I will sketch why the homotopy category $\Ho(\mD)$ of a stable $\infty$-category is a triangulated category. Coming back to the original examples $\mD^{\#}(\mM)$, this will recover the equivalence of triangulated categories
\begin{equation}
\label{eq:mDtoD}
\Ho\big(\mD^{\#}(\mM)\big) \simeq D^{\#}(\mM)\qquad\textnormal{for}\quad \# = +,b\,.
\end{equation} 
\begin{proposition}[{\cite[Theorem 1.1.2.14]{LurieHT}}]
\label{prop:stabletotriang}
    Let $\mD$ be a stable $\infty$-category and $\Ho(\mD)$ its homotopy category. Then $\Ho(\mD)$ is a triangulated category with
    \begin{enumerate}
        \item the additive structure determined uniquely by $\mD$, such that $\Ho(\mD)$ inherits the zero object, 
        \item the suspension functor induced by the suspension $\infty$-functor of $\mD$,
        \item the class of distinguished triangles determined by sequences of {\color{red}red arrows} in diagrams of the form
        $$
        \begin{tikzcd}
\arrow[d]X_1\arrow[r,red]&\arrow[d,red]X_2\arrow[r]&0\arrow[d]\\
   0\arrow[r]&  X_3\arrow[r,red]&\Sigma(X_1)       
        \end{tikzcd}
        $$
        where each rectangle, including the exterior one, is a (co)fiber sequence.
    \end{enumerate}
  
\end{proposition}
\begin{proof}
I will only discuss the octahedral axiom here, to get the reader accustomed to working with stable $\infty$-categories. Starting from the morphisms $X_1\to X_2\to X_3$ in $\mD$, it is a good exercise in using the functorial properties of cones discussed in Example \ref{ex:functomD} to construct the homotopy commutative diagram 
    $$
    \begin{tikzcd}
        \arrow[d]X_1\arrow[r]&\arrow[d]X_2\arrow[r]&\arrow[d]X_3\arrow[r]&\arrow[d]0&\\
        0\arrow[r]&\arrow[d]X_2/X_1\arrow[r]&\arrow[d]X_3/X_1\arrow[r]&\arrow[d]\Sigma(X_1)\arrow[r]&\arrow[d]0\\
        &0\arrow[r]&X_3/X_2\arrow[r]&\Sigma(X_2)\arrow[r]&\Sigma(X_2/X_1)
    \end{tikzcd}
    $$
    such that each rectangle is a homotopy pushout. Using the second point of the proposition, one recovers the octahedral axiom in $\Ho(\mD)$. It states that starting from the commutative diagram
$$
\begin{tikzcd}[column sep={5.5em,between origins},
  row sep={4.2em,between origins}, ampersand replacement=\&]
    \eqmathbox{X_1}\arrow[dr]\arrow[rr, bend left = 35]\&\& \eqmathbox{X_3}\\
\&\eqmathbox{X_2}\arrow[ur]\&
\end{tikzcd}\,,
    $$
  in $\Ho(\mD)$, one can construct
$$
\begin{tikzcd}[column sep={5.5em,between origins},
  row sep={4.2em,between origins}, ampersand replacement=\&]
    \eqmathbox{X_1}\arrow[dr]\arrow[rr, bend left = 38]\&\& \eqmathbox{X_3}\arrow[dr] \arrow[rr, bend left =38, ]\&\&\eqmathbox{X_3/X_2}\arrow[dr]\arrow[rr,bend left = 38]\&\&\eqmathbox{(X_2/X_1)[1]}\\
\&\eqmathbox{X_2}\arrow[ur]\arrow[dr]\&\&\eqmathbox{X_3/X_1}\arrow[ur]\arrow[dr]\&\&\eqmathbox{X_2[1]}\arrow[ur]\&\\
\&\&\eqmathbox{X_2/X_1}\arrow[rr, bend right = 35]\arrow[ur, dashed]\&\&\eqmathbox{X_1[1]}\arrow[ur]\&\&
\end{tikzcd}\,.
$$
Here, each curved line consists of distinguished triangles.
\end{proof}
When $\mD =\mD^{\#}(\mM)$, he morphisms satisfy 
$$
\Hom_{D^{\#}(\mM)}\big(I^\bullet_1,I^\bullet_2\big) \cong H^0\big(\Hom^\bullet(I^\bullet_1,I^\bullet_2)\big)
$$
in a natural way. Using the description of the shift functor and the class of distinguished triangles in Proposition \ref{prop:stabletotriang} combined with the data in the proof of Proposition \ref{prop:mDstable}, this reproduces one of the possible constructions of $D^{\#}(\mM)$ as a triangulated category. Here I used that $\mM$ has enough injectives in its category of quasi-coherent sheaves -- see \cite[Proposition 96.15.2]{stacks-project}. Therefore, the proof of \eqref{eq:mDtoD} follows.
\begin{remark}
 Starting from a commutative diagram 
 \begin{equation}
 \label{eq:DIxI}
 \begin{tikzcd}
     \arrow[d]X_1\arrow[r]&X_2\arrow[d]\\
     Y_1\arrow[r]&Y_2
 \end{tikzcd}
 \end{equation}
 in $D^{\#}(\mM)$. one can extend it to 
  $$
 \begin{tikzcd}
     \arrow[d]X_1\arrow[r]&X_2\arrow[d]\arrow[r]&\arrow[d, red]X_2/X_1\arrow[r]&\arrow[d]X_1[1]\\
     Y_1\arrow[r]&Y_2\arrow[r]&Y_2/Y_1\arrow[r]&Y_1[1]
 \end{tikzcd}
 $$
 where the {\color{red}red  arrow} may not be unique. If one wants to do a computations in $D^{\#}(\mM)$ that requires one to show that two choices of this morphism are the same, one may instead try to find a lift of \eqref{eq:DIxI} to $\mD^{\#}(\mM)$ which produces by Proposition \ref{prop:contractiblesec} a unique up to contractible choices diagram \eqref{eq:IxIxI}. Completing the computation in $\mD^{\#}(\mM)$, one can recover the required statement in $D^{\#}(\mM)$ by \eqref{eq:mDtoD}.
\end{remark}
\subsection{Derived functors}
Exact functors of stable $\infty$-categories are refinements of the exact functors of triangulated categories as can be seen from Proposition \ref{prop:stabletotriang} and the following definition.
\begin{definition}
An $\infty$-functor $F:\mD_1\to \mD_2$ between stable $\infty$-categories is \textit{exact} if it preserves finite homotopy limits and finite homotopy colimits. By \cite[Proposition 1.1.4.1]{LurieHA}, this is equivalent to $F$ preserving either of them or even just the fiber sequences or cofiber sequences. 
\end{definition}
The notion of left and right exactness of functors between abelian categories has its analog in stable $\infty$-categories. It depends on the choice of a t-structure, which is defined in terms of the triangulated homotopy category.
\begin{definition}
\sloppy    A pair $(\mD^{\leq 0},\mD^{\geq 0})$ of full $\infty$-subcategories of $\mD$ is a \textit{t-structure} if $\big(\Ho(\mD^{\leq 0}),\Ho(\mD^{\geq 0})\big)$ is a t-structure of $\Ho(\mD)$ as a triangulated category. Setting $\mD^{\leq n} = \Omega^{n}(\mD^{\leq 0})$, one says that the t-structure is right complete if $\mD$ can be expressed as the homotopy limit of the infinite sequence of morphisms
$$
\cdots\longrightarrow\mD^{\geq n}\longrightarrow\mD^{\geq n+1}\longrightarrow\mD^{\geq n+2}\longrightarrow\cdots \,.
$$

    Choose t-structures $(\mD_i^{\leq 0},\mD_i^{\geq 0})$ of stable $\infty$-categories $\mD_i$ for $i=1,2$. An exact functor $F:\mD_1\to \mD_2$ is said to be 
    \begin{enumerate}[label=\roman*)]
        \item \textit{left t-exact} with respect to the chosen t-structures if $F(\mD^{\geq 0}_1)\subset \mD^{\geq 0}_2$,
        \item \textit{right t-exact} with respect to the chosen t-structures if $F(\mD^{\leq 0}_1)\subset \mD^{\leq 0}_2$.
    \end{enumerate}
\end{definition}
The heart $\mD^{\heartsuit} = \mD^{\leq 0}\cap \mD^{\geq 0}$ of the t-structure $(\mD^{\leq 0},\mD^{\geq 0})$ in $\mD$ is an abelian category as it coincides with the heart of the associated t-structure on $\Ho(\mD)$ by \cite[Remark 1.2.1.12]{LurieHA}.  Left and right t-exact $\infty$-functors can be obtained, respectively, from left and right exact functors between hearts of t-structures.

Consider again $\mD^{\#}(\mM)$. There is the natural t-structure $\big(\mD^{\#,\leq 0}(\mM),\mD^{\#,\geq 0}(\mM)\big)$ defined as follows:
\begin{enumerate}
    \item $\mD^{\#,\leq 0}(\mM)$ consists of complexes $I^\bullet$ in $\mD^{\#}(\mM)$ satisfying $H^i(I^\bullet)=0$ for $i>0$,
    \item $\mD^{\#,\geq 0}(\mM)$ consists of complexes $I^\bullet$ in $\mD^{\#}(\mM)$ satisfying $H^i(I^\bullet)=0$ for $i<0$.
\end{enumerate}
Its heart is precisely the category of quasi-coherent sheaves $\QCoh(X)$. Fixing a t-structure $(\mD^{\leq 0},\mD^{\geq 0})$ of another stable $\infty$-category $\mD$, one has the following result.
\begin{proposition}[{\cite[Theorem 1.3.3.2]{LurieHA}}]
\label{prop:exactfun}
    Consider the full $\infty$-subcategory $$\mE\subset \Fun\big(\mD^+(\mM), \mD\big)$$ consisting of left t-exact functors that map injective sheaves in $\QCoh(X)$ to objects in $\mD^{\heartsuit}$ and the category $\mE^{\heartsuit}$ of left exact functors between the abelian categories $\QCoh(X)$ and $\mD^{\heartsuit}$. If $(\mD^{\leq 0},\mD^{\geq 0})$ is right complete, then the map
    $$
    F\mapsto \tau^{\leq 0}\circ F|_{\QCoh(X)}\,,
    $$
    where $\tau^{\leq 0}: \Ho(\mD)\mapsto \Ho\big(\mD^{\leq 0}\big)$ is the standard truncation functor, induces an equivalence $\xi:\mE\xrightarrow{\sim} \mE^{\heartsuit}$. 
\end{proposition}
\begin{proof}
   While the full proof in \cite[§1.3.3]{LurieHA} is lengthy, the main idea is simple. One needs to construct an inverse of $\xi$ which should map any left exact functor $G: \QCoh(X)\to \mD^{\heartsuit}$ to a functor acting on complexes 
$$
\cdots\longrightarrow 0\longrightarrow I^{-n}\longrightarrow I^{-n+1}\longrightarrow I^{-n+2}\longrightarrow  \cdots
$$
where $I^k$ are all injective. The value of this functor can be defined inductively. Consider the truncated complex
$$
I^\bullet_m = (\cdots\longrightarrow 0\longrightarrow  I^{-n}\longrightarrow I^{-n+1}\longrightarrow \cdots \longrightarrow I^{m-1}\longrightarrow I^{m} \longrightarrow 0\longrightarrow \cdots)  
$$
for $m> -n$. Then there is a fiber sequence $I^\bullet_{m}\to I^\bullet_{m-1}\to I^{m}[-m+1]$ which needs to be preserved by $\xi^{-1}(G)$ because this functor should be exact. By induction, one can reconstruct $\xi^{-1}(G)(I^\bullet_{m})$ from the fiber sequence 
$$
\begin{tikzcd}
\xi^{-1}(G)(I^\bullet_{m})\arrow[r]& \xi^{-1}(G)(I^\bullet_{m-1})\arrow[r]& \xi^{-1}(G)(I^{m})[-m]\,.
\end{tikzcd}
$$
To include $I^\bullet$ that are not bounded from above, one takes homotopy limits along the natural maps $\cdots \to \xi^{-1}(G)(I^\bullet_{m+2})\to \xi^{-1}(G)(I^\bullet_{m+1})\to \xi^{-1}(G)(I^\bullet_m)\to \cdots$ induced by the projections. 
\end{proof}
Suppose that $f: \mN\to \mM$ is a quasi-compact quasi-separated morphism of Artin stacks, then $f_*: \QCoh(\mM)\to \QCoh(\mM)$ is a left-exact functor of abelian categories. As such, it induces the left t-exact $\infty$-functor 
$$
\begin{tikzcd}
Rf_*: \mD^{+}(\mN)\longrightarrow \mD^+(\mM)
\end{tikzcd}
$$
by applying Proposition \ref{prop:exactfun}. 

I will be working with two more $\infty$-functors between stable derived categories. Both of them are constructed separately here because there are not enough projectives in $\QCoh(\mM)$ in general. 
\begin{example}
\label{ex:inftyfunc}
    \begin{enumerate}[label=\roman*)]
    \item \textit{(Derived $\infty$-pullback)}
    
    
In \cite[Chapter I.3]{GR1}, the authors constructed $\mD^{\#}(\mM)$ in an alternative way and related it to the present definition in \cite[Proposition 2.4.3 of Chapter I.3]{Gait}. Their construction also produces the functor $Lf^*$ immediately from the definition. Since $\Ho(Lf^*)$ is left adjoint to $\Ho(Rf_*)$ by \cite[Proposition 5.2.2.9]{LurieHT}, and the latter is the correct derived pushforward $D^{\#}(\mM)\to D^{\#}(\mN)$, one can see that $\Ho(Lf^*)$ is also the usual derived pullback.
        \item \textit{(Derived $\infty$-dual)} For an Artin stack $\mM$, I will need the $\infty$-categorical version of the derived dual $(-)^\vee:D^b(\mM)\to D^b(\mM)$. For this, fix an injective resolution $O^\bullet$ of $\mO_{\mM}$ which is a contractible choice in $\mD^b(\mM)$ by Remark \ref{rem:choosingresols}. For any $I^\bullet$ in $\mD^{b}(\mM)$, set $\mHom^\bullet_{\mM}(I^\bullet,\mO_{\mM})$ to be the total complex of the double complex $\mHom_{\mM}(I^\bullet,O^\bullet)$. This defines a dg-functor $\mC^b(\mM)\to \mC(\mM)^b$ where the latter is the dg-category of all quasi-coherent complexes on $\mM$ with bounded cohomology. Setting $\mD(\mM)^b = N_{\dg}\big(\mC(\mM)^b\big)$ in terms of the dg-nerve recalled in Definition \ref{def:dgnerve}, the above induces an $\infty$-functor $\mHom^\bullet_{\mM}(-,\mO_{\mM}):\mD^b(\mM)\to \mD(\mM)^b$. There is, moreover, the functor $\mD(\mM)^b\to \mD^b(\mM)$ corresponding to taking injective resolutions \footnote{The existence and uniqueness can be shown similarly to Proposition \ref{prop:exactfun}.}. The derived $\infty$-dual 
        \begin{equation}
        \label{eq:inftyvee}
        (-)^\vee: \mD^b(\mM)\to \mD^b(\mM)
        \end{equation}
        is obtained by composing with this functor. By recalling the description of (co)fiber sequences from the proof of Proposition \ref{prop:mDstable}, one sees that $\mHom^{\bullet}_{\mM}(-,\mO_{\mM})$ is coexact. Thus, I have constructed a coexact $\infty$-functor \eqref{eq:inftyvee} that induces the classical derived dual on $D^b(\mM)$.
    \end{enumerate}
\end{example}
Having set up the necessary language for using stable $\infty$-categories, I will lift problems from $D^{\#}(\mM)$ to $\mD^{\#}(\mM)$ where one can often solve them. The following explains this procedure more rigorously.
\begin{definition}
\label{def:lifting}
Given a diagram of morphisms and distinguished triangles in $D^{\#}(\mM)$, I will say that a diagram in $\mD^b(\mM)$ is the \textit{lift} of the original diagram if
\begin{enumerate}
    \item any commutative part of the diagram in $D^{\#}(\mM)$ is replaced by a homotopy commutative one in $\mD^{\#}(\mM)$ (including higher homotopies),
    \item any distinguished triangle is replaced by a (co)fiber sequence,
    \item quasi-isomorphisms are replaced by homotopy equivalences,
    \item derived duals of complexes and morphisms are replaced by derived $\infty$-duals,
    \item shifts of complexes are replaced by repeated applications of loop or suspension functors on the lifts of complexes,
    \item projecting back to $D^{\#}(\mM)$ recovers the original diagram.
\end{enumerate}
I will often use the same notation for the object in $D^b(\mM)$ and its lift. 
\end{definition}
\subsection{The construction of $\infty$-PVP diagrams}
\label{Asec:inftyPvp}
From now on, I will focus on constructing the lift of \eqref{eq:diagsym} to $\mD^b(\mN)$, which I will call the \textit{ $\infty$-Pvp} diagram. The idea is to start from a small amount of data in $\mD^b(\mN)$ to construct this diagram. 
\begin{proposition}
\label{prop:sympullback}
  Let $f:\mN\to \mM$ be a quasi-smooth morphism of stacks with a perfect obstruction theory $\MM\xrightarrow{\tau}\LL_f$ and let $\EE\xrightarrow{\psi} \LL_{\mM}$ be a CY4 obstruction theory with higher self-duality. Suppose further that the bottom morphism of the horizontal distinguished triangles in \eqref{eq:diagsym} exists and that $\eta$ admits a lift to $\mD^b(\mN)$ denoted by \footnote{Here $f^*(\EE)$ can be taken to be the derived $\infty$-pullback of a lift of $\EE$.}
   $$
\eta^{\wedge}: \MM[-1]\longrightarrow f^*(\EE)\,.
   $$
Using the notation $\ov{\eta}^{\wedge}:f^*(\EE)\cong f^*(\EE)^\vee[2]\longrightarrow \MM^\vee[3]$ to denote the derived  $\infty$-dual of $\eta^{\wedge}$ in $\mD^b(\mN)$, assume that
   \begin{equation}
   \label{eq:startingIxI}
\begin{tikzcd}
\arrow[d]\MM[-1]\arrow[r,"\eta^{\wedge}"]&f^*(\EE)\arrow[d,"\ov{\eta}^{\wedge}"]\\
  0\arrow[r]&\MM^\vee[3]
\end{tikzcd}
\end{equation}
   is homotopy commutative in $\mD^b(\mN)$ and the homotopy is invariant under $(-)^\vee[2]$. Then there exists a unique up to contractible choices complex $\FF\in \mD^b(\mN)$ which completes the diagram \eqref{eq:diagsym} and $\phi\circ \mu:\FF\to \LL_{\mN}$ is a CY4 obstruction theory with higher self-duality.

 Additionally, if the commutative diagram
\begin{equation}
\label{eq:relativesquareininfty}
\begin{tikzcd}
\arrow[d,"\tau"]\MM[-1]\arrow[r,"\eta"]&f^*(\EE)\arrow[d]\\
  \LL_{f}[-1]\arrow[r]&f^*(\LL_{\mM})
\end{tikzcd}
\end{equation}
can be lifted to $\mD^b(\mN)$, then there exists a lift of \eqref{eq:diagsym} unique up to contractible choices. 
\end{proposition}
\begin{proof}
  I will only provide the proof of the first statement, because the second one is simpler. Replacing the diagram in \eqref{eq:IxI} with \eqref{eq:startingIxI}, one can follow the construction at the end of Example \ref{ex:functomD} step-by-step. Except that this time, one takes cofiber sequences when going from \eqref{eq:IxIxI} to \eqref{eq:IxIxIxI}. The resulting analog of \eqref{eq:IxIxIxI} is given by 
\begin{equation}
\label{eq:Parkhomotopylift}
        \begin{tikzcd}
\MM[-1]\arrow[d,equal]\arrow[r,"\lambda^{\wedge}"]&\wt{\FF}^\vee[2]\arrow[d,"\ov{\kappa}^{\wedge}"]\arrow[r,"\ov{\mu}^{\wedge}"]&\arrow[d,"\mu^{\wedge}"]\FF\\
            \arrow[d]\MM[-1]\arrow[r,"\eta^{\wedge}"]&\arrow[d,"\ov{\eta}^{\wedge}"]f^*(\EE)\arrow[r,"\kappa^{\wedge}"]&\arrow[d, "\ov{\lambda}^{\wedge}"]\wt{\FF}\\
0\arrow[r]&\MM^\vee[3]\arrow[r,equal]&\MM^\vee[3]
        \end{tikzcd}
\end{equation}        
in $\mD^b(\mN)$. Because  \eqref{eq:startingIxI} is preserved under $(-)^\vee[2]$, the uniqueness of cones in $\mD^b(\mN)$ and $(-)^\vee$ being a co-exact functor by Example \ref{ex:inftyfunc} imply that the full $3\times 3$-diagram is self-dual. This constructs higher self-duality of $\FF$ in $\mD^b(\mN)$. Projecting to $D^b(\mN)$, one obtains \eqref{eq:diagsym}. 

The virtual admissibility of $\FF\xrightarrow{\phi\circ \mu} \LL_{\mN}$ will follow from Lemma \ref{lem:conditionsFF} below.
\end{proof}
This is just a generalization of \cite[Lemma 2.3]{Park} based on the argument following \cite[(127)]{BKP}.
\begin{lemma}
\label{lem:conditionsFF}
    The morphism $\phi\circ \mu: \FF\to \LL_{\mN}$ from Proposition \ref{prop:sympullback} is an orientable obstruction theory satisfying the isotropy of cones condition. Moreover, if $\EE$ is even, then so is $\FF$. 
\end{lemma}
\begin{proof}
    I will begin by proving that $\phi\circ \mu$ defines an obstruction theory. Taking the long exact sequence of cohomologies associated to 
    $$
    \begin{tikzcd}
  \MM^{\vee}[2]\arrow[r]& \FF\arrow[r,"\mu"]&\wt{\FF}\arrow[r]&\MM^{\vee}[3]\,,
  \end{tikzcd}
    $$
  produces isomorphisms
  $$
  h^{1}(\FF) \xrightarrow{\sim} h^{1}(\FFna)\quad \textnormal{and} \quad  h^{0}(\FF) \xrightarrow{\sim} h^{0}(\FFna)
  $$
  and the surjection $h^{-1}(\FF) \twoheadrightarrow h^{-1}(\FFna)$.
As $\FFna\to \LL_{\mN}$ is an obstruction theory, so is $\FF$. 

The property that $\FF$ is orientable and even if $\EE$ is follows immediately from the diagram \eqref{eq:diagsym} leaving me to check the isotropy property. To show it, I begin with the diagram
$$
\begin{tikzcd}
 \arrow[d] \bTot_{\mN}\big(\LL_{\mN}^\vee[1]\big) \arrow[r]& \bTot_{\mN}\big(\wt{\FF}^\vee[1]\big)\arrow[d]\arrow[r]& \bTot_{\mN}\big(\FF^\vee[1]\big)\arrow[d,"q_{\FF}"]\\
   \pi^*\bTot_{\mM}\big(\LL_{\mM}^\vee[1]\big)\arrow[r] &\pi^*\bTot_{\mM}\big(\EE^\vee[1]\big)\arrow[r,"q_{\EE}"]&\CC
\end{tikzcd}\,.
$$
The right square is commutative because of 
$$
\begin{tikzcd}
\arrow[d,"\ov{\kappa}"]\FFna^\vee[2]\arrow[r,"{\ov{\mu}}"]&\FF\arrow[d,"\mu"]\\
f^*\big(\EE\big)\arrow[r,"\kappa"]&\FFna
\end{tikzcd}
$$
being self-dual. Applying the functor $t_0(-)$, on is left with the commutative diagram
$$
\begin{tikzcd}[column sep=large]
 \arrow[d] \mathfrak{C}_{\mN}\arrow[r]& \mathfrak{C}_{\mN}\big(\FFna\big)\arrow[d]\arrow[r]& \mathfrak{C}_{\mN}\big(\FF\big)\arrow[d,"q_{\FF}"]\\
   \pi^*\mathfrak{C}_{\mM}\arrow[r] &\pi^*\mathfrak{C}_{\mM}\big(\EE\big)\arrow[r,"q_{\EE}"]&\CC
\end{tikzcd}
$$
which shows that $\FF$ satisfies the isotropy condition since $\EE$ does.
\end{proof}
\subsection{Functoriality of $\infty$-PVP diagrams}
\label{sec:functorialitysympull}
The diagram \eqref{eq:relativesquareininfty} in $\mD^b(\mN)$ can be obtained by working with derived stacks. Take $\boldsymbol{f}: \bmN\to \bmM$ to be a derived enrichment of $f: \mN\to \mM$ such that
\begin{equation}
\label{eq: derivedpi}
\begin{tikzcd}
\mN\arrow[d]\arrow[r,"f"]& \mM\arrow[d]\\
   \bmN\arrow[r, "\boldsymbol{f}"] &\bmM  
\end{tikzcd}
\end{equation}
is a commutative diagram of derived stacks, and $\boldsymbol{f}$ is quasi-smooth. By \cite[Prop. 3.2.12]{LurieDAG}, there exists a morphism of fiber sequences
$$
\begin{tikzcd}
\arrow[d]\boldsymbol{f}^*\LL_{\bmM}|_{\mN}\arrow[r]& \arrow[d]\LL_{\bmN}|_{\mN}\arrow[r]& \arrow[d]\LL_{\boldsymbol{f}}|_{\mN}\\
f^*\LL_{\mM}\arrow[r]&\LL_{\mN}\arrow[r]&\LL_f
\end{tikzcd}
$$
in $\mD^b(\bmN)$. Note that I am now using $\LL_{(-)}$ to denote the untruncated cotangent complex complex. Setting $\EE=\LL_{\bmM}|_{\mM}$, $\wt{\FF} =\LL_{\bmN}|_{\mN} $, and $\MM=\LL_{\boldsymbol{f}}|_{\mN}$, one obtaines the usual virtual pullback diagram containing \eqref{eq:relativesquareininfty}. This was used in \cite{JoyceWC} when proving wall-crossing in lower dimensions.

One may consider a situation with the following commutative triple of morphisms of derived stacks:
\begin{equation}
\label{eq:ffactor}
 \begin{tikzcd}
     \boldsymbol{\mN_2}\arrow[r,"\boldsymbol{f_2}"] \arrow[rr,bend left = 50, "\boldsymbol{f}"]&\boldsymbol{\mN_1}\arrow[r, "\boldsymbol{f_1}"]&\bmM
 \end{tikzcd}\,.
 \end{equation} 
It is the derived refinement of
$$
 \begin{tikzcd}
     \mN_2\arrow[r,"f_2"] \arrow[rr,bend left = 50, "f"]&\mN_1\arrow[r, "f_1"]&\mM
 \end{tikzcd}\,.
$$
The diagram of derived stacks \eqref{eq:ffactor} induces the following homotopy commutative diagram in $\mD^b(\boldsymbol{\mN_2})$
   \begin{equation}
   \label{eq:naivefunctorialityinmD}
    \begin{tikzcd}
        \arrow[d]\boldsymbol{f}^*\LL_{\boldsymbol{\mM}}\arrow[r]&\arrow[d]\boldsymbol{f_2}^*\LL_{\boldsymbol{\mN_1}}\arrow[r]&\LL_{\boldsymbol{\mN_2}}\arrow[d]\\
        0\arrow[r]&\arrow[d]\boldsymbol{f_2}^*\LL_{\boldsymbol{f_1}}\arrow[r]&\arrow[d]\LL_{\boldsymbol{f}}\\
        &0\arrow[r]&\LL_{\boldsymbol{f_2}}
    \end{tikzcd}\,.
    \end{equation}
    where each rectangle is a homotopy Cartesian diagram. Set $\EE=\LL_{\bmM}|_{\mM}$, $\wt{\FF}_{1} =\LL_{\bmN_1}|_{\mN_1} $, $\wt{\FF}=\LL_{\bmN_2}|_{\mN_2}$, $\MM_1=\LL_{\boldsymbol{f_1}}|_{\mN_1}$, $\MM_2=\LL_{\boldsymbol{f_2}}|_{\mN_2}$, and $\MM=\LL_{\boldsymbol{f}}|_{\mN_2}$. Applying the arguments of the proof of Proposition \ref{prop:stabletotriang} to the restriction of  \eqref{eq:naivefunctorialityinmD} to $\mN_2$, I recover the octahedral diagram
\begin{equation}
\label{eq:functoriality}
\begin{tikzcd}[column sep={5.5em,between origins},
  row sep={4.2em,between origins}, ampersand replacement=\&]
    \eqmathbox{f_2^*(\MM_1)[-1]}\arrow[dr]\arrow[rr, bend left = 35]\&\& \eqmathbox{f^*(\EE)}\arrow[dr] \arrow[rr, bend left =35, ]\&\&\eqmathbox{\wt{\FF}}\arrow[dr]\arrow[rr,bend left = 35]\&\&\eqmathbox{\MM_2}\\
\&\eqmathbox{\MM[-1]}\arrow[ur]\arrow[dr]\&\&\eqmathbox{f^*_2(\wt{\FF}_{1})}\arrow[ur]\arrow[dr]\&\&\eqmathbox{\MM}\arrow[ur]\&\\
\&\&\eqmathbox{\MM_2[-1]}\arrow[rr, bend right = 35]\arrow[ur]\&\&\eqmathbox{f_2^*(\MM_1)}\arrow[ur]\&\&
\end{tikzcd}
\end{equation}    
in $D^b(\mN_2)$. By the same reasoning, there is a morphism from \eqref{eq:functoriality} down to 
\begin{equation}
\label{eq:functoriality2}
\begin{tikzcd}[column sep={5.5em,between origins},
  row sep={4.2em,between origins}, ampersand replacement=\&]
    \eqmathbox{ f_2^*(\LL_{f_1})[-1]}\arrow[dr]\arrow[rr, bend left = 35]\&\& \eqmathbox{f^*(\LL_{\mM_1})}\arrow[dr] \arrow[rr, bend left =35, ]\&\&\eqmathbox{\LL_{\mN_2}}\arrow[dr]\arrow[rr,bend left = 35]\&\&\eqmathbox{\LL_{f_2}}\\
\&\eqmathbox{\LL_f[-1]}\arrow[ur]\arrow[dr]\&\&\eqmathbox{f^*_2(\LL_{\mM_2})}\arrow[ur]\arrow[dr]\&\&\eqmathbox{\LL_f}\arrow[ur]\&\\
\&\&\eqmathbox{\LL_{f_2}[-1]}\arrow[rr, bend right = 35]\arrow[ur]\&\&\eqmathbox{f_2^*(\LL_{f_1})}\arrow[ur]\&\&
\end{tikzcd}\,.
\end{equation}
Joyce used this (without spelling out the details) in his proof of wall-crossing in \cite{JoyceWC}. It amounts to functoriality for the lower halves of the diagrams \eqref{eq:diagsym}. Alternatively, this holds if one is given sufficient starting data in stable $\infty$-categories as in \eqref{eq:MM1andMMtoEE}. 

It is natural to ask for a similar statement for full Pvp diagrams. This problem is more complex as one needs to go at least to 3-simpleces (see \eqref{eq:3simplex})  to formulate it correctly.
 \begin{definition}
 \label{def:compatible}
 Let 
  $$\begin{tikzcd}
\mN_2\arrow[r,"f_2"]\arrow[rr,bend left = 50, "f"]&\mN_1\arrow[r,"f_1"]&\mM  
  \end{tikzcd}$$ be a diagram of quasi-smooth morphisms between stacks such that there is a homotopy comutative diagram 
\begin{equation}
\label{eq:MM1andMMtoEE}
\begin{tikzcd}[column sep={5.5em,between origins},
  row sep={4.2em,between origins}, ampersand replacement=\&]
   \arrow[dd] \eqmathbox{f_2^*(\MM_1)[-1]}\arrow[dr]\arrow[rr, bend left = 35,"\eta_1^\wedge"]\&\& \arrow[dd,"\psi^{\wedge}"]\eqmathbox{f ^*(\EE)}\\
\&\arrow[dd]\eqmathbox{\MM[-1]}\arrow[ur,"\eta^\wedge"]\&\\[-1em] \eqmathbox{f_2^*(\LL_{f_1})[-1]}\arrow[dr]\arrow[rr, bend left = 35]\&\&\eqmathbox{f ^*(\LL_{\mM})}\\
\&\eqmathbox{\LL_f[-1]}\arrow[ur]\&
\end{tikzcd}
\end{equation}
with vertical arrows being lifts of obstruction theories.
Let 
\begin{equation}
\label{eq:Mhom0}
\begin{tikzcd}\arrow[d]\MM[-1]\arrow[r,"\eta ^\wedge"]&f^*(\EE)\arrow[d,"\ov{\eta}^\wedge"]\\
0\arrow[r]&\MM^\vee[3]
\end{tikzcd}
\end{equation}
be a self-dual homotopy commutative diagram. Then it induces together with the upper floor of \eqref{eq:MM1andMMtoEE} the self-dual homotopy commutative diagram 
$$
\begin{tikzcd}\arrow[d]\MM_1[-1]\arrow[r,"\eta_1 ^\wedge"]&f_1^*(\EE)\arrow[d,"\ov{\eta_1}^\wedge"]\\
0\arrow[ur, Rightarrow, blue, "h_1"']\arrow[r]&\MM_1^\vee[3]
\end{tikzcd}\,.
$$
Here \B{$h_1$} makes $\ov{\eta_1}^\wedge\circ \eta_1 ^\wedge$ null-homotopic. For another choice of such a self-dual homotopy \B{$g_{1}$}, I will say that it is \textit{compatible with \eqref{eq:Mhom0}} if a self-dual 3-simplex with the 2-skeleton
\begin{equation}
\label{eq:MM1compatibility}
\includegraphics{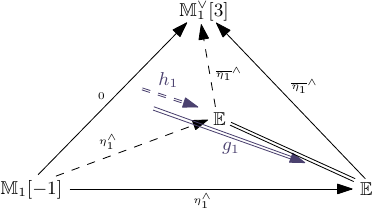}
\end{equation}
is specified.  I.e., there is a self-dual 2-homotopy between \B{$g_1$} and \B{$h_1$} in the sense of Example \ref{ex:simplices}.iii). 
 \end{definition}
\begin{theorem}
\label{thm:functsympull}
 Continue working in the situation represented by \eqref{eq:MM1andMMtoEE}, and assume that the obstruction theory $\EE\to \LL_{\mM}$ satisfies the conditions of Proposition \ref{prop:sympullback}.  All three statements below hold, if either of the them is true:
\begin{enumerate}[label =\roman*)]
\item The homotopy commutative, self-dual diagrams
\begin{equation}
\begin{tikzcd}\arrow[d]\MM_1[-1]\arrow[r,"\eta_1 ^\wedge"]&f_1^*(\EE)\arrow[d,"\ov{\eta_1}^\wedge"]\\
0\arrow[r]&\MM_1^\vee[3]
\end{tikzcd}
\label{eq:M1hom0}
\end{equation}
and \eqref{eq:Mhom0} are provided. They are compatible in the sense of Definition \ref{def:compatible}. 
\item There is a given homotopy commutative and self-dual diagram \eqref{eq:M1hom0}. Using it, one can construct a self-dual object
\begin{equation}
\label{eq:symmetricpullbackMM1}
 \begin{tikzcd}[column sep=large]   \MM_1[-1]\arrow[d,equal]\arrow[r]&\arrow[d](\wt{\FF}_{1})^\vee[2]\arrow[r]&\FF_1\arrow[d]\\
    \MM_1[-1]\arrow[d]\arrow[r,"\eta_1^ \wedge"]&\arrow[d,"\ov{\eta_1}^\wedge"] f_1^*\big(\EE_1\big)\arrow[r]&\wt{\FF}_{1}\arrow[d]\\
    0\arrow[r]&\MM_1^\vee[3]\arrow[r]&\MM_1^\vee[3]
    \end{tikzcd}
\end{equation}
in $\Diag^{3\times 3}(\mD^b(\mN_1))$ (see Example \ref{ex:functomD}.v)).
The resulting diagram 
\begin{equation}
\label{eq:M2toM10}
\begin{tikzcd}\arrow[d]\MM_2[-1]\arrow[r]&f_2 ^*(\wt{\FF}_{1})\arrow[d]\\
0\arrow[r]&\MM_1^\vee[3]
\end{tikzcd}
\end{equation}
is homotopy commutative, which gives rise to
\begin{equation}
\label{eq:Miihom0}
\begin{tikzcd}\arrow[d]\MM_2[-1]\arrow[r]&f_2 ^*(\FF_{1})\arrow[d]\\
0\arrow[r]&\MM_2^\vee[3]
\end{tikzcd}\,.
\end{equation}
This diagram is also homotopy commutative and self-dual.
\item The homotopy commutative and self-dual diagram \eqref{eq:Mhom0} is given. As such, it induces $\FF\xrightarrow{\theta}\MM_2$ in $\mD^b(\mN_2)$. There is a fixed homotopy for 
\begin{equation}
\label{eq:FFtoMM}
\begin{tikzcd}    \arrow[d]\MM^\vee[2]\arrow[r]&0\arrow[d]\\
    \FF \arrow[r]&\MM
\end{tikzcd}
\end{equation}
inducing the homotopy \B{$h_2$} in 
\begin{equation}
\label{eq:FFtoMM2}
\begin{tikzcd}
\arrow[d,"\overline{\theta}"']\MM_2^\vee[2]\arrow[r]&0\arrow[d]\\
    \FF\arrow[ur, Leftarrow, blue, "h_2"'] \arrow[r, "\theta"']&\MM_2\,.
\end{tikzcd}
\end{equation}
There is another such homotopy \B{$g_2$} that fits into a self-dual 3-simplex with the 2-skeleton
\begin{equation}
\label{eq:MM2comatibleMM}
\includegraphics{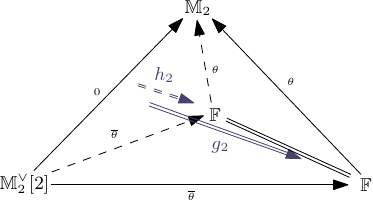}\,.
\end{equation}
\end{enumerate}
\vspace{-4pt}
\vspace{8pt}

Assume that i), ii) or iii) holds. Applying Proposition \ref{prop:sympullback} to the diagram \eqref{eq:M1hom0} constructs a lift of a CY4 obstruction theory $\FF_1\to \LL_{\mN_1}$ with higher self-duality. Doing the same with \eqref{eq:Miihom0} constructs such an obstruction theory $\FF\to\LL_{\mN_2}$. The result of using Proposition \ref{prop:sympullback} directly on \eqref{eq:Mhom0} determines the same $\FF\to \LL_{\mN_2}$ in $\mD^b(\mN_2)$ up a contractible choice of homotopy equivalences. 
\end{theorem} 
\begin{proof}
In this proof, I will ommit specifying the pull-backs of complexes and their maps. In both i) and ii), I can construct the diagram \eqref{eq:symmetricpullbackMM1}. The assumptions of i) allow me to further construct in $\mD^b(\mN_2)$ the diagram
\begin{equation}
\label{eq:firststep}
 \begin{tikzcd}[row sep = 8pt]
&&&\B{\MM_2^\vee[3]}\arrow[dd,blue]& \\
    \MM_1[-1]\arrow[dd,equal]\arrow[rr]&&\arrow[dd](\FF_{\textnormal{na},1})^\vee[2]\arrow[ur,blue,]\arrow[rr]&&\FF_1\arrow[dd]\\
&&&\B{\MM^\vee[3]}\arrow[dddl, blue, bend left = 18]&
 \\
    \MM_1[-1]\arrow[dd]\arrow[dr, bend right = 18 ,blue]\arrow[rr, blue]&&\arrow[dd, blue] \EE\arrow[ur,blue, "\ov{\eta}^\wedge"]\arrow[rr]&&\FF_{\textnormal{na},1}\arrow[dd]\\
  &\B{\MM[-1]}\arrow[ur, blue, "\eta^\wedge"]\arrow[rr, blue]&&\B{\MM_2[-1]}\arrow[ur, blue]&
    \\
    0\arrow[rr]&&\MM_1^\vee[3]\arrow[rr]&&\MM_1^\vee[3]
    \end{tikzcd}
\end{equation}
where any composition of two \B{blue arrows} meeting at $\EE_1$ is homotopic to zero. The homotopy for the diagram 
\begin{equation}
\label{eq:MM1toMM}
\begin{tikzcd}\arrow[d]\MM_1[-1]\arrow[r]&\EE\arrow[d]\\
0\arrow[r]&\MM^\vee[3]
\end{tikzcd}
\end{equation}
is taken to be the one determined by \eqref{eq:Mhom0} and the roof of \eqref{eq:MM1andMMtoEE}. This produces the 2-homotopy commutative box diagram 
 \begin{equation}
 \label{eq:startingdiag}
        \begin{tikzcd}[row sep=small, column sep=small]
           &\MM_1[-1]\arrow[dd]\arrow[rr]&&\EE\arrow[dd]
            \\
\arrow[ur] \arrow[dd]\MM_1[-1]\arrow[rr]&&\arrow[ur]\MM[-1]\arrow[dd]&\\
&0\arrow[rr]&&\MM^\vee[3]\\
           \arrow[ur] 0\arrow[rr]&& \arrow[ur]0 &
        \end{tikzcd}\,.
        \end{equation}
There are natural maps from the vertices of this diagram down to the vertices of a box diagram \eqref{eq:boxdiagram} that has 0's everywhere except for $I^\bullet_5 = I^\bullet_7 = \MM_1^\vee[3]$.  I claim that this produces an object of $\Fun\big((\Delta^1)^{\times 4}, \mD^b(\mN_2)\big)$. 

The only 2-simplices that can have non-zero homotopies come from \eqref{eq:MM1andMMtoEE}, \eqref{eq:Mhom0}, \eqref{eq:M1hom0}, \eqref{eq:MM1toMM}, and its dual. Furthermore, there are only two boxes with non-zero 2-homotopies determined by \eqref{eq:MM1compatibility} in the sense of Example \ref{ex:simplices}.iii).

Taking co-cones along the morphisms down from \eqref{eq:startingdiag} produces the 2-homotopy commutative diagram 
\begin{equation}
\label{eq:seconddiag}
        \begin{tikzcd}[row sep=small, column sep=small] &\arrow[dd]\MM_1[-1]\arrow[rr]&&\arrow[dd]\wt{\FF}_{1}^\vee[2]
            \\
\arrow[ur] \arrow[dd]\MM_1[-1]\arrow[rr]&&\arrow[ur] \arrow[dd]\MM[-1]&\\
&0\arrow[rr]&&\MM_2^\vee[3]\\
           \arrow[ur] 0\arrow[rr]&& \arrow[ur]0&           
        \end{tikzcd}\,.
        \end{equation}
which already contains the dual of \eqref{eq:M2toM10}. Taking cones of this diagram along the horizontal direction constructs \eqref{eq:Miihom0}. The lifts of CY4 obstruction theories $\FF\to \LL_{\mN_2}$ resulting both from \eqref{eq:Mhom0} and \eqref{eq:Miihom0} can be seen to be equivalent. This follows from the$(\Delta^1)^{\times 4}$ diagram discussed under \eqref{eq:startingdiag} because the order of taking (co)cones is permutable in stable $\infty$-categories by \cite[Chapter 7, Corollary 7.3.8.20]{kerodon}.

In ii), let me first explain where $\MM_2[-1]\to f^*_2(\FF_1)$ comes from. From the assumptions, I construct the diagram
$$
        \begin{tikzcd}[row sep=small, column sep=small]
&\wt{\FF}_{1}\arrow[dd]\arrow[rr]&&\MM_1\arrow[dd]\\
            \arrow[dd]\arrow[ur] \MM_2[-1]\arrow[rr]&& \arrow[ur]\MM_1 \arrow[dd]&\\
            &\MM_1^\vee[3]\arrow[rr]&&0
            \\
\arrow[ur]0\arrow[rr]&&\arrow[ur]0&
        \end{tikzcd}\,.
 $$
 It is 2-homotopy commutative because the 7th vertex is 0. Taking cocones along the horizontal and then along the vertical direction, one obtains the \B{horizontal blue sequence} with \B{morphisms} to the original diagram in
\begin{equation}
\label{eq:secondstep}
 \begin{tikzcd}[row sep = 8pt]
&&&&&\B{\MM_2^\vee[3]}\arrow[dd,blue] \\
    \MM_1[-1]\arrow[dd,equal]\arrow[rr]\arrow[dr, bend right = 18 ,blue]&&\arrow[dd](\FF_{\textnormal{na},1})^\vee[2]\arrow[rr]&&\FF_1\arrow[dd]\arrow[ur,blue]\\
&\B{\MM[-1]}\arrow[ur, blue]\arrow[rr, blue]&&\B{\MM_2[-1]}\arrow[ur, blue]&&\B{\MM^\vee[3]}\arrow[ddddl, blue, bend left = 25]
 \\
    \MM_1[-1]\arrow[ddd]\arrow[rr]&&\arrow[ddd,] f_1^*(\EE_1)\arrow[rr]&&\FF_{\textnormal{na},1}\arrow[ddd]\arrow[ur,blue]\\
&&&&\\[10pt]   
&&&&
    \\
    0\arrow[rr]&&\MM_1^\vee[3]\arrow[rr]&&\MM_1^\vee[3]
    \end{tikzcd}\,.
\end{equation}
The dual argument produces the \B{vertical blue sequence}, so \eqref{eq:M2toM10} can now be stated. 

To go from ii) to i) one reverses the arguments for getting ii) out of i). I use the 2-homotopy commutative diagram
$$
        \begin{tikzcd}[row sep=small, column sep=small]
 &\arrow[dd]\FF_1\arrow[rr]&&\arrow[dd]\MM_1
            \\
\arrow[ur]\arrow[dd]\MM_2[-1]\arrow[rr]&&\arrow[ur]\MM_1\arrow[dd]&\\
&\MM^\vee_2[3]\arrow[rr]&&0\\
            \arrow[ur] 0\arrow[rr]&& \arrow[ur]0 &           
        \end{tikzcd}\,.
 $$
Taking cocones, it induces \eqref{eq:seconddiag}. This can be completed to a $(\Delta^1)^{\times 4}$ diagram by taking the map from a box diagram \eqref{eq:boxdiagram} with all $I_i^\bullet=0$ except $I_5^\bullet =  \MM^\vee_1[2] = I^\bullet_7$. Taking cones produces \eqref{eq:startingdiag} and moreover the $(\Delta^1)^{\times 4}$ diagram following \eqref{eq:startingdiag}. The only difference is that now the two 2-homotopies equivalent to \eqref{eq:MM1compatibility} are not self-dual, and are not equal. They are however dual and equal to each other up to a 3-homotopy, so we can strictify them to the data of i).

In iii), the map  $\FF\xrightarrow{\theta} \MM_2$ is a composition of the natural $\FF\to \MM$ from \eqref{eq:Parkhomotopylift} and $\MM\to \MM_2$. I leave it to the reader to undertake similar diagram-chasing as above to recover from \eqref{eq:MM2comatibleMM} the data of ii).  To show the converse, it is not difficult to derive \eqref{eq:MM2comatibleMM} from \eqref{eq:MM1compatibility} in i). 

\end{proof}
I will now discuss one simple situation when the above results can be applied. Consider a morphism $f:\mN\to \mM$ of stacks with a $\infty$-Pvp diagram lifting \eqref{eq:diagsym}. I would like to produce a compatible $\infty$-Pvp diagram for the rigidified morphism 
$$
\begin{tikzcd}f^{\rig}: \mN^{\rig}\arrow[r]&\mM^{\rig}
\end{tikzcd}\,.
$$
Here, I assumed that $\mN$ and $\mM$ admit a $B\GG_m$-action and $f$ is equivariant. The precise compatibility condition is spelt out in the next Lemma, which states that such a diagram exists. 
\begin{lemma}
\label{lem:rigidifyingPark}
\begin{enumerate}[label=\roman*)]
    \item For the morphisms $f:\mN\to \mM$ consider the commutative diagram of stacks
\begin{equation}
\label{eq:NMrig}
\begin{tikzcd}
    \mN\arrow[d,"\Pi_{\mN}"']\arrow[r,"f"]&\mM\arrow[d,"\Pi_{\mM}"]\\
    \mN^{\rig}\arrow[r,"f^{\rig}"']&\mM^{\rig}
\end{tikzcd}
\end{equation}
where $\Pi_{(-)}$ denotes projections to appropriate rigidifications. Given an $\infty$-Pvp diagram for $f$ which is $B\GG_m$ weight 0, so descends to $\mN^\rig$, there are unique up to contractible choices $\infty$-Pvp diagrams for all four morphisms such that they are compatible in the sense of Theorem \ref{thm:functsympull} for all consecutive pairs of morphisms.
\item For a commutative diagram
$$
\begin{tikzcd}
\arrow[d,"\Pi_{\mN_2}"]\mN_2\arrow[r,"f_2"]&\arrow[d, "\Pi_{\mN_1}"]\mN_1\arrow[r,"f_1"]&\arrow[d,"\Pi_{\mM}"]\mM\\
\mN^{\rig}_2\arrow[r,"f^{\rig}_2"']&\mN^{\rig}_1\arrow[r,"f^{\rig}_1"']&\mM^{\rig}
\end{tikzcd}\,,
$$
assume that there are $\infty$-Pvp diagrams along $f_2$ and $f_1$ inducing one along $f = f_1\circ f_2$ by Theorem \ref{thm:functsympull}. Then the $\infty$-Pvp diagram for $f^{\rig}$ obtained by applying Theorem \ref{thm:functsympull}.ii) to the induced $\infty$-Pvp diagrams for $f^{\rig}_1$ and $f^{\rig}_2$ is equivalent to the diagram constructed in i) for $f$. 
\end{enumerate}
 
\end{lemma}
\begin{proof}
i) The morphism $\Pi_{\mM}$ induces the right square of
\begin{equation}
   \label{eq:epsO}
\begin{tikzcd}\arrow[d]\wt{\EE}^{\rig}\arrow[r]&\arrow[d]\EE\arrow[r,"\delta"]&\arrow[d, equal] \mO_{\mM}[-1]\\
(\Pi_{\mM})^*\big(\LL_{\mM^{\rig}}\big)\arrow[r]&\LL_{\mM}\arrow[r]&\mO_{\mM}[-1]
  \end{tikzcd}\,.
\end{equation}
The full diagram is obtained by taking fibers. A similar diagram for the obstruction theory $\FF$ on $\mN$ is induced by $\Pi_{\mN}$. Since $$\Map_{\mD^b(-)}\big(\mO_{(-)}[3],\mO_{(-)}[-1] \big)$$ is 3-connected, one can construct the top right homotopy commutative square of 
 \begin{equation}
 \label{eq:rigidifyingEE}
\begin{tikzcd}
\mO_{\mM}[3]\arrow[d]\arrow[r,equal]&\mO_{\mM}[3]\arrow[d,"\ov{\delta}"]\arrow[r]&0\arrow[d]
\\
\arrow[d]\wt{\EE}^{\rig}\arrow[r]&\arrow[d]\EE\arrow[r,"\delta"]&\arrow[d, equal] \mO_{\mM}[-1]\\
\EE^{\rig}\arrow[r]&(\wt{\EE}^{\rig})^\vee[2]\arrow[r]&\mO_{\mM}[-1]
\end{tikzcd}
 \end{equation}
 uniquely up to contractible choices. The resulting obstruction theory $\EE^{\rig}$ of $\mM^{\rig}$ was often called the \textit{rigidification} of $\EE$ in the main text. This produces the situation of Theorem \ref{thm:functsympull}.ii) for the consecutive morphisms $\Pi_{\mM}$ and $f$ in \eqref{eq:NMrig} because one clearly has a natural diagram \eqref{eq:MM1andMMtoEE} and
$$
\begin{tikzcd}
    \arrow[d]\MM[-1]\arrow[r]&\EE\arrow[d]\\
    0\arrow[ur, Rightarrow,"h", Blue]\arrow[r]&\MM^\vee[3]
\end{tikzcd}\,,
$$
which plays the role of \eqref{eq:Miihom0}.
By this theorem, there is an $\infty$-Pvp diagram for the composition $f^{\rig}\circ \Pi_{\mN} =\Pi_{\mM}\circ f$. Thus, I am left to construct the data of Theorem \ref{thm:functsympull}.iii)  for the consecutive morphisms $f^{\rig}$ and $\pi_{\mN}$. Because we have assumed that the relative obstruction theory for $f$ is $B\GG_m$ weight 0, the diagram \eqref{eq:MM1andMMtoEE} is given uniquely again. Then \eqref{eq:FFtoMM} is determined by the above leading to \eqref{eq:FFtoMM2} for $\MM_2=\mO_{\mN}[-1]$. The same argument as in \eqref{eq:rigidifyingEE} produces \B{$g_2$} from \eqref{eq:MM2comatibleMM}. Using that $\Map_{\mD^b(\mN_k)}\big(\MM_2^{\vee}[2],\MM_2\big)$ is 3-connected, the 3-simplex \eqref{eq:MM2comatibleMM} is given uniquely up to contractible choices. Theorem \ref{thm:functsympull} constructs \eqref{eq:M1hom0} which induces an $\infty$-Pvp diagram for $f^{\rig}$. Its compatibility with $\Pi_{\mN}$ is also a consequence of Theorem \ref{thm:functsympull}.

ii) Using the construction in i), I need to show that the $\infty$-Pvp diagram for $\Pi_{\mM}\circ f_1\circ f_2$ (constructed by Theorem \ref{thm:functsympull}) is equal to the one along $f^{\rig}_1\circ f^{\rig}_2\circ \Pi_{\mN_2}$. This follows first by the equivalence of the former to the $\infty$-Pvp diagram for $f^{\rig}_1\circ \Pi_{\mN_1}\circ f_2$, which is in turn equivalent to the latter. 
\end{proof}
\section{Equivariant homology of stacks}
\label{app:B}
In \cite{BB1}, equivariant homology theories for stacks will be discussed in detail. There, we propose to construct them using (non-equivariant) bivariant theories. We motivate this approach by studying examples presented in §\ref{app:examples} which allows us to formulate an overarching perspective. The 6-functor formalism discussed in \cite{KhanEH} provides such a bivariant theory leading to equivariant singular homology. In this section, I will summarize the axioms of bivariant theories and how equivariant homology is produced by taking limits. A more refined approach in terms of pro-systems is presented in \cite{BB1}. I will then recall Khan's construction and explain a simple example that already covers the quiver moduli spaces from §\ref{sec:dgquivers}. 

\subsection{Bivariant theories}
\label{app:bivarianttheories}
This subsection is a shortened version of the corresponding one in \cite{BB1}. There, we additionally provide motivation for this approach. Bivariant theories used in the present form were introduced in \cite[§2.2]{FMP}. I recall their definition, except that I allow  all morphisms here instead of restricting to some confined class of them. The same applies for independent squares.

We work with the category $\Art$ of Artin stacks and Artin morphisms between them as defined, for example, in \cite{TV,  LurieDAG, KhanEH}. From now on, all morphisms and diagrams are assumed to be in $\Art$, and I fix a ring $R$.
\begin{definition}
\label{def:bivariant}
    A bivariant theory for $\Art$ is a way to assign to each morphism  $\mX\xrightarrow{f}\mY$ in $\Art$ a graded $R$-module
    $$
    \BB_*\big(\mX\xrightarrow{f}\mY\big)
    $$
   together with the following graded $R$-morphisms:
    \begin{itemize}
\item (\textit{Composition Product}) For any pair of maps $\mX\xrightarrow{f}\mY\xrightarrow{g}\mZ$, there is a morphism
$$
\begin{tikzcd}
m: \BB_*\big(\mX\xrightarrow{f}\mY\big)\otimes  \BB_*\big(\mY\xrightarrow{g}\mZ\big)\arrow[r]& \BB_*\big(\mX\xrightarrow{g\circ f}\mZ\big)\,.
\end{tikzcd}
$$
Sometimes, I will simply write $m(a,b) = a\cdot b$ for composable elements.
\item (\textit{Base-Change Pullback}) 
Given a base change diagram \[\begin{tikzcd}
  \mX' \ar{r}{}\arrow[d,"{f'}"'] & \mX\ar{d}{f} \\
  \mY' \arrow[r,"g"'] & \mY
\end{tikzcd}\,,\]a natural map
$$
\begin{tikzcd}
 g^*:\BB_*\big(\mX\xrightarrow{f}\mY\big)\arrow[r]& \BB_*\big(\mX'\xrightarrow{f'}\mY'\big)
 \end{tikzcd}
$$
is given.
\item (\textit{Relative Pushforward}) For a commutative diagram 
$$
\begin{tikzcd}[column sep=small]
\mX'\arrow[dr,"{f'}"']\arrow[rr,"g"]&&\mX\arrow[dl,"f"]\\
   & \mY&
\end{tikzcd}\,,
$$
one has a natural morphism
$$
\begin{tikzcd}
  g_*: \BB_*\big(\mX'\xrightarrow{f'}\mY\big) \arrow[r]&\BB_*\big(\mX\xrightarrow{f}\mY\big)
\end{tikzcd}\,.
$$
\end{itemize}
These need to satisfy compatibilities noted down in \cite[§2.2]{FMP}. We only recall them in name here:
\begin{itemize}
    \item (\textit{Associativity of Products})
      \item (\textit{Functoriality of Pullback})
     \item (\textit{Functoriality of Pushforward})
       \item (\textit{Commutativity Between Product and Pushforward})
       \item (\textit{Base Change Transformation - for pushforward and pullback})
    \item (\textit{Projection Formula})
    \item (\textit{Skew-Commutativity})
\end{itemize}
\end{definition}
Just as in \cite[§2.3]{FMP}, one can conclude that 
$$
H^*(\mX) := \BB_{-*}\big(\mX\xrightarrow{\id_{\mX}}\mX\big)\,,\qquad H_*(\mX) := \BB_*\big(\mX\to\pt\big)
$$
give rise to compatible cohomology and homology theories. 

We introduce two new assumptions that are required to hold only when
$\QQ\subset R$.
\begin{itemize}
\item (\textit{Rational Künneth Isomorphism})
Consider the diagram
\begin{equation}
\label{eq:rationalkunnethdiag}
\begin{tikzcd}
\arrow[d,"\pi_2"']\mX\times \mX'\arrow[r,"\pi_1"]&\mX\arrow[d,"p"]\arrow[r,"f"]&\mY\\
\mX'\arrow[r,"q"]&\pt&
\end{tikzcd}
\end{equation}
where $\pi_i$ denotes the projection to the  $i$'th factor.
Then the composition of
\begin{equation}
\label{eq:Kunnethmorph}
\begin{tikzcd}
 H_*(\mX')\otimes \BB_*(\mX\to \mY)\arrow[r,"p^*\otimes\id"]&\BB_*(\mX\times\mX'\to\mX)\otimes \BB_*(\mX\to \mY) \arrow[r,"m"]&\BB_*(\mX\times\mX'\to \mY)
\end{tikzcd}
\end{equation}
denoted by $\boxtimes$ is an isomorphism.
\item (\textit{Rational triviality})
If $f:\mX\to\mY$ is a $B\ZZ_n$-torsor, then for any $g:\mY\to \mZ$ the pushforward
 $$
\begin{tikzcd}
f_*:\BB_*\big(\mX\xrightarrow{g\circ f}\mZ\big)\arrow[r]&\BB_*\big(\mY\xrightarrow{g}\mZ\big)
\end{tikzcd}
$$
is an isomorphism.
\end{itemize}
These assumption are needed when constructing Lie algebras from deformations of vertex algebras on equivariant homology.

\subsection{Equivariant homology as a limit}
To additionally introduce equivariance, we use the Borel construction of Totaro \cite{Totaro} and Edidin-Graham \cite{EdGrEH} via approximations. However, as we are working with stacks, we need to take limits (or work with pro-systems as in \cite{BB1}) because it does not suffice to work with a fixed approximation. 

Let $\T$ be an algebraic torus acting on $\mX$ in $\Art$.  For finite-dimensional $\T$-representations, consider their open subsets $E_i\T$ where the action is free. Combine them into the ind-system $\{E_i\T\}_{i\in I}$ with transition maps $\kappa_{i,j}:E_i\T\hookrightarrow E_j\T$ being $\T$-equivariant embeddings. In this case, we will write $B_i\T:=E_i\T/\T$, and we note that each $\kappa_{i,j}$ induces the map $\iota_{i,j}:B_i\T\rightarrow B_j\T$ the colimit of which is homotopically equivalent to $B\T$. We write $\iota_i : B_i\T\to B\T$ for the resulting natural morphism.

Setting 
$$
[\mX/\T]_i:=\mX\times_\T E_i\T\,,\qquad \tau_{i,j} = \id_{\mX}\times_\T\kappa_{i,j}\,,
$$
we obtain the commutative diagram
\begin{equation}
\label{eq:XiGfiberprod}
\begin{tikzcd}[column sep=large]
   \arrow[d][\mX/\T]_i\arrow[r,"\tau_{i,j}"]&{[\mX/\T]}_j\arrow[d]\arrow[r,"\tau_j"]&{[\mX/\T]}\arrow[d]\\
   B_i\T\arrow[r,"\iota_{i,j}"]&B_{j}\T\arrow[r,"\iota_j"]&B\T
\end{tikzcd}
\end{equation}
where all squares are Cartesian. Here, the vertical arrows were constructed using the projection $\mX\to \pt$.

\begin{definition}
\label{def:eqhom}
 For a fixed bivariant theory $\BB_*$, define the \textit{$\T$-equivariant homology} $H^{\T}_*(\mX)$ of $\mX$ as the graded limit
$$H^\T_*(\mX):= \varprojlim_{i\in I} \BB_*\big([\mX/\T]_i\to B_i\T\big)$$
where transition morphisms are the base change pullbacks $\iota_{i,j}^*$ for the fiber product \eqref{eq:XiGfiberprod}.

The \textit{$\T$-equivariant cohomology} is defined as 
\begin{equation}
\label{eq:defcoh}
H^*_\T(\mX):=\varprojlim_{i\in I} H^*\big([\mX/\T]_i\big)\,.
\end{equation}
Both $H^\T_*(\mX)$ and $H^*_\T(\mX)$ are considered as graded topological $R$-modules for the graded limit topology. They are also both topological  $H^*_\T(\pt) = \varprojlim H^*(B_i\T)$ modules because $\BB_*\big([\mX/\T]_i\to B_i\T\big)$ and $H^*\big([X/\T]_i\big)$ are  $H^*(B_i\T)$ modules by the composition product. 
\end{definition}
\begin{remark}
A graded limit of graded objects will be, in general, different from the one obtained after forgetting grading, because it is taken in each degree separately. Consider for example the pro-system $\big\{R[u]/u^n\big\}$ with transition maps the quotients $R[u]/u^{n+1}\to R[u]/u^{n}$. Setting $\deg(u)=2$, the graded limit of this system is $R[u]$. If the grading is ignored, the limit becomes $R\llbracket u\rrbracket$.
\end{remark}
A more suitable definition for the purpose of introducing and constructing deformations of vertex algebras is given in \cite{BB1} in terms of pro-systems. As representation theoretic aspects are not the focus of the present work, I will not use this perspective here. 
\subsection{Operations on equivariant homology}
\label{sec:equivariantOPs}
In \cite{BB1}, we define all the necessary operations level-wise on the pro-system. This induces the corresponding operations on $H^{\T}_*(\mX)$. All of these operations are immediate consequences of bivariance.
\begin{itemize}
\item (\textit{Pullbacks and pushforwards})
Let $\mX,\mY$ be stacks with a $\T$-action. Then for any $\T$-equivariant map $f:\mX\to \mY$ one defines the $\T$-equivariant pullback $f^*:H_{\T}^*(\mY)\to H_{\T}^*(\mX)$ as the limit of pullbacks along $[f/\T]_i:[\mX/\T]_i\to [\mY/\T]_i$. 

We also construct the pushforward
$
 \begin{tikzcd}
f_*:H_*^{\T}(\mX)\arrow[r]&H_*^{\T}(\mY)
\end{tikzcd}
$
determined level-wise by 
$$
\begin{tikzcd}
([f/\T]_i)_*:\BB_*\big([\mX/\T]_i\to B_i\T)\arrow[r]& \BB_*\big([\mY/\T]_i\to B_i\T)\,.
\end{tikzcd}
$$
Compatibility with transition morphisms follows from the base change property as shown in \cite{BB1}.
\item (\textit{Cap product})
The cap product
\begin{equation}
\label{eq:capproduct}
\begin{tikzcd}
H^*_{\T}(\mX)\widehat{\otimes}_{H^*_\T(\pt)}H^\T_*(\mX)\arrow[r,"\cap"]&H^\T_*(\mX)\,.
\end{tikzcd}
\end{equation}
is itself obtained as a limit. Levelwise, this map acts on $\alpha_i\in H^*\big([\mX/\T]_i\big)$, $a_i\in \BB_*\big([\mX/\T]_i\to B_i\T\big)$ by
$$
\begin{tikzcd}
\alpha\otimes a_i\arrow[r,mapsto]&\alpha_i\cdot a_i\in \BB_*\big([\mX/\T]_i\to B_i\T\big)
\end{tikzcd}
$$
Compatibility with respect to transition morphisms is an immediate consequence of the commutativity of products and pullbacks. Since the above map is $H^*(B_i\T)$-bilinear, we obtain the morphism \eqref{eq:capproduct} after taking limits.
\item (\textit{Equivariant Künneth morphism})
Let $\mX$ and $\mY$ be $\T$-stacks and consider the diagonal action on $\mX\times \mY$. This induces the Cartesian diagram
\begin{equation}
\label{eq:equivKunneth}
\begin{tikzcd}[column sep=large, row sep =large]
{\big[(\mX\times \mY)/\T\big]}\arrow[d,]\arrow[r,]&{[\mY/\T]}\arrow[d, "y"]\\
{[\mX/\T]}\arrow[r,"x"']&B\T
\end{tikzcd}\,.
\end{equation}
We construct a morphism 
$$
\begin{tikzcd}
H^\T_*(\mX)\widehat{\otimes}_{H^*_\T(\pt)} H^\T_*(\mY)\arrow[r,"\boxtimes_{\T}"]&H^{\T}_*(\mX\times \mY)\,.
\end{tikzcd}
$$
Level-wise, we consider the approximation 
$$
\begin{tikzcd}[column sep=large, row sep =large]
{\big[(\mX\times \mY)/\T\big]_i}\arrow[d,]\arrow[r,]&{[\mY/\T]_i}\arrow[d, "y_i"]\\
{[\mX/\T]_i}\arrow[r,"x_i"']&B_i\T
\end{tikzcd}\,.
$$
Taking $a_i\in \BB_{|a|}\big([\mX/\T]_i\to B_i\T\big)$ and $b_i\in \BB_{|b|}\big([\mY/\T]_i\to B_i\T\big)$ the map acts by 
$$
\begin{tikzcd}
a_i\otimes b_i\arrow[r,mapsto]&y_i^*(a_i)\cdot b_i = (-1) ^{|a||b|}x_i^*(b_i)\cdot a_i
\end{tikzcd}
$$
where the last equality used the skew-commutativity assumption. The map is $H^*(B_i\T)$-bilinear 
and compatible with transition morphisms.
\item (\textit{Reduction to subgroups}) 
For a subtorus $\S\subset \T$, restricting the $\T$-action on $E_i\T$ to $\S$ makes $\{E_i\T\}_{i\in I}$ into a cofinal subsystem of $\{E_k\S\}_{k\in K}$. Then there are Cartesian diagrams
$$
\begin{tikzcd}[column sep=large, row sep =large]
{\mX\times_{\S}E_i\T}\arrow[d]\arrow[r]&E_i\T/\S\arrow[d]\\
{[\mX/\T]_i}\arrow[r]&B_i\T
\end{tikzcd}
$$
where both vertical arrows are induced by the inclusion $\S\hookrightarrow \T$. The reduction morphism
\begin{equation}
\label{eq:redST}
\begin{tikzcd}
H^{\T}_*(\mX)\arrow[r,"\red^{\S\subset \T}_{\mX}"]&[1cm]H^{\S}_*(\mX)
\end{tikzcd}
\end{equation}
is defined as the limit of pullbacks along $E_i\T/\S\to B_i\T$. 
\end{itemize}

\subsection{Bivariant theory from 6-functor formalism}
Here, I will recall the definition of relative chains from \cite{KhanEH}. Taking their homology gives rise to a bivariant theory as shown in \cite{BB1}.

For a (higher) Artin stack $\mX$ denote by $\Sh(\mX)$ the stable $\infty$-category of sheaves of $R$-modules on $\mX$. The corresponding 6-functor formalism for $\Art$ is discussed at length in \cite{KhanEH, KhanVFC, KhanLV}, so I will only mention the constituent functors. For any morphism $f:\mX\to\mY$ of Artin stacks, there is the adjoint pair consisting of the pullback and the direct image functor:
\begin{equation}
\label{eq:pushpulladj}
        \begin{tikzcd}
            \Sh(\mY)\arrow[r, shift left=1ex, "Lf^*"{name=G}] & \Sh(\mX)\arrow[l, shift left=.5ex, "Rf_*"{name=F}]
            \arrow[phantom, from=F, to=G, , "\scriptscriptstyle\boldsymbol{\bot}"]
        \end{tikzcd}\,.
 \end{equation}
 If $f$ is additionally a morphism in $\Art$, one also has the adjoint pair
 \begin{equation}
 \label{eq:shriekpushpulladj}
        \begin{tikzcd}
            \Sh(\mX)\arrow[r, shift left=1ex, "Lf_!"{name=G}] & \Sh(\mY)\arrow[l, shift left=.5ex, "f^!"{name=F}]
            \arrow[phantom, from=F, to=G, , "\scriptscriptstyle\boldsymbol{\bot}"]
        \end{tikzcd}
 \end{equation}
 where $Lf_{!}$ is the direct image with compact support. I will also denote by $(-\otimes^{\LL}-)$ the derived tensor product in $\Sh(\mX)$ for which the constant sheaf $\un{R}_{\mX}$ is the unit.
 \begin{definition}[\cite{KhanEH}]
 \label{def:relhom}
Let $f:\mX\to \mY$ be a morphism in $\Art$. The \textit{relative chains over $\mY$} with coefficients in $R$ are
        $$
        C^{/\mY}_{\bullet}(\mX):= Lf_! f^!(\un{R}_{\mY})
        $$
        where we view the expression on the right-hand side as a chain complex by reversing the grading. Using this, the bivariant theory is defined by 
        \begin{equation}
        \label{eq:HoverS}
        \BB_*(\mX\to\mY, R):=H_{*}\Big(C_{\bullet}^{/\mY}(\mX)\Big)\,.
        \end{equation}
\end{definition}
The composition product, base-change pullback, and relative pushforward have already been constructed on the level of chains in \cite{KhanEH}. They are a direct consequence of the 6-functor yoga. In \cite{BB1}, we check that the conditions in §\ref{app:bivarianttheories} are indeed satisfied.
 \begin{proposition}[\cite{BB1}]
    Definition \ref{def:relhom} constructs a bivariant theory with rational Künneth isomorphisms and satisfying rational triviality. 
\end{proposition}
\subsection{Properties of the equivariant homology}
The resulting equivariant homology and cohomology theories are ideal for wall-crossing and the construction of $\T$-deformations of vertex algebras  due to the following additional properties.
\begin{enumerate}[label=\roman*)]
\item (\textit{Non-equivariant limit}) By \cite[Proposition 2.8]{KhanEH}, both $H^{\T}_*(\mX)$ and $H_{\T}^*(\mX)$ become the usual Betti homology and cohomology of stacks as in \cite[Definition 4.2]{gross}, \cite[Example 2.4.3 (c)]{Joycehall} when $\T=\{1\}$. This means that one first applies the topological realization functor mapping a stack $\mX$ to the topological space $\mX^{\ttop}$. Then one acts with $H_*(-)$ or $H^*(-)$.
\item (\textit{Equivariant Chern character})
Let $\PPerf_r$ be the stack of perfect complexes of $\CC$-vector spaces of rank $r$. By i), there is an isomorphism
$$H^*(\PPerf_r)\cong H^*\big((\PPerf_r)^{\ttop}\big)\,.$$
Because $(\PPerf_r)^{\ttop} = BU$ by \cite[§4.2]{B16}, its cohomology can be expressed as 
$$
H^*(\PPerf_r) =  R[ \ch_1,\ch_2,\cdots ]\,,
$$
where $\ch_i$ is the $i$'th Chern character of the universal rank $r$ K-theory class on $BU$.

Let $\mE$ be a $\T$-equivariant rank $r$ perfect complex on $\mX$, then it descends to $[\mE/\T]$ on $[\mX/\T]$ with the corresponding natural morphism $p_{[\mE/\T]}:[\mX/\T]\to \PPerf_r$. One defines the $i$'th equivariant Chern character of $\mE$ to be 
$$
\ch^{\T}_i(\mE):=p_{[\mE/\T]}^*(\ch_i)\in H_{\T}^*(\mX)\,.
$$
\item (\textit{Completeness of cohomology}) By \cite[Lemma 3.3]{KhanEH}, there is a natural isomorphism
$$
H^*\big([\mX/\T]\big)\cong H^*_{\T}(\mX)\,.
$$
\item (\textit{Homotopy invariance}) Let $E\xrightarrow{\pi}\mX$ be a $\T$-equivariant vector bundle on $\mX$, then $\pi_*: H^{\T}_*(E)\to H^{\T}_*(\mX)$ is an isomorphism. To see this, note that every approximation
$$
\begin{tikzcd}[column sep = large]
{[E/\T]}_i\arrow[r,"{[\pi/\T]}_i"]& {[\mX/\T]}_i
\end{tikzcd}
$$
is an equivariant vector bundle, so the statement follows from \cite[Theorem 2.7 (ii)]{KhanEH}.
\item (\textit{Cycle-class map})
Let $X$ be a proper algebraic space with a $\T$-action. By \cite[Corollary 3.4]{KhanEH}, its equivariant homology can be computed as 
$$
H^{\T}_*(X)  \cong H_*\Big(C^{/B_i\T}_{\bullet}\big([X/T]_i\big)\Big)
$$
for a fixed $i\in I$ and the natural projection $q_i:[X/\T]_i\to B_i\T$. The same corollary also shows that the equivariant Borel--Moore homology $H^{\BM,\T}_{*}(X)$ is recovered as
$$
H^{\BM,\T}_{*}(X)\cong H_*\Big(R(q_i)_*q_i^{!}\big(\un{R}_{B_i\T}\big)\Big) \cong H^{\BM}_{*+2\dim_{\CC}(B_i\T)}\big([X/\T]_i\big)\,.
$$
Due to $X$ being proper, the natural transformation $L(q_i)_!\to R(q_i)_*$ is an isomorphism. This induces the identification
$H^{\BM,\T}_{*}(X)\cong H^{\T}_{*}(X)
$. Moreover, there is a natural morphism $A_{*}^{\T}(X)\to H^{\BM,\T}_{2*}\big(X\big)$  by \cite[§2.8]{EdGrEH}, which leads to the equivariant cycle-class map
$$\begin{tikzcd}A_{*}^{\T}(X)\arrow[r]&H^{\T}_{2 *}(X)\end{tikzcd}$$ 
when combined with the above discussion.
\end{enumerate}
Using this, I constructed $V_{*}$ and $V_{\loc,*}$ for the global approach in §\ref{sec:equivariantVA}. One can also define the invariants \eqref{eq:Msigalphaglobal} in $H^{\T}_*\big(\mM^{\rig}_{\mA}\big)_{\loc}$. To identify this equivariant homology with $L_{\loc,*}$ and define a Lie algebra on it, I prove the following result for the natural projection
$$
\begin{tikzcd}
\Pi_{\alpha}:\mM^{\rig}_{\alpha}\arrow[r]&\mM_{\alpha}
\end{tikzcd}
$$
from Definition \ref{def:categoryA}.b).
\begin{lemma}[\cite{BB1}]
    Let 
    \begin{equation}
    \label{eq:BZZnbasechange}
     \begin{tikzcd}
         \arrow[d]B\GG_m\times \mY_{\alpha}\arrow[r,"\pi_2"]&\arrow[d,"r"]\mY_{\alpha}\\
\mM_{\alpha}\arrow[r,"\Pi_{\alpha}"]&\mM^{\rig}_{\alpha}
     \end{tikzcd}
    \end{equation}
    be a $\T$-equivariant Cartesian diagram in $\Art$ such that $r$ is a $B\ZZ_n$-torsor, and $\pi_2$ is a trivial $B\GG_m$-torsor with trivial $\T$-action on $B\GG_m$. If $\QQ\subset R$, then we have the natural isomorphism
    $$
    H^{\T}_*\big(\mM_{\alpha}\big)/T\cong H^{\T}_*\big(\mM^{\rig}_{\alpha}\big)\,.
    $$
\end{lemma}
The non-equivariant diagram \eqref{eq:BZZnbasechange} was constructed in \cite[Proposition 3.24 (b)]{Joycehall} when $\alpha\notin \Ker(\chi)$. This argument can be modified to work equivariantly.
\begin{proposition}
\label{prop:rigidisom}
    Let $\alpha\in \ov{K}(\mA)$ be such that for some $\beta\in \ov{K}(\mA)$, one has
$$
\chi(\beta,\alpha)\neq 0\,,\qquad \mM^{\T}_{\beta}\neq  \emptyset\,.
$$
Then there exists a $\T$-equivariant Cartesian diagram \eqref{eq:BZZnbasechange}. In particular, there is a natural isomorphism  $
    H^{\T}_*\big(\mM_{\alpha}\big)/T\cong H^{\T}_*\big(\mM^{\rig}_{\alpha}\big)
    $ whenever additionally $\QQ\subset R$. 
\end{proposition}
\begin{proof}
Choose an equivariant object $B\in \mM^{\T}_{\beta}$ and construct
$$\Ext(B,\alpha) := \RHom_{\{B\}\times \mM_{\alpha}}(B,\mE_{\alpha})$$
on $\mM_{\alpha}$. It is clearly $\T$-equivariant of $B\GG_m$-weight $1$. As its rank is $\chi(\beta,\alpha)$, the line bundle $\mL:= \det\big(\Ext(B,\alpha)\big)$ has $B\GG_m$-weight $n:=\chi(\beta,\alpha)$. Let $\mY_{\alpha}$ be the total space of $\mL$ without the zero section. Following the reasoning in the proofs of \cite[Proposition 3.24, Proposition 2.29]{Joycehall}, the composition of 
$$
\begin{tikzcd}
\mY_{\alpha}\arrow[r]&\mM_{\alpha}\arrow[r]&\mM^{\rig}_{\alpha}
\end{tikzcd}
$$
is a $\T$-equivariant $B\ZZ_n$-bundle. Moreover, setting this composition equal to $r$, one obtains \eqref{eq:BZZnbasechange}.
\end{proof}
\subsection{Examples}
\label{app:examples}
This subsection considers the simplest situation when the action of $\T$ is trivial. Due to stackiness, this already introduces non-trivial behavior because the $\T$-action mixes with the local isotropy groups. This was explained to me by Anton Mellit \cite{Mellit} who came up with Example \ref{ex:equivariantpushforw} below. 
\begin{example}
\label{ex:trivialThom}
\begin{enumerate}
\item Consider the situation when $\T =\CC^*$ as it is easy to extrapolate to any finite-dimensional torus. We can choose the cofinite system of opens
$$
E_{n}\T = \CC^{n+1}\backslash\{0\}\subset \CC^{n+1}
$$
with $B_n\T = \PP^n$. 

  Let $\mX$ be a stack with trivial $\T$-action. Then $\mX\times \PP^n =[\mX/\T]_n \to B_n\T= \PP^n $ is just the projection to the second factor. The Künneth isomorphism for $\QQ\subset R$ induces
    $$
   \BB_*(\mX\times \PP^n\to \PP^n) \cong H_*(\mX)\otimes H^*(\PP^n)\cong H_*(\mX)\otimes R[u]/u^{n+1}
    $$
    where $u = c_1\big(\mO_{\PP^n}(1)\big)$.
Taking limits and using the notation \eqref{eq:Mtgr}, one obtains
\begin{equation}
\label{eq:triveqhom}
H^{\T}_*\big(\mX\big)\cong \varprojlim_n  H_*(\mX)\otimes R[u]/u^{n+1}=H_*(\mX)\llbracket u\rrbracket^{\gr}.
\end{equation}
For a more general torus, the same argument shows that $H^{\T}_*\big(\mX\big)\cong H_*(\mX)\llbracket \Ft \rrbracket^{\gr}.$
    \item Let $Q^{\bullet}$ be a dg-quiver (not necessarily CY4) as in Definition \ref{def:CY4quiver}. A $\T$-action on the stack $\mM_{\un{d}}$ often originates from rescaling morphisms at degree 0 edges. Thus there is a $\T$-equivariant map from $\mM_{\un{d}}$ to $\mM^0_{\un{d}}$ -- the moduli stack for the quiver $Q$ which forgets all edges of $Q^{\bullet}$ in degrees less than 0. Sometimes, this map induces an isomorphism of equivariant homologies. More importantly, it always induces a map of ($\T$-deformations of) vertex algebras after choosing the appropriate complexes $\Theta_{\mA}$ -- see \cite[Remark 3.8]{BoVir}.  Thus, one may reduce considerations to a usual quiver $Q$.
    
   Recall the group $\GL(\un{d})$ from \eqref{eq:AdSdGLd} and take its classifying stack $B\GL(\un{d})$. There is then a projection
    $$
    \mathscr{C}\textit{ont}: \begin{tikzcd}\mM_{\un{d}}\arrow[r]&B\GL(\un{d})\end{tikzcd}
    $$
contracting morphisms at all edges to 0. For the trivial $\T$-action on $B\GL(\un{d})$, this map is a projection from a $\T$-equivariant vector bundle over $B\GL(\un{d})$. By homotopy invariance, the pushforward $\mathscr{C}\textit{ont}_{\,*}$ induces an isomorphism of equivariant homologies. Using \eqref{eq:triveqhom}, shows that 
$$
H^{\T}_*\big(\mM_{\un{d}}\big):= H_*\big(B\GL(\un{d})\big)\llbracket \Ft\rrbracket^{\gr}\,. 
$$
\end{enumerate}
Example i) additionally confirms that the definition of equivariant homology for the local approach in §\ref{sec:equivariantVA} is compatible with Definition \ref{def:eqhom} when working with the fixed point stack. 
\end{example}
To understand, why it is paramount to consider power series of the form \eqref{eq:examplepowerseries}, I describe a particular case of equivariant pushforward acting on a point class.
\begin{example}
\label{ex:equivariantpushforw}
    Consider $B\GG_m$ with the trivial $\T=\CC^*$-action leading to
    $$
    H^{\T}_*(B\GG_m) = R[p]\llbracket u \rrbracket^{\gr}
    $$
    with $p^n$ defined in \eqref{eq:pnofBGm}. Let $\phi: \pt\to B\GG_m$ be the $\T$-equivariant morphism inducing 
    $$
\begin{tikzcd}
{[\phi/\T]}:B\T\arrow[r]&B\GG_m\times B\T
\end{tikzcd}
    $$
  given by the diagonal morphism after identifying $\T=\GG_m$. Its pullback 
  $$
  \begin{tikzcd}
  \phi^*:R[\tau,u] = H^*_{\T}(B\GG_m) \arrow[r]& R[u]= H^*_{\T}(\pt)
  \end{tikzcd}
  $$
  acts by mapping the generator $\tau$ to $u$.

Let $1\in R[u] = H^{\T}_{*}(\pt)$ be the point class. I would like to compute $\phi_*(1)\in R[p]\llbracket u \rrbracket^{\gr}$. In this simple case, it is determined by its action $H^*_{\T}(B\GG_m)\to H^*_{\T}(\pt)$ which is given by $\phi^*$. Using the above description of $\phi^*$, it follows by a simple computation that 
$$
\phi_*(1) = \sum_{n\geq 0}p^nu^n\in R[p]\llbracket u \rrbracket^{\gr}\,.
$$
\end{example}
In conclusion, the absence of such power
series from any equivariant theory would imply that equivariant pushforwards are in general not defined. Note that this example is not pathological, as it already appeared in \eqref{eq:ezTorigin} when computing $e^{zT}$ from the $\GG_m$-localization on the master space. It can be moreover generalized to study the classes $\big[\Hilb^n(\CC^4)\big]^{\vir}$ from Corollary \ref{cor:hilbWC}. I will pursue this direction in my future work. 
\bibliography{mybib.bib} 
\bibliographystyle{mybstfile}
\end{document}